\documentclass{article}
\usepackage[utf8]{inputenc}
\usepackage{amsmath} 
\usepackage{amssymb,mathrsfs, mathtools} 
\usepackage{bbm}
\usepackage{a4wide} 
\usepackage{graphicx}
\usepackage{upgreek}
\usepackage{physics}
\usepackage{color} 
\usepackage{enumerate}
\usepackage{todonotes}
\usepackage[normalem]{ulem}
\usepackage{graphicx}
\usepackage{epsfig}
\usepackage{xcolor}
\usepackage{comment}
\usepackage{url}
\usepackage{xargs}
\usepackage{caption}
\usepackage{subcaption}
\usepackage{array,booktabs}
\usepackage{amsthm}
\usepackage{hyperref}
\usepackage[nameinlink,capitalize]{cleveref}

\newtheorem{theorem}{Theorem}

\newtheorem{lemma}[theorem]{Lemma}
\Crefname{lemma}{Lemma}{Lemma}
\crefname{lemma}{lemma}{lemma}

\newtheorem{corollary}[theorem]{Corollary}
\Crefname{corollary}{Corollary}{Corollary}
\crefname{corollary}{corollary}{corollary}

\newtheorem{remark}{Remark}
\Crefname{remark}{Remark}{Remark}
\crefname{remark}{remark}{remark}

\newtheorem{proposition}[theorem]{Proposition}
\Crefname{proposition}{Proposition}{Proposition}
\crefname{proposition}{proposition}{proposition}

\newtheorem{definition}{Definition}
\Crefname{definition}{Definition}{Definition}
\crefname{definition}{definition}{definition}

\Crefname{example}{Example}{Example}
\crefname{example}{example}{example}

\DeclareMathAlphabet{\mathpzc}{OT1}{pzc}{m}{it}
\newcommand{\dps}{\displaystyle} 
\newcommand{\rme}{\mathrm{e}}
 
\newcommand{\cL}{\mathcal{L}}

\newcommand{\R}{\mathbb{R}}
\newcommand{\Id}{\mathrm{Id}} 
\newcommand{\Ran}{\mathrm{Ran}}

\renewcommand{\leq}{\leqslant}
\renewcommand{\le}{\leqslant}
\renewcommand{\geq}{\geqslant}
\renewcommand{\ge}{\geqslant}
\def\N{\mathbb{N}}
\def\R{\mathbb{R}}
\def\T{\mathbb{T}}
\def\Z{\mathbb{Z}}

\def\E{\mathbb{E}} 
\def\P{\mathbb{P}}

\newcommand{\A}{\mathcal{A}} 
\newcommand{\Diff}{\mathcal{D}} 
\newcommand{\Diffset}{\mathfrak{D}} 
\newcommand{\Df}{\mathscr{D}}
 
\newcommand{\diff}{D} 
\newcommand{\diffset}{\delta}

\renewcommandx{\norm}[2][2=]{\left\lvert{#1}\right\rvert_{#2}}
\newcommand{\F}{\mathrm{F}}
\newcommand{\normF}[1]{\left| #1 \right|_{\F}}

\renewcommand{\dim}{d}
\newcommand{\cLD}{\cL_\Diff}

\newcommand{\calA}{\mathcal{A}}
\newcommand{\calB}{\mathcal{B}}
\newcommand{\calM}{\mathcal{M}}
\newcommand{\calS}{\mathcal{S}}
\newcommand{\calW}{\mathcal{W}}
\newcommand{\calC}{\mathcal{C}}

\newcommand{\rmF}{\mathrm{F}}

\let\div\relax
\DeclareMathOperator{\div}{div}
\DeclareMathOperator{\supp}{supp}
\DeclareMathOperator{\Ker}{Ker}
\DeclareMathOperator{\argmin}{argmin}

\DeclareRobustCommand{\varlambda}{\text{\usefont{OML}{txmi}{m}{it}\symbol{"15}}}

\begin{document}

\title{Optimizing the diffusion coefficient of overdamped Langevin dynamics}
\author{T. Leli\`evre$^{1,2}$, G. A. Pavliotis$^3$, G. Robin$^{4}$, R. Santet$^{1,2}$ and G. Stoltz$^{1,2}$ \\
  \small 1: CERMICS, \'Ecole nationale des ponts et chaussées, IP Paris, Marne-la-Vallée, France \\
  \small 2: MATHERIALS project-team, Inria Paris, France \\
  \small 3: Department of Mathematics, Imperial College, London, United-Kingdom \\
  \small 4: CNRS \& Laboratoire de Mathématiques et Modélisation d'Évry, Évry, France
}

\maketitle


\begin{abstract}
  Overdamped Langevin dynamics are reversible stochastic differential equations which are commonly used to sample probability measures in high-dimensional spaces, such as the ones appearing in computational statistical physics and Bayesian inference. By varying the diffusion coefficient, there are in fact infinitely many overdamped Langevin dynamics which are reversible with respect to the target probability measure at hand. This suggests to optimize the diffusion coefficient in order to increase the convergence rate of the dynamics, as measured by the spectral gap of the generator associated with the stochastic differential equation. We analytically study this problem here, obtaining in particular necessary conditions on the optimal diffusion coefficient. We also derive an explicit expression of the optimal diffusion in some appropriate homogenized limit. Numerical results, both relying on discretizations of the spectral gap problem and Monte Carlo simulations of the stochastic dynamics, demonstrate the increased quality of the sampling arising from an appropriate choice of the diffusion coefficient.
\end{abstract}



\section{Introduction}
\label{sec:introduction}

Predicting properties of materials and macroscopic physical systems in the framework of statistical physics~\cite{Balian}, using for instance molecular dynamics~\cite{FrenkelSmit,Tuckerman,LM15,lelievre_stoltz_2016,AT17} and obtaining the distribution of parameter values in Bayesian inference~\cite{Robert07}, both require the sampling of probability measures in high-dimensional spaces. Methods of choice to sample such probability measures rely on stochastic dynamics, in particular Markov Chain Monte Carlo (MCMC) methods. The convergence of these methods may however be quite slow because the target measure is typically concentrated on a few high probability modes separated by low probability regions, or because the Hessian of the corresponding potential, see~\eqref{eq:mu}, is badly conditioned. Various algorithms have been proposed over the years to improve the sampling of such probability measures, for instance using importance sampling strategies, and/or interacting replicas; see for instance the extensive review~\cite{HLSVD22} in the context of molecular dynamics.

\paragraph{Sampling Boltzmann--Gibbs measures.}
We focus in this work on overdamped Langevin dynamics, a popular choice in molecular dynamics, computational statistics~\cite{RobertsTweedie1996}, and also machine learning (see for instance~\cite{CFG14}). This dynamics is ergodic (and, in fact, reversible) with respect to the Boltzmann--Gibbs distribution, which is the probability measure with density
\begin{equation}
  \label{eq:mu}
  \mu(q) = Z^{-1}\rme^{-V(q)}, \qquad Z = \int_\mathcal{Q} \rme^{-V} < +\infty,
\end{equation}
where~$V \in \calC^\infty(\mathcal{Q})$. In molecular dynamics, the measure~$\mu$ is the distribution sampled by the possible states of a system in the canonical ensemble. In that case, the system is at thermal equilibrium with a heat bath of fixed temperature~$T$, which is taken into account by considering~$\beta V$ instead of~$V$ in~\eqref{eq:mu} with~$\beta=1/(k_{\rm B} T)$, where $k_{\rm B}$ is the Boltzmann constant (in this work, we simply set~$\beta =1$). Because periodic boundary conditions are commonly used for molecular dynamics simulations, we restrict ourselves to the case when the configuration space~$\mathcal{Q}$ is the~$\dim$-dimensional torus~$\T^\dim$ (with~$\T = \mathbb{R}/\mathbb{Z}$ the one dimensional torus).

The overdamped Langevin dynamics is given by the following stochastic differential equation (SDE):
\begin{equation}\label{e:langevin}
  dq_t = -\nabla V(q_t) \, dt + \sqrt{2} \, dW_t,
\end{equation}
where~$(W_t)_{t \geq 0}$ is a standard~$\dim$-dimensional Brownian motion. This dynamics enjoys many nice properties. In particular, it is reversible with respect to the Gibbs measure~\eqref{eq:mu}, and, under appropriate assumptions on the potential energy function~$V$ (see for instance~\cite{bakry}), its marginal law in time converges exponentially fast to the target distribution~\eqref{eq:mu} in various norms and distances, \emph{e.g.}~total variation, relative entropy (Kullback--Leibler), $L^2(\mu)$, or weighted $L^{\infty}$ norms. For example, when the overdamped Langevin dynamics in~$\R^\dim$ is considered, it is sufficient for the potential~$V$ to be smooth and strongly convex at infinity to ensure the existence of Poincaré and logarithmic Sobolev inequalities for the Gibbs measure which, in turn, leads to an exponentially fast convergence to the target distribution in~$L^2(\mu)$ and in relative entropy, respectively. This class of potentials  includes many potentials of interest, for example multiwell potentials. Naturally, in the nonconvex case, the convergence rate typically scales very badly with respect to the dimension~$d$ (and the temperature). In this work, we will consider convergence in~$L^2(\mu)$, \emph{i.e.}~the $\chi^2$-divergence. This is the relevant setting \emph{e.g.}~for obtaining error estimates on trajectory averages, relying on the Central Limit Theorem.

It is important to note that the constant in the Poincar\'{e} and in the logarithmic Sobolev inequalities is a property of the Gibbs measure. In particular, the rate of convergence to the target measure for the overdamped Langevin dynamics~\eqref{e:langevin} depends only on the potential $V$, and convergence can be very slow when entropic or energetic barriers are present. It is therefore natural to consider alternatives to the overdamped Langevin dynamics that include ``hyperparameters'' that can be tuned to improve the rate of convergence to equilibrium and reduce the asymptotic variance.

There are in fact infinitely many overdamped Langevin dynamics which admit~\eqref{eq:mu} as invariant probability measure; see~\cite[Section~4.6]{Pavl2014}. A comprehensive list of stochastic dynamics that can be used in order to sample from a Gibbs measure can be found in~\cite{DuncanNuskenPavliotis2017}. For instance, dynamics of the form
\begin{equation*}
  dq_t = \left( -\nabla V(q_t) + \gamma(q_t)\right) dt + \sqrt{2} \, dW_t,
\end{equation*}
where the vector field~$\gamma:\T^\dim\to \R^\dim$ is smooth and such that
  $\div\left( \gamma \mu \right) = 0$
leave invariant the target probability measure with density~\eqref{eq:mu}; and are ergodic, but not reversible, with respect to this probability measure under certain assumptions on~$V,\gamma$. The hope is that the added drift term~$\gamma$ that renders the dynamics nonreversible, can accelerate convergence, measured in the $L^2(\mu)$ sense, and also reduce the asymptotic variance. This approach was extensively studied in the past decade~\cite{HHS93,HHS05,LelievreNierPavliotis2013,RS15,RBS16,DLP16,DPZ17}. For example, for Gaussian targets, the optimal nonreversible perturbation~$\gamma$ can be obtained in an algorithmic manner~\cite{LelievreNierPavliotis2013}. The presence of the divergence-drift term $\gamma$ can lead to computational difficulties, since the resulting dynamics can be stiff. Appropriate numerical methods need to be developed to address this issue~\cite{DLP16, DPZ17}. In fact, the discretized nonreversible dynamics can perform worse than the reversible dynamics if the stiffness issue is not addressed adequately~\cite{ZMS22, LS18}.

\paragraph{Optimizing reversible dynamics for better sampling.}
Another alternative, which is the focus of this work, is to stay within the class of reversible dynamics, but with a space-dependent diffusion matrix. For a given diffusion matrix~$\Diff \in \calC^1(\T^\dim,\mathcal{S}_\dim^{+})$ (with~$\mathcal{S}_\dim^{+}$ the set of real, symmetric, positive matrices of size~$\dim \times \dim$), the associated overdamped Langevin dynamics reads
\begin{equation}
  \label{eq:dynamics_mult}
  dq_t = \left(- \Diff(q_t)\nabla V(q_t) + \div\Diff(q_t) \right) dt + \sqrt{2}\,\Diff^{1/2}(q_t) \, dW_t,
\end{equation}
where~$\Diff^{1/2}$ is defined by functional calculus, and~$\div\, \Diff$ is the vector whose~$i$-th component is the divergence of the~$i$-th column of the matrix~$\Diff = [\Diff_1,\dots,\Diff_\dim]$, \emph{i.e.}~$\div\, \Diff = \left(\div\, \Diff_1,\dots,\div\, \Diff_\dim\right)^{\top}$. Modulating the diffusion to improve the sampling was explored to some extent both in the computational statistics literature~\cite{RobertsStramer} and in molecular dynamics~\cite{RBS16,ABDULLE2019349}; and also for simulated annealing~\cite{FQG97}.

The goal of this paper is to optimize the diffusion matrix in order to accelerate convergence. Intuitively, it seems relevant to accelerate the diffusion in regions of low probability under the target measure (equivalently, regions of high values of the energy function~$V$), in order to more efficiently and rapidly find another mode, and slow down the diffusion in regions of high probability since these are the zones where sampling should be favored. Clearly, in order for this optimization problem to make sense, appropriate constraints on the magnitude of the diffusion matrix need to be introduced. These constraints on the diffusion matrix and their effects on the solution of the associated constrained optimization problem are discussed in detail later on.

The $L^2(\mu)$ convergence rate of the overdamped Langevin diffusion~\eqref{eq:dynamics_mult} is given by the spectral gap of the infinitesimal generator. The latter operator acts on test functions~$\varphi$ as
\begin{equation}
  \label{eq:generator_cLD}
  \cLD \varphi = \left(- \Diff \nabla V + \div(\Diff)\right)^{\top} \nabla \varphi + \Diff : \nabla^2 \varphi.
\end{equation}
It is standard to show (see, \emph{e.g.}~\cite{LelievreNierPavliotis2013,bakry,lelievre_stoltz_2016}) that,  $\cLD$ admits a positive spectral gap $\Lambda(\Diff)$ if and only if,
for any initial condition~$f(0) = \mu_0/\mu\in L^2(\mu)$, and denoting by~$f(t)\mu$ the probability density function of the process~$q_t$ at time~$t$,
\begin{equation}
  \label{eq:cvg}
\forall t\ge 0,\qquad   \left\lVert f(t) - 1\right\rVert_{L^2(\mu)} \leq \mathrm{e}^{-\Lambda (\Diff)t}\left\lVert f(0) - 1\right\rVert_{L^2(\mu)},
\end{equation}
where~$\left\lVert\cdot\right\rVert_{L^2(\mu)}$ is the norm on~$L^2(\mu)$, see~\Cref{subsec:general-problem} for precise definitions. For a fixed target probability density~$\mu$ (\emph{i.e.}~a fixed potential energy function~$V$), the generator~$\cLD$, and thus the spectral gap~$\Lambda(\Diff)$, are parameterized by the diffusion matrix~$\Diff$.

The aim of this work is to compute, explicitly or numerically, the optimal diffusion matrix leading to the largest spectral gap, and thus to the largest convergence rate. For reversible Markov chains on discrete spaces, this question was explored in~\cite{Boyd_fastMC,Boyd_fastMC2}. A subtle issue in this endeavor is the normalization of the diffusion matrix~$\Diff$. Indeed, the convergence rate in~\eqref{eq:cvg} can trivially be increased by a factor~$\alpha\geq 1$ upon multiplying~$\Diff$ by~$\alpha$. In fact, it seems advantageous to let~$\alpha$ go to infinity. The catch is of course that one should compare dynamics which are defined on similar timescales. From a practical point of view, one way of setting a timescale is to consider discretizations of overdamped Langevin dynamics. Larger diffusion matrices require smaller values of the time step. The issue of normalization is discussed more thoroughly in~\Cref{subsec:normalization}.

\paragraph{Related works.}
The optimization of the diffusion matrix is related to a body of literature on accelerating the convergence of Langevin-like dynamics through particular choices for the diffusion matrix~$\Diff$. This is very much related to the optimization of the damping coefficient~\cite{chak2023optimal} or of the mass matrix in the Langevin Monte Carlo method, a.k.a.~as preconditioning, see~\cite{Bennett75,Girolami_RiemannMC,GW10,riemannianLangevinSimplex,LMW18,mirror_langevin,SanzSernaPreconditioning,chen2023gradient,MR4807088} as well as~\cite{Haario1999, haario_2001,MR2749836} for adaptive constructions of preconditioners. One example of such methods is the Riemannian manifold Langevin Monte Carlo method, introduced in~\cite{Girolami_RiemannMC}, and which reduces to~\eqref{eq:dynamics_mult} in the so-called overdamped limit. There, for strictly convex potentials, the diffusion matrix~$\Diff$ is chosen as the inverse of the Fisher--Rao metric tensor associated with the target measure, namely the inverse of the Hessian of~$V$. From a functional analytical point of view, this can be understood, for a problem set on~$\R^d$, from the Brascamb--Lieb inequality~\cite{Helffer1998,Saumard_Wellner_2014}:
\begin{equation*}
  \mbox{Var}_{\mu}(h) \leq  \int_{\R^d} \langle \nabla h, (\mbox{Hess} (V))^{-1} \nabla h \rangle \, \mu(q) \, d q,
\end{equation*}
for an arbitrary observable $h \in \calC^1$ with $\mbox{Var}_{\mu}(h) < +\infty$. Choosing~$\Diff=\left(\nabla^2 V\right)^{-1}$ in~\eqref{eq:generator_cLD}, one can check that the right hand side of the above inequality is the Dirichlet form associated with the generator~$\cL_{(\nabla^2 V)^{-1}}$, see~\eqref{eq:cLD} below. In particular, this implies that the spectral gap of this generator is 1, which leads to an exponential convergence with rate~1, independently of the strongly convex potential~$V$ and the dimension. We emphasize that this unconstrained optimization problem for the diffusion matrix is quite different from the one we consider in this paper: we optimize the spectral gap for the dynamics on the torus, with no convexity assumptions on~$V$ and with constraints on the magnitude of~$\mathcal{D}$.

Optimization problems on eigenvalues is a subject which has already been extensively discussed in the literature~\cite{Henrot}. Note in particular that, for Markov chains, it was already observed in~\cite{Boyd_fastMC} that maximizing the spectral gap can be formulated as a convex optimization problem. In some sense, our work can be seen as a generalization of~\cite{Boyd_fastMC} in an infinite dimensional setting.
Let us also mention the nice work~\cite{overton1992large} which discusses in detail the problems induced by the degeneracy of the eigenvalues at optimality, and specific algorithms to improve the convergence on this setting (in the present work, we will stick to elementary numerical techniques to solve the optimization problem).

At the time we completed this work, we became aware of the preprint~\cite{cui2024optimal} where the authors consider a very similar problem, that they solve using different techniques, namely Stein kernels and moment measures. More precisely, by using an appropriate transportation map, they transform the nonconvex log-probability to a convex one, for which the Brascamp--Lieb inequality yields the optimal diffusion, for a well-chosen normalization constraint. It is unclear how to apply their technique, which requires to work in the configuration space~$\R^\dim$, to our setting, namely the~$d$-dimensional torus. Notice that the results also seem to differ qualitatively, for example in terms of the order of degeneracy of the second eigenvalue at optimality. We intend to further explore the links between this recent preprint and ours in future works.

Let us also mention results obtained via the theory of partial differential equations for eigenvalue problems associated with operators of the form~$-\div(\Diff \nabla \cdot)$, see~\cite[Chapter~10]{Henrot}. The main difference with such results and ours is that we consider the Laplace operators on the weighted space~$L^2(\mu)$ as we are interested in sampling the nontrivial probability measure~$\mu$.

\paragraph{Main contributions.} The main contributions of this work are the following.
\begin{itemize}
  \item First, we formulate the optimization of the convergence rate of the Langevin dynamics~\eqref{eq:dynamics_mult} with respect to~$\Diff$ as a convex optimization program, for which the well-posedness is guaranteed by adapting standard results in the literature~\cite{Henrot}. We also discuss some important properties of the resulting optimal diffusion matrix, such as its formal characterization using the Euler--Lagrange equation as well as its positivity. The theoretical results are illustrated by numerical experiments in dimension $d=1$, using a method combining a finite element parametrization and an optimization algorithm to compute the optimal diffusion matrix in practice. All the methods and experiments are provided in open source Python and Julia codes, available at \url{https://github.com/rsantet/Optimal_Overdamped_Langevin_Diffusion_Python} and \url{https://github.com/rsantet/Optimal_Overdamped_Langevin_Diffusion_Julia} respectively.
  \item Our second main contribution is to study the behavior of the optimal diffusion matrix in the homogenized limit where the target measure is associated with a highly oscillating potential. In a one-dimensional setting, we show that the optimal diffusion matrix has an analytical expression, proportional to the inverse of the target density, which is in accordance with various previous heuristics (as in~\cite{RobertsStramer}). This analytical solution can be used as an initialization guess for the optimization algorithm mentioned above, or as a proxy for the optimal diffusion matrix which does not require costly convex optimization procedures.
  \item Our third main contribution is to propose a sampling algorithm based on a simple Random Walk Metropolis Hastings algorithm with a proposal variance depending on the current state, which takes advantage of the precomputed optimal diffusion matrix. We show that the associated Markov chain converges to the Langevin dynamics with optimal diffusion matrix in the diffusive limit. We also present some numerical experiments illustrating the behavior of the complete procedure on simple examples, highlighting in particular that the dynamics is less metastable when using the optimal diffusion matrix or its homogenized approximation. We show that the results obtained with the diffusion from the homogenization approximation are similar to the ones obtained with the optimal diffusion matrix, which makes the diffusion from the homogenization approximation a good choice in practice since it has an analytical expression.
\end{itemize}

\paragraph{Outline of the work.}
We formally define the maximization of the spectral gap in~\Cref{sec:general} and make precise the normalization of the diffusion matrix. We then study the well posedness of the maximization problem in~\Cref{sec:scalar}, and characterize the optimal diffusion matrix. We present in~\Cref{sec:numerical} numerical results for various one-dimensional probability distributions. As an alternative to a full-scale numerical simulation, or in order to start the optimization procedure with a relevant initial guess, we consider in~\Cref{sec:homog} the optimal diffusion matrix arising in the homogenization limit where the given probability density is periodically replicated with a decreasing spatial period. We finally demonstrate in~\Cref{sec:sampling} with Monte Carlo simulations that an optimized diffusion matrix is beneficial for the convergence of the dynamics. We conclude in~\Cref{sec:perspectives} by listing perspectives this work calls for. In the appendices~\ref{app:thm:well-posedness-1D},~\ref{app:smooth_max} and~\ref{app:homogenization} we present the proofs of our main results. Appendix~\ref{app:numerical} provides details about the methodology underlying the discrete optimization and the algorithms used in the illustrative numerical experiments, and~\Cref{app:differential_lambda} makes precise the computation of the differential of~$\Lambda$.

\section{Formulation of the optimization problem}
\label{sec:general}

We present in this section the problem of interest, namely the maximization of the spectral gap of the generator of the Langevin dynamics~\eqref{eq:dynamics_mult} with invariant measure~$\mu$ given by~\eqref{eq:mu}. We start by precisely formulating the target functional in~\Cref{subsec:general-problem}. Then, in~\Cref{subsec:normalization}, we discuss normalization constraints which are required for the well-posedness of the problem. Let us recall that we restrict ourselves to  diffusion processes on the~$d$-dimensional torus~$\T^\dim$.

\subsection{Maximizing the spectral gap of Langevin dynamics}
\label{subsec:general-problem}

As discussed in the introduction around~\eqref{eq:cvg}, the rate of convergence of the dynamics~\eqref{eq:dynamics_mult} towards the invariant measure~$\mu$ is governed by the spectral gap of the operator~$\cLD$ (see \emph{e.g.}~\cite[Chapter~4]{bakry}). In order to precisely define this quantity in our context, we work on the Hilbert space~$L^2(\mu)$. For a given diffusion matrix~$\Diff : \T^\dim \to \mathcal{S}_\dim^{+}$ (where we recall that~$\calS_d^{+}$ denote the set of real, symmetric,  semi-definite positive matrices of size~$d\times d$), a simple computation shows that the generator~$\cLD$, considered as an operator on~$L^2(\mu)$, can be written as
\begin{equation}
  \label{eq:cLD}
  \cLD = - \nabla^*\Diff\nabla =  -\sum_{i,j=1}^\dim \partial_{q_j}^* \Diff_{j,i} \partial_{q_i},
\end{equation}
where we denote by~$A^*$ the~$L^2(\mu)$-adjoint of a closed operator~$A$. In particular,~$\partial_{q_i}^* = -\partial_{q_i} + \partial_{q_i}V$. The quadratic form associated with~$\cLD$ is easily obtained from the expression~\eqref{eq:cLD} of the generator: formally, for any~$u : \T^\dim \to \R$,
\begin{equation*}
  \left\langle u, \cLD u \right\rangle_{L^2(\mu)} = - \int_{\T^\dim} \nabla u(q)^{\top} \Diff(q) \nabla u(q) \, \mu(q) \, dq.
\end{equation*}
The right-hand side of the above equality is nonpositive, and possibly infinite at this level of generality. It is finite for instance when~$\Diff \in L^\infty(\T^\dim,\mathcal{S}_\dim^{+})$ and~$u \in H^1(\mu)$, the subspace of~$L^2(\mu)$ composed of functions in~$L^2(\mu)$ whose (distributional) partial derivatives with respect to~$q_i$ also belong to~$L^2(\mu)$. We will indeed consider in the following an optimization problem for~$\Diff$ belonging to a subspace of $L^\infty(\T^\dim,\mathcal{S}_\dim^{+})$ (see~\eqref{eq:maximizer_in_thm} below).

Since~0 is an eigenvalue of~$\cLD$ whose associated eigenvectors are constant functions, we introduce the following subspace to define the spectral gap:
\begin{equation}
  \label{eq:H10}
  H^{1,0}(\mu) = \left\{ u \in H^1(\mu) \ \middle| \ \int_{\T^\dim} u(q) \, \mu(q) \, dq = 0\right\}.
\end{equation}
We also define~$H^{1,0}(\T^\dim)$ as the space~$H^{1,0}(\mathbf{1})$, namely the space~\eqref{eq:H10} with~$\mu$ replaced by the uniform probability distribution on~$\T^\dim$.
The spectral gap is finally defined as
\begin{equation}
  \label{eq:lambdaD-init}
  \Lambda(\Diff) = \inf_{u \in H^{1,0}(\mu) \setminus\{0\}} \frac{\dps \int_{\T^\dim} \nabla u (q)^{\top}\Diff(q) \nabla u (q) \, \mu(q)\,dq}{\dps \int_{\T^\dim} u(q)^2 \mu(q)\,dq}.
\end{equation}

For a bounded positive definite diffusion matrix~$\Diff$ satisfying~$\Diff(q) \geq c \Id_\dim$ almost everywhere for some~$c>0$ (the inequality being understood in the sense of symmetric matrices, a.k.a.~the Loewner order, \emph{i.e.}~$M_1 \le M_2$ if and only if~$\xi^\top M_1\xi\leqslant \xi^\top M_2\xi$ for all~$\xi\in\R^\dim$), standard arguments from calculus of variations show that the infimum is attained for an eigenvector~$u_\Diff$ associated with the second eigenvalue of~$-\cLD$ (ranked by increasing values), \emph{i.e.}~with the first nonzero eigenvalue of~$-\cLD$; and that~$\Lambda(\Diff)$ is in fact the first nonzero eigenvalue of~$-\cLD$ (see for instance~\cite[Chapter~1]{Henrot}).

\begin{remark}[Degenerate diffusion matrices]
  \label{rmk:SG_when_dim_cancels}
  If there exists~$q\in\T^\dim$ such that~$\Diff(q)$ is rank-deficient (\emph{i.e.}~$\xi^{\top}\Diff(q)\xi=0$ for some~$\xi\in\R^\dim \setminus \{0\}$), the domain of the quadratic form associated with~$\cLD$ can in fact be larger than~$H^{1,0}(\mu)$; see~\cite{Zhikov_1998}. In this case,~$\Lambda(\Diff)$ is in general larger than
 the spectral gap of the operator~$\cLD$ (and hence provides a too optimistic convergence rate) since the infimum in~\eqref{eq:lambdaD-init} is taken on a set of functions not sufficiently large. See~\cite[Remark~1.2.3]{Henrot} for a discussion on sufficient conditions for the operator~$\cLD$ to have a discrete spectrum even in the presence of degeneracies of the diffusion matrix (by which we mean that the diffusion matrix is positive everywhere, but not positive definite at some points).
\end{remark}


We propose to maximize the spectral gap of the overdamped Langevin dynamics with respect to the diffusion matrix~$\Diff$. This corresponds here to solving the following optimization problem:
\begin{equation}
  \label{eq:optim-continuous-init}
  \Lambda^{\star} = \sup_{\Diff\in\Diffset}\Lambda(\Diff),
\end{equation}
where~$\Diffset$ is a subset of{\label{minus:bounded_diffusion}} measurable functions from~$\T^\dim$ to~$\mathcal{S}_\dim^{+}$. The choices we consider in the sequel for~$\Diffset$ take into account that the diffusion needs to be normalized in some way since, for any~$\Diff:\T^\dim\rightarrow\mathcal{S}_\dim^{+}$,
\begin{equation}
  \label{eq:t_Lambda_scaling}
  \forall t\geq 0, \qquad \Lambda(t \Diff) = t\Lambda(\Diff),
\end{equation}
so that a maximization over a set containing a line~$\R \Diff$ with~$\Lambda(\Diff) > 0$ would lead to~$\Lambda^{\star} = +\infty$. Notice that, more generally, it holds~$\Lambda(\Diff_1)\leqslant\Lambda(\Diff_2)$ whenever~$\Diff_1\leqslant\Diff_2$ in the sense of matrix-valued operators, that is, for almost every~$q\in\T^\dim$, it holds~$\Diff_1(q)\leqslant\Diff_2(q)$ in the sense of symmetric matrices.

\begin{remark}[A numerical motivation for normalizing the diffusion matrix]
  \label{rmk:normalization_numerics}
  The need for a normalization can also be motivated by numerical considerations. For instance, an Euler--Maruyama discretization of~\eqref{eq:dynamics_mult} with a time step~$\Delta t>0$ reads
  \begin{equation}
    \label{eq:EM_ovd_multiplicative}
    q^{n+1} = q^n + \Delta t \left[- \Diff \nabla V + \div\,\Diff \right](q^n) + \sqrt{2\Delta t}\, \Diff^{1/2}(q^n) \, G^n,
  \end{equation}
  where~$(G^n)_{n \geq 0}$ is a family of independent and identically distributed standard~$d$-dimensional Gaussian random variables. It is apparent from the above formula that the behavior of the numerical scheme actually depends on~$\Delta t \, \Diff$ (and the divergence of this quantity), and not on~$\Diff$ or~$\Delta t$ alone, so that multiplying~$\Diff$ by a constant can be equivalently seen as multiplying the time step by the same constant. In this context, one should think of normalization criteria as ways to fix the amount of numerical error involved in the discretization, for instance by fixing the average rejection probability in a Metropolis--Hastings scheme whose proposal is provided by~\eqref{eq:EM_ovd_multiplicative}. The normalization we discuss in~\Cref{subsec:normalization} is conceptually simpler than such a criterion, and not related to a specific numerical scheme. We intend to further study this perspective in following works.
\end{remark}

\subsection{Normalization constraint}
\label{subsec:normalization}

We make precise in this section the normalization we consider for diffusion matrices, in view of the discussion after~\eqref{eq:optim-continuous-init}. In essence, we choose a (weighted) Lebesgue norm to be smaller than~1.

\paragraph{The $L^p_V$ normalization.} A quantity which naturally appears in the mathematical formulation of the problem is the product~$\Diff(q) \rme^{-V(q)}$, present for instance in the numerator of the Rayleigh ratio~\eqref{eq:lambdaD-init}. This quantity also measures the average size of the Brownian moves at stationarity (since~$\Diff(q)$ is the covariance of the noise at a given configuration~$q$, the latter configuration having a likelihood proportional to~$\rme^{-V(q)}$). Note also that the divergence of this product is the drift in~\eqref{eq:dynamics_mult} (up to the factor~$\rme^{-V(q)}$). This motivates normalizing the product~$\Diff \rme^{-V}$ rather than~$\Diff$ itself.

Let~$1\leqslant p\leqslant +\infty$. For~$\Diff:\T^\dim\to\calS_\dim^{+}$ measurable, we define
\begin{equation}
  \label{eq:norm_L^p_mu_matrices}
  \left\lVert \Diff \right\rVert_{L^p_V} := \left( \int_{\T^\dim} \normF{\Diff(q)}^p \, \rme^{- p V(q)} \, dq \right)^{1/p},
\end{equation}
when~$p<+\infty$ and~$\left\lVert \Diff \right\rVert_{L^\infty_V} := \left\| \normF{\Diff} \rme^{-V} \right\|_{L^\infty(\T^\dim)}$ when $p=+\infty$. Here,~$|\cdot|_{\rmF}$ is any monotonic norm with respect to the Loewner order on~$\calS_\dim^+$: it holds~$\normF{M_1} \leq \normF{M_2}$ when~$0 \le M_1 \leq M_2$ (see~\Cref{rem:normF} as to why we make this assumption). Examples of such norms include the Frobenius and spectral norms. We use in particular the Frobenius norm as a specific choice to obtain characterizations of optimal diffusions in~\Cref{subsec:euler-lagrange}, since the map~$\Diff\mapsto\normF{\Diff}^p$ is differentiable for~$1<p<+\infty$.
In order to normalize the optimal diffusion, throughout this work, we will therefore maximize the spectral gap $\Lambda(\Diff)$ over diffusion matrices $\Diff: \T^d \to \calS_d^+$ such that:
\begin{equation}\label{eq:normLp}
\|\Diff\|_{L^p_V} \le 1.
\end{equation}

Finally, we will only consider the~$L^{p}_V$ constraint~\eqref{eq:normLp} for~$1\leqslant p<+\infty$. It is natural since sampling relies on averages with respect to the equilibrium measure~$\mu$. Moreover, Euler--Lagrange equations for the optimal diffusions are less explicit when $p=+\infty$ as~$\Diff\mapsto\left\lVert\Diff\right\rVert_{L_V^{\infty}}$ is not a differentiable map (in contrast to~$\Diff\mapsto\left\lVert\Diff\right\rVert_{L^{p}_V}^{p}$). Besides, considering $p=+\infty$ may lead to constant and isotropic optimal diffusions for some norms (see~\Cref{rem:spectral_norm_l_infty_constraint}) whereas we would like to define a problem which leads to diffusion matrices which fit to the local characteristics of the underlying potential, possibly anisotropic (see related discussions in~\cite[Section~10.1]{Henrot}).

\begin{remark}[On the monotonicity assumption of~$|\cdot|_{\rmF}$]
  \label{rem:normF}
  The existence of optimal diffusions can actually be obtained without the monotonicity assumption on~$|\cdot|_{\rmF}$. In our work, this assumption is used in some theoretical arguments when studying the homogenization limits, see~\Cref{thm:commutation} and its proof in~\Cref{app:thm:commutation}.
\end{remark}

\begin{remark}[On the choice $p=+\infty$ in~\eqref{eq:normLp}]
\label{rem:spectral_norm_l_infty_constraint}
  Let~$|\cdot|_{\rmF}$ be the spectral norm, \emph{i.e.}~$\left\lvert \Diff(q)\right\rvert_{\rmF}$ yields the largest eigenvalue of~$\Diff(q)$ (which is real and nonnegative). Then a maximizer of~$\Diff\mapsto\Lambda(\Diff)$ over the set $\left\lbrace\Diff: \T^d \to \calS_d^+ \,\middle|\,\left\lVert\Diff\right\rVert_{L_V^{\infty}}\leqslant1\right\rbrace$ is simply given by~$\Diff^{\star}(q)=\rme^{V(q)}\Id_\dim$. Indeed, if~$\left\lVert\Diff\right\rVert_{L^{\infty}_V}\leqslant1$, then for almost every~$q\in\T^\dim$ and~$\xi\in\R^\dim$, it holds~$\xi^{\top}\Diff(q)\xi\leqslant\left\lvert\Diff(q)\right\rvert_{\rmF}\xi^\top\xi\leqslant\rme^{V(q)}\xi^\top\xi=\xi^\top\Diff^\star(q)\xi$ so that~$\Lambda(\Diff)\leqslant\Lambda(\Diff^{\star})$. The fact that~$\left\lVert\Diff^\star\right\rVert_{L_V^{\infty}}=1$ concludes the proof. 
\end{remark}

\begin{remark}[On the choice of the weight in~\eqref{eq:norm_L^p_mu_matrices}]
    \label{rem:weight_norm_Lp}
    Note that~\eqref{eq:norm_L^p_mu_matrices} is simply the usual~$L^{p}$ norm applied to the product~$\normF{\Diff}\rme^{-V}$. Weights other than~$\rme^{-pV}$ in the integrand of~\eqref{eq:norm_L^p_mu_matrices} could be considered, see~\Cref{rmk:choice_normalization_motivated} in~\Cref{app:thm:commutation}. One motivation for the weight in~\eqref{eq:norm_L^p_mu_matrices} is that it leads to expressions for optimal diffusion matrices in the homogenized limit which are independent of~$p$, see~\Cref{sec:optimization_homog_lim}.
\end{remark}

\paragraph{Definition of the set of diffusion matrices.}
To obtain some of the theoretical results, the normalization constraint~\eqref{eq:normLp} will not be sufficient, and we will require sometimes the diffusion matrices to be bounded from above, or from below by a positive constant (uniform positiveness). In particular,
\begin{itemize}
  \item to obtain the existence of a maximizer of~$\Lambda$, we use a  compactness argument (for the weak-* $L^\infty$ topology), which  requires to work on a set of diffusions bounded from above by a fixed constant (see Appendix~\ref{app:thm:well-posedness-1D});
  \item a set of diffusions bounded from below and from above by a fixed positive constant has to be considered when studying the homogenized limit of the optimal diffusion matrices (see~\Cref{sec:homog}).
\end{itemize}
This leads us to introduce, for~$p\in[1,+\infty]$ and $a,b\geqslant0$, the set
$$L^p_V(\T^\dim,\mathcal M_{a,b})=\left\{ \Diff: \T^\dim \to \calS_\dim^+\, \middle|\, \Diff(q) \rme^{-V(q)} \in \mathcal M_{a,b} \text{ for a.e. } q \in \T^d,\, \|\Diff\|_{L^p_V} < +\infty \right\},$$
where
\begin{equation}
  \label{eq:M_ab}
  \mathcal{M}_{a, b} = \left\{ M\in\mathcal{S}_\dim^+ \, \middle| \, \forall \xi \in \R^\dim, \ a|\xi|^2\leq \xi^{\top}M\xi \leq \frac{1}{b} |\xi|^2 \right\},
\end{equation}
with~$|\cdot|$ the Euclidean norm on~$\R^\dim$, and the convention that the right-hand side of the last inequality is~$+\infty$ when~$b=0$. 
The optimization problem that we consider in the following is thus:
\begin{equation}
  \label{eq:maximizer_in_thm}
  \textrm{Find }\Diff^{\star} \in \Diffset_p^{a,b}\textrm{ such that }\Lambda(\Diff^{\star}) = \sup_{\Diff\in\Diffset_p^{a,b}} \Lambda(\Diff),
\end{equation}
where
\begin{equation}
  \label{eq:constrained-set-matriciel}
  \Diffset_p^{a,b} = \left\{\Diff\in L^{\infty}_V(\T^{\dim},\calM_{a,b}) \,\middle|\, \|\Diff\|_{L^p_V}\leq 1 \right\}.
\end{equation}
Notice that we will therefore always consider bounded diffusion matrices (since, even if $b=0$, any~$\Diff \in \Diffset_p^{a,b}$ is such that~$\|\Diff\|_{L^\infty_V} < +\infty$). This is in particular useful since~$L^\infty\left(\T^d,\calS_\dim^+\right) \ni \Diff \mapsto \Lambda(\Diff)$ is upper-semicontinous.

In order for the set~\eqref{eq:constrained-set-matriciel} to be nonempty, there are constraints on $a$ and $b$.
\begin{lemma}
  \label{lem:D_p^ab_nonempty}
  Fix~$1\leq p < +\infty$, and~$a,b \geq 0$. The set $\Diffset_p^{a,b}$ is nonempty if and only if
  \begin{equation}
\label{eq:compatibility_conditions_continuous_level}
    ab \leq 1, \qquad a \leq a_\mathrm{max} := \frac{1}{\normF{\Id_\dim}}.
  \end{equation}
\end{lemma}

\begin{proof}
  First, $\mathcal M_{a,b}$ is nonempty if and only if $a \le b^{-1}$.
  Note that any element of~$\Diffset_p^{a,b}$ is lower bounded by~$a\rme^{V}\Id_\dim$, so that $\|\Diff\|_{L^p_V} \le 1$ requires $a \le a_{\rm max}$.
\end{proof}

\begin{remark}[On the importance of considering $a>0$ and/or $b>0$]
\label{rem:a_b}
As already stated above, the lower and upper bounds $a$ and $b$ are only introduced in order to obtain certain theoretical results. Our primary objective in this work is to study the maximization problem~\eqref{eq:maximizer_in_thm} with $a=b=0$. We will always try to set~$a=0$
 or $b=0$ whenever possible, or consider situations where these constraints are not active at the optimal solution. In particular, we will assume that only the~$L^{p}_V$ constraint~\eqref{eq:norm_L^p_mu_matrices} prevents the scaling phenomenon~\eqref{eq:t_Lambda_scaling} (and not the constraint due to the upper bound~$b$).
 
 For the numerical experiments, we mostly take $a=b=0$ and the discretized optimization problem is actually well posed in this setting, see~\Cref{app:numerical}. Moreover, it can be shown that even if $a=0$, the optimal diffusion for the discretized problem is positive definite (see~\Cref{prop:well-posed-pwc}).
 \end{remark}

\section{Theoretical analysis of the spectral gap optimization}
\label{sec:scalar}

We present in this section some theoretical results on the mathematical analysis of the optimization problem~\eqref{eq:maximizer_in_thm}. We start by well-posedness results in~\Cref{subsec:well-posedness}, and {\label{minus:we}then formally characterize in~\Cref{subsec:euler-lagrange} the optimal diffusion matrix.

\subsection{Well posedness of the optimization problem}
\label{subsec:well-posedness}

The aim of this section is to show that the maximization of the spectral gap~$\Lambda(\Diff)$ on the set~$\Diffset_p^{a,b}$ defined in~\eqref{eq:constrained-set-matriciel} is a well-posed problem. This is easily done by adapting arguments from~\cite{Henrot}. Recall that the potential~$V$ is assumed to be~$\calC^{\infty}(\T^\dim)$ in all this work. We first show that the map~$\Lambda$ is well-defined over the set of matrices we consider for the optimization problem.

\begin{proposition}
\label{prop:lambda_bounded}
  Fix~$1\leq p < +\infty$, and~$a,b \geq 0$ such that~\eqref{eq:compatibility_conditions_continuous_level} is satisfied. The function~$\Lambda$ is nonnegative and bounded on~$\Diffset_p^{a,b}$, uniformly with respect to~$a,b$.
\end{proposition}

\begin{proof}
  Since~$\Diff\in\Diffset_p^{a,b}$ is a.e.~positive semi-definite, it holds~$\Lambda(\Diff)\geq 0$. On the other hand, for a fixed function~$u \in \calC^{\infty}(\T^\dim)$ satisfying~$\int_{\T^\dim}u(q)\, \mu(q) \, dq = 0$ and~$\int_{\T^\dim}u(q)^2\mu(q)\,dq = 1$, the function~$|\nabla u(q)|^2$ is bounded on~$\T^\dim$. Since there exists a constant~$K\in\R_{+}$ such that~$0 \leq \xi^\top \Diff(q) \xi \leq K\normF{\Diff(q)} |\xi|^2$ (the constant~$K$ being 1 for both the Frobenius and spectral norms), it holds (recall that~$Z$ is defined in~\eqref{eq:mu})
  \begin{align*}
    \Lambda(\Diff) & \leq Z^{-1}K\left\lVert |\nabla u|^2\right\rVert_{L^\infty(\T^\dim)}\int_{\T^\dim}\normF{\Diff(q)}\, \rme^{-V(q)}\,dq\\
    & \leq Z^{-1} K\left\lVert|\nabla u|^2\right\rVert_{L^\infty(\T^\dim)}\left(\int_{\T^\dim}\normF{\Diff(q)}^p \rme^{- p V(q)}\, dq\right)^{1/p} <+\infty,
  \end{align*}
  using H\"older's inequality on~$\T^d$. This shows that there exists~$C>0$ such that, for any~$a,b \geq 0$ and any~$\Diff \in \Diffset_p^{a,b}$, it holds~$0 \le \Lambda(\Diff) \le C$.
 \end{proof}
 
A first useful property to prove the well posedness of the maximization problem~\eqref{eq:maximizer_in_thm} is the following classical result on the concavity of the functional~$\Lambda$, directly obtained from the fact that~$\Lambda(\Diff)$ is an infimum of linear functionals in~$\Diff$ (see for instance~\cite[Theorem~10.1.1]{Henrot}).

\begin{lemma}
  \label{lem:concave-mat}
  The functional~$\Diff \mapsto\Lambda(\Diff)$ is concave on the set of bounded, measurable functions~$\T^\dim \to \mathcal{S}_\dim^+$.
\end{lemma}

The following theorem, proved in~\Cref{app:thm:well-posedness-1D} states the well-posedness of the optimization problem we consider. Note that we assume that~$b>0$ for the maximizing sequence to be compact in $\Diffset_p^{a,b}$ for the weak-* $L^\infty(\T^\dim,\calS_\dim^+)$ topology. Recall that~$a_{\rm max}$ is defined in~\eqref{eq:compatibility_conditions_continuous_level}.

\begin{theorem}
  \label{thm:well-posedness-1D}
  Fix~$p \in [1,+\infty)$. For any~$a \in [0,a_\mathrm{max}]$ and~$b > 0$ such that~$ab \leq 1$, there exists a solution~$\Diff^{\star} \in \Diffset_p^{a,b}$ to~\eqref{eq:maximizer_in_thm}. In addition, for any open set~$\Omega \subset \T^\dim$, the optimal diffusion matrix~$\Diff^{\star}$ is not identically zero on $\Omega$.
\end{theorem}

The second property stated in Theorem 4 is of course only meaningful for $a=0$. It guarantees that~$\Diff^{\star}$ cannot be identically~0 on open sets, but a priori~$\Diff^{\star}(q)$ need not be positive definite, and could vanish on sets of measure~0. Let us recall that if~$\Diff^{\star}$ vanishes at some point, it may not be optimal in terms of the spectral gap of the operator, see~\Cref{rmk:SG_when_dim_cancels}. We  will further discuss the positive definiteness of~$\Diff^{\star}$ in~\Cref{subsec:euler-lagrange}, using the Euler--Lagrange equation. From a numerical perspective, we prove in~\Cref{subsubsec:disc_optim} that the optimal diffusion matrix is bounded from below by a positive constant when the minimization is performed over the intersection of~$\Diffset_p^{a,b}$ with a finite dimensional linear space. However, in some cases, this lower bound is numerically observed to decrease when the spatial discretization is refined, which seems to indicate that $\Diff^{\star}$ may indeed vanish (see {\em e.g.}~\Cref{fig:res_1} below). 

Besides, it is not clear that the maximizer~$\Diff^{\star}$ saturates the constraint~\eqref{eq:normLp}, because of the upper bound~$b$. In practice, it is always the case in our numerical experiments since we perform the optimization without any point-wise upper bound (we take $b=0$).

\subsection{Characterization of uniformly positive definite maximizers~$\Diff^{\star}$}
\label{subsec:euler-lagrange}

The objective of this section is to formally characterize the maximizers~$\Diff^{\star}$ using the Euler--Lagrange condition satisfied by critical points of the functional~$\Lambda$.
When the optimal diffusion is uniformly positive definite, we first show in~\Cref{subsec:euler_lagrange_nondegenerate} that~$\Lambda(\Diff^{\star})$ is degenerate (\emph{i.e.}~this eigenvalue has a multiplicity larger or equal to~2). Therefore, the mapping $\Diff \mapsto \Lambda(\Diff)$ is not necessarily differentiable at~$\Diff^{\star}$, and only partial information can be obtained by considering the superdifferential of~$\Lambda$ at the optimum. In order to obtain an Euler--Lagrange equation characterizing~$\Diff^{\star}$, we circumvent this difficulty in~\Cref{subsec:euler_lagrange_degenerate} by relying on a smooth approximation of the definition of~$\Lambda$. The theoretical results obtained in this section are partially formal and will be corroborated by numerical experiments in~\Cref{sec:numerical}.

\subsubsection{The eigenvalue~$\Lambda(\Diff^\star)$ is degenerate}
\label{subsec:euler_lagrange_nondegenerate}

Let us consider the solution~$\Diff^{\star}$ to the problem~\eqref{eq:maximizer_in_thm} for $p \in (1,+\infty)$,
\begin{equation*}
  a=0,
\end{equation*}
and $b>0$. Note that we exclude the case~$p=1$ as the map~$M\mapsto\left\lvert M\right\rvert_{\rmF}$ is not differentiable at any rank-deficient~$M \in \mathcal{M}_{0,b}$. By~\Cref{thm:well-posedness-1D}, there exists a matrix valued function~$\Diff^{\star}$ which is a maximizer of~$\Lambda$ over the constrained set~$\Diffset_p^{0,b}$. Let us recall that, necessarily, either the upper bound~\eqref{eq:M_ab} on~$b$ or the $L^p$ constraint~\eqref{eq:norm_L^p_mu_matrices} has to be saturated, otherwise one can increase the spectral gap by simply multiplying the diffusion by a constant. In the remainder of this section, we assume that 
\begin{equation}\label{eq:b_not_active}
  \exists b_+>b, \quad \Diff^{\star}(q) \, \rme^{-V(q)} \le \frac{1}{b_+} \,   \Id_\dim \textrm{ for a.e. }q\in\T^\dim,
\end{equation}
so that necessarily
\begin{equation}\label{eq:phip=1}
\Phi_p(\Diff^{\star})=1,
\end{equation}
where~$\Phi_p$ encodes the normalization constraint~\eqref{eq:norm_L^p_mu_matrices}:
\begin{equation}
  \label{eq:normalization_constraint}
  \Phi_p(\Diff) = \int_{\T^\dim} \normF{\Diff(q)}^p \, \rme^{- p V(q)} \,dq.
\end{equation}
In our numerical experiments, we always observe that~\eqref{eq:b_not_active} and~\eqref{eq:phip=1} are satisfied upon choosing $b$ sufficiently small. 

Moreover, we assume that
\begin{equation}\label{eq:unif_SDP}\exists c>0,\quad \Diff^{\star}(q)\geq c \, \Id_\dim \textrm{ for a.e. }q\in\T^\dim.
\end{equation}
As explained below,~\eqref{eq:unif_SDP} is in particular natural for potentials which are $1/k$-periodic with $k \in \N \setminus \{0\}$, and $k$ sufficiently large to approach the homogenized limit, for which the optimal diffusion is indeed uniformly positive definite; see \Cref{sec:homog}.

A standard argument based on the Euler--Lagrange equation satisfied by the minimizer of the Rayleigh quotient in~\eqref{eq:lambdaD-init} shows that there exists a function~$u_{\Diff^{\star}}\in H^{1,0}(\mu)$ such that $\|u_{\Diff^{\star}}\|_{L^2(\mu)}=1$ and $\Lambda(\Diff^{\star}) =
\int_{\T^\dim} \nabla u_{\Diff^{\star}} (q)^{\top}\Diff^{\star}(q) \nabla u_{\Diff^{\star}} (q) \, \mu(q)\,dq$. We show in this section that, when~\eqref{eq:unif_SDP} holds,~$\Lambda(\Diff^\star)$ is a degenerate eigenvalue of the operator~$-\mathcal{L}_{\Diff^\star}$ (under some additional regularity assumption in the case~$d=1$). The proof proceeds by contradiction and relies on the fact that, when~$\Lambda(\Diff^{\star})$ is non-degenerate and~\eqref{eq:unif_SDP} holds, one can write the Euler--Lagrange equation satisfied for~$\Diff^{\star}$. This leads to a characterization of the optimal diffusion matrix which contradicts~\eqref{eq:unif_SDP}. In order to state the result, we need to distinguish for technical reasons the cases~$\dim=1$ and~$\dim>1$.

\begin{proposition}
  \label{prop:EL_degenerate}
  Choose the Frobenius norm for the matrix norm~$|\cdot|_{\rmF}$, consider~$p \in (1,+\infty)$ and~$a=0$. Assume that the maximizer~$\Diff^\star$ solution to~\eqref{eq:maximizer_in_thm}  satisfies~\eqref{eq:b_not_active}. Assume moreover that~$\Diff^{\star} \in \calC^0(\T,\R_+)$ when~$\dim=1$. If~\eqref{eq:unif_SDP} holds, then the optimal spectral gap~$\Lambda(\Diff^{\star})$ is a degenerate eigenvalue of the diffusion operator~$-\mathcal{L}_{\Diff^\star}$.
\end{proposition}

\begin{proof}
  Suppose that~$\Diff^{\star}\geqslant c\Id_\dim$ for~$c>0$. We denote by~$u_{\Diff^{\star}}$ an eigenvector satisfying~$-\cL_{\Diff^{\star}}u_{\Diff^{\star}}= \Lambda(\Diff^{\star})u_{\Diff^{\star}}$.
  Observe first that, for any~$\delta \Diff \in L^{\infty}(\T^\dim,\mathcal{S}_\dim)$, where~$\mathcal{S}_\dim$ is the set of real, symmetric, matrices of size~$d \times d$ (note that~$\delta \Diff$ does not output a positive matrix a priori), the matrix valued function~$\Diff^{\star} + t\delta\Diff$ has values in the space of symmetric positive definite matrices when~$|t|$ is sufficiently small.

  The proof proceeds by contradiction. If the eigenvalue~$\Lambda(\Diff^{\star})$ is non-degenerate, the eigenvalue~$\Lambda(\Diff^{\star} + t\delta\Diff)$ remains non-degenerate and isolated for~$|t|$ sufficiently small, and it is possible to choose the eigenvector~$u_{\Diff^{\star} + t\delta\Diff}$ associated with this eigenvalue so that~$\|u_{\Diff^{\star} + t\delta\Diff}\|_{L^2(\mu)} = 1$ and the mappings~$t\mapsto u_{\Diff^{\star} + t\delta\Diff}$ and~$t\mapsto\Lambda(\Diff^{\star} + t\delta\Diff)$ are analytic in an open neighborhood of~$t=0$ (see {\em e.g.} Theorem~II.6.1 and the discussion in Section~VII.3.1 of~\cite{Kato}). The optimal diffusion matrix~$\Diff^{\star}$ is therefore characterized by the following Euler--Lagrange equation: there exists a Lagrange multiplier~$\gamma \in \R_{+}$ such that
  \begin{equation}
    \label{eq:EL}
    \forall \delta\Diff\in L^{\infty}(\T^\dim,\mathcal{S}_\dim),\qquad
    \frac{d}{dt}{\Lambda}(\Diff^\star + t\delta\Diff)\Big|_{t=0} + \gamma \frac{d}{dt}\Phi_p(\Diff^\star + t\delta\Diff)\Big|_{t=0} = 0.
  \end{equation}
  Note that there are no constraints in~\eqref{eq:EL} related to positivity or boundedness because of~\eqref{eq:b_not_active}-\eqref{eq:unif_SDP}.

  Relying on the choice of the Frobenius norm for~$|\cdot|_{\rmF}$ to compute the differential of~$\Phi_p$, \eqref{eq:EL} can be rewritten as (see~\Cref{app:differential_lambda} for details on the derivation of the differential of~$\Lambda$): for any~$\delta \Diff \in L^{\infty}(\T^\dim,\mathcal{S}_\dim)$,
  \begin{equation*}
    \int_{\T^\dim} \delta\Diff(q) : \left( \nabla u_{\Diff^{\star}}\otimes\nabla u_{\Diff^{\star}} \right) \mu(q)\,dq = p\gamma\int_{\T^\dim}\normF{\Diff^{\star}(q)}^{p-2}\Diff^{\star}(q):\delta\Diff(q) \, \rme^{- p V(q)}\,dq,
  \end{equation*}
  where~$:$ and~$\otimes$ are respectively the double contraction and outer product operators: for any~$M_1,M_2 \in \mathbb{R}^{\dim \times \dim}$ and~$\xi,\zeta \in \mathbb{R}^\dim$,
  \begin{equation*}
    M_1:M_2 = \mathrm{Tr}\left(M_1^\top M_2\right) = \sum_{i,j=1}^\dim \left[M_1\right]_{i,j} \left[M_2\right]_{i,j},
    \qquad
    \left[\xi \otimes \zeta\right]_{i,j} = \xi_i\zeta_j.
  \end{equation*}
 We deduce from the Euler--Lagrange condition~\eqref{eq:EL} that the optimal diffusion matrix~$\Diff^{\star}$ satisfies

  \begin{equation}
    \label{eq:EL-prop-1}
    \Diff^{\star}(q) = \gamma_p \normF{\Diff^{\star}(q)}^{2-p} \rme^{(p-1)V(q)} \nabla u_{\Diff^{\star}}(q)\otimes \nabla u_{\Diff^{\star}}(q),
  \end{equation}
  where the constant~$\gamma_p > 0$ is determined by the constraint~$\Phi_p(\Diff^{\star}) = 1$.

  When~$\dim \geq 2$, we immediately obtain a contradiction with~$\Diff^{\star}\geq c \Id_\dim$ since~\eqref{eq:EL-prop-1} implies that~$\Diff^{\star}(q)$ is rank-deficient (in fact, rank 1) for a.e.~$q\in\T^\dim$.

  When~$\dim=1$, \eqref{eq:EL-prop-1} simplifies (after some straightforward manipulations) as
  \begin{equation}
    \label{eq:EL-prop}
    \Diff^{\star}(q) = \widetilde{\gamma}_p \rme^{V(q)} \left| u'_{\Diff^{\star}}(q) \right|^{2/(p-1)},
  \end{equation}
  where the constant~$\widetilde{\gamma}_p > 0$ is determined by~$\Phi_p(\Diff^{\star}) = 1$. We now show that~$u'_{\Diff^{\star}}$ vanishes on~$\T$, thus contradicting that~$\Diff^{\star}(q)\geq c > 0$ for all~$q\in\T$ (the inequality is indeed valid for all~$q \in \T$, and not up to a set of Lebesgue measure~0 since~$\Diff^\star$ is assumed to be continuous). To do so, we start by noticing that the continuity of~$\Diff^\star$ implies that the eigenfunction~$u_{\Diff^{\star}}$ belongs to~$\calC^1(\T,\R)$ (see~\cite[Remark 1.2.11]{Henrot} and~\cite[Theorem 9.15]{gilbargEllipticPartialDifferential2001}), so that~$u_{\Diff^{\star}}'$ is continuous on~$\T$. Since~$u_{\Diff^{\star}}$ is periodic on~$\T$, its derivative~$u_{\Diff^{\star}}'$ vanishes at least once on~$\T$. In view of~\eqref{eq:EL-prop}, the diffusion~$\Diff^\star$ then also vanishes at this point. The contradiction finally allows to obtain the desired result for~$\dim=1$ as well.
\end{proof}

\subsubsection{Euler--Lagrange equation for~$\Diff^\star$}
\label{subsec:euler_lagrange_degenerate}

In view of~\Cref{prop:EL_degenerate}, the Euler--Lagrange equation cannot be used to obtain a characterization of the maximizer, at least when the maximizer is uniformly positive definite. In fact, we observed in our numerical experiments that the spectral gap is always degenerate, even when the maximizer is rank-deficient, which suggests that~$\Lambda$ is never differentiable at~$\Diff^{\star}$. Partial information can however still be obtained by considering the superdifferential of~$\Lambda$ at~$\Diff^{\star}$: adapting the arguments from~\cite[Proposition~4]{cui2024optimal} (which rely on~\cite[Theorem~2.4.18]{zalinescu_2002}), the following characterization of~$\Diff^{\star}$ can be written when the eigenspace associated with the spectal gap is~$M$-dimensional:
\begin{equation}
  \label{eq:maximizer_superdifferential}
  \Diff^{\star}=\gamma_p\normF{\Diff^{\star}}^{2-p}\rme^{(p-1)V}\left(\sum_{i=1}^{M}\upsilon_i\nabla u_{i,\Diff^{\star}}\otimes\nabla u_{i,\Diff^{\star}}\right),\quad\forall 1\leqslant i\leqslant M,\quad\upsilon_i\geqslant 0,\quad \sum_{i=1}^{M}\upsilon_i=1,
\end{equation}
where~$\gamma_p>0$, and~$\left(u_{i,\Diff^{\star}}\right)_{1\leqslant i\leqslant M}$ are normalized orthogonal eigenvectors associated with the eigenvalue~$\Lambda(\Diff^{\star})$.  However, this characterization does not provide the precise convex combination that recombines into~$\Diff^{\star}$. In this section, we show how to estimate the coefficients~$(\upsilon_i)_{1\leqslant i\leqslant M}$ by introducing a smooth-min approximation of the objective function. More precisely, these coefficients can be estimated by a limiting procedure on these approximated optimization problems for which the Euler--Lagrange equations yield an explicit characterization of the optimum.
The so-obtained characterization is corroborated by numerical experiments, see~\Cref{sec:numerical}. The positivity of~$\Diff^{\star}$ is discussed at the end of this section.

\paragraph{Smooth-min approximation of the objective function $\Lambda$.}
Our approach relies on the fact that, for a positive~$\alpha$, the map~$\mathsf{m}_\alpha:\R^n \to \R$ defined by
\begin{equation*}
  \label{eq:m_alpha}
  \mathsf{m}_\alpha(x_1,\dots,x_n)=\sum_{i=1}^n x_i\rme^{-\alpha x_i} \, \Big/ \,\sum_{i=1}^n\rme^{-\alpha x_i}
\end{equation*}
converges to the~$\min$ function when~$\alpha\to+\infty$, \emph{i.e.}
\begin{equation}
  \label{eq:conv_m_alpha}
  \forall (x_1,\dots,x_n) \in \R^n,
  \qquad
  \mathsf{m}_{\alpha}(x_1,\dots,x_n)\xrightarrow[\alpha\to+\infty]{}\min\limits_{1\leqslant i\leqslant n}x_i.
\end{equation}
We adapt this remark to obtain a smooth approximation of the spectral gap~$\Lambda(\Diff)$, from which one can compute the gradient, and exchange the derivative and the limit~$\alpha\to+\infty$ in order to obtain a formal characterization of the optimal diffusion matrix similar to~\eqref{eq:EL-prop-1}. The methodology to obtain this formal characterization is detailed in~\Cref{app:smooth_max}, and we only briefly present the results here.

Applying a smooth-min approximation to the spectral gap $\Lambda$ defined by~\eqref{eq:lambdaD-init} leads us to introduce the map $f_\alpha$ defined by:
for any~$\alpha>0$, for any $\Diff$, 
\begin{equation}
  \label{eq:f_alpha}
  f_{\alpha}(\Diff)=\frac{\Tr\left(-\cLD\,\rme^{\alpha\cLD}\right)}{\Tr(\rme^{\alpha\cLD})-1},
\end{equation}
where the traces are taken on~$L^{2}(\mu)$. It can be checked that~$f_{\alpha}(\Diff)$ converges to~$\Lambda(\Diff)$ when~$\alpha\to+\infty$, see~\eqref{eq:limit_f_alpha_to_infty} in~\Cref{app:smooth_max}. This motivates considering the following approximation of the problem~\eqref{eq:maximizer_in_thm}. Fix~$p \in [1,+\infty)$. For any~$a \in [0,a_\mathrm{max}]$ and~$b > 0$ such that~$ab \leq 1$,
\begin{equation}\label{eq:maximizer_in_thm_alpha}
  \textrm{Find }\Diff^{\star,\alpha} \in \Diffset_p^{a,b}\textrm{ such that }f_\alpha(\Diff^{\star,\alpha}) = \sup_{\Diff\in\Diffset_p^{a,b}} f_\alpha(\Diff).
\end{equation}
We expect that $\Diff^{\star,\infty}$ (which is the limit of $\Diff^{\star,\alpha}$ as~$\alpha\to+\infty$ whenever it exists) is an optimal diffusion $\Diff^{\star}$ solution to~\eqref{eq:maximizer_in_thm}. The interest of the function $f_\alpha$ compared to $\Lambda$ is that it is differentiable so that, under assumptions similar to~\eqref{eq:b_not_active}-\eqref{eq:unif_SDP}, one obtains a characterization for~$\Diff^{\star,\alpha}$ similar to~\eqref{eq:EL-prop-1}, which writes
\begin{equation}
  \label{eq:diff_alpha_F}
  \Diff^{\star,\alpha}(q)=\gamma_{\alpha}\left\lvert \Diff^{\star,\alpha}(q)\right\rvert_{\mathrm{F}}^{2-p}\rme^{(p-1)V(q)}\sum_{k\geqslant2}\left[
    \frac{\mathcal{G}_{\alpha}(1-\alpha\varlambda_{k,\alpha})+\alpha \mathcal{H}_{\alpha}}{\mathcal{G}_{\alpha}^{2}}\rme^{-\alpha\varlambda_{k,\alpha}}
    \right]\nabla e_{k,\alpha}(q)\otimes\nabla e_{k,\alpha}(q),
\end{equation}
with~$\gamma_{\alpha}\in\R_{+}$ such that~$\Phi_p(\Diff^{\star,\alpha})=1$,~$(\varlambda_{k,\alpha})_{k\geqslant1}$ and~$(e_{k,\alpha})_{k\geqslant1}$ are the eigenelements of the operator~$-\cL_{\Diff^{\star,\alpha}}$ (with eigenvalues ordered and counted with their multiplicities), and the quantities~$\mathcal{G}_\alpha$ and~$\mathcal{H}_\alpha$ are defined by
\begin{equation}
  \label{eq:G_alpha_H_alpha_maximizer}
  \mathcal{G}_\alpha=\sum\limits_{j\geqslant2}\rme^{-\alpha\varlambda_{j,\alpha}},\qquad
  \mathcal{H}_\alpha=\sum\limits_{j\geqslant2}\varlambda_{j,\alpha}\rme^{-\alpha\varlambda_{j,\alpha}}.
\end{equation}
As a consequence, one obtains in the case~$\dim=1$,
\begin{equation}
  \label{eq:diff_alpha_1d}
  \Diff^{\star,\alpha}(q)=\widetilde{\gamma}_{\alpha}\rme^{V(q)}\left(
  \sum_{k\geqslant2}\left[
    \frac{\mathcal{G}_{\alpha}(1-\alpha\varlambda_{k,\alpha})+\alpha \mathcal{H}_{\alpha}}{\mathcal{G}_{\alpha}^{2}}\rme^{-\alpha\varlambda_{k,\alpha}}
    \right]\left\lvert e_{k,\alpha}'(q)\right\rvert^2\right)^{1/(p-1)}
\end{equation}
with~$\widetilde{\gamma}_{\alpha}>0$. We observed numerically on various test cases that this formula indeed yields a good approximation of the optimal diffusion, even for rather small values
of~$\alpha$, see~\Cref{fig:res_3_b} in~\Cref{app:smooth_max}.

\paragraph{Formal characterization in the limit~$\alpha\to+\infty$.}

For notational simplicity, we consider the case when the limiting problem obtained for~$\alpha\to+\infty$ has a leading nonzero eigenvalue with a degeneracy of at most 2. This is in line with our numerical experiments in dimension one. It is straightforward to adapt the argument to degeneracies of arbitrary order. We make the following assumptions:
\begin{itemize}
  \item all the eigenelements converge when~$\alpha\to+\infty$: for any~$k\geqslant 1$, there exists~$\varlambda_{k,\infty}\in \R_{+}$ such that~$\varlambda_{k,\alpha}\xrightarrow[\alpha\to+\infty]{}\varlambda_{k,\infty}$, and there exists~$e_{k,\infty}\in H^{1,0}(\mu)$ such that~$e_{k,\alpha}\xrightarrow[\alpha\to+\infty]{}e_{k,\infty}$ in~$H^{1,0}(\mu)$;
  \item if~$k\geqslant 4$, then~$\varlambda_{k,\infty}>\varlambda_{3,\infty}$ (the eigenspace associated to the eigenvalue~$\varlambda_{2,\infty}$ is either 1 or 2 dimensional);
  \item when~$\varlambda_{2,\infty}=\varlambda_{3,\infty}$, we further assume that the following limit is well defined:
  $$\lim_{\alpha \to \infty} \alpha(\varlambda_{3,\alpha}-\varlambda_{2,\alpha})= \eta. $$
\end{itemize}
One then obtains the following characterization of an optimal diffusion matrix in the limit~$\alpha\to+\infty$, and in dimension 1:
\begin{equation}
\label{eq:Diff_characterization_formal_alpha_to_infty_degenerate_case_1d_case}
  \Diff^{\star,\infty}(q)=\widetilde{\gamma}_{\infty}\rme^{V(q)}\left(
  \left\lvert e_{2,\infty}'(q)\right\rvert^{2}+
  \frac{\rme^{-\eta}(1+\rme^{-\eta}-\eta)}{1+\rme^{-\eta}+\eta\rme^{-\eta}}\left\lvert e_{3,\infty}'(q)\right\rvert^{2}
  \right)^{1/(p-1)},
\end{equation}
with~$\widetilde{\gamma}_{\infty}>0$. Formula~\eqref{eq:Diff_characterization_formal_alpha_to_infty_degenerate_case_1d_case} is a special instance of formula~\eqref{eq:maximizer_superdifferential}, where the convex combination is uniquely determined by the scalar~$\eta$, which can be estimated by solving~\eqref{eq:maximizer_in_thm_alpha} for one or many values of~$\alpha$. The correctness of formula~\eqref{eq:Diff_characterization_formal_alpha_to_infty_degenerate_case_1d_case} is demonstrated on two numerical examples in~\Cref{sec:numerical}, see~\Cref{fig:res_2_right,fig:res_3_e} below.

In the case~$d=1$, it is clear from~\eqref{eq:Diff_characterization_formal_alpha_to_infty_degenerate_case_1d_case} that if the value~$\eta$ satisfies~$1+\rme^{-\eta}-\eta=0$, then $\Diff^{\star,\infty}$ vanishes where~$e'_{2,\infty}$ vanishes. In that case, the uniform positivity assumption (similar to~\eqref{eq:unif_SDP}, see~\eqref{eq:unif_SDP_alpha}), which is needed to derive the characterization~\eqref{eq:diff_alpha_F}, is not satisfied. However, we still believe that the formula~\eqref{eq:Diff_characterization_formal_alpha_to_infty_degenerate_case_1d_case} holds, from two numerical experiments in~\Cref{sec:numerical}, see~\Cref{fig:eta_star_2,fig:eta_star_1} below. The particular value of~$\eta$ that satisfies~$1+\rme^{-\eta}-\eta=0$ is
\begin{equation}
  \label{eq:eta_star}
  \eta^\star=1+W\left(\frac{1}{\rme}\right) \approx 1.27846,
\end{equation}
where~$W$ is the Lambert~$W$ function. In fact, a necessary condition (actually sufficient if~$e_{2, \infty}'$ and~$e_{3,\infty}'$ do not vanish at the same points) for~\eqref{eq:Diff_characterization_formal_alpha_to_infty_degenerate_case_1d_case}  to define a positive definite diffusion coefficient is~$\eta\in[0,\eta^\star)$. In practice, we only observed numerical values of~$\eta \in [0,\eta^\star]$, which is equivalent to the non-negativeness of the multiplicative coefficient of $\left\lvert e_{3,\infty}'(q)\right\rvert^{2}$ in~\eqref{eq:Diff_characterization_formal_alpha_to_infty_degenerate_case_1d_case}. This is indeed consistent with the fact that we should obtain convex combinations, see~\eqref{eq:maximizer_superdifferential}.

\section{Numerical results}
\label{sec:numerical}

In this section, we illustrate the theoreticals results obtained in~\Cref{sec:scalar} by solving the optimization problem for several one-dimensional examples on the torus, with the normalization constraint~\eqref{eq:normLp} with~$p=2$.

We describe in~\Cref{app:numerical} the general methodology that was used to solve~\eqref{eq:maximizer_in_thm} in practice.  We also give theoretical results on the existence of a maximizer for the discretized problem, and show that is it necessarily positive, see~\Cref{prop:well-posed-disc}. The existence of a maximizer holds even when~$b=0$, so that we use this value for all numerical experiments. We also include make precise how the problem~\eqref{eq:maximizer_in_thm_alpha} was discretized and solved in practice. We recall that the code used to obtain these numerical results is available on GitHub, see~\Cref{sec:introduction}.

Following the methodology we suggest, the maximizers obtained by solving the discrete approximations of~\eqref{eq:maximizer_in_thm} and~\eqref{eq:maximizer_in_thm_alpha} can be seen either as piecewise functions, defined on the partition~$(K_n)_{1\leqslant n\leqslant N}=\left([(n-1)/N,n/N)\right)_{1\leqslant n\leqslant N}$ of~$\T$, or simply as~$N$-dimensional vectors, which we name diffusion coefficients and denote by~$D$. We choose~$N=1000$. In each figure, the target distribution is~$\mu$ defined by~\eqref{eq:mu}; the optimal diffusion coefficient~$\diff^{\star}$ (respectively~$\diff^{\star,\alpha}$) is obtained numerically by solving the discretized version of the optimization problem~\eqref{eq:maximizer_in_thm} (respectively~\eqref{eq:maximizer_in_thm_alpha}); the constant diffusion coefficient is~$\diff_{\rm cst}=\gamma\mathbbm{1}_N$ with~$\gamma>0$ chosen to saturate the (discrete version of the) constraint~\eqref{eq:normLp}; and the homogenized diffusion coefficient corresponds to the finite-dimensional representation of the proxy~$\Diff_{\mathrm{hom}}^\star$ for the optimal diffusion coefficient obtained in a periodic homogenization limit, whose expression is given by~\eqref{eq:homog_diff_opt} (see~\Cref{sec:homog} for the presentation of the homogenized limit):
\begin{equation}
  \label{eq:approximate_d_hom_star}
  \forall n\in\left\lbrace1,\dots,N\right\rbrace,\qquad
  \diff^{\star}_{\mathrm{hom},n}=\Diff^{\star}_{\rm hom}\left(\frac{n-1}{N}\right)=\rme^{V\left(\frac{n-1}{N}\right)}.
\end{equation}

In order to numerically check that 
the smooth-min optimization procedure yields a good approximation of an optimal diffusion coefficient, we introduce the numerical approximation~$D^{\star,\infty}$  of~\eqref{eq:Diff_characterization_formal_alpha_to_infty_degenerate_case_1d_case}. This vector is defined componentwise as follows: for~$n\in\left\lbrace1,\dots,N\right\rbrace$,
\begin{equation}
  \label{eq:approximation_d_star_infty}
  D^{\star,\infty}_{n}=\widetilde{\gamma}_{\infty}\rme^{V((n-1)/N)}
  \left(
  \left\lvert U_{2,n}-U_{2,n-1}\right\rvert^{2}+
  \frac{\rme^{-\upeta}(1+\rme^{-\upeta}-\upeta)}{1+\rme^{-\upeta}+\upeta\rme^{-\upeta}}\left\lvert U_{3,n}-U_{3,n-1}\right\rvert^{2}
  \right)^{1/(p-1)},
\end{equation}
where~$U_2$ (respectively~$U_3$) is a normalized eigenvector of the discretized operator with the optimal diffusion $D^{\star}$, associated with the eigenvalues~$\sigma_2(D^{\star})$ (respectively~$\sigma_3(D^{\star})$)\footnote{Even though we numerically obtain very close values for~$\sigma_2(D^{\star})$ and~$\sigma_3(D^{\star})$, the normalized eigenvectors $U_2$ and $U_3$ are in practice always uniquely defined (up to an irrelevant sign) since we always have $\sigma_2(D^{\star})<\sigma_3(D^{\star})$ (a strict equality is never observed numerically).}. Here,~$\sigma_i(D)$ denotes the~$i$-th eigenvalue of the discrete approximation of the operator~$-\cLD$. The constant~$\widetilde{\gamma}_\infty$ is such that the (discrete version of the) constraint~\eqref{eq:normLp} is saturated. Note that periodic boundary conditions are imposed on the eigenvalue problems so that {\em e.g.}~$U_{2,0}=U_{2,N}$ and~$U_{3,0}=U_{3,N}$. The scalar~$\upeta$ is an approximation of~$\eta$ (see~\eqref{eq:Diff_characterization_formal_alpha_to_infty_degenerate_case_1d_case}), which is obtained by approximating numerically the limit~$\upeta=\lim\limits_{\alpha\to+\infty}\alpha(\sigma_{3}(D^{\star,\alpha})-\sigma_{2}(D^{\star,\alpha}))$.

We present the numerical results by first considering cases where~$a=0$ in~\Cref{sec:results_a_0}, and then situations where the lower bound~$a$ is positive (see~\Cref{sec:results_a>0}). The results for~$a=0$ are further distinguished depending on the value of the approximation of~$\eta$ among the cases~$\eta = 0$, $\eta \in (0,\eta^\star)$ and~$\eta = \eta^\star$. We finally draw conjectures from our numerical observations in~\Cref{subsubsec:conj}.

\subsection{Results for~$a=0$}
\label{sec:results_a_0}

\paragraph{A case where~$\eta=0$.}
We plot in~\Cref{fig:res_2_left} various diffusion coefficients as well as the target distribution for a four-well potential. We first observe that the optimal diffusion coefficient~$\diff^{\star}$ takes larger values in regions where the target distribution is small. The analytical proxy~$\diff_{\mathrm{hom}}^\star$ has the same general shape as~$\diff^{\star}$. In fact, the two coefficients are almost superimposed. This observation is formalized in~\Cref{subsec:optim-homog} below.

The numerical values of the spectral gaps obtained for~$\diff^{\star}$,~$\diff_{\mathrm{hom}}^\star$ and~$\diff_{\rm cst}$ are respectively~30.24, 30.19 and 14.70. This simple example thus demonstrates the improvement obtained by optimizing the diffusion coefficient in terms of the spectral gap of the operator, and thus the enhanced convergence rate to equilibrium (further demonstrated by Monte Carlo simulations in~\Cref{sec:sampling}). It also confirms the relevance of the homogenized proxy~$\Diff_{\mathrm{hom}}^\star$ derived in~\Cref{sec:homog}, which can be computed explicitly and yields a spectral gap comparable to the one obtained through the optimization procedure.

The optimization procedure leads to an optimal diffusion coefficient which is uniformly lower bounded on the torus. In particular, we checked that the minimum of the optimal diffusion coefficient does not decrease as the mesh gets finer. A different behavior is observed in other numerical examples, see~\Cref{fig:res_1_left,fig:res_cos_4a} below.

The first three nonzero eigenvalues obtained for the optimal diffusion coefficient are~$\sigma_2(D^{\star})=30.23813406$, $\sigma_3(D^{\star})=30.23813407$, and~$\sigma_4(D^{\star})=86.06$. This means that the second eigenvalue can be considered degenerate, which is in accordance with~\Cref{prop:EL_degenerate} (we expect that both eigenvalues are equal in the limit~$N\to+\infty$).~\Cref{fig:res_2_right} displays the optimal diffusion coefficient~$D^{\star}$ and the numerical approximation~$D^{\star,\infty}$ of~\eqref{eq:Diff_characterization_formal_alpha_to_infty_degenerate_case_1d_case} given by~\eqref{eq:approximation_d_star_infty}. We have observed that the two eigenvalues~$\sigma_{2}(D^{\star,\alpha})$ and~$\sigma_{3}(D^{\star,\alpha})$ are essentially the same (up to numerical precision) even for very small values of~$\alpha$, the values of~$\alpha(\sigma_{3}(D^{\star,\alpha})-\sigma_{2}(D^{\star,\alpha}))$ when~$\alpha\in[1,7]$ being at most of order~$2.5\times10^{-5}$. We therefore consider that~$\upeta=0$: in view of~\eqref{eq:Diff_characterization_formal_alpha_to_infty_degenerate_case_1d_case}, the two eigenvectors equally contribute in the convex combination. With this choice, the two coefficients~$D^{\star}$ and~$D^{\star,\infty}$ are superimposed, as predicted by~\eqref{eq:Diff_characterization_formal_alpha_to_infty_degenerate_case_1d_case}. This confirms the relevance of the smooth-min approach used in~\Cref{subsec:euler_lagrange_degenerate}.

\begin{figure}
  \centering
  \begin{subfigure}[t]{0.49\linewidth}
    \includegraphics[width=\textwidth]{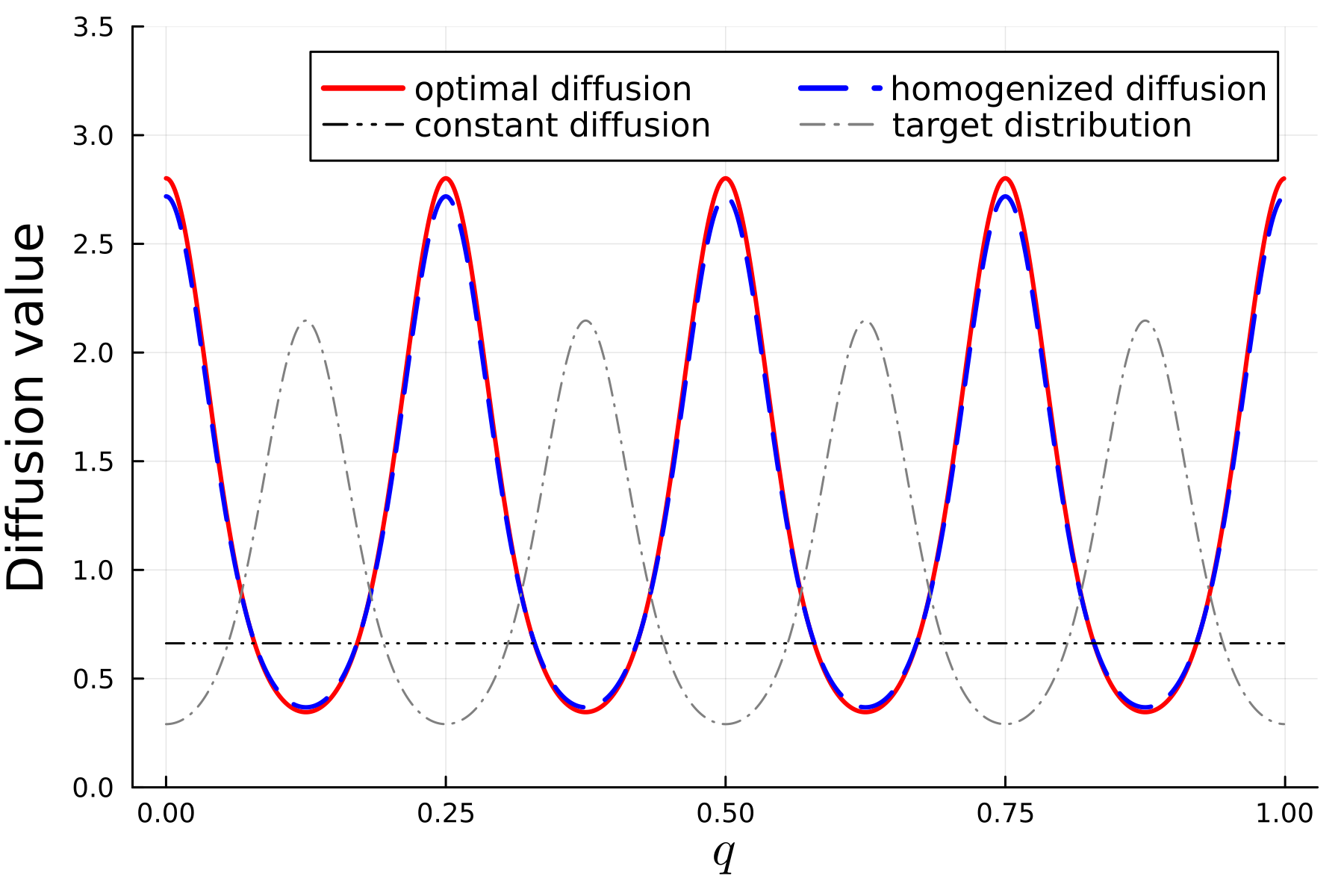}
    \caption{\label{fig:res_2_left}Various diffusion coefficients and target distribution~$\mu$. The solid curve in red is the optimal diffusion coefficient~$D^{\star}$.}
  \end{subfigure}
  \hfill
  \begin{subfigure}[t]{0.49\linewidth}
    \includegraphics[width=\textwidth]{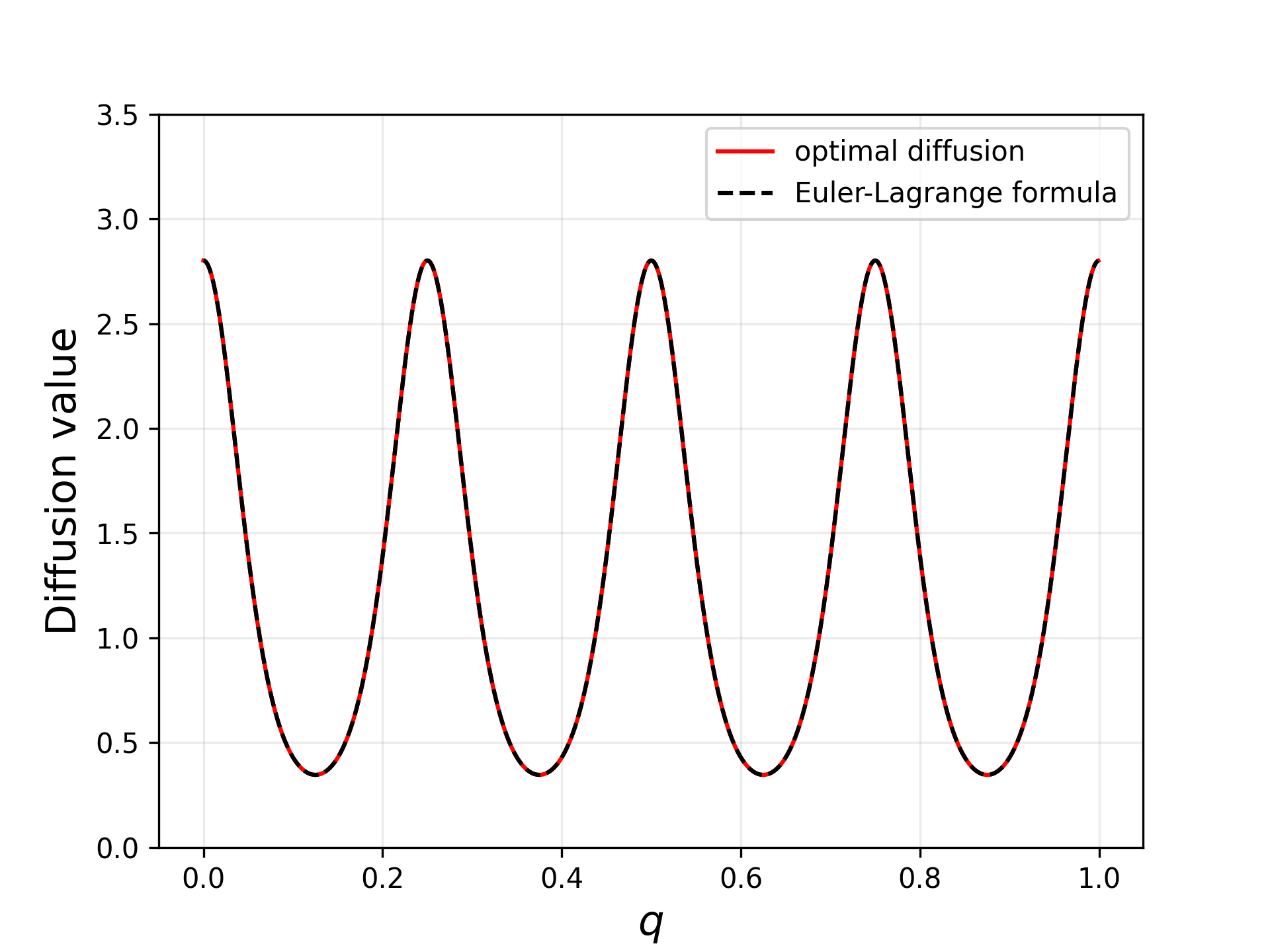}
    \caption{\label{fig:res_2_right}Comparison with the formal Euler--Lagrange characterizations~\eqref{eq:Diff_characterization_formal_alpha_to_infty_degenerate_case_1d_case} (approximated by~\eqref{eq:approximation_d_star_infty}) for the case $\eta=0$.}
  \end{subfigure}
  \caption{
  Results of the optimization procedure on the one-dimensional four-well potential~$V(q) =  \cos(8\pi q)$.}
\end{figure}

\paragraph{A case where~$\eta\in(0,\eta^\star)$.}
We plot in~\Cref{fig:res_3} the numerical results for a one-well potential. The numerical values of the spectral gaps obtained for~$\diff^\star,\diff^{\star,\alpha}$,~$\diff_{\mathrm{hom}}^\star$ and~$D_{\rm cst}$ are respectively~36.88, 36.75, 32.43 and 30.47. We also compute~$f_\alpha(D^{\star,\alpha})=36.94$. This shows that the map~$f_{\alpha}$ is a good approximation of the objective function~$\Lambda$, and that the diffusion coefficient obtained when solving~\eqref{eq:maximizer_in_thm_alpha} yields a spectral gap comparable to the one obtained when solving~\eqref{eq:optim-discrete}.

 Let us again compute an approximation $\upeta$ of~$\eta$ in order to compare the optimal diffusion with the formal Euler--Lagrange formula~\eqref{eq:Diff_characterization_formal_alpha_to_infty_degenerate_case_1d_case}. We plot in~\Cref{fig:res_3_c} the two leading nonzero eigenvalues using the maximizer~$\diff^{\star,\alpha}$ for increasing values of~$\alpha$. The behavior is very different from the previous case: the eigenvalues seem to converge to the same value but are quite different, even for rather large values of~$\alpha$. In~\Cref{fig:res_3_d}, we plot the behavior of~$\alpha(\sigma_{3}(D^{\star,\alpha})-\sigma_{2}(D^{\star,\alpha}))$ as~$\alpha$ increases, and fit the data with a function~$\alpha\mapsto K/\alpha +\upeta$, parameterized by~$K,\upeta\in\R$, using the method~\texttt{curve\_fit} from the~\texttt{scipy.optimize} module. This gives a limiting value~$\upeta\approx 0.51$: this corresponds to the convex combination~$(\upsilon_1,\upsilon_2)=(0.744,0.256)$ in~\eqref{eq:maximizer_superdifferential}. The comparison between~$D^{\star}$ and~$D^{\star,\infty}$ (as in~\eqref{eq:approximation_d_star_infty}) is given in~\Cref{fig:res_3_e}: the curves are superimposed, which again confirms that the formal characterization obtained in~\Cref{subsec:euler_lagrange_degenerate} is relevant.
\begin{figure}
  \centering
  \begin{subfigure}[t]{0.49\linewidth}
    \includegraphics[width=\textwidth]{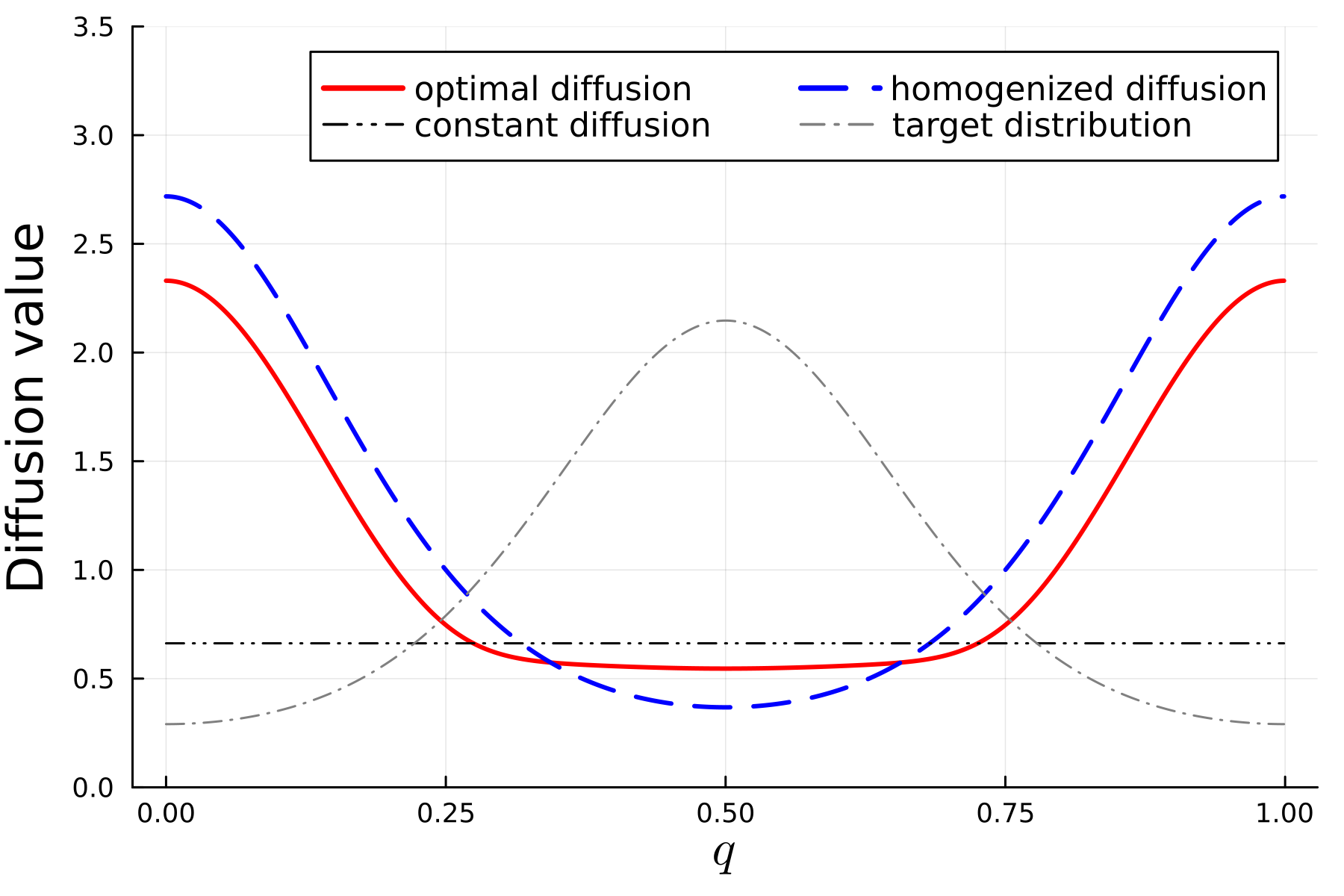}
    \caption{\label{fig:res_3_a}Various diffusion coefficients and target distribution~$\mu$. The solid curve in red is the optimal diffusion coefficient~$D^{\star}$.}
  \end{subfigure}
  \hfill
  \begin{subfigure}[t]{0.49\linewidth}
    \includegraphics[width=\textwidth]{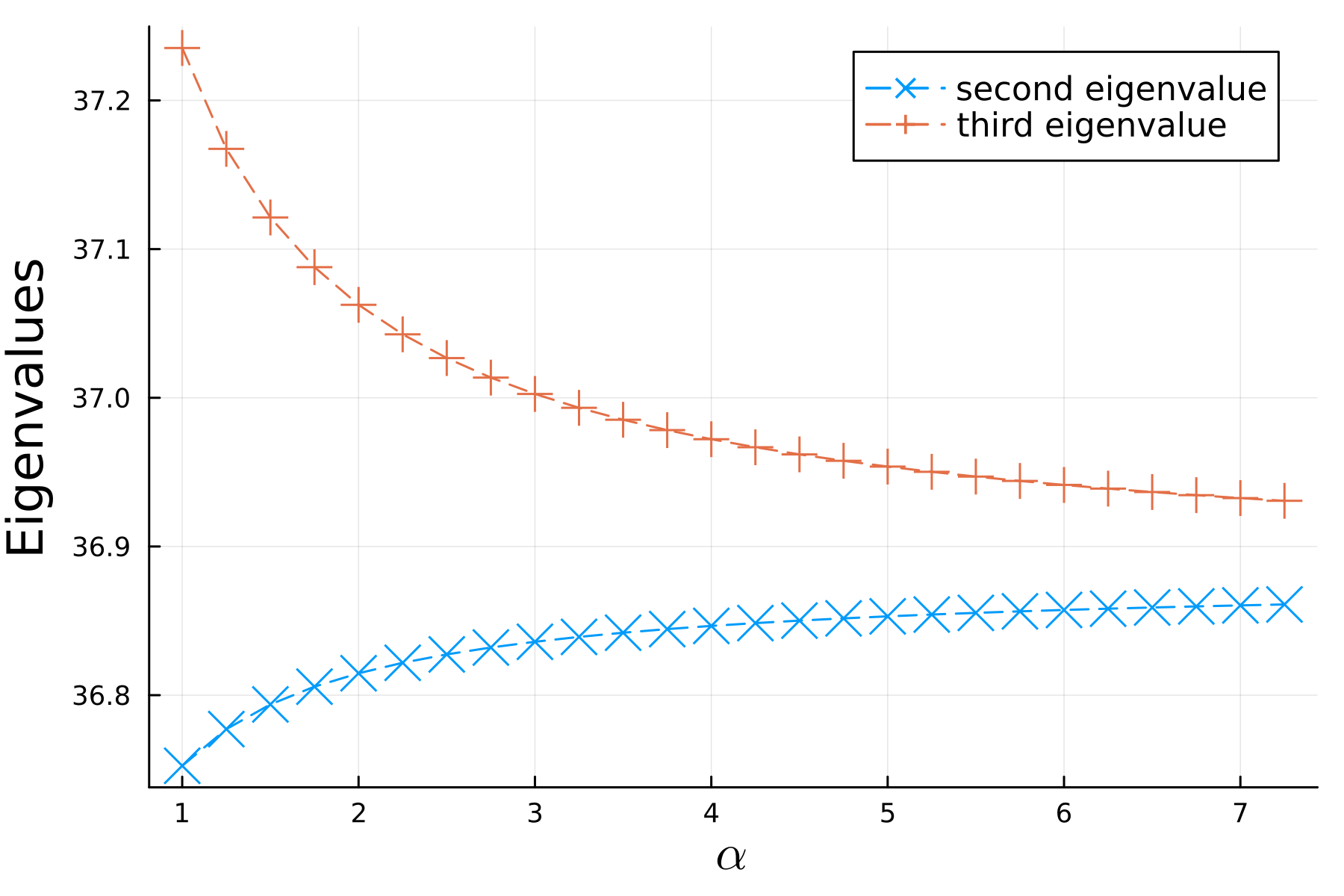}
    \caption{\label{fig:res_3_c}Second and third eigenvalues associated with the maximizer~$\diff^{\star,\alpha}$ as a function of~$\alpha$.}
  \end{subfigure}
  \hfill
  \begin{subfigure}[t]{0.49\linewidth}
    \includegraphics[width=\textwidth]{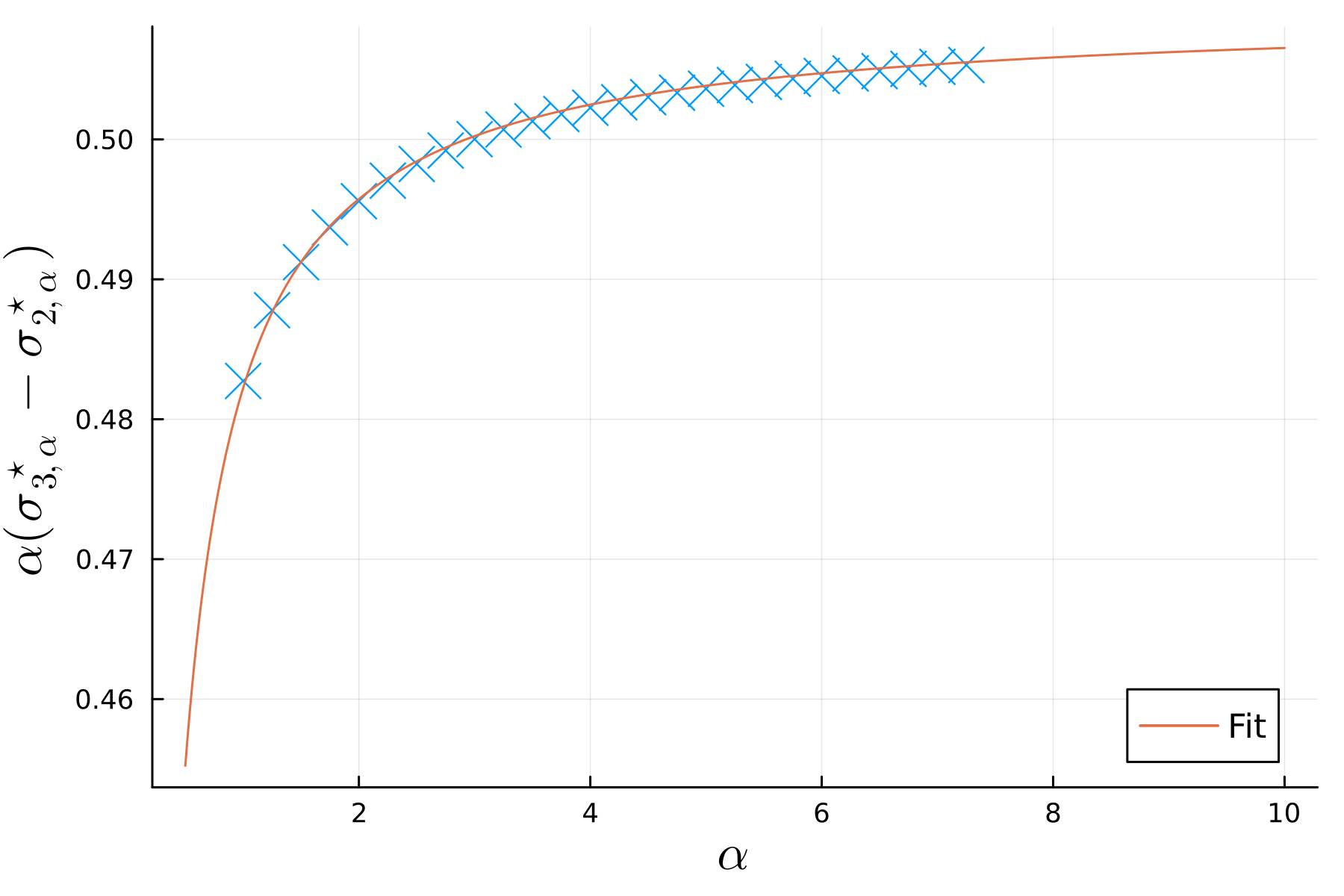}
    \caption{\label{fig:res_3_d}$\alpha(\sigma_{3}(D^{\star,\alpha})-\sigma_{2}(D^{\star,\alpha}))$ as a function of~$\alpha$. The fit is given by~$K/\alpha +\upeta$ with~$K\approx-0.027$ and~$\upeta\approx0.51$.}
  \end{subfigure}
  \hfill
  \begin{subfigure}[t]{0.49\linewidth}
    \includegraphics[width=\textwidth]{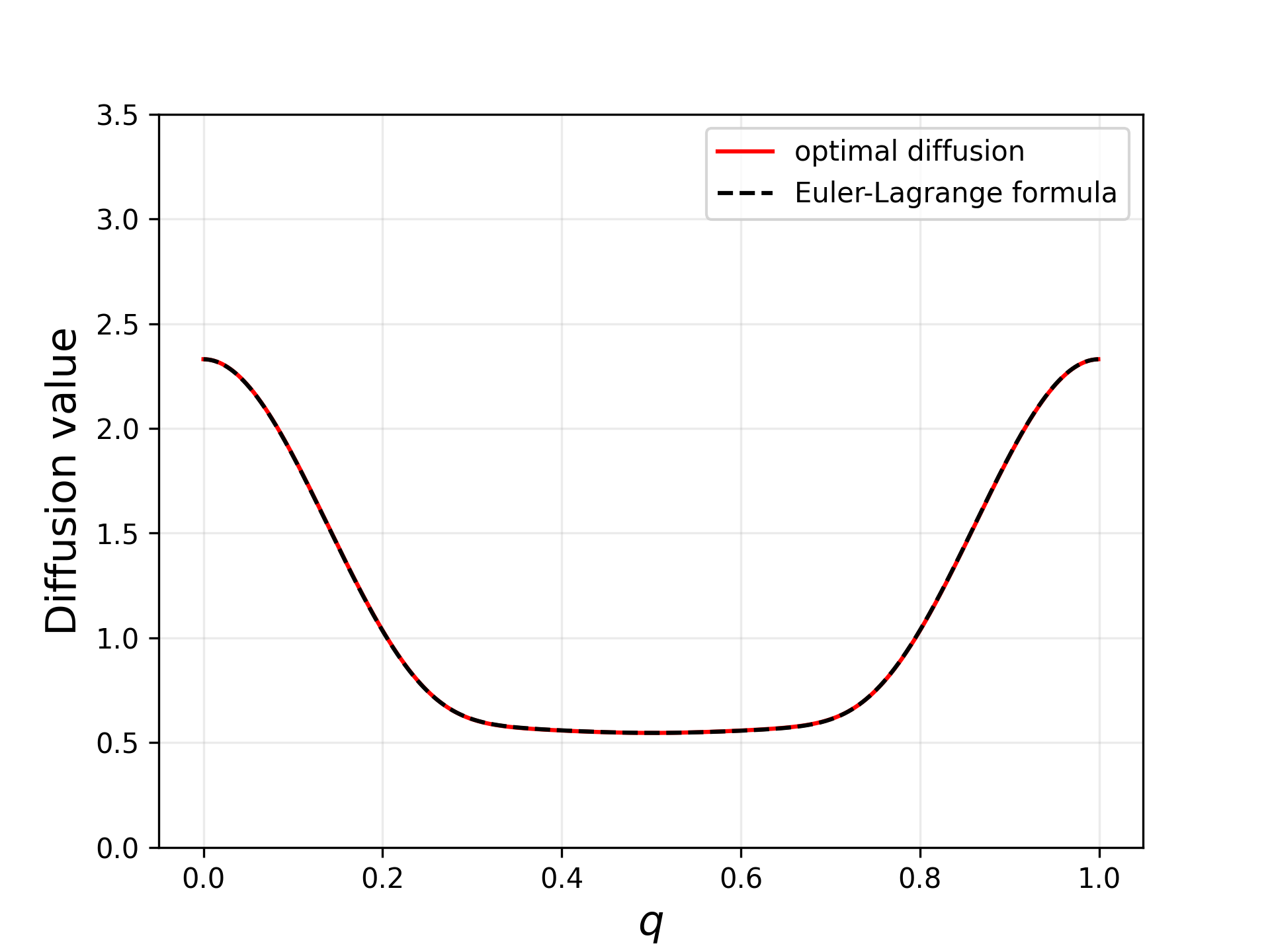}
    \caption{\label{fig:res_3_e}Comparison with the formal Euler--Lagrange characterizations~\eqref{eq:Diff_characterization_formal_alpha_to_infty_degenerate_case_1d_case} (approximated by~\eqref{eq:approximation_d_star_infty}) for the case $\eta\in(0,\eta^\star)$. The solid curve in red is the optimal diffusion coefficient~$D^{\star}$.}
  \end{subfigure}
  \caption{\label{fig:res_3} Results of the optimization procedure on the one-dimensional one-well potential~$V(q) =  \cos(2\pi q)$.}
\end{figure}

\paragraph{Cases where~$\eta=\eta^\star$.\label{ans:O_11}}

\begin{figure}
  \centering
    \begin{subfigure}[t]{0.49\linewidth}
    \includegraphics[width=\textwidth]{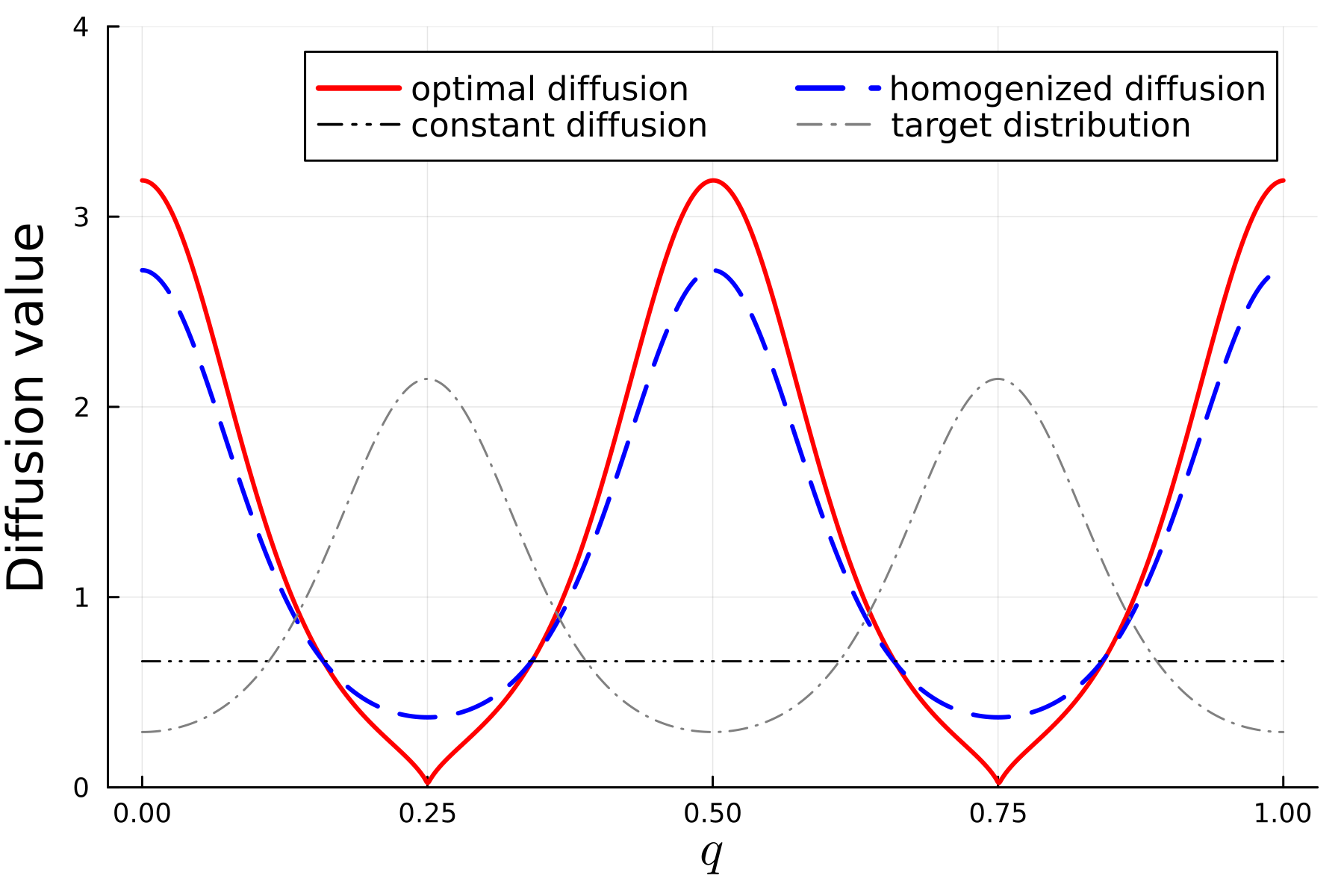}
    \caption{\label{fig:res_cos_4a}Various diffusion coefficients and target distribution~$\mu$. The solid curve in red is the optimal diffusion coefficient~$D^{\star}$. The numerical values of the spectral gap obtained for~$\diff^{\star}$,~$\diff_{\mathrm{hom}}^\star$ and~$D_{\rm cst}$ are respectively~22.84, 21.18 and 8.46. }
  \end{subfigure}
  \hfill
  \begin{subfigure}[t]{0.49\linewidth}
    \includegraphics[width=\textwidth]{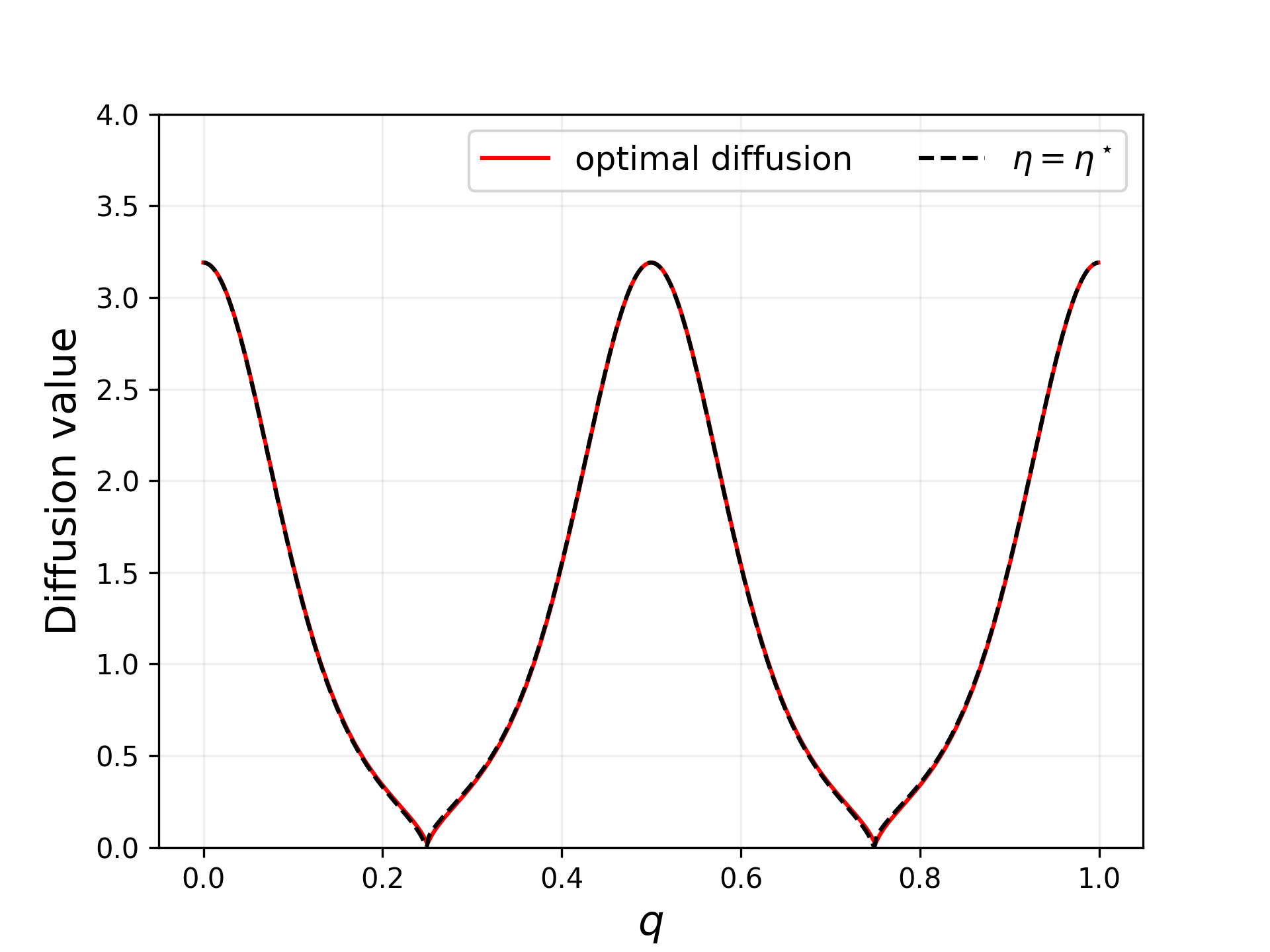}
    \caption{\label{fig:eta_star_2}
    Comparison with the formal Euler--Lagrange characterizations~\eqref{eq:Diff_characterization_formal_alpha_to_infty_degenerate_case_1d_case} (approximated by~\eqref{eq:approximation_d_star_infty}), using the value~$\eta=\eta^\star$ given by~\eqref{eq:eta_star}.}
  \end{subfigure}\\
  \begin{subfigure}[t]{0.49\linewidth}
    \includegraphics[width=\textwidth]{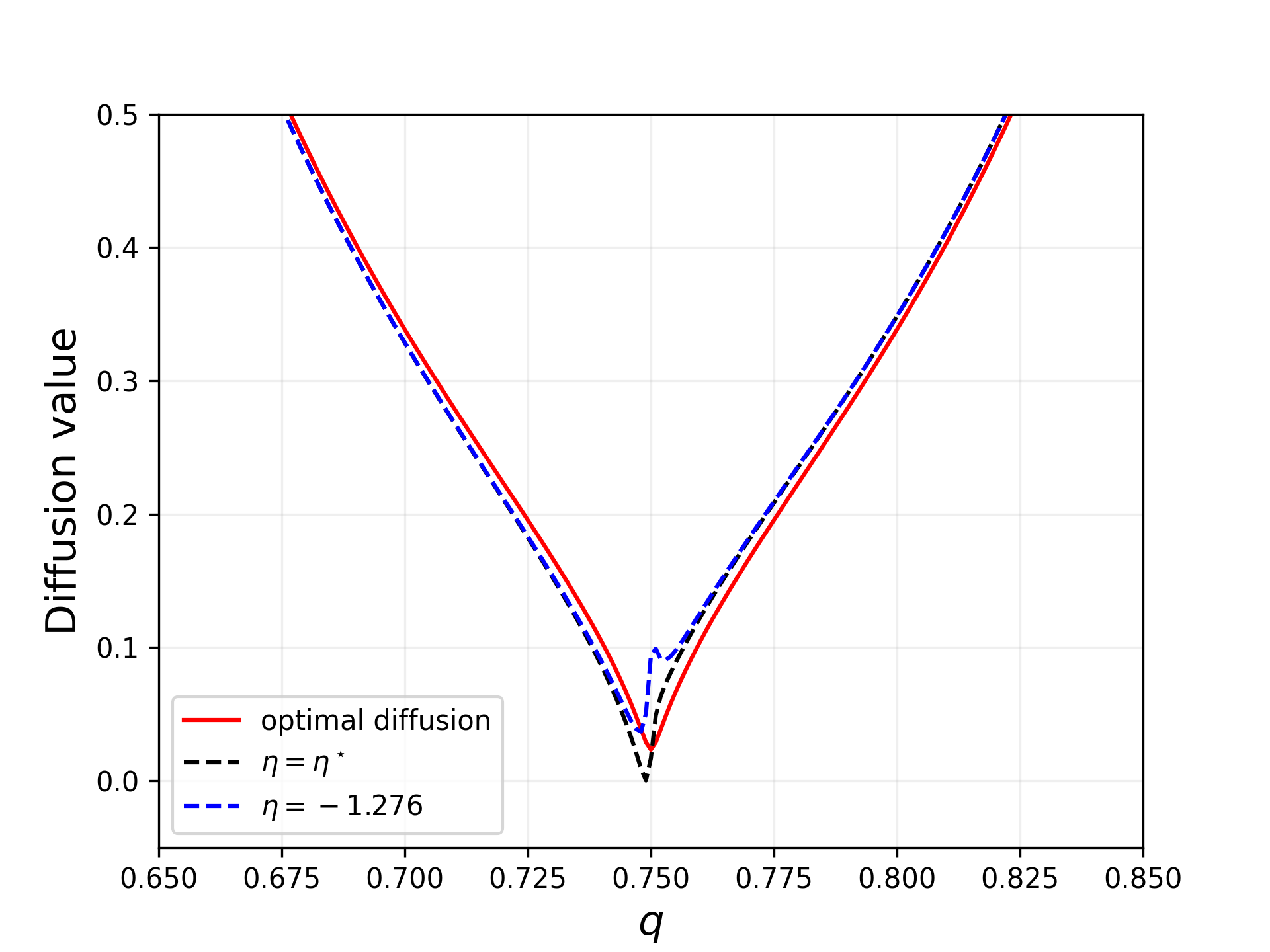}
    \caption{\label{fig:eta_star_3}Zoom around~$q\approx0.75$, displaying the coefficient given by~\eqref{eq:Diff_characterization_formal_alpha_to_infty_degenerate_case_1d_case} with~$\eta=\eta^\star$ and~$\eta=\upeta$.}
  \end{subfigure}
  \hfill
  \begin{subfigure}[t]{0.49\linewidth}
    \includegraphics[width=\textwidth]{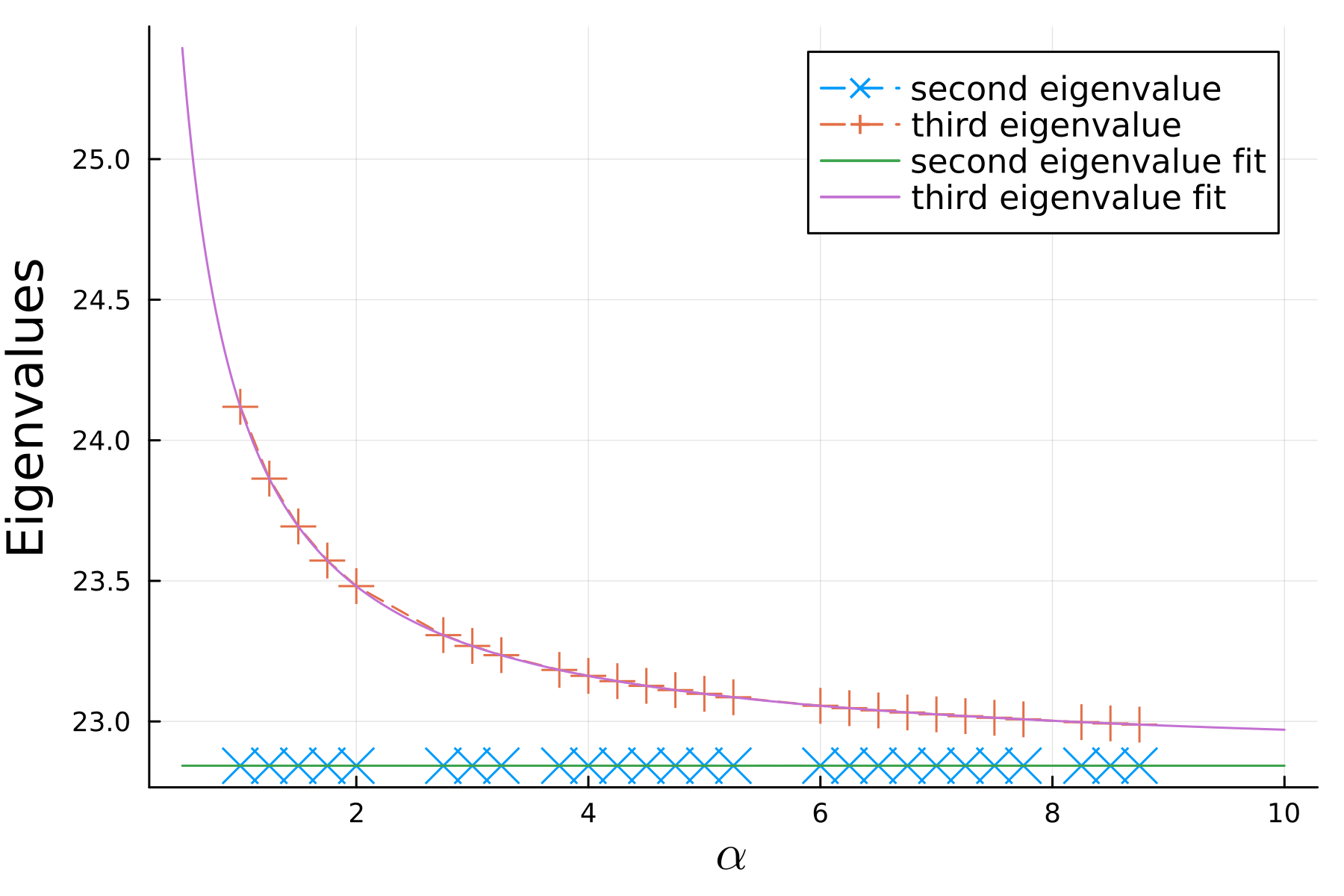}
    \caption{\label{fig:res_cos4_b}Second and third eigenvalues associated with the maximizer~$D^{\star,\alpha}$ as a function of~$\alpha$. The fit for~$\sigma_{2}(D^{\star,\alpha})$ is the constant~$K_2\approx22.842$ and the fit for~$\sigma_{3}(D^{\star,\alpha})$ is given by~$\upeta/\alpha+K_3$ with~$\upeta\approx1.276$ and~$K_3\approx22.843$.}
  \end{subfigure}
  \caption{
  Results of the optimization procedure on the one-dimensional double-well potential~$V(q) =  \cos(4\pi q)$.}
\end{figure}

Let us finally exhibit situations where $\eta=\eta^\star$. We will see in particular that in this case the optimal diffusion vanishes at some points, as expected from~\eqref{eq:Diff_characterization_formal_alpha_to_infty_degenerate_case_1d_case} with $\eta=\eta^\star$.

Let us first consider the double-well potential $V(q)=\cos(4\pi q)$, see~\Cref{fig:res_cos_4a}. We observe on~\Cref{fig:eta_star_2} that the optimal diffusion can be obtained from
the characterization~\eqref{eq:Diff_characterization_formal_alpha_to_infty_degenerate_case_1d_case} (approximated by~\eqref{eq:approximation_d_star_infty}) with $\eta=\eta^\star$ (the value of~$\eta^\star$ is approximated using the \texttt{lambertw} function of the \texttt{scipy.special} module). We see in~\Cref{fig:eta_star_3} that the optimal diffusion vanishes at two points around~$q\approx0.25$ and~$q\approx0.75$, as expected from formula~\eqref{eq:Diff_characterization_formal_alpha_to_infty_degenerate_case_1d_case} with $\eta=\eta^\star$.  
We further back the statement that $\eta=\eta^\star$ by considering the numerical limit $\upeta$ by solving~\eqref{eq:maximizer_in_thm_alpha} for various values of~$\alpha$. We present in~\Cref{fig:res_cos4_b} the behavior of the two leading nonzero eigenvalues for the values of~$\alpha$ for which the optimization problems converged. We also fit the data similarly to what was done in~\Cref{fig:res_3_d} to extract the limiting value~$\upeta\approx1.276$, which is very close to the critical value~$\eta^\star$ introduced in~\Cref{subsec:euler_lagrange_degenerate}, see~\eqref{eq:eta_star}. 
We emphasize that, even though only one eigenvector contributes in the formula~\eqref{eq:Diff_characterization_formal_alpha_to_infty_degenerate_case_1d_case}, the eigenvalue is degenerate.


\begin{figure}
  \centering
  \begin{subfigure}[t]{0.49\linewidth}
    \includegraphics[draft=false,width=\textwidth]{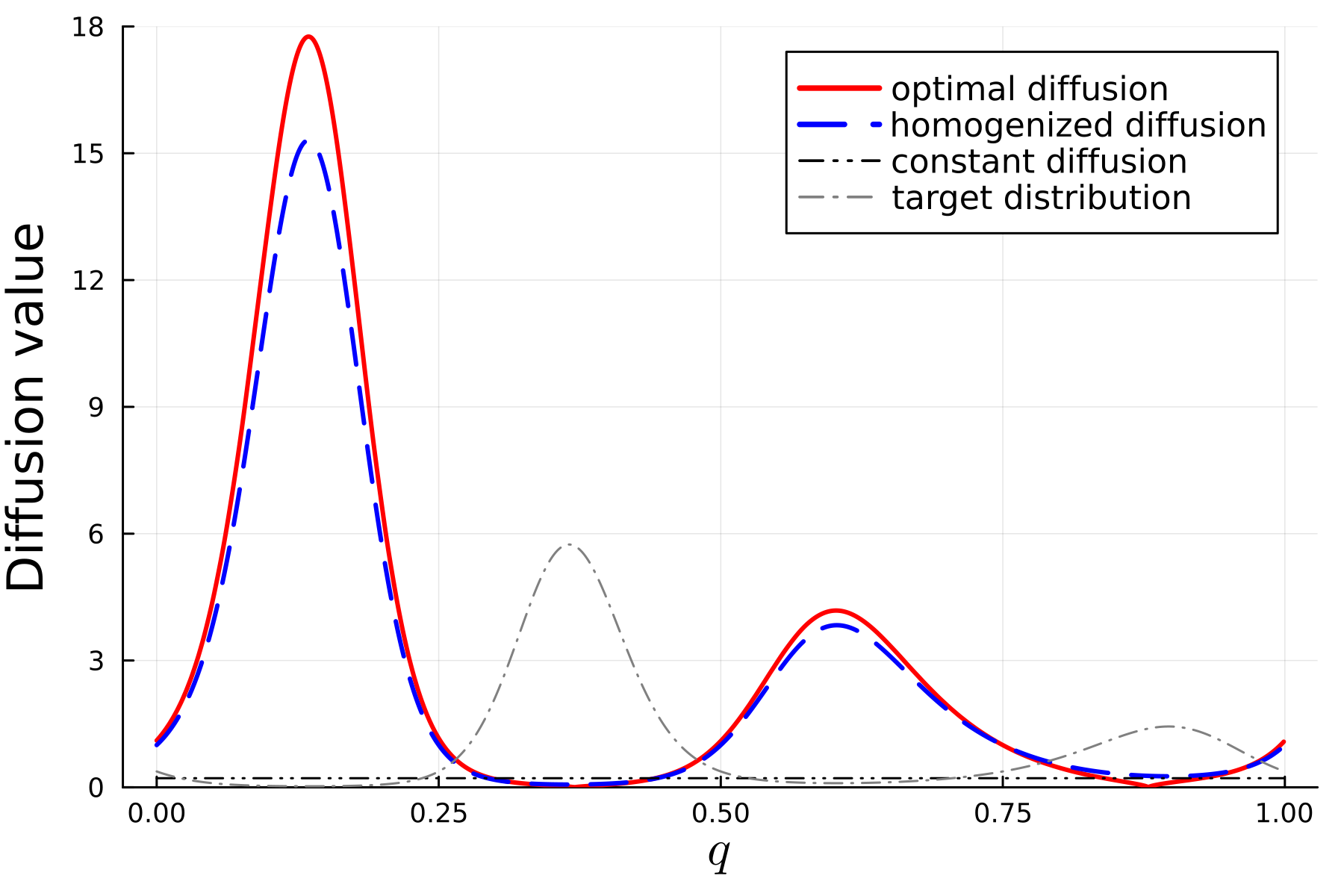}
    \caption{\label{fig:res_1_left}Various diffusion coefficients and target distribution~$\mu$. The solid curve in red is the optimal diffusion coefficient~$D^{\star}$. The numerical values of the spectral gaps obtained for~$\diff^{\star}$,~$\diff_{\mathrm{hom}}^\star$ and~$D_{\rm cst}$ are respectively~11.23, 10.58 and 0.81.}
  \end{subfigure}
  \hfill
    \begin{subfigure}[t]{0.49\linewidth}
    \includegraphics[width=\textwidth]{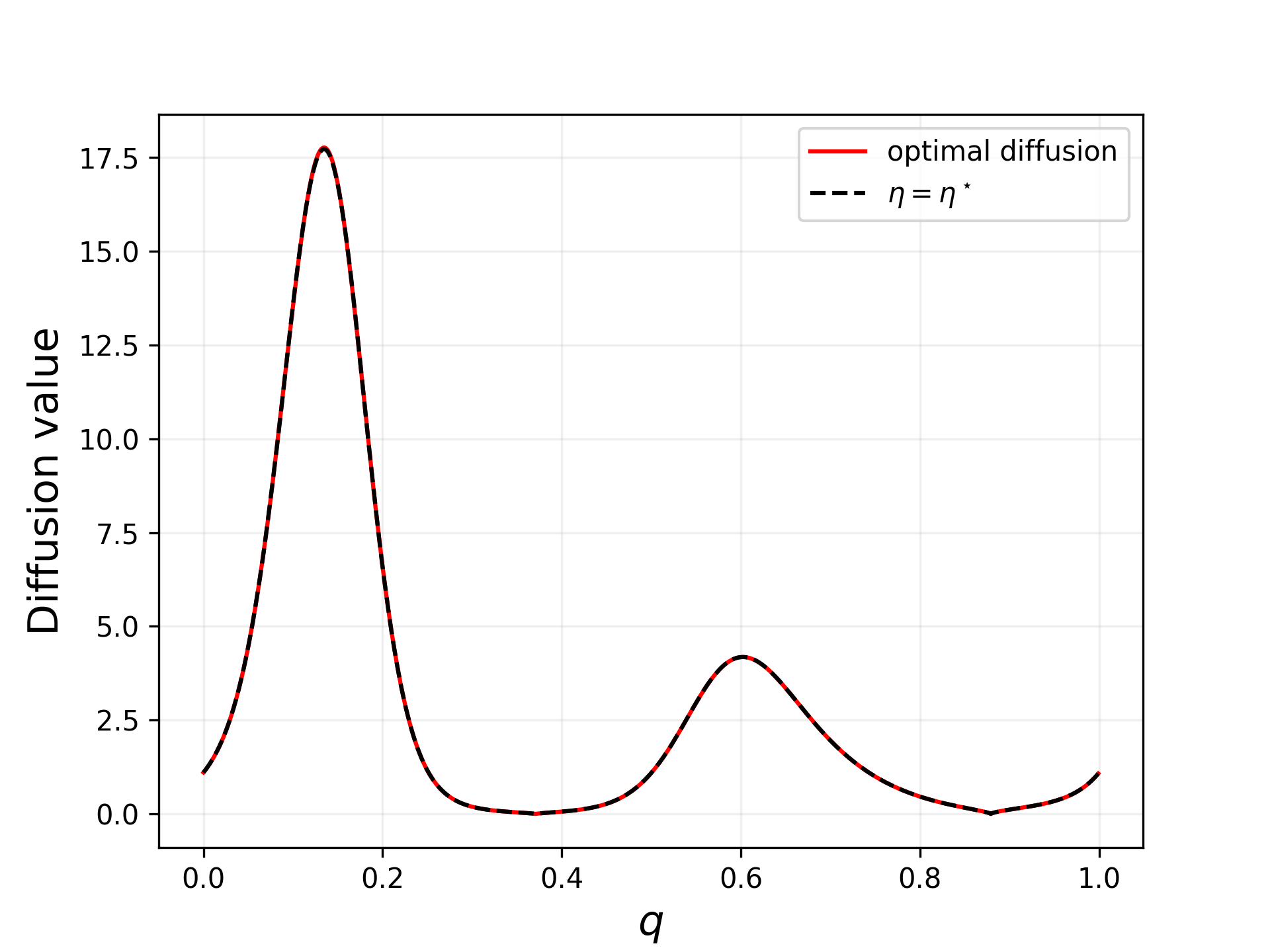}
    \caption{\label{fig:eta_star_1}
    Comparison with the formal Euler--Lagrange characterizations~\eqref{eq:Diff_characterization_formal_alpha_to_infty_degenerate_case_1d_case} (approximated by~\eqref{eq:approximation_d_star_infty}), using the value~$\eta=\eta^\star$ given by~\eqref{eq:eta_star}.}
  \end{subfigure}
  \hfill
  \begin{subfigure}[t]{0.49\linewidth}
    \includegraphics[width=\textwidth]{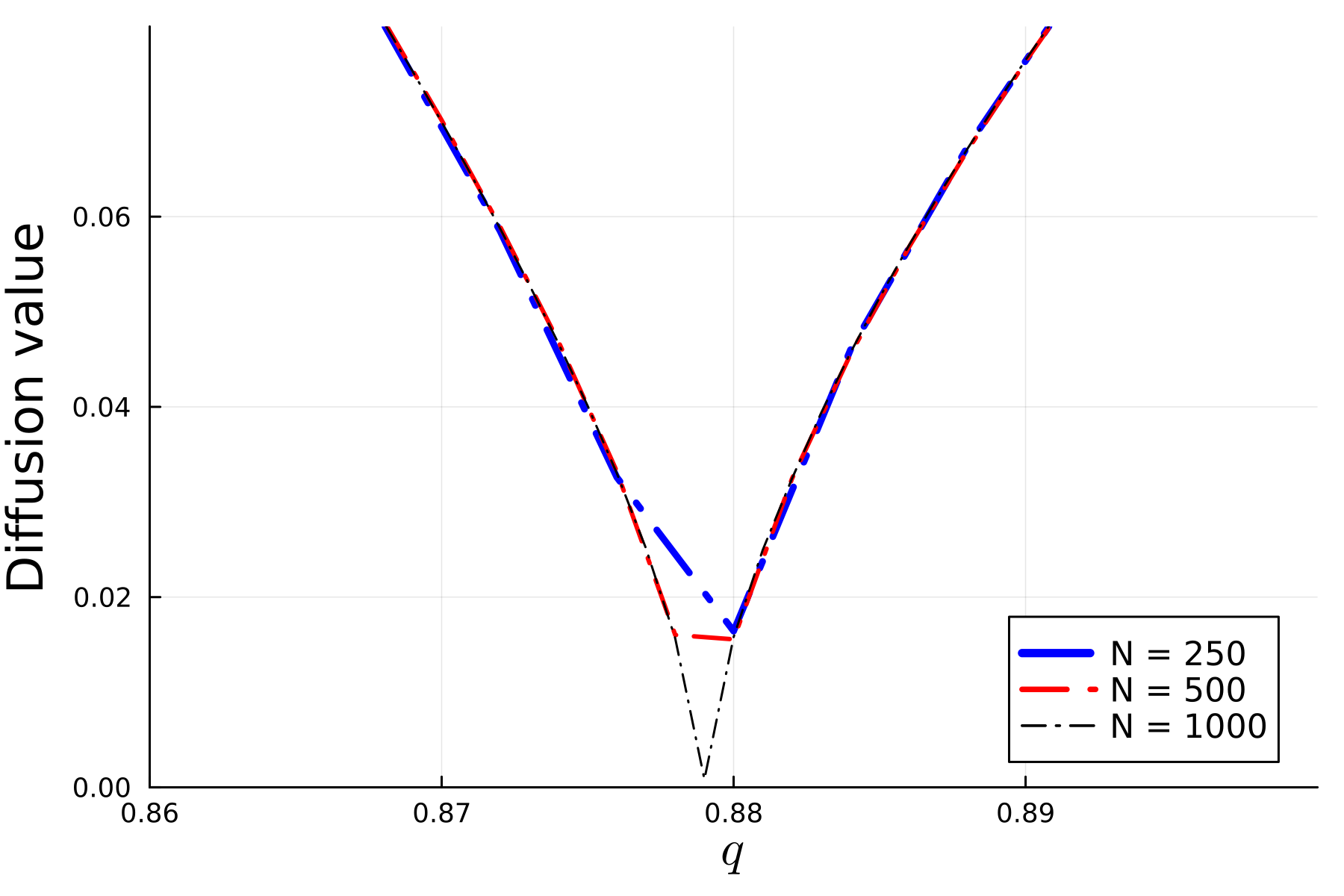}
    \caption{\label{fig:res_1_bottom}Zoom on the optimal diffusion coefficient~$D^{\star}$ for increasing values of~$N$.}
  \end{subfigure}
  \caption{\label{fig:res_1} Results of the optimization procedure on the one-dimensional double-well potential~$V(q) =  \sin(4\pi q) (2 + \sin(2\pi q))$.}
\end{figure}

Let us now consider another double-well potential, namely $V(q)=\sin(4\pi q)(2+\sin(2\pi q))$. We plot in~\Cref{fig:res_1_left} the optimal diffusions. We observe in~\Cref{fig:eta_star_1} that, as in the previous case, the formal Euler--Lagrange characterization~\eqref{eq:Diff_characterization_formal_alpha_to_infty_degenerate_case_1d_case} using the value~$\eta=\eta^\star$ is satisfied. As a consequence, the optimal diffusion coefficient vanishes at two points, around~$q \approx 0.37$ and~$q \approx 0.88$. We zoom around the point~$q\approx 0.88$ in~\Cref{fig:res_1_bottom} for various meshes and observe that, indeed, as the mesh gets finer ($N$ increases), the minimum value of the optimal diffusion coefficient gets smaller. In this case, it was difficult to compute $\upeta$.  It seems that the cancellation of the optimal diffusion at some point induces some difficulties in the optimization procedure, see~\Cref{rem:numerical_issues} in~\Cref{app:numerical}.

\subsection{Results for~$a>0$}
\label{sec:results_a>0}

We now present results of the optimization procedure for the double-well potential energy function~$V(q)=\sin(4\pi q)(2+\sin(2\pi q))$ already considered above (see Figure~\ref{fig:res_1}), when imposing various lowers bounds~$a\in\left\lbrace 0, 0.2, 0.4, 0.6, 0.8, 1\right\rbrace$. In~\Cref{fig:res_various_lower_bounds_optimal_diffusions}, we plot the optimal diffusion coefficients. The case~$a=0$ has already been covered in~\Cref{fig:res_1}. When~$a=1$, the optimal diffusion coefficient is fully determined by the constraints, and actually coincides with the optimal diffusion coefficient in the homogenized limit~$\diff_{\rm hom}^{\star}$. We highlight in~\Cref{fig:res_various_lower_bounds_normalized_diffusions} how the lower bounds~$a$ affects the shape of the optimal diffusion coefficient: we plot the optimal diffusion coefficients divided by~$\diff_{\rm hom}^{\star}$, and observe that the pointwise lower bound~$a$ is saturated on larger intervals as~$a$ increases to~1. We present in~\Cref{tab:spectral_gap_various_lower_bounds} the values of the spectral gap for the same values of the lower bound~$a$. Note that the spectral gap for~$a=0.2$ is almost the same as the one for~$a=0$. This suggests that imposing a (small) positive lower bound does not deteriorate too much the spectral gap, while ensuring a non-degenerate optimal diffusion. This may be useful to ensure the irreducibility of Markov Chain Monte Carlo algorithms based on the optimal diffusion.
 This point is further discussed below in~\Cref{subsec:sampling} (see~\Cref{fig:rejection_probabilities_various_lower_bounds}).  Besides, we observed that imposing a lower bound $a>0$ also improves the convergence of the optimization procedure to get the optimal diffusion.

\begin{figure}
  \centering
  \begin{subfigure}[t]{0.45\textwidth}
    \includegraphics[width=\linewidth]{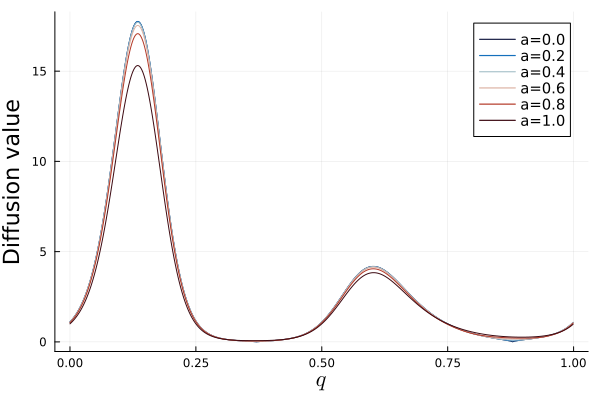}
    \caption{Optimal diffusion coefficients for various lower bounds~$a$.}
    \label{fig:res_various_lower_bounds_optimal_diffusions}
  \end{subfigure}
  \hfill
  \begin{subfigure}[t]{0.45\textwidth}
    \includegraphics[width=\linewidth]{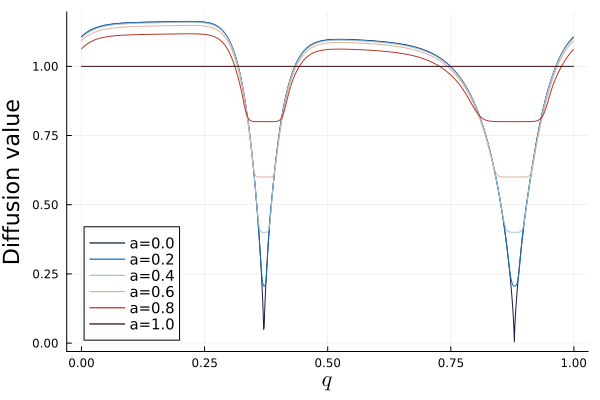}
    \caption{Optimal diffusion coefficients normalized by~$\diff_{\hom}^{\star}$.}
    \label{fig:res_various_lower_bounds_normalized_diffusions}
  \end{subfigure}
  \caption{Results of the optimization procedure for various lower bounds~$a$ and~$V(q)=\sin(4\pi q)(2+\sin(2\pi q))$.}
  \label{fig:res_various_lower_bounds}
\end{figure}

\begin{table}
  \centering
  \begin{tabular}{c|cccccc}
    \toprule
    Lower bound~$a$ & 0.0 & 0.2 & 0.4 & 0.6 & 0.8 & 1.0 \\\midrule
    Spectral gap    &
    11.227          &
    11.226          &
    11.208          &
    11.145          &
    10.983          &
    10.572
    \\ \bottomrule
  \end{tabular}
  \caption{Spectral gap associated to optimal diffusion coefficients obtained with the optimization procedure using different lower bounds~$a$ in the same setting as~\Cref{fig:res_various_lower_bounds}.}
  \label{tab:spectral_gap_various_lower_bounds}
\end{table}

\subsection{Conjectures and  remarks}\label{subsubsec:conj}
In view of the results above, we conjecture that in all situations, in dimension 1, for the optimal diffusion coefficient:
\begin{itemize}
  \item The first three nonzero eigenvalues of the operator~$-\cL_{\Diff^{\star}}$ satisfy $\varlambda_2 = \varlambda_3 < \varlambda_4$.
  \item The limiting procedure outlined in Section~\ref{subsec:euler_lagrange_degenerate} to define the optimal diffusion by~\eqref{eq:Diff_characterization_formal_alpha_to_infty_degenerate_case_1d_case} holds; in particular $\lim_{\alpha \to \infty} \alpha(\varlambda_{3,\alpha}-\varlambda_{2,\alpha})=\eta \in [0,\eta^\star]$ is well defined.
  \item The independent and normalized eigenvectors~$e_2$ and~$e_3$ (associated to~$\varlambda_2=\varlambda_3$) have derivatives that do not vanish at the same points and there are then two situations (see~\eqref{eq:Diff_characterization_formal_alpha_to_infty_degenerate_case_1d_case}):       \begin{itemize}
          \item either~$\eta=\eta^\star$, which is equivalent to the fact that the optimal
                diffusion vanishes at some points;
          \item or~$\eta<\eta^\star$, which is equivalent to the fact that the optimal
                diffusion is uniformly positive.
        \end{itemize}
\end{itemize}
It would be interesting to prove these conjectures, and also to understand which features of the potential $V$ determine whether $\eta \in [0,\eta^\star)$ (the optimal diffusion is positive) or $\eta=\eta^\star$ (the optimal diffusion vanishes at some point).
 For example, we will see in the next section that when~$V$ is periodic with a period~$1/k$ where~$k\in\mathbb{N}\setminus\left\lbrace0\right\rbrace$, the optimal diffusion converges rapidly to the optimal diffusion for the homogenized limit, which does not vanish.

\section{Optimal diffusion matrix in the homogenized limit}
\label{sec:homog}

A drawback of the optimization method described in~\Cref{sec:numerical} is the computational burden associated with the calculation of the optimal diffusion matrix. To obtain an approximation of the optimal diffusion, which can be used for instance as a good initial guess to the optimization procedure, we propose a simple model for the effective diffusion of the optimized overdamped Langevin dynamics. This model is obtained by studying the asymptotic behavior of the optimal diffusion matrix in the homogenized limit of periodic potentials with infinitely small spatial periods. The limit is explicit either for linear constraints, or for generic~$L^p_V$ constraints in the one dimensional case~$\dim=1$, and displays a simple dependence on the target density~$\mu$. The form of this dependence is backed up by prior results in the literature~\cite{RobertsStramer}.

From a physical viewpoint, considering infinitely periodic materials is relevant, for example to study diffusions of adatoms on (periodic) surfaces or propagation of defects in (periodic) crystalline lattices~\cite[Chapitre~3]{bensoussan_2011}. We also refer to~\Cref{rem:homog_sampling} where we explain that under a proper rescaling in time and space, the diffusion process can be seen as a solution to an effective Brownian motion on the torus whose diffusion is the homogenized diffusion. Maximizing this homogenized diffusion therefore amounts to enhancing the exploration capabilities of this limiting process.

In order to present our results, we start by introducing in~\Cref{subsec:homog-compact} the setting of homogenization theory, in particular the notion of~$H$-convergence, and give some useful and standard results, adapted to our framework of periodic boundary conditions. Next, in~\Cref{subsec:per-homog}, we study the homogenized limit of the diffusion matrices and then explicitly optimize the spectral gap of the homogenized limit for one-dimensional systems or linear constraints. In~\Cref{subsec:optim-homog}, we study the commutation of the homogenization and optimization procedures. We show that the diffusion matrix obtained by optimizing first and then taking the homogenized limit coincides with the optimal diffusion matrix for the homogenized problem, therefore proving that maximizing an effective diffusion is equivalent to optimizing a spectral gap on the periodized problem before homogenizing. The second approach, namely optimizing the homogenized problem, is more practical, and this is the main motivation for the commutation result we establish. Lastly, we present in~\Cref{subsec:homog:numerical_results} numerical experiments that confirm our theoretical analysis.

\subsection{Definitions and compactness results}
\label{subsec:homog-compact}

We present in this section compactness results for sequences of matrices~$(\A^k)_{k\geq 1} \subset L^{\infty}(\T^\dim, \mathcal{M}_{a,b})$. Let us emphasize that we restrict ourselves to~$a,b>0$ in all this section. We start by recalling the notion of~$H$-convergence in the setting of~$\mathbb{Z}^\dim$-periodic functions. Compared to the classical~$H$-convergence presented in the literature (see for instance~\cite[Chapter~1]{allaire_homogeneisation} and references therein, as well as~\cite{lebris_blanc_homog_book}), the main difference is that we use periodic boundary conditions rather than Dirichlet boundary conditions.

\begin{definition}[$H$-convergence]
  \label{def:Hcvg}
  Fix~$a,b>0$, and consider a sequence~$(\A^k)_{k\geq 1} \subset L^{\infty}(\T^\dim, \mathcal{M}_{a,b})$. The sequence~$(\A^k)_{k\geq 1}$ is said to~$H$-converge to~$\overline{\A} \in L^{\infty}(\T^\dim, \mathcal{M}_{a,b})$ (denoted by~$\A^k\xrightarrow{H}\overline \A$) if, for any~$f\in H^{-1}(\T^\dim)$ such that~$\langle f, \mathbf{1}\rangle_{H^{-1}(\T^\dim),H^1(\T^\dim)}=0$, the sequence~$(u^k)_{k\geq 1} \subset H^1(\T^\dim)$ of solutions to
  \begin{equation}
    \label{eq:periodic-boundary}
    \left\{\begin{aligned}
       & -\div\left(\A^k\nabla u^k\right)= f & \text{ on }\T^\dim,
      \\
       & \int_{\T^\dim}u^k(q)\, dq = 0,      &
    \end{aligned}\right.
  \end{equation}
  satisfies, in the limit~$k \to +\infty$,
  \begin{equation*}
    \left\{\begin{aligned}
      u^k            & \rightharpoonup u \text{ weakly in }H^1(\T^\dim),                          \\
      \A^k\nabla u^k & \rightharpoonup \overline \A\nabla u  \text{ weakly in }L^2(\T^\dim)^\dim,
    \end{aligned}\right.
  \end{equation*}
  where~$u \in H^1(\T^\dim)$ is the solution of the homogenized problem
  \begin{equation}\label{eq:Abaru}
    \left\{\begin{aligned}
       & -\div\left(\overline \A\nabla u\right)= f & \text{ on }\T^\dim,
      \\
       & \int_{\T^\dim}u(q) \, dq=0.               &
    \end{aligned}\right.
  \end{equation}
\end{definition}

Notice that the functions~$u^k$ in~\eqref{eq:periodic-boundary} and~$u$ in~\eqref{eq:Abaru} are indeed well-defined since~$\langle f, \mathbf{1}\rangle_{H^{-1}(\T^\dim),H^1(\T^\dim)}=0$ (in application of the Lax--Milgram theorem, see~\Cref{lem:LaxMilgram} in~\Cref{app:homogenization} below).
The next theorem, proved for completeness in~\Cref{app:thm:Hcvg-compact}, states a compactness result for sequences of periodic functions.

\begin{theorem}[Murat--Tartar]
  \label{thm:Hcvg-compact}
  Fix~$a,b>0$, and consider a sequence~$(\A^k)_{k\geq 1} \subset L^{\infty}(\T^\dim,\mathcal{M}_{a,b})$. Up to extraction of a subsequence,~$(\A^k)_{k \ge 1}$~$H$-converges to a matrix~$\overline{\A}\in L^{\infty}(\T^\dim,\mathcal{M}_{a,b})$ (non constant in general).
\end{theorem}

The following theorem, adapted from~\cite[Theorem~1.3.18]{allaire_homogeneisation} (and proved for completeness in~\Cref{subsec:proof-Hcvg-constant}, see~\Cref{rmk:cv_H_periodic}), states an~$H$-convergence result for sequences of matrices defined through periodic homogenization. It is a particular case of~\Cref{thm:Hcvg-compact} specific to periodic functions. 

\begin{theorem}
  \label{thm:Hcvg}
  Fix~$a,b>0$, and let~$\A\in L^{\infty}(\T^\dim,\mathcal{M}_{a,b})$. For~$k\geq 1$, consider the sequence~$(\A_{\#,k})_{k \ge 1}$ defined as
  \begin{equation*}
    \forall q \in \T^\dim, \qquad \A_{\#,k}(q)=\A(kq).
  \end{equation*}
  The sequence~$(\A_{\#,k})_{k \ge 1}$~$H$-converges to the constant matrix~$\overline{\mathcal{A}}\in\mathcal{M}_{a,b}$ with entries
  \begin{equation}
    \label{eq:Abar1}
    \forall 1\leq i,j\leq \dim, \qquad \overline{\A}_{ij} = \int_{\T^\dim} (e_i + \nabla w_i(q))^\top \A(q) (e_j + \nabla w_j(q)) \, dq,
  \end{equation}
  where~$(e_1,\ldots , e_\dim)$ is the canonical basis of~$\R^\dim$ and~$\{w_i\}_{1\leq i\leq \dim} \subset H^1(\T^\dim)$ is the family of unique solutions to the problems
  \begin{equation}\label{eq:Abar2}
    \left\{\begin{aligned}
       & -\div \left( \A(e_i+\nabla w_i) \right) = 0, \\
       & \int_{\T^\dim}w_i(q) \, dq =  0.
    \end{aligned}\right.
  \end{equation}
\end{theorem}

The last theorem, proved in~\Cref{app:thm:Hcvg-spectral} by a straightforward adaptation of~\cite[Theorem~1.3.16]{allaire_homogeneisation}, states that~$H$-convergence implies the convergence of the sequence of spectral gaps.

\begin{theorem}
  \label{thm:Hcvg-spectral}
  Fix~$a,b>0$ and let~$(\A^k)_{k\geq 1}\subset L^{\infty}(\T^\dim,\mathcal{M}_{a,b})$ be a sequence~$H$-converging to a homogenized matrix~$\overline{\mathcal{A}} \in L^{\infty}(\T^\dim,\mathcal{M}_{a,b})$. Let~$(\rho^k)_{k\geq 1} \subset L^{\infty}(\T^\dim)$ be a sequence of positive functions converging weakly-$*$ in~$L^{\infty}(\T^\dim)$ to a limiting function~$\rho$, and such that there exists~$\rho_-,\rho_+>0$ for which~$0<\rho_{-}\leq \rho^k(q)\leq \rho_+<+\infty$ for all~$k\geq 1$ and for almost all~$q\in\T^\dim$. Let~$\lambda^k$ be the smallest nonzero eigenvalue and~$u^k$ an associated normalized eigenvector of the spectral problem
  \begin{equation*}
    \left\{\begin{aligned}
       & -\div\left(\A^k(q)\nabla u^k(q)\right) = \lambda^k\rho^k(q)u^k(q), \\
       & \int_{\T^\dim} u^k(q)^2 \, dq = 1.
    \end{aligned}\right.
  \end{equation*}
  Then,
  \begin{equation*}
    \lim_{k\rightarrow +\infty}\lambda^k = \overline{\lambda}
  \end{equation*}
  and, up to a subsequence,~$(u^k)_{k\geq 1}$ converges weakly in~$H^1(\T^\dim)$ to~$u$,
  where $(\overline{\lambda},u)$ is an eigenpair of the homogenized operator:
  \begin{equation*}
    \left\{\begin{aligned}
       & -\div\left(\overline{\mathcal{A}}\nabla u(q)\right) = \overline{\lambda}\rho(q)u(q), \\
       & \int_{\T^\dim} u(q)^2 \, dq = 1,
    \end{aligned}\right.
  \end{equation*}
  and~$\overline{\lambda}$ is the smallest nonzero eigenvalue of this homogenized eigenvalue problem.
\end{theorem}

\subsection{Optimization of the periodic homogenization limit}
\label{subsec:per-homog}

We still fix~$a,b>0$, and apply the results of~\Cref{subsec:homog-compact} to a particular class of sequences of matrix valued-functions~$(\A_{\#,k})_{k\geq 1}$, defined for~$\Diff\in\Diffset_p^{a,b}$ by decreasing the period of the~$\mathbb{Z}^\dim$-periodic matrix
\begin{equation}
  \label{eq:def_A_from_D}
  \A(q) = \Diff(q) \exp(-V(q))
\end{equation}
to obtain the~$(\Z/k)^\dim$-periodic function
\begin{equation}
  \label{eq:A_sharp_k}
  \A_{\#,k}(q) = \Diff_{\#,k}(q)\exp\left(-V_{\#,k}(q)\right), \qquad \Diff_{\#,k}(q) = \Diff(kq), \qquad V_{\#,k}(q) = V(kq).
\end{equation}
Note that~$\A_{\#,k}(q)\in\mathcal{M}_{a,b}$ for almost every~$q \in \T^\dim$ and for all~$k\geq 1$.
The periodization procedure~\eqref{eq:A_sharp_k} allows us to prove in~\Cref{sec:periodic_homog_on_A} the~$H$-convergence of the sequence of matrices, and the convergence of the associated eigenvalues and eigenvectors (see~\Cref{cor:Hcvg-spectral}). We then provide in~\Cref{sec:optimization_homog_lim} explicit expressions for diffusion matrices~$\Diff$ which maximize the spectral gap of the homogenized problem; see in particular~\Cref{prop:max-homog} for the one dimensional case~$\dim = 1$.

\subsubsection{Periodic homogenization of the spectral gap problem}
\label{sec:periodic_homog_on_A}

We consider the process~$(q^k_t)_{t \geq 0}$ defined on~$\T^\dim$ by the dynamics
\begin{equation}
  \label{eq:dynamics-periodic}
  dq^k_t = \left[-\Diff_{\#,k}(q^k_t)\nabla V_{\#,k}(q^k_t) + \div(\Diff_{\#,k})(q^k_t)\right]dt+\sqrt{2}\,\Diff_{\#,k}^{1/2}(q^k_t)\,dW_t.
\end{equation}
The generator associated to the dynamics acts on test functions~$u:\T^\dim \to \R$ as
\begin{equation*}
  \mathcal{L}_k u =  \exp\left(V_{\#,k}\right) \div\left[\Diff_{\#,k}\exp\left(-V_{\#,k}\right)\nabla u\right].
\end{equation*}
Since the eigenfunctions associated with the first eigenvalue~0 are constant functions, the Rayleigh--Ritz principle implies that the spectral gap of~$\mathcal{L}_k$ considered as an operator on~$L^{2}(Z^{-1}\rme^{-V_{\#,k}})$ (recall that~$Z$ is defined in~\eqref{eq:mu}) writes (recalling the notation~\eqref{eq:A_sharp_k}):
\begin{equation}
  \label{eq:lambda_per_k}
  \Lambda_{\#,k}(\Diff) = \min_{u\in H^1(\T^\dim) \setminus\{0\}} \left\{ \frac{\dps \int_{\T^\dim}\nabla u(q)^\top \A_{\#,k}(q)\nabla u(q) \, dq}{\dps \int_{\T^\dim} u^2(q) \, \rme^{-V_{\#,k}(q)} \, dq} \ \middle| \ \int_{\T^\dim} u \, \rme^{-V_{\#,k}} = 0 \right\}.
\end{equation}
Let us emphasize that the test functions to find the minimum on the right-hand side of the previous equality are~$\Z^\dim$-periodic, while the diffusion and the potential are~$(\Z/k)^\dim$-periodic. Note also that the minimization over the functions~$u$ is performed on~$H^1(\T^\dim)$, which is equivalent in fact to the functional space~$H^1(Z^{-1}\rme^{-V_{\#,k}})$\ since $\T^\dim$ is compact and~$V$ is smooth. This allows to work with a fixed functional space for the test functions~$u$ whatever the value of~$k$.

The spectral gap~$\Lambda_{\#,k}(\Diff)$ is the smallest nonzero solution~$\lambda^{k}$ of the following eigenvalue problem posed on~$\T^\dim$: find~$(u^{k},\lambda^{k})\in H^{1}(\T^{\dim})\setminus\left\lbrace0\right\rbrace\times(0,+\infty)$ such that
\begin{equation*}
  -\div\left[\Diff_{\#,k}\exp\left(-V_{\#,k}\right)\nabla u^k\right]= \lambda^k \exp\left(-V_{\#,k}\right)u^k,\qquad \int_{\T^{\dim}}u^{k}\rme^{-V_{\#,k}}=0.
\end{equation*}
By applying~\Cref{thm:Hcvg-spectral} to the sequence~$(\mathcal{A}_{\#,k})_{k\geq 1}$, we therefore obtain the following result on the convergence of the sequence of spectral gaps~$\left(\Lambda_{\#,k}(\Diff)\right)_{k\geq 1}$ for a given diffusion~$\Diff\in \Diffset_p^{a,b}$.

\begin{corollary}
  \label{cor:Hcvg-spectral}
  Fix~$p \in [1,+\infty)$,~$a,b>0$ satisfying~\eqref{eq:compatibility_conditions_continuous_level} and~$\Diff\in \Diffset_p^{a,b}$. Then,~$\mathcal{A}_{\#,k} \xrightarrow{H} \overline{\mathcal{A}}$ with~$\overline{\A}$ defined in~\eqref{eq:Abar1}. Defining~$\overline{\Diff} = Z^{-1}\overline{\A}$, the sequence~$\left(\Lambda_{\#,k}(\Diff)\right)_{k\geq 1}$ converges as~$k\to +\infty$ towards
  \begin{equation*}
    \Lambda_{\mathrm{hom}}(\Diff) : = \inf_{u\in H^{1,0}(\T^\dim) \setminus\{0\}}\frac{\dps \int_{\T^\dim} \nabla u(q)^\top \overline{\Diff}\nabla u(q) \, dq}{\dps \int_{\T^\dim}u^2(q) \, dq}.
  \end{equation*}
\end{corollary}

\begin{remark}
  \label{rem:homog_sampling}
  Consider the process~$(q_t^{k})_{t\geqslant0}$ defined in~\eqref{eq:dynamics-periodic}, and define the rescaled process~$\left(Q_t^k\right)_{t\geqslant 0}=\left(k^{-1}q^{k}_{k^2t}\right)_{t\geqslant0}$ for~$k\geqslant 1$. By adapting the arguments from~\cite[Theorem~1]{fathi_2015} (see also~\cite[Chapter~3]{bensoussan_2011}), this rescaled process can be shown to converge, in the limit $k\to+\infty$, to an effective Brownian motion on the $d$-dimensional torus whose diffusion is the homogenized diffusion $\overline{\Diff}$ introduced in Corollary~\ref{cor:Hcvg-spectral}.
\end{remark}

In the next section, we seek to maximize the spectral gap~$\Lambda_{\mathrm{hom}}(\Diff)$ with respect to the baseline diffusion matrix~$\Diff$ from which~$\overline{\Diff}$ arises as the~$H$-limit of~$\Diff_{\#,k}(q)\exp(- V_{\#,k}(q))/Z$. There exists a large body of literature on optimizing homogenized coefficients, for example in the context of shape optimization of materials, typically by considering laminated structures which would correspond here to situations where the diffusion can only take two values; see for instance~\cite{Sigmund94,HD95,BT10,AGDP19} and~\cite[Chapter~10]{Henrot}.

It is useful to provide a more explicit expression of the limiting matrix~$\overline{\Diff}$ in order to maximize~$\Lambda_{\mathrm{hom}}(\Diff)$. Consider~$\xi = \xi_1e_1 + \dots + \xi_\dim e_\dim \in\R^\dim$ and note that
\begin{equation*}
  \xi^{\top} \overline{\Diff}\xi = Z^{-1} \int_{\T^\dim} (\xi + \nabla w_{\xi}(q))^\top \A(q)(\xi + \nabla w_{\xi}(q)) \, dq,
\end{equation*}
where~$w_{\xi} = \xi_1 w_1 + \dots + \xi_\dim w_\dim$ and the family~$(w_i)_{1 \leq i \leq \dim}$ is defined in~\eqref{eq:Abar2}. Recalling~\eqref{eq:def_A_from_D}, it follows that
\begin{align}
  \xi^{\top} \overline{\Diff}\xi & = \int_{\T^\dim} \left(\xi+\nabla w_\xi(q)\right)^{\top} \Diff(q) \left(\xi+\nabla w_\xi(q)\right) \, \mu(q) \, dq \notag                                                          \\
  & = \int_{\T^\dim}\left(\xi^{\top} \Diff(q)\xi  + 2 \xi^{\top}\Diff(q)\nabla w_\xi(q)+ \nabla w_\xi(q)^{\top} \Diff(q) \nabla w_\xi(q) \right)\mu(q) \, dq. \label{eq:optim-homog-3}
\end{align}
By the weak formulation of~\eqref{eq:Abar2}, it holds, for any~$v \in H^1(\T^\dim)$,
\begin{equation}
  \label{eq:FV_def_mat_homog}
  \int_{\T^\dim}\xi^{\top} \Diff(q)\nabla v(q) \, \mu(q) \, dq = -\int_{\T^\dim}\nabla w_\xi(q)^{\top} \Diff(q)\nabla v(q) \, \mu(q) \, dq.
\end{equation}
Upon replacing~$v$ by~$w_\xi$, we obtain with~\eqref{eq:optim-homog-3} that
\begin{equation}
  \label{eq:optim-homog-4}
  \xi^{\top} \overline{\Diff}\xi  = \xi^\top \left( \int_{\T^\dim} \Diff(q)\, \mu(q)\, dq \right)\xi - \int_{\T^\dim}\nabla w_\xi(q)^{\top} \Diff(q) \nabla w_\xi(q) \, \mu(q) \, dq.
\end{equation}

\subsubsection{Optimization of the homogenized limit}
\label{sec:optimization_homog_lim}

We now focus on the optimization with respect to~$\Diff \in \Diffset_p^{a,b}$ of the homogenized limit~$\Lambda_{\mathrm{hom}}({\Diff})$ given in~\Cref{cor:Hcvg-spectral}, under some normalization constraint:
\begin{equation}
  \label{eq:optim-homog-def}
  \Lambda_{\mathrm{hom}}^\star = \sup_{\Diff \in \Diffset_p^{a,b}}\Lambda_{\mathrm{hom}}(\Diff).
\end{equation}
This optimization problem is well posed by arguments similar to the ones used in~\Cref{subsec:well-posedness}, and using the fact that the mapping~$\Diff\mapsto\overline{\Diff}$ in~\eqref{eq:optim-homog-4} is continuous on~$\Diffset_p^{a,b}$ (relying on the continuity of the mapping~$L^{\infty}(\T^\dim,\calM_{a,b})\ni\mathcal{A}\mapsto w_i\in H^{1}(\T^d)$ where~$w_i$ is the solution of~\eqref{eq:Abar2}). In the most general case, one of the difficulties is to understand how the constraint~$\Diff \in \Diffset_p^{a,b}$ homogenizes into a constraint on~$\overline{\Diff}$. We start by analyzing a specific example of constrained set with linear constraints, which is not the constraint~$\| \Diff \|_{L^1_V} \leq 1$ with the norm defined in~\eqref{eq:norm_L^p_mu_matrices}. The interest of the linear constraint we consider is that the optimization problem becomes trivial. We next derive a similar result for the constraint~$\| \Diff \|_{L^p_V} \leq 1$, but for the one-dimensional case~$\dim=1$. Both results hold provided~$a,b>0$ are sufficiently small so that the optimal diffusion matrix automatically satisfies the pointwise lower and upper bounds involving~$a$ and~$b$.

\paragraph{Linear constraint.} We start by considering the following constraint on~$\Diff$, for a given constant matrix~$M\in\mathcal{S}_\dim^{++}$ (with~$\calS_\dim^{++}$ the set of real, symmetric, positive, definite matrices of sizes~$\dim \times \dim$):
\begin{equation}
  \label{eq:linear-constraint}
  \int_{\T^\dim}\Diff(q) \, \mu(q) \, dq = M.
\end{equation}
Note that this imposes~$\dim(\dim+1)/2$ scalar constraints on~$\Diff$ instead of only~1 scalar constraint when considering the~$L^{p}_V$ constraint~\eqref{eq:normLp}. This is the only setting where we are able to provide an explicit formula for the homogenized optimization problem in dimension $d >1$.

\begin{proposition}
\label{prop:linear_constraints}
  Consider~$M\in\mathcal{S}_\dim^{++}$, with~$M_-,M_+ > 0$ such that~$M_- \Id_\dim \leq M \leq M_+^{-1} \Id_\dim$. Fix~$a,b > 0$ such that
  \begin{equation}
    \label{eq:bounds_a_b_linear_constraint}
    0 < a \leq M_- Z, \qquad 0 < b \leq \frac{1}{M_+ Z}.
  \end{equation}
  Then, the diffusion~$\Diff^\star_\mathrm{hom}(q) = M/\mu(q)$ is a solution to the optimization problem
  \begin{equation*}
    \max \left\{ \Lambda_{\rm{hom}}(\Diff)\, \middle| \, \Diff\in L_V^{\infty}(\T^\dim,\mathcal{M}_{a,b}), \ \ \int_{\T^\dim}\Diff(q) \, \mu(q) \, dq = M \right\}.
  \end{equation*}
\end{proposition}

\begin{proof}
  In view of~\eqref{eq:optim-homog-4}, for any~$\Diff \in L_V^{\infty}(\T^\dim,\mathcal{M}_{a,b})$ satisfying~\eqref{eq:linear-constraint}, it holds
  \begin{equation*}
    \forall \xi\in\R^\dim, \qquad \xi^{\top}\overline \Diff\xi \leq \xi^{\top}\left(\int_{\T^\dim}\Diff(q) \, \mu(q) \, dq\right) \xi = \xi^{\top}M\xi.
  \end{equation*}
  Thus, for any~$\Diff\in L_V^{\infty}(\T^\dim,\mathcal{M}_{a,b})$ satisfying the constraint~\eqref{eq:linear-constraint},
  \begin{equation*}
    \inf_{u\in H^{1,0}(\T^\dim) \setminus\{0\}}\frac{\dps \int_{\T^\dim} \nabla u(q)^\top \overline{\Diff} \nabla u(q)\, dq}{\dps \int_{\T^\dim} u^2(q) \, dq} \leq \inf_{u\in H^{1,0}(\T^\dim) \setminus\{0\}}\frac{\dps \int_{\T^\dim} \nabla u(q)^\top M\nabla u(q) \, dq}{\dps \int_{\T^\dim}u^2(q) \, dq}.
  \end{equation*}
  Now, note that~$\Diff(q) = M/\mu(q)$ belongs to~$L_V^{\infty}(\T^\dim,\mathcal{M}_{a,b})$ in view of the choice~\eqref{eq:bounds_a_b_linear_constraint}, and satisfies~$\overline{\Diff} = M$ in view of~\eqref{eq:optim-homog-4} since~$w_\xi=0$ for any~$\xi\in\R^d$ by~\eqref{eq:Abar2}. This diffusion matrix therefore maximizes the limiting spectral gap in view of the above inequality, which concludes the proof.
\end{proof}

\paragraph{$L^p_V$ constraint in the one-dimensional case.}
We now restrict ourselves to the case where~$\dim=1$. In this situation,~$\Lambda_{\mathrm{hom}}(\Diff)$ is proportional to the nonnegative real number~$\overline{\Diff}$. Thus, maximizing~$\Lambda_{\mathrm{hom}}(\Diff)$ boils down to maximizing~$\overline{\Diff} \geq 0$. In the one-dimensional case, denoting~$w_\xi$ as~$w_\Diff$ (we omit the irrelevant parameter~$\xi$, but make the dependence on~$\Diff$ explicit), equation~\eqref{eq:Abar2} on~$w_{\Diff}$ can be rewritten as
\begin{equation}
  \label{eq:eq_w_D_1D}
  \left[\rme^{-V}\Diff(1+w_\Diff')\right]' = 0,
\end{equation}
and the expression~\eqref{eq:optim-homog-4} simplifies as
\begin{equation*}
  \overline{\Diff} = \int_{\T^\dim} \Diff(q)\left( 1-w'_\Diff(q)^2 \right) \mu(q) \, dq,
\end{equation*}
while the normalization constraint~$\| \Diff \|_{L^p_V} \leq 1$ simply reads
\begin{equation*}
  \int_{\T^\dim}\Diff(q)^p \rme^{- p V(q)} \, dq \leq 1.
\end{equation*}
The optimization problem~\eqref{eq:optim-homog-def} is therefore equivalent to
\begin{equation*}
  \max_{\Diff \in L^p_V(\T,\R_+)} \left\{\int_{\T^\dim} \Diff(q)\left( 1-w'_\Diff(q)^2 \right) \mu(q) \, dq \ \middle| \ a \leq \Diff \rme^{-V} \leq b^{-1}, \ \int_{\T^\dim}\Diff(q)^p \,\rme^{- p V(q)}\,dq \leq 1 \right\}.
\end{equation*}
This optimization problem admits an explicit solution provided~$a,b\leqslant 1$, as made precise in the following result.

\begin{proposition}
  \label{prop:max-homog}
  Let~$\dim=1$. Fix~$p \in [1,+\infty)$, and consider~$0 <a \le 1 \le b^{-1} < +\infty$. Then, a maximizer for the spectral gap in the homogenized limit~\eqref{eq:optim-homog-def} is
  \begin{equation}
    \label{eq:homog_diff_opt}
    \Diff^\star_{\mathrm{hom}}(q) = \rme^{V(q)}.
  \end{equation}
\end{proposition}

Note that the maximizer~$\Diff^\star_{\mathrm{hom}}$ is independent of~$p$ due to our specific choice of normalization, and that no extra normalization factor is needed.

\begin{proof}
  Equation~\eqref{eq:eq_w_D_1D} can be integrated as~$w_{\Diff}' + 1 = K \Diff^{-1} \rme^{V}$ for some constant~$K \in \R$, so that (up to an irrelevant additive constant)
  \begin{equation*}
    w_{\Diff}(q) = -q + K \int_0^q \frac{\rme^{V}}{\Diff}.
  \end{equation*}
  The constant~$K$ is chosen to ensure the periodicity of the function. This gives
  \begin{equation*}
    w_{\Diff}(q) = -q +  \left(\int_{\T} \frac{\rme^{V}}{\Diff} \right)^{-1} \int_0^q \frac{\rme^{V}}{\Diff},
  \end{equation*}
  so that, by~\eqref{eq:optim-homog-4} and~\eqref{eq:FV_def_mat_homog} (with~$\xi=1$ in this one-dimensional context, and $v=w_\xi=w_\Diff$),
  \begin{equation}
    \label{eq:D_homogenized_1d}
    \overline{\Diff} = \int_{\T}  \Diff \left[ 1 - \left(w_{\Diff}'\right)^2 \right] \mu = \int_\T \Diff\left(1 + w'_{\Diff}\right) \mu = \left(\int_\T \rme^{-V}\right)^{-1}\left(\int_\T \Diff^{-1} \, \rme^{V}\right)^{-1}.
  \end{equation}
  The maximization of the effective diffusion~$\overline{\Diff}$ with respect to~$\Diff$ therefore amounts to solving the following optimization problem:
  \begin{equation*}
    \min\left\{ \int_{\T} \Diff^{-1} \rme^{V} \ \middle| \ a \leq \Diff \rme^{-V} \leq b^{-1}, \ \int_{\T} \Diff(q)^p \, \rme^{- p V(q)} \, dq \leq 1 \right\}.
  \end{equation*}
  The Euler--Lagrange equation associated with the minimization problem with the normalization constraint only (forgetting the pointwise upper and lower bounds), assuming that the constraint is saturated at its value~1, implies that~$\Diff^{-2} \rme^{V}$ is proportional to~$\Diff^{p-1} \rme^{-p V}$, \emph{i.e.}~$\Diff$ is proportional to~$\rme^{V}$; and in fact equal to the latter function in view of the normalization constraint. The pointwise upper and lower bounds are then automatically satisfied given our choice for~$a,b$, which allows to conclude the proof.
\end{proof}

The optimal diffusion in the homogenized limit~\eqref{eq:homog_diff_opt} therefore modulates the speed of the exploration of the potential energy surface based on the configuration's energy level: the system slows down in important regions of the configuration space, where~$V$ is smaller, and tends to quickly cross low probability regions, where~$V$ is larger. Intuitively, this favors transitions between high-probability regions (modes of the distribution, corresponding to low energy regions).

The extension of this result to the multi-dimensional case~$\dim \geq 1$ is not as clear as computations cannot be easily performed analytically and we leave this open question for future work.

\subsection{Homogenization of the optimal diffusion}
\label{subsec:optim-homog}

We consider in this section the~$H$-limit of diffusion matrices which maximize the spectral gap for potentials with decreasing spatial periodicity. This corresponds to first performing the optimization of the spectral gap, then the homogenization of the associated maximizers; whereas the analysis of~\Cref{sec:periodic_homog_on_A} relied on first homogenizing the spectral gap problem, then optimizing it in~\Cref{sec:optimization_homog_lim}.

For~$k\geq 1$, define the constrained set associated with the oscillating potential~$V_{\#,k}$:
\begin{equation*}
  \Diffset_{\#,k,p}^{a,b} = \left\{\Diff\in L^{\infty}_{V_{\#,k}}(\T^\dim,\mathcal{M}_{a,b}) \,\middle|\, \int_{\T^\dim}\normF{\Diff(q)}^p \, \rme^{-p V_{\#,k}(q)} \, dq \leq 1 \right\}.
\end{equation*}
Note that, if~$\Diff \in \Diffset_{p}^{a,b}$, then~$\Diff_{\#,k} \in \Diffset_{\#,k,p}^{a,b}$, where we recall that~$\Diff_{\#,k}$ is the periodized diffusion~$\Diff_{\#,k}(q) = \Diff(k q)$. Of course, the set~$\Diffset_{\#,k,p}^{a,b}$ contains more functions than those obtained by periodization. Let us indeed emphasize that the diffusion matrices in~$\Diffset_{\#,k,p}^{a,b}$ are a priori only~$\mathbb{Z}^\dim$-periodic. For~$\Diff\in\Diffset_{\#,k,p}^{a,b}$ a~$\mathbb{Z}^\dim$-periodic diffusion, consider the spectral gap associated with the oscillating potential~$V_{\#,k}$:
\begin{equation}
  \label{eq:lambda-periodic-def}
  \Lambda^k(\Diff) = \min_{u \in H^1(\T^\dim) \setminus\{0\}} \left\{ \frac{\dps \int_{\T^\dim} \nabla u(q)^\top \Diff(q)\nabla u(q)\, \rme^{-V_{\#,k}(q)}\, dq}{\dps \int_{\T^\dim} u^2(q) \, \rme^{-V_{\#,k}(q)} \, dq} \ \middle| \ \int_{\T^\dim} u \, \rme^{-V_{\#,k}} = 0 \right\}.
\end{equation}
The optimal diffusion matrix corresponding to~\eqref{eq:lambda-periodic-def} satisfies the following optimization problem:
\begin{equation}
  \label{eq:lambda-periodic-optim}
  \Lambda^{k,\star} = \max_{\Diff \in \Diffset_{\#,k,p}^{a,b}} \Lambda^k(\Diff).
\end{equation}
Note that the minimization problem in~\eqref{eq:lambda-periodic-def} is different from the one in~\eqref{eq:lambda_per_k} since the diffusion matrix~$\Diff$ is not~$(\mathbb{Z}/k)^\dim$-periodic, but only~$\mathbb{Z}^\dim$-periodic a priori. It turns out however that the optimization problem~\eqref{eq:lambda-periodic-optim} admits a~$(\Z/k)^\dim$-periodic solution, as made precise in the following proposition, proved in~\Cref{app:lem:periodicity}.

\begin{proposition}
  \label{prop:periodicity}
  Fix~$p \in [1,+\infty)$ and consider~$a \in [0,a_\mathrm{max}]$ and~$b > 0$ such that~$ab \leq 1$. For~$k\in \N_+$, there exists~$\Diff^{k,\star} \in \Diffset_{p}^{a,b}$ such that, denoting by~$\Diff_{\#,k}^{k,\star}(q) = \Diff^{k,\star}(kq)$,
  \begin{equation*}
    \Lambda^k(\Diff_{\#,k}^{k,\star}) = \Lambda_{\#,k}(\Diff^{k,\star}) = \Lambda^{k,\star}.
  \end{equation*}
\end{proposition}
\Cref{prop:periodicity} shows that, while solving the optimization problem~\eqref{eq:lambda-periodic-optim} (well posed in application of~\Cref{thm:well-posedness-1D}), we may restrict ourselves to~$(\Z/k)^\dim$-periodic diffusion matrices.~\Cref{thm:commutation} below, proved in~\Cref{app:thm:commutation}, states that the limit~$\overline{\Lambda}^\star$ when~$k \to +\infty$ of the sequence~$(\Lambda^{k,\star})_{k \geq 1}$ is in fact equal to the limit~$\Lambda^\star_{\mathrm{hom}}$ defined in~\eqref{eq:optim-homog-def}, showing the commutation of homogenization and optimization procedures.

\begin{theorem}
  \label{thm:commutation}
  Fix~$p \in [1,+\infty)$ and consider~$a \in (0,a_\mathrm{max}]$ and~$b > 0$ such that~$ab \leq 1$.
  Then the sequence~$\left(\Lambda^{k,\star} \right)_{k\geq 1}$ converges to~$\Lambda^\star_{\mathrm{hom}}$ as~$k\to+\infty$.
\end{theorem}
~\Cref{thm:commutation} suggests that, in the homogenized limit, the optimal diffusion matrix yields the same spectral gap as its proxy obtained by solving~\eqref{eq:optim-homog-def} -- which, we recall, can be analytically written out in the simple cases discussed in~\Cref{sec:optimization_homog_lim}.

\subsection{Numerical results}
\label{subsec:homog:numerical_results}

We present two one-dimensional numerical examples illustrating the convergence towards the homogenized regime when the frequency~$k$ goes to infinity. The space discretizations of the optimization problem~\eqref{eq:lambda-periodic-optim}, the optimization algorithms, the notation and the hyperparameters are the same as in~\Cref{sec:numerical}. The mesh size used to obtain these results is refined as~$k$ increases, using~$N=200k$. Denote by~$\diff^{k,\star}$ the maximizer obtained by solving the approximation of the optimal spectral gap~\eqref{eq:lambda-periodic-optim} for a given frequency~$k\geqslant 1$ and by~$\sigma_2(\diff^{k,\star})$ the corresponding optimal value. The optimal homogenized coefficient is $\overline{\Diff}^\star_{\rm hom}=1/Z$, which is easily obtained from the formula~\eqref{eq:D_homogenized_1d} together with the expression of the optimal diffusion coefficient in the homogenized limit~\eqref{eq:homog_diff_opt}. Therefore, the value of the homogenized limit of the spectral gap is the product of the first non-zero eigenvalue of $-\Delta$ on $\T$ with $\overline{\Diff}^\star_{\rm hom}$, which yields~$\Lambda_{\mathrm{hom}}^\star=4\pi^2/Z$. 

The convergence of~$\diff^{k,\star}$ towards~$\diff^{\star}_{\rm hom}$ is illustrated in~\Cref{fig:homogenization-pic}. In particular, we represent on the upper right and left plots the optimal diffusion coefficients~$\diff^{k,\star}$ for increasing values of~$k$, along with the diffusion coefficient~$\diff^\star_\mathrm{hom}$. Note that we always observe that the diffusion~$\diff^{k,\star}$ is~$1/k$-periodic, and we only plot its trace on~$[0,1/k]$ (and stretch it on~$[0,1]$) so that optimal diffusion coefficients for various values of~$k$ can be compared: more precisely, the function plotted for the optimal diffusion coefficient for the frequency~$k\geqslant 1$ is~$q\mapsto\diff^{k,\star}_{\left\lfloor\frac{Nq}{k}\right\rfloor+1}$.

We observe that, on these simple one-dimensional examples, the optimal and homogenized diffusion coefficients coincide almost perfectly for relatively small values of the frequency, namely already for~$k=3$. We also observe that the optimal diffusion coefficient seems to vanish at two points for~$k=1$ (and for $k=2$ for the numerical example presented in~\Cref{fig:hom_cv_sin_sin}), whereas it is uniformly positive definite for~$k \geq 3$.

\begin{figure}
  \centering
  \begin{subfigure}[t]{0.49\linewidth}
    \includegraphics[width=\linewidth]{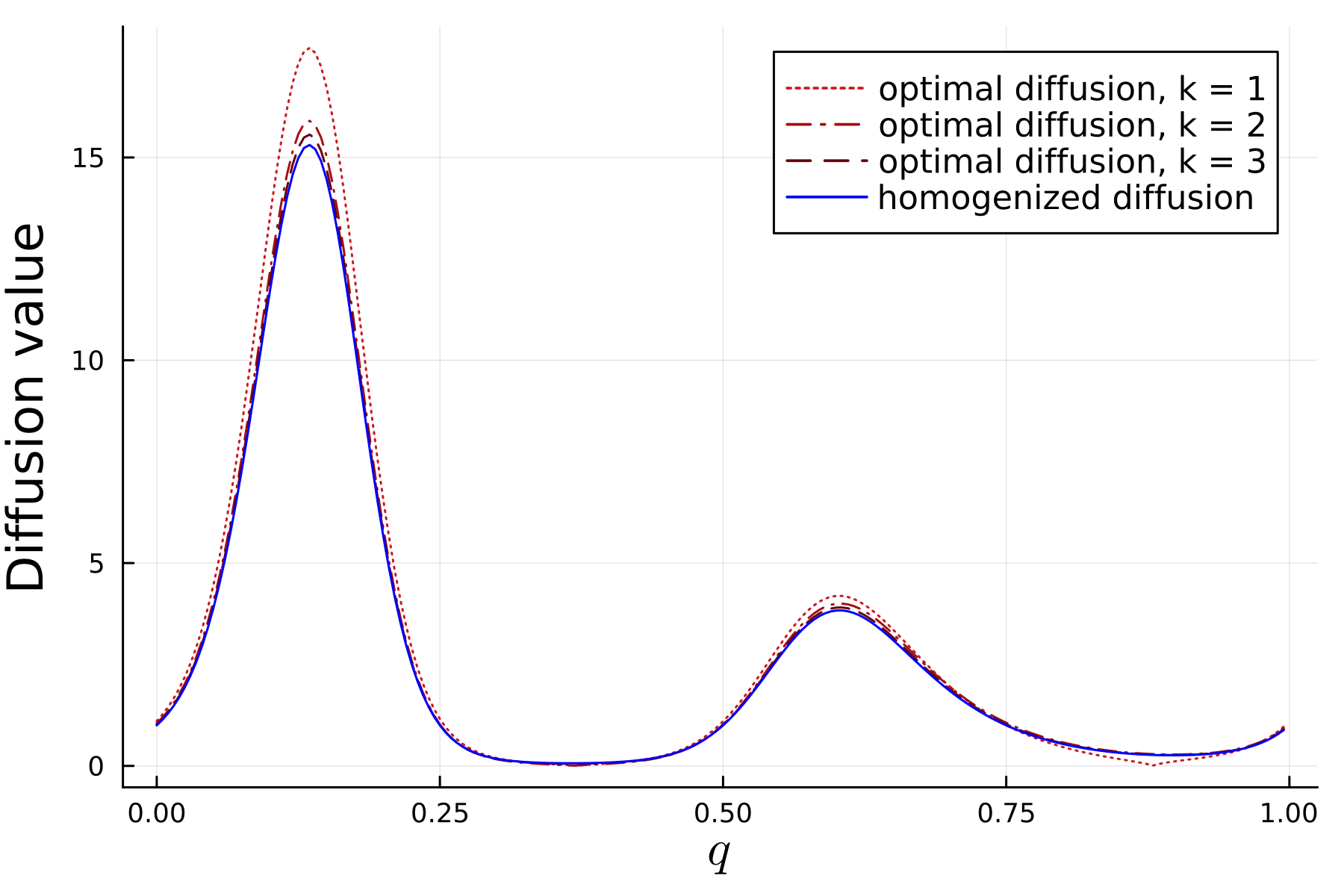}
    \includegraphics[width=\linewidth]{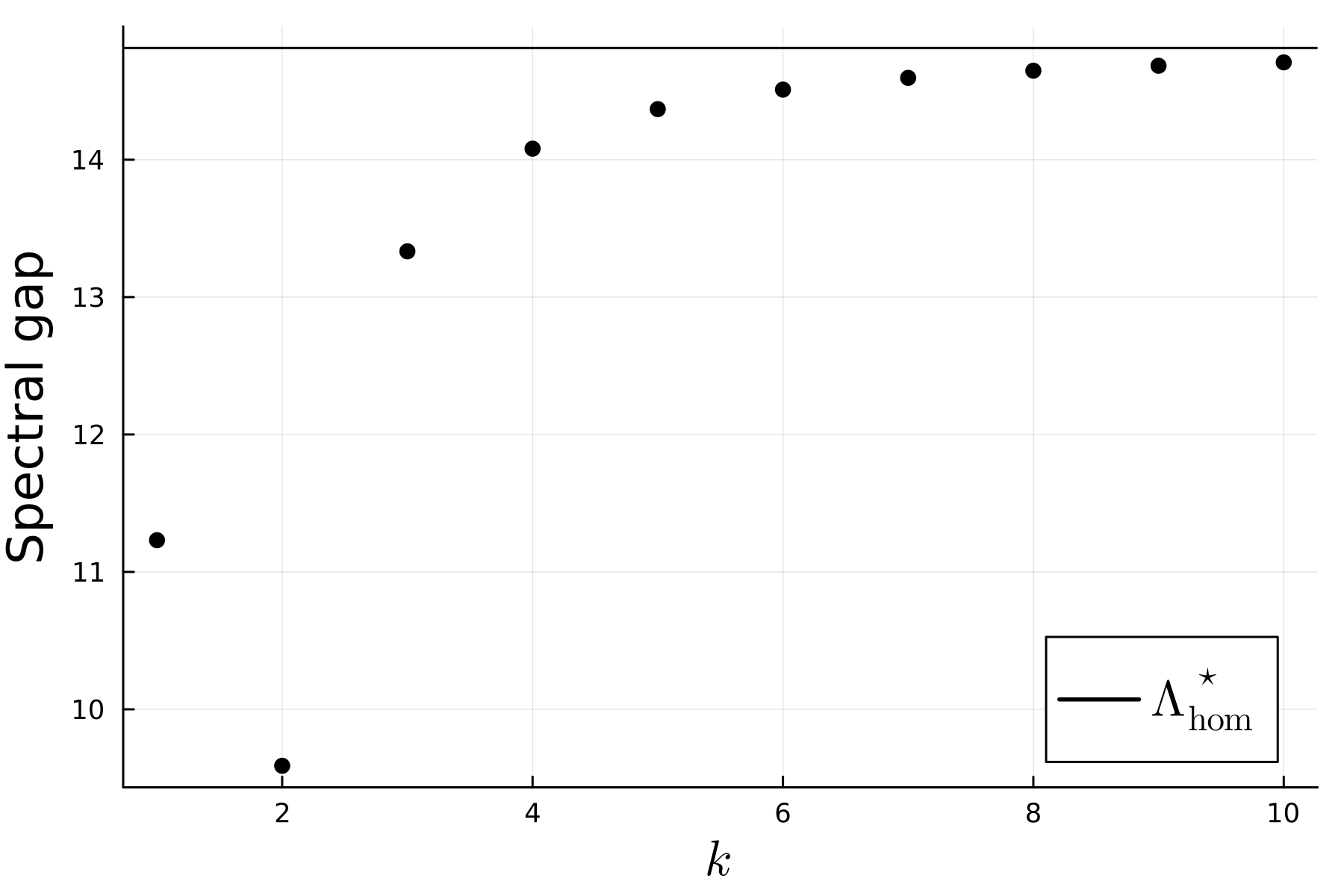}
    \caption{\label{fig:hom_cv_sin_sin}Potential:~$V(q) =   \sin(4\pi q) (2 + \sin(2\pi q))$.}
  \end{subfigure}
  \hfill
  \begin{subfigure}[t]{0.49\linewidth}
    \includegraphics[width=\linewidth]{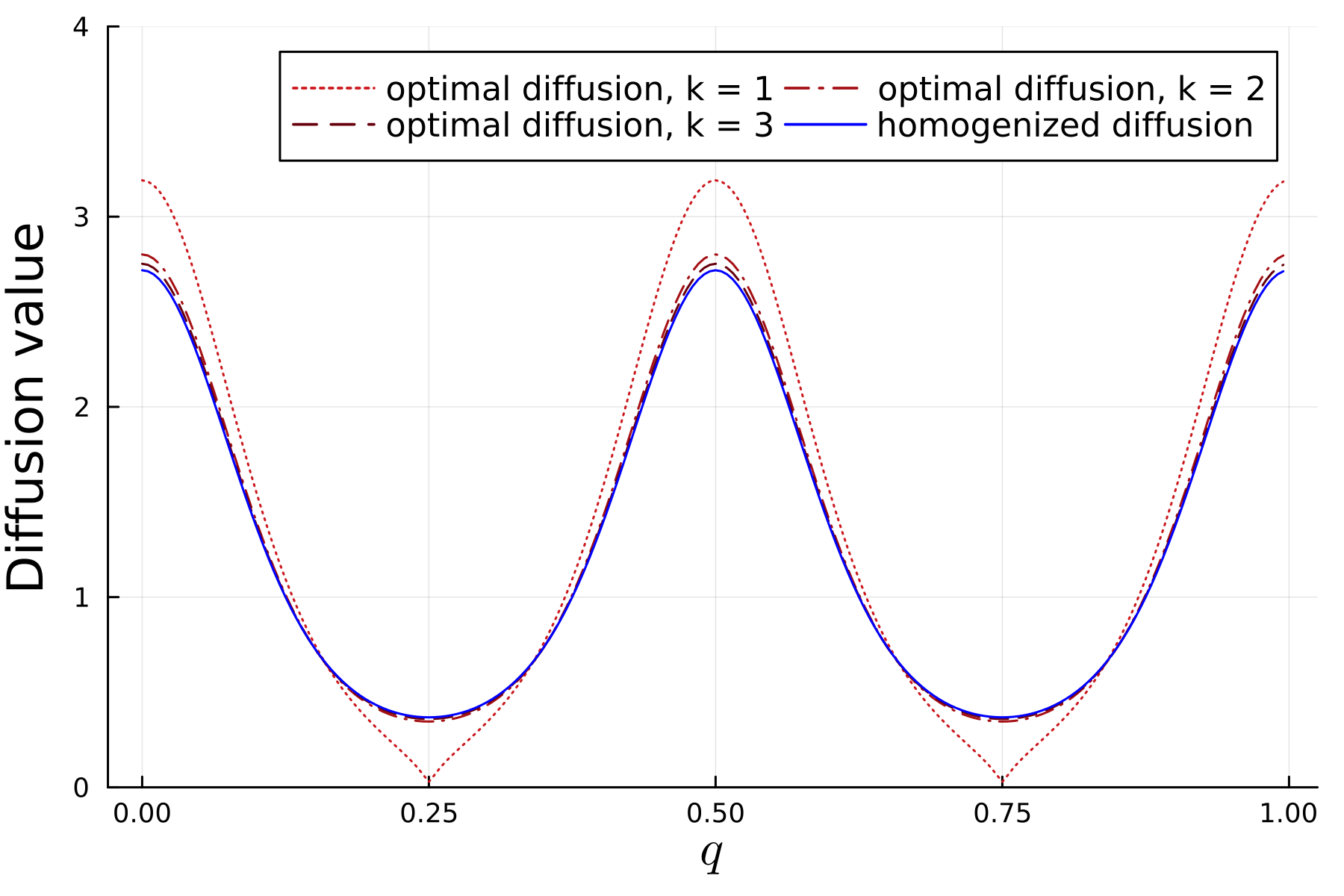}
    \includegraphics[width=\linewidth]{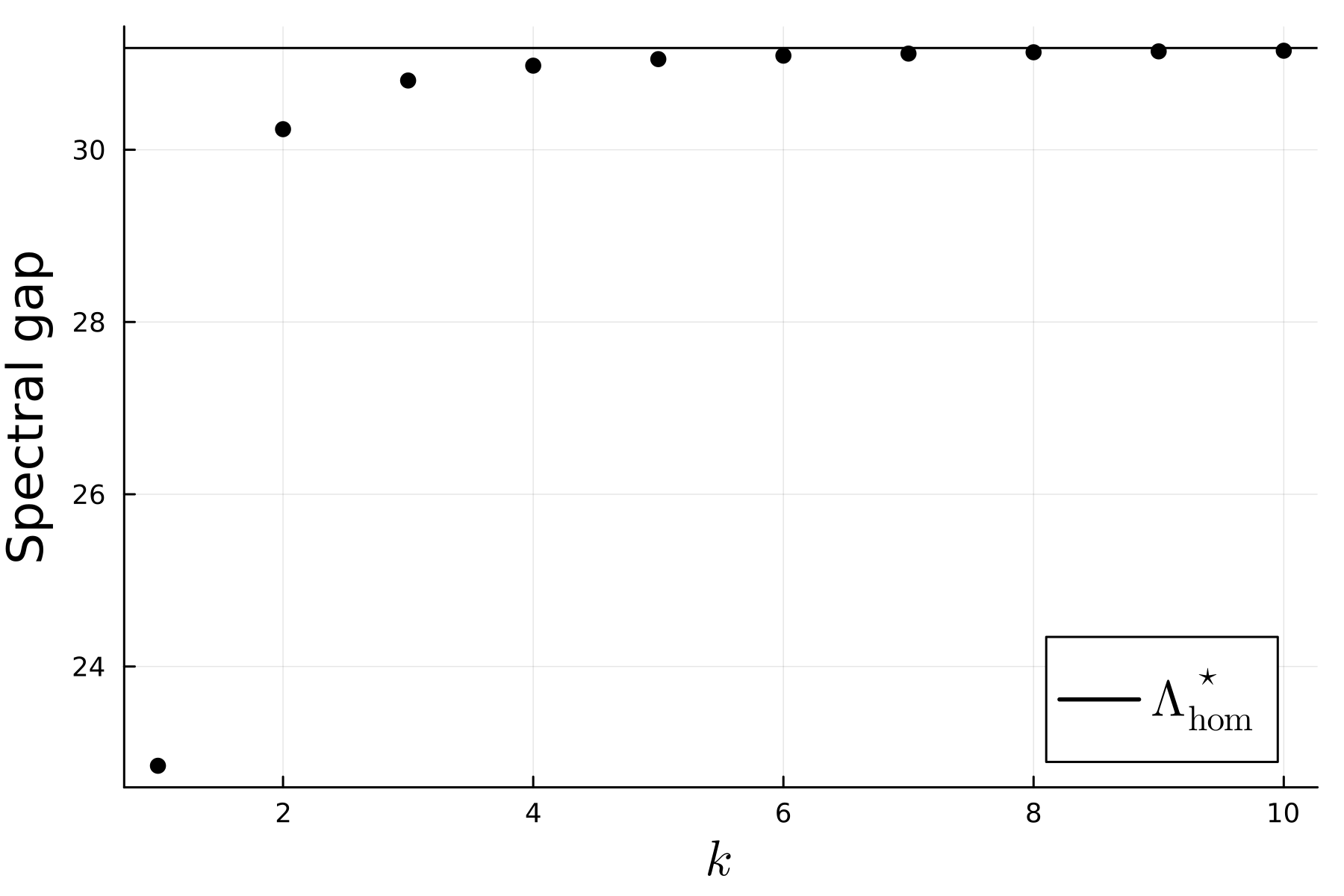}
    \caption{Potential:~$V(q) =   \cos(4\pi q)$. }
  \end{subfigure}
  \caption{Comparison of the two procedures: homogenize then optimize (leading to the solid blue line `homogenized diffusion'), versus optimize then homogenize (series of curves indexed by increasing values of~$k$). The left and right columns correspond to two different target distributions. The figures in the second row illustrate the convergence of the sequence~$\left(\sigma_2(\diff^{k,\star})\right)_{k\geqslant 1}$ towards~$\Lambda_{\mathrm{hom}}^{\star}$.}
  \label{fig:homogenization-pic}
\end{figure}

\section{Application to sampling algorithms}
\label{sec:sampling}

We present in this section an application of the optimization procedure to accelerate sampling algorithms. We focus on the Random Walk Metropolis--Hastings (RWMH) algorithm~\cite{metropolis,hastings} to unbiasedly sample from a Boltzmann--Gibbs target distribution. As we show in~\Cref{subsec:Donsker}, this algorithm provides a consistent discretization of the continuous dynamics~\eqref{eq:dynamics_mult}. More precisely, we show that the RWMH with proposals constructed using a space dependent diffusion converges, in the limit~$\Delta t \to 0$, to a Langevin diffusion with space dependent diffusion given by the proposal variance. Notice that we do not discuss the ergodicity of such dynamics, see~\cite{livingstone2021} for results in that direction. We next present in~\Cref{subsec:sampling} numerical experiments demonstrating that optimizing the diffusion leads to a more efficient sampling of the target measure. 
We numerically compare three variants of RWMH corresponding to three different choices for the variance of the proposal. The first variant is a simple RWMH algorithm with constant proposal variance; the two others have a space-dependent proposal variance~$\Delta t \, \Diff$, with~$\Diff$ given by (i) the pre-computed optimal diffusion matrix and (ii) the explicit homogenized limit of~\Cref{sec:optimization_homog_lim}.

\begin{remark}[On the choice of the dynamics]
We consider a simple Random-Walk algorithm with moves of size~$\mathrm{O}(\sqrt{\Delta t})$ rather than a Metropolization of a Euler--Maruyama discretization of the overdamped Langevin dynamics (the SmartMC method~\cite{RDF78} in molecular dynamics, known as MALA~\cite{RobertsTweedie1996} in computational statistics). Indeed, the latter dynamics have rejection probabilities of order~$\Delta t^{3/2}$ for constant diffusion matrices, but a much larger rejection probability of order~$\sqrt{\Delta t}$ for genuinely position dependent diffusion matrices~\cite{FS17} (see also~\Cref{lem:acceptance-rate-DL} in~\Cref{sec:technical_lemmas_pathwise_cv} below). In contrast, the rejection probability is of order~$\sqrt{\Delta t}$ in all cases for Random Walk Metropolis--Hastings. We therefore choose RWMH in order to have algorithms with rejection probabilities of the same order of magnitude, which allows for a fairer comparison of the dynamics. Let us mention that numerical schemes leading to a rejection probability scaling as~$\Delta t^{3/2}$ for Metropolizations of overdamped Langevin dynamics with position dependent diffusion matrices have recently been proposed in~\cite{LSS22}. We also refer to~\cite{DML23} for discussions about efficient discretizations of the overdamped Langevin dynamics with position dependent diffusions, using time-rescaling transforms.
\end{remark}

\subsection{Consistency of the Random Walk Metropolis--Hastings algorithm}
\label{subsec:Donsker}

We describe in this section the RWMH algorithm we consider, and then provide consistency results of the method in the limit~$\Delta t \to 0$.

\paragraph{Description of the Metropolis algorithm.}
The Random Walk Metropolis--Hastings algorithm~\cite{metropolis,hastings} is obtained by first proposing moves constructed from the current configuration by adding a Gaussian increment, and then accepting or rejecting the proposal move according to a Metropolis criterion. More precisely, starting from the current configuration~$q^i_{\Delta t}$ for~$i \geq 0$ and using a discretization time step~$\Delta t>0$, a new configuration is proposed as
\begin{equation}
  \label{eq:proposal_move_RWMH}
  \widetilde{q}^{i+1}_{\Delta t} = q^i_{\Delta t} + \sqrt{2\Delta t}\,\Diff^{1/2}(q^i_{\Delta t})\,G^{i+1},
\end{equation}
where~$(G^i)_{i\geq 1}$ is a sequence of independent and identically distributed (i.i.d.) normal random variables. The transition kernel~$\mathscr{T}(q^i_{\Delta t},q^{i+1}_{\Delta t})$ associated with this transition has density
\begin{equation*}
  \mathscr{T}(q,q') = \left(\frac{1}{4\pi\Delta t}\right)^{\dim/2}\det\Diff(q)^{-1/2} \exp\left(-\frac{1}{4\Delta t} (q'-q)^\top \Diff(q)^{-1} (q'-q)\right).
\end{equation*}
Note that the reverse move associated with the proposal~\eqref{eq:proposal_move_RWMH} is
\begin{equation*}
  q^i_{\Delta t} = \widetilde{q}^{i+1}_{\Delta t} - \sqrt{2\Delta t}\,\Diff^{1/2}(\widetilde{q}^{i+1}_{\Delta t})\, \widetilde{G}^{i+1},
  \qquad
  \widetilde{G}^{i+1} = \Diff^{-1/2}(\widetilde q^{i+1}_{\Delta t})\Diff^{1/2}(q^i_{\Delta t})\,G^{i+1}.
\end{equation*}
This allows to compute the Metropolis ratio as
\begin{equation}
  \label{eq:def_R_dt}
  \begin{aligned}
    R_{\Delta t}(q^i_{\Delta t},G^{i+1}) & = \min\left\{ 1, \frac{\mu(\widetilde{q}^{i+1}_{\Delta t})\mathscr{T}(\widetilde{q}^{i+1}_{\Delta t},q^i_{\Delta t})}{\mu(q^i_{\Delta t})\mathscr{T}(q^i_{\Delta t},\widetilde{q}^{i+1}_{\Delta t})} \right\}\\
    & = \min\left\{1,\left(\frac{\det \Diff(q^i_{\Delta t})}{\det \Diff(\widetilde q^{i+1}_{\Delta t})}\right)^{1/2} \mathrm{e}^{- [ V(\widetilde q^{i+1}_{\Delta t})-V(q^i_{\Delta t})] - (|\widetilde G^{i+1}|^2-|G^{i+1}|^2)/2} \right\}.
  \end{aligned}
\end{equation}
Accepting the proposal with probability~$R_{\Delta t}(q^i_{\Delta t},G^{i+1})$ using a sequence of i.i.d. random variables~$(U^i)_{i\geq 1}$ with uniform law on~$[0,1]$, independent of~$(G^i)_{i\geq 1}$, then leads to setting the new configuration as
\begin{equation}
  \label{eq:RWMH-variance}
  q^{i+1}_{\Delta t} = q^i_{\Delta t} + \sqrt{2\Delta t}\,\Diff^{1/2}(q^i_{\Delta t})\,G^{i+1}\mathbbm{1}_{\{U^{i+1}\leq R^n(q^i_{\Delta t},G^{i+1})\}}.
\end{equation}

\paragraph{Weak error estimates and pathwise weak convergence.}
The following result, similar to~\cite[Lemma~4]{FS17}, implies that the numerical scheme~\eqref{eq:RWMH-variance} is weakly consistent. The proof is omitted since it exactly follows the proof of~\cite[Lemma~4]{FS17} as a particular case when the drift term is 0.

\begin{proposition}
  \label{prop:weak_consistency_rwmh}
  For any~$\varphi \in \calC^\infty(\T^\dim)$, there exists~$K \in \mathbb{R}_+$ and~$\Delta t^\ast > 0$ such that, for any~$0<\Delta t\leqslant \Delta t^{\ast}$,
  \begin{equation*}
    \sup_{q \in \T^\dim}\left\lvert
      \frac{\mathbb{E}^q\left( \varphi(q_{\Delta t}^1) \right) - \varphi(q) - \Delta t \cLD \varphi(q)}{\Delta t^{3/2}},
    \right\rvert\leqslant K,
  \end{equation*}
  where~$\mathbb{E}^q$ denote the expectation with respect to all realizations of the Markov chain~\eqref{eq:RWMH-variance} starting from~$q_{\Delta t}^0 = q$.
\end{proposition}

We next state a weak consistency result of the Metropolis scheme with the baseline continuous dynamics~\eqref{eq:dynamics_mult} at the level of the paths of the process. Define the rescaled, linearly interpolated continuous-time process~$(Q^{\Delta t}_t)_{t \geq 0}$ by
\begin{equation}
  \label{eq:process}
  Q^{\Delta t}_t =\left (\left \lceil \frac{t}{\Delta t} \right \rceil - \frac{t}{\Delta t}\right)q_{\Delta t}^{\left \lfloor \frac{t}{\Delta t} \right \rfloor} + \left(\frac{t}{\Delta t} - \left \lfloor \frac{t}{\Delta t} \right \rfloor\right)q_{\Delta t}^{\left \lceil \frac{t}{\Delta t} \right \rceil},
\end{equation}
with~$Q_0^{\Delta t} = q_0$ given. We can then state the following pathwise consistency result, proved in~\Cref{app:thm:Donsker}.

\begin{theorem}
  \label{thm:Donsker}
  Assume that~$\Diff\in \calC^2(\T^\dim,\mathcal{S}_\dim^{++})$, and denote by~$P_{\Delta t}$ the law of the process~$(Q^{\Delta t}_t)_{t \geq 0}$ on~$\calC^0([0,+\infty), \T^\dim)$, starting from a given initial condition~$Q_0^{\Delta t} = q_0$. Then,~$P_{\Delta t}$ converges weakly to~$P$ as~$\Delta t \to 0$, where~$P$ is the law of the unique solution to the stochastic differential equation
  \begin{equation}
    \label{eq:dynamics_donsker}
    dQ_t = \left(- \Diff(Q_t)\nabla V(Q_t) + \div\Diff(Q_t) \right) dt + \sqrt{2}\, \Diff^{1/2}(Q_t) \, dW_t,
  \end{equation}
  where~$(W_t)_{t \geq 0}$ is a standard~$\dim$-dimensional Brownian motion, with initial condition $Q_0=q_0$.
\end{theorem}

Note that we assume that~$\Diff\in\mathcal{C}^2(\T,\mathcal{S}_d^{++})$ for some technical estimates, in particular those of~\Cref{lem:acceptance-rate-DL} in~\Cref{sec:technical_lemmas_pathwise_cv} below.

\subsection{Sampling experiments}
\label{subsec:sampling}

In this section, we illustrate the benefits of using a position dependent diffusion coefficient for sampling a Boltzmann--Gibbs target distribution. We consider a linear interpolation of the optimal diffusion coefficient~$\diff^{\star}$ (introduced in~\Cref{sec:numerical}) obtained with~$N=I=1000$ and~$a=b=0$, the diffusion coefficient for the homogenized problem~$\diff^{\star}_{\mathrm{hom}}$ (see~\Cref{sec:homog} and equation~\eqref{eq:homog_diff_opt}), and the constant diffusion coefficient~$\diff_{\rm cst}$ (see~\Cref{sec:numerical}).

We choose the one-dimensional double well potential~$V(q)=\sin(4\pi q)(2+\sin(2\pi q))$ defined on~$\T$ (the associated Gibbs distribution is displayed in~\Cref{fig:res_1_left}). The potential exhibits a global minimum on~$\T$ at~$x_0\approx 0.36544$ and a local minimum at~$x_1 \approx 0.89714$.

\paragraph{Sample trajectories.}  We show in~\Cref{fig:MC_samples} sample trajectories following the discretization~\eqref{eq:RWMH-variance} for each diffusion coefficient. In order to compare the simulations, the same Gaussian variables~$(G^{i})_{i\geqslant1}$ are used in~\eqref{eq:proposal_move_RWMH}. They are run with time step~$\Delta t=10^{-4}$ for~$N_{\rm it}=10^6$ iterations, with the same initial condition~$q_0=0$. This time step is chosen so that the rejection probability in the Metropolis--Hastings procedure is of order 5\%; see~\Cref{tab:MH_rejection_probability} for the values of the rejection probabilities in each case, as well as~\Cref{fig:transition_time_MH_ratio} for a study of the scaling of the rejection probability as a function of the time step. We observe that the dynamics using a constant diffusion coefficient is stuck longer in wells, whereas the dynamics using either the optimal or homogenized diffusion coefficient transitions faster between metastable states.

\paragraph{Mean Squared Displacement and effective diffusion.} We next compute the mean squared displacement (MSD) averaged over~$N_{\mathrm{sim}}=10000$ realizations of the dynamics for the various optimal diffusion coefficients. Each realization starts at position~$q_0=0$, and the physical simulation time is set to~$10$ units. Therefore, the MSD at iteration~$n\geqslant1$ is
\begin{equation*}
  (\mathrm{MSD})_n=\frac{1}{N_{\mathrm{sim}}}\sum_{i=1}^{N_{\mathrm{sim}}}\left\lvert q^{i,n}_{\Delta t}-q_0\right\rvert^{2},
\end{equation*}
where~$q_{\Delta t}^{i,n}$ is the (unperiodized) position of the~$i$-th realization of the dynamics at iteration~$n$ for a given time step~$\Delta t>0$. The results for~$\Delta t=10^{-7}$ are presented in~\Cref{fig:msd_iterations}. The MSD is computed every~$1000$ steps. The ribbons represent 95\% confidence intervals, quantified using the variance estimated over the~$N_{\rm sim}$ independent realizations. We observe more displacement when using either the optimal or the homogenized diffusion coefficient, which confirms that the associated dynamics diffuse more over the configuration space than the dynamics using a constant diffusion coefficient.

Observe that the MSD depends linearly on time: this is an illustration of the celebrated Enstein formula, the linear coefficient being twice the \emph{effective diffusion} (we refer for instance to~\cite{FS17} and references therein for more theoretical details). The effective diffusion is a function of the diffusion coefficient~$\Diff$, and is maximal when choosing the optimal homogenized diffusion~$\Diff_{\rm hom}^\star$ defined by~\eqref{eq:D_homogenized_1d}. In fact, for this particular choice, the effective diffusion has a simple analytical expression, given by~$1/Z\approx0.375$ where~$Z$ is the normalization constant introduced in~\eqref{eq:mu}. In~\Cref{fig:msd_slopes}, we compute approximations of the effective diffusion: we run MSD computations for various time steps~$\Delta t\in\left\lbrace 10^{-7},10^{-6},\dots,10^{-1}\right\rbrace$, and estimate the effective diffusion using a linear regression on the time interval~$[5,10]$. These simulations show that we obtain significantly larger MSD whatever the time step~$\Delta t$ when choosing any of the two optimal diffusions rather than the constant diffusion. The relative positions of the curves stay the same, even for time steps that are not close to the continuous-in-time limit. Lastly, the effective diffusion for the optimal homogenized diffusion for~$\Delta t=10^{-7}$ is equal to~$0.375$, which coincides with the reference value $1/Z$ for the continuous-in-time dynamics. We checked that these results were stable when increasing the simulation time up to~$100$ units.

\begin{figure}
  \begin{center}
    \begin{subfigure}[t]{0.45\textwidth}
      \centering \includegraphics[draft=false,width=\linewidth]{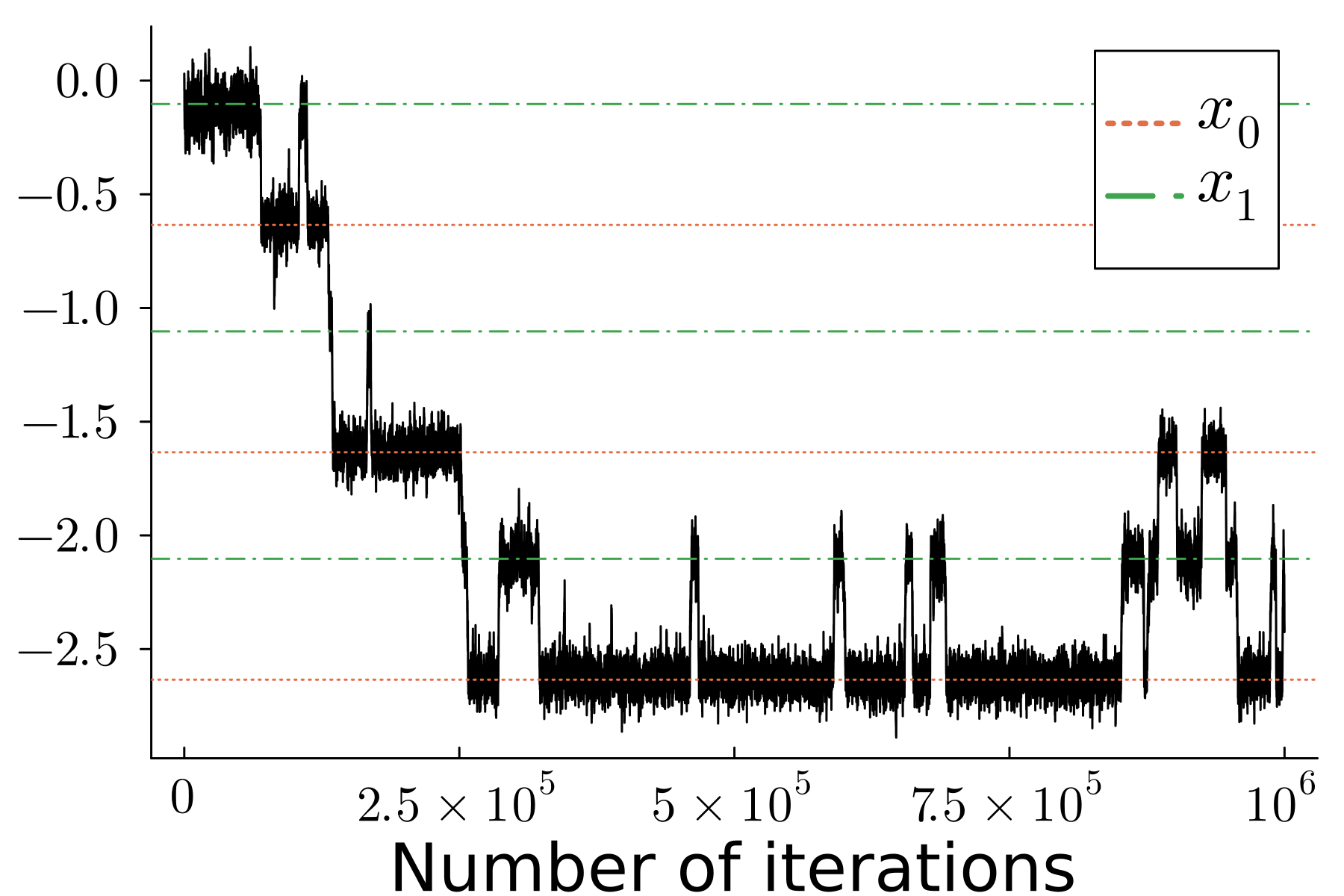}
      \caption{Constant diffusion coefficient.}
    \end{subfigure}
    \hfill
    \begin{subfigure}[t]{0.45\textwidth}
      \centering \includegraphics[draft=false,width=\linewidth]{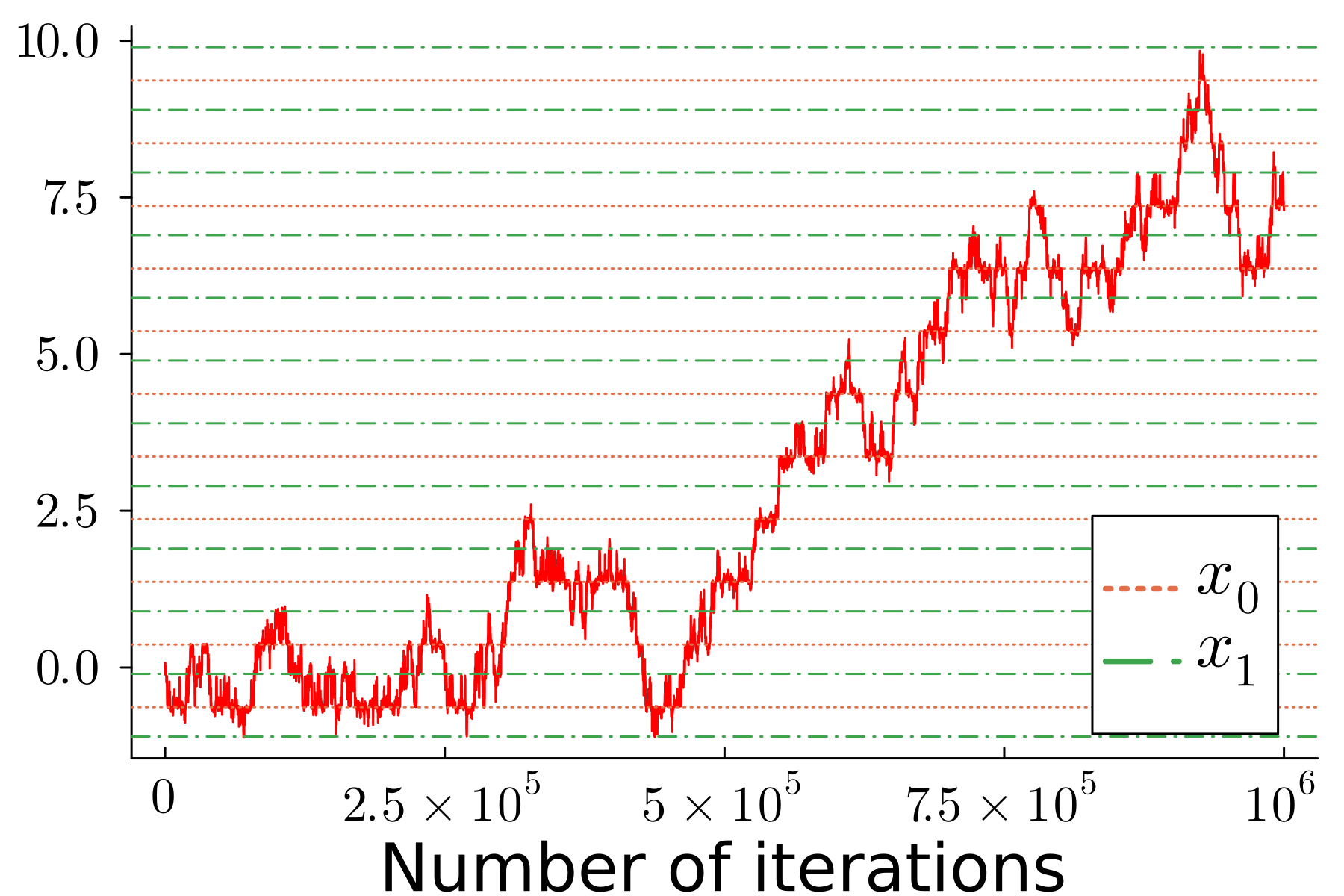}
      \caption{Optimal diffusion coefficient,~$a=0$.}
    \end{subfigure}
    \hfill
    \begin{subfigure}[t]{0.45\textwidth}
      \centering \includegraphics[draft=false,width=\linewidth]{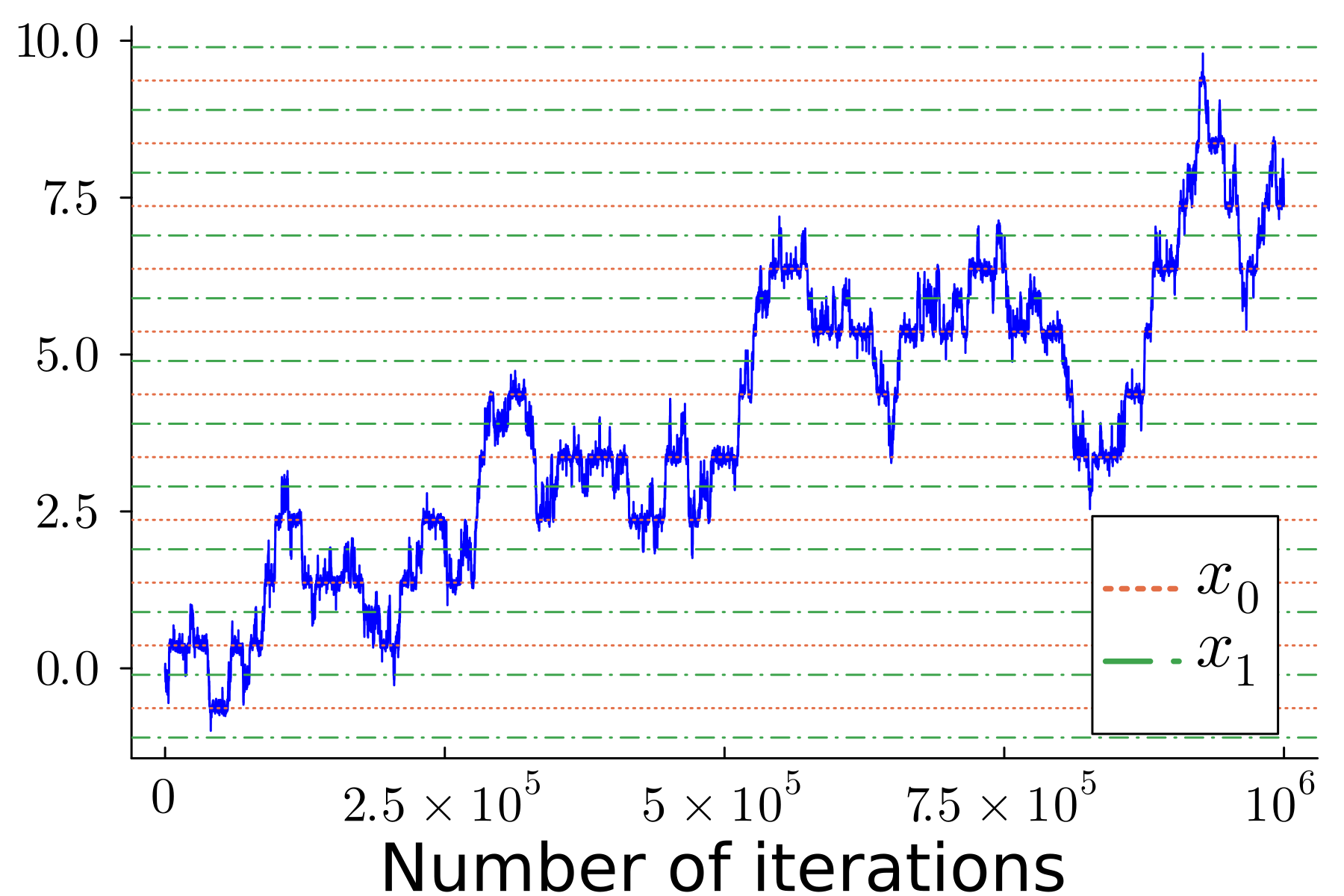}
      \caption{Homogenized diffusion coefficient.}
    \end{subfigure}
    \caption{\label{fig:MC_samples} Sample trajectories of the RWMH algorithm for a periodic, one-dimensional potential with double well~$V(q)=\sin(4\pi q)(2 + \sin(2\pi q))$. Notice that we plot the trajectories over the unfolded configuration space $\R$ instead of the torus $\T=\R/\Z$. The two wells corresponding to~$x_0\approx0.36544$ and~$x_1\approx0.89714$ are indicated using dashed and dotted horizontal lines in each periodic cell.}
  \end{center}
\end{figure}

\begin{figure}
    \centering
    \begin{subfigure}[t]{0.45\textwidth}
      \centering \includegraphics[draft=false,width=\linewidth]{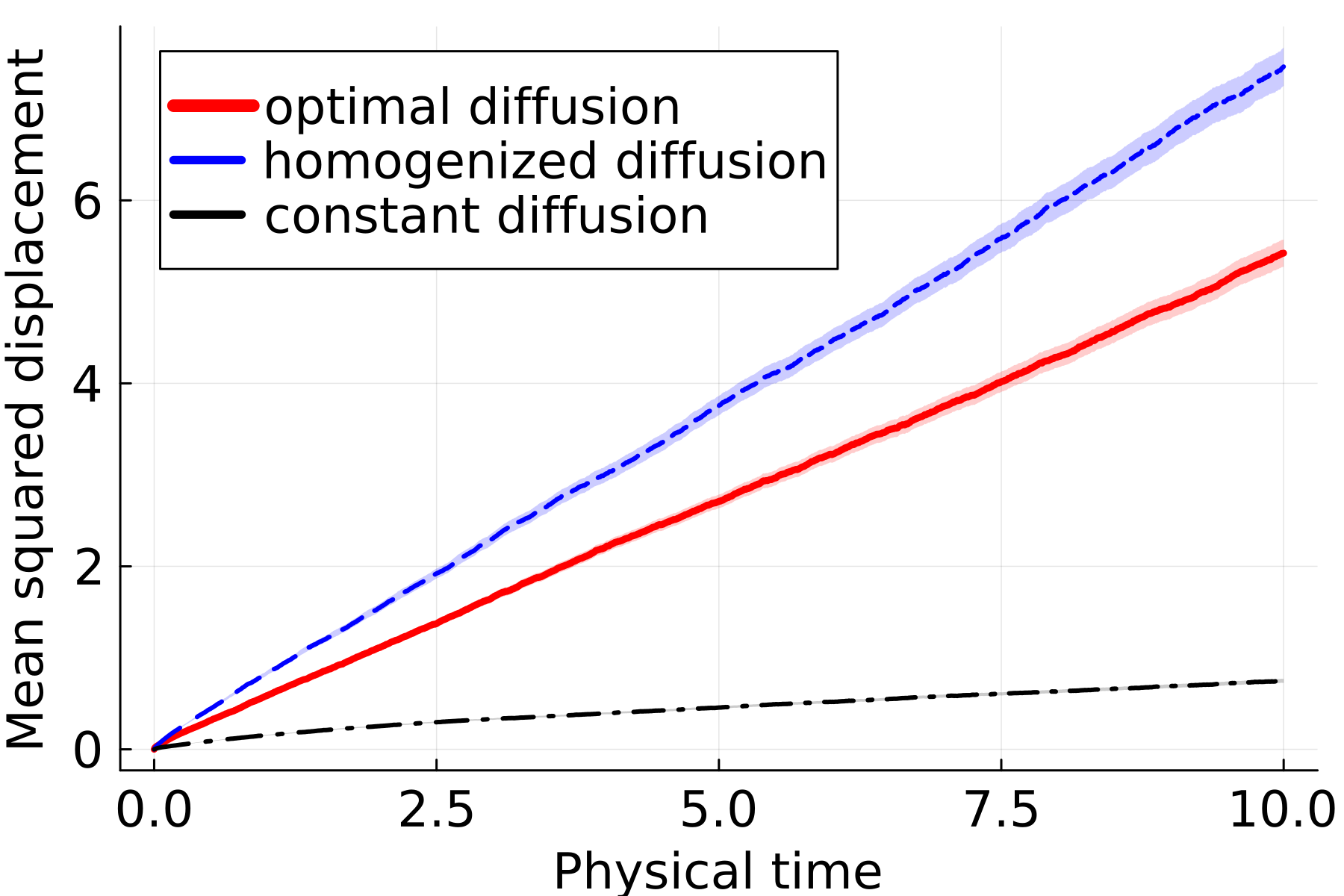}
      \caption{Mean Squared Displacement over time.}\label{fig:msd_iterations}
    \end{subfigure}
    \hfill
    \begin{subfigure}[t]{0.45\textwidth}
      \centering \includegraphics[draft=false,width=\linewidth]{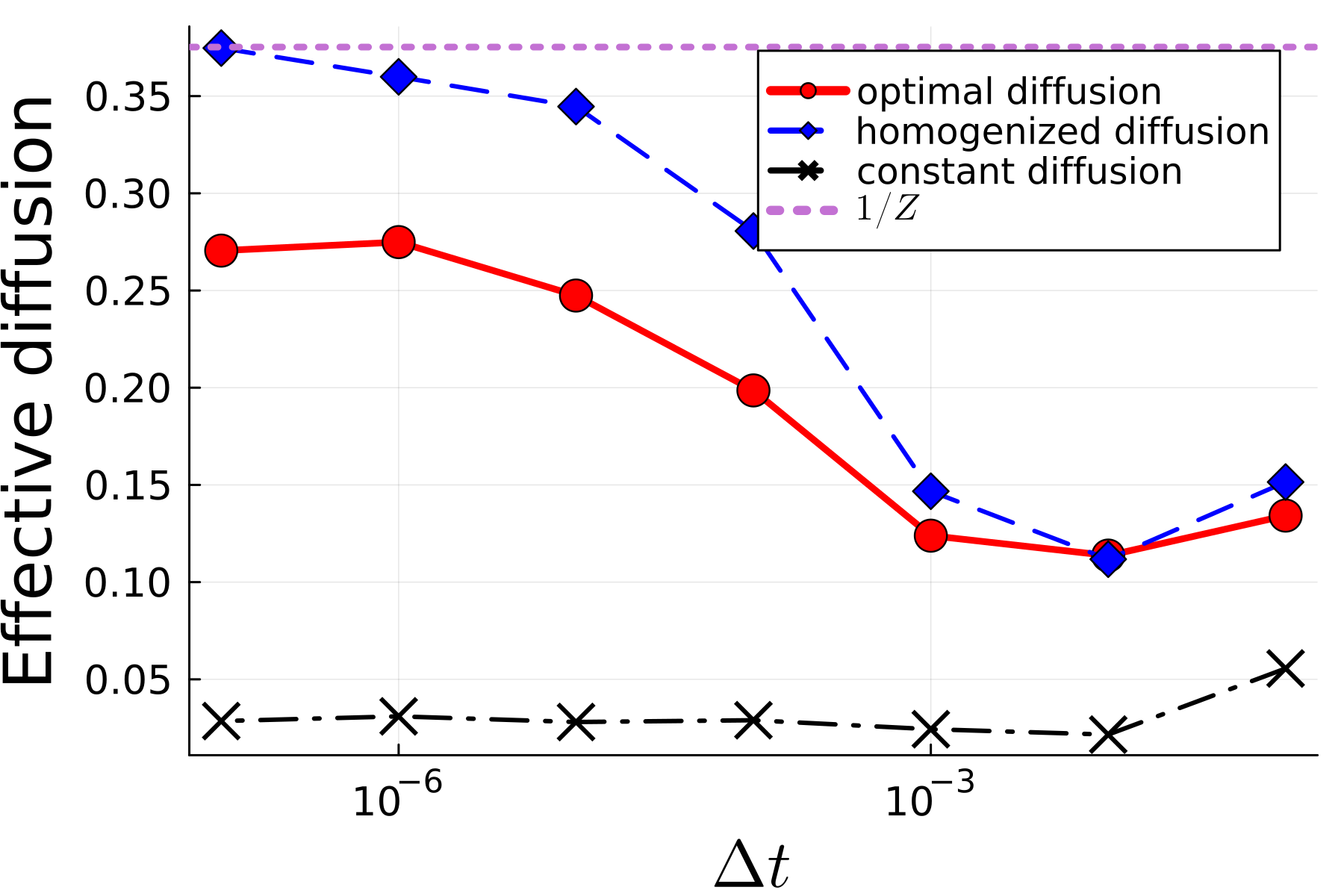}
      \caption{Effective diffusion for various time steps~$\Delta t$.}\label{fig:msd_slopes}
    \end{subfigure}
    \caption{Mean Squared Displacement and effective diffusion for the various diffusions coefficients.}
    \label{fig:msd}
\end{figure}

\begin{table}
  \centering
  \begin{tabular}{c|cccc}
    \toprule
    Diffusion coefficient           & Constant & Homogenized & Optimal \\\midrule
    Spectral gap                    & 0.81     & 10.6        & 11.2    \\\midrule
    M--H rejection probability (\%) & 3.72     & 4.00        & 6.42    \\
    \bottomrule
  \end{tabular}
  \caption{Spectral gap and Metropolis--Hastings rejection probability for the simulations in~\Cref{fig:MC_samples}.}
  \label{tab:MH_rejection_probability}
\end{table}

\paragraph{Convergence of the empirical law towards the Gibbs measure.}

In order to further compare the various diffusion coefficients and their respective effects on the sampling efficiency, we numerically compute the weighted~$L^{2}$ error ($\chi^2$ divergence) between the empirical distribution of the discretized process and the Gibbs measure, see~\eqref{eq:chi_square_disc} below. As pointed out in~\eqref{eq:cvg}, for the limiting continuous-in-time dynamics~\eqref{eq:dynamics_donsker}, this error converges exponentially fast towards~0, at a rate given by the spectral gap of the generator of the  dynamics.

More precisely, we set~$\Delta t=10^{-6}$ and consider an initial distribution~$\mu^{0}$ made of~$N_{\rm samples}=10^{5}$ samples generated independently and uniformly over~$[0,1]$:
\begin{equation*}
  \mu^{0}=\frac{1}{N_{\rm samples}}\sum_{i=1}^{N_{\rm samples}}\delta_{q_{\Delta t}^{i,0}},\qquad q_{\Delta t}^{i,0}\sim\mathcal{U}([0,1]).
\end{equation*}
Each sample is then independently updated using the RWMH algorithm presented in~\Cref{subsec:Donsker}, so that the empirical distribution~$\mu^{n}$ at time~$n\Delta t$ is
\begin{equation*}
  \mu^{n}=\frac{1}{N_{\rm samples}}\sum_{i=1}^{N_{\rm samples}}\delta_{q_{\Delta t}^{i,n}}.
\end{equation*}

In practice, the weighted~$L^{2}$ error appearing in the left-hand side of~\eqref{eq:cvg} is approximated by discretizing the interval~$[0,1)$ into~$N_{\rm bins}$ bins of equal sizes: for~$1\leqslant k\leqslant N_{\rm bins}$, we denote by~$B_k=[(k-1)/N_{\rm bins}, k/N_{\rm bins})$ the $k$-th bin. The weighted~$L^{2}$ error is then approximated as
\begin{equation}\label{eq:chi_square_disc}
  \left(
  \frac{1}{N_{\rm bins}}\sum_{k=1}^{N_{\rm bins}}\frac{\left(\widehat{\mu}_k^{n}-\widehat{\mu}_k\right)^{2}}{\widehat{\mu}_k}
  \right)^{1/2}
\end{equation}
where~$\widehat{\mu}_k^{n},\widehat{\mu}_k$ are proportional to the fraction of samples in the~$k$-th bin, and normalized as
\begin{equation*}
  \frac{1}{N_{\rm bins}}\sum_{k=1}^{N_{\rm bins}}\widehat{\mu}_k^{n}=\frac{1}{N_{\rm bins}}\sum_{k=1}^{N_{\rm bins}}\widehat{\mu}_k=1.
\end{equation*}
In particular, in our case,
\begin{equation*}
  \forall 1\leqslant k\leqslant N_{\rm bins},\qquad
  \widehat{\mu}_k=N_{\rm bins}\rme^{-V\left(\tfrac{k-1/2}{N_{\rm bins}}\right)} \Big/ \sum_{j=1}^{N_{\rm bins}}\rme^{-V\left(\frac{j-1/2}{N_{\rm bins}}\right)},
\end{equation*}
approximates the Gibbs distribution, and
\begin{equation*}
  \forall 1\leqslant k\leqslant N_{\rm bins},\qquad
  \widehat{\mu}^{n}_k=\frac{N_{\rm bins}}{N_{\rm samples}}\left\lvert\left\lbrace 1\leqslant i\leqslant N_{\rm samples},\, q_{\Delta t}^{i,n}\in B_k\right\rbrace\right\rvert,
\end{equation*}
approximates the empirical distribution~$\mu^{n}$, where~$\left\lvert A\right\rvert$ denotes the cardinal of a set~$A$. We run the RWMH algorithm for~$N_{\rm it}=10^{6}$ iterations and use~$N_{\rm bins}=100$ to approximate the weighted~$L^{2}$ error.

The error between the empirical distribution and the Gibbs distribution as a function of the physical time is presented in~\Cref{fig:L2_error}. This confirms that the use of the optimal diffusion or the homogenized coefficient greatly improves the convergence of the law of the process towards equilibrium. The exponential rate of convergence, corresponding to the dotted lines, is quickly reached. Recall that the spectral gaps for the optimal, homogenized and constant diffusion coefficients are respectively~11.23, 10.58 and 0.81. This is in agreement with the numerical results of~\Cref{fig:L2_error} which show a slightly lower error for the optimal diffusion coefficient than for the homogenized diffusion coefficient, and a significantly lower error than for the constant diffusion coefficient.

\begin{figure}
  \centering
  \includegraphics[width=0.8\textwidth]{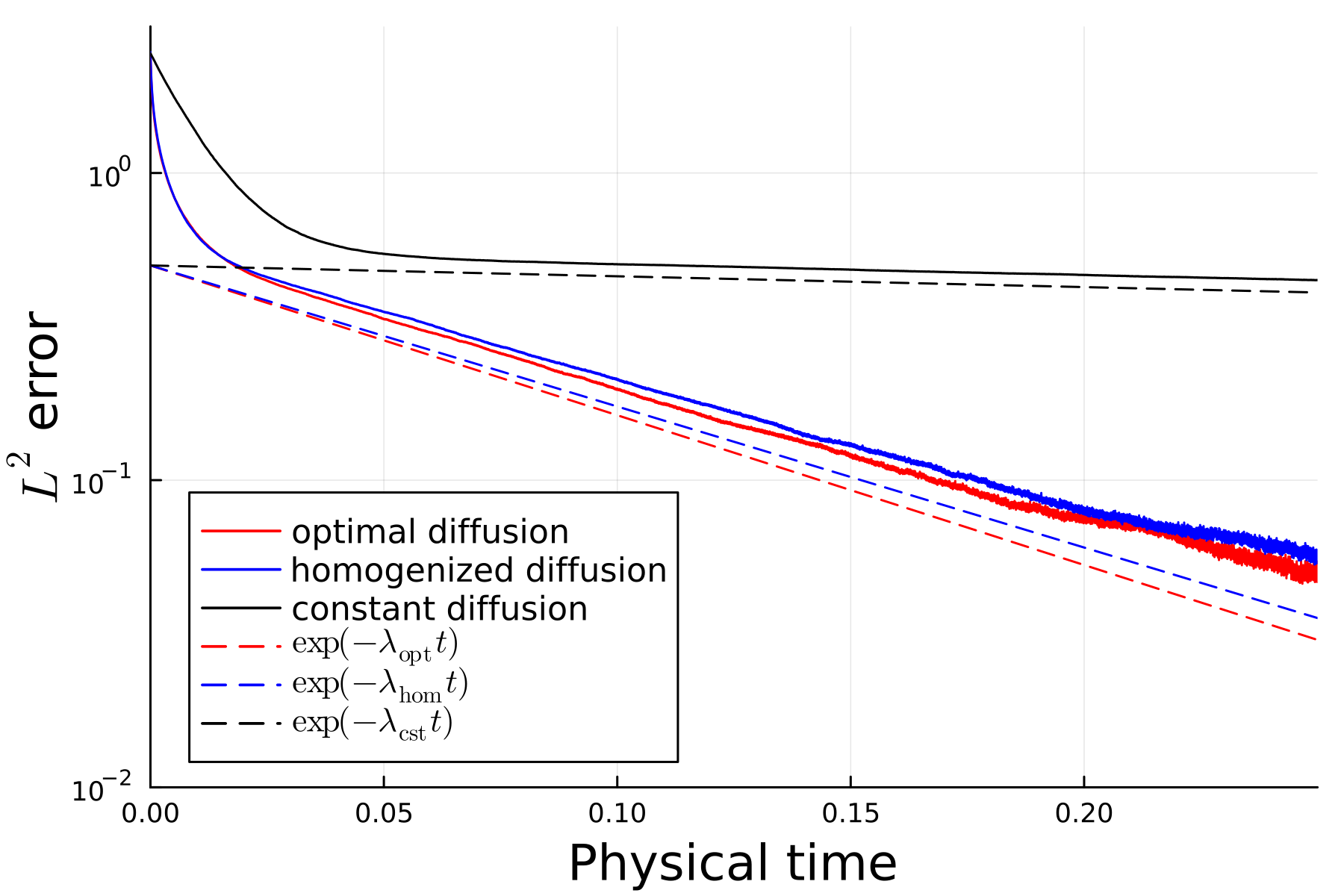}
  \caption{Weighted~$L^{2}$ error (in logarithmic scale) between the empirical distribution and the Gibbs distribution over time for various diffusion coefficients.}
  \label{fig:L2_error}
\end{figure}

\paragraph{Spatial analysis of the rejection probability of the Metropolis--Hastings procedure.}

We compute the average rejection probability of the Metropolis--Hastings procedure for a fixed time step~$\Delta t=0.01$. We discretize the torus with a set of points~$q^i=i/1000$ for~$0\leqslant i<1000$. For each point~$q^i$, we propose $10^{5}$ moves following the rule~\eqref{eq:proposal_move_RWMH}, and count how many proposals are accepted when performing the Metropolis--Hastings procedure, \emph{i.e.}~computing~\eqref{eq:def_R_dt} and comparing its value to a random number uniformly distributed over~$[0,1]$.

We represent on~\Cref{fig:rejection_probabilities}, for various diffusion coefficients, the three following functions: the target distribution, the diffusion coefficient and the average rejection probability. We discarded the error bars as they were too small to be even noticed in the plot.

For the constant diffusion coefficient, for which the results are presented in~\Cref{fig:rejection_probabilities_constant_diffusion}, the proposals are most likely rejected when~$q$ is at the bottom of a well. When using either the optimal diffusion coefficient or the homogenized diffusion coefficient, the behavior is the opposite: the proposals are most likely accepted when~$q$ is located at the bottom of a well. This is shown in~\Cref{fig:rejection_probabilities_homogenized_diffusion,fig:rejection_probabilities_optimal_diffusion}. Notice that for the case of the optimal diffusion, higher rejection probabilities are observed where the optimal diffusion coefficient almost vanishes. This is because of the ratio of determinant factors appearing in the denominator of~\eqref{eq:def_R_dt}: even though both terms are of similar orders, they are almost zero, which causes numerical instabilities. We plot in~\Cref{fig:rejection_probabilities_various_lower_bounds} the average rejection probabilities when using the optimal diffusion coefficients obtained with the optimization procedure and imposing various lower bounds~$a\in\left\lbrace 0.2,0.4,0.6,0.8,1\right\rbrace$. Of course, using a positive lower bound $a$ leads to uniformly positive diffusion coefficients, which is desirable to get stable and ergodic samplers.

\begin{figure}
  \centering
  \begin{subfigure}[t]{0.45\textwidth}
    \centering
    \includegraphics[width=\linewidth,draft=false]{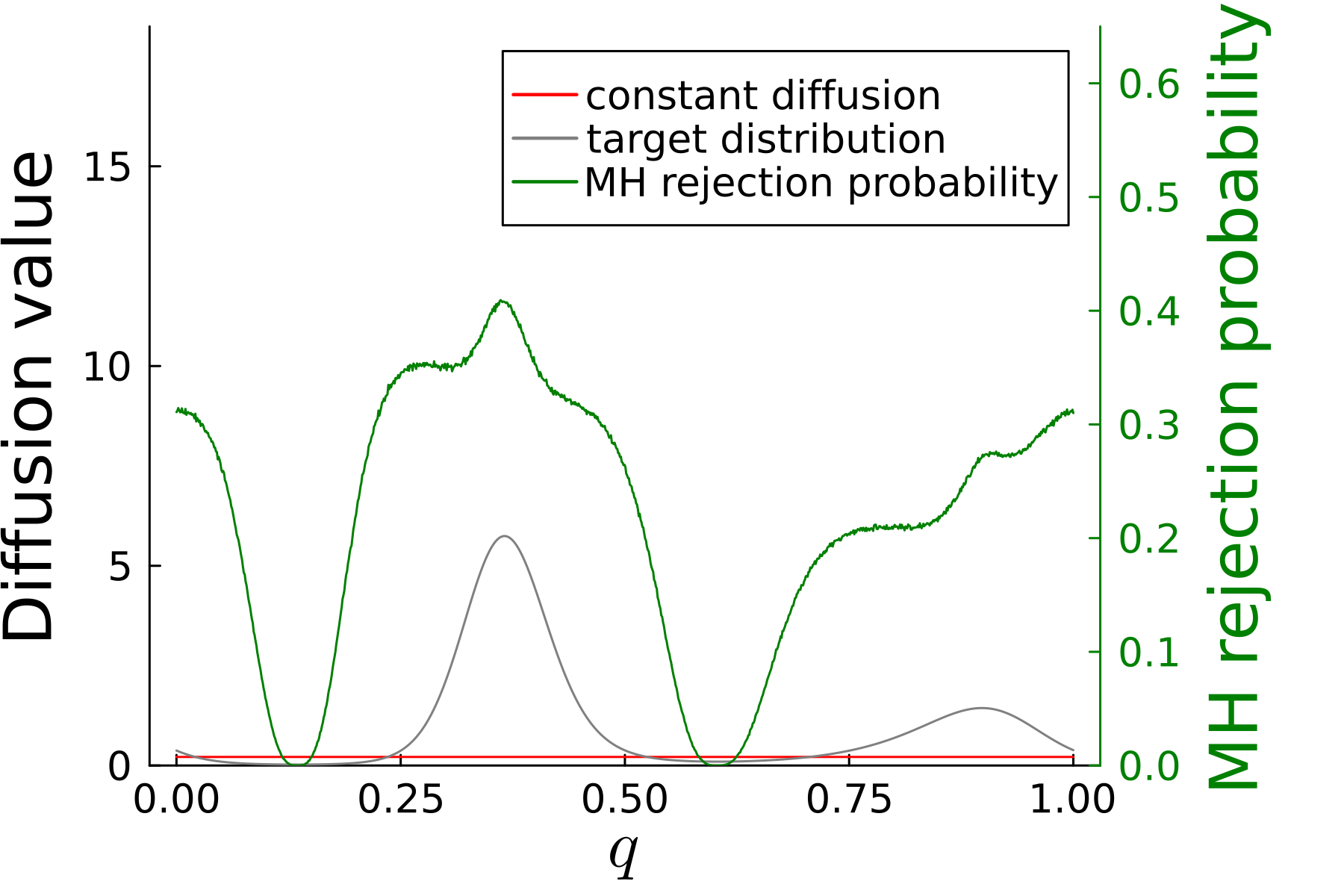}
    \caption{\label{fig:rejection_probabilities_constant_diffusion}
    Constant diffusion coefficient.}
  \end{subfigure}
  \hfill
  \begin{subfigure}[t]{0.45\textwidth}
    \centering
    \includegraphics[width=\linewidth,draft=false]{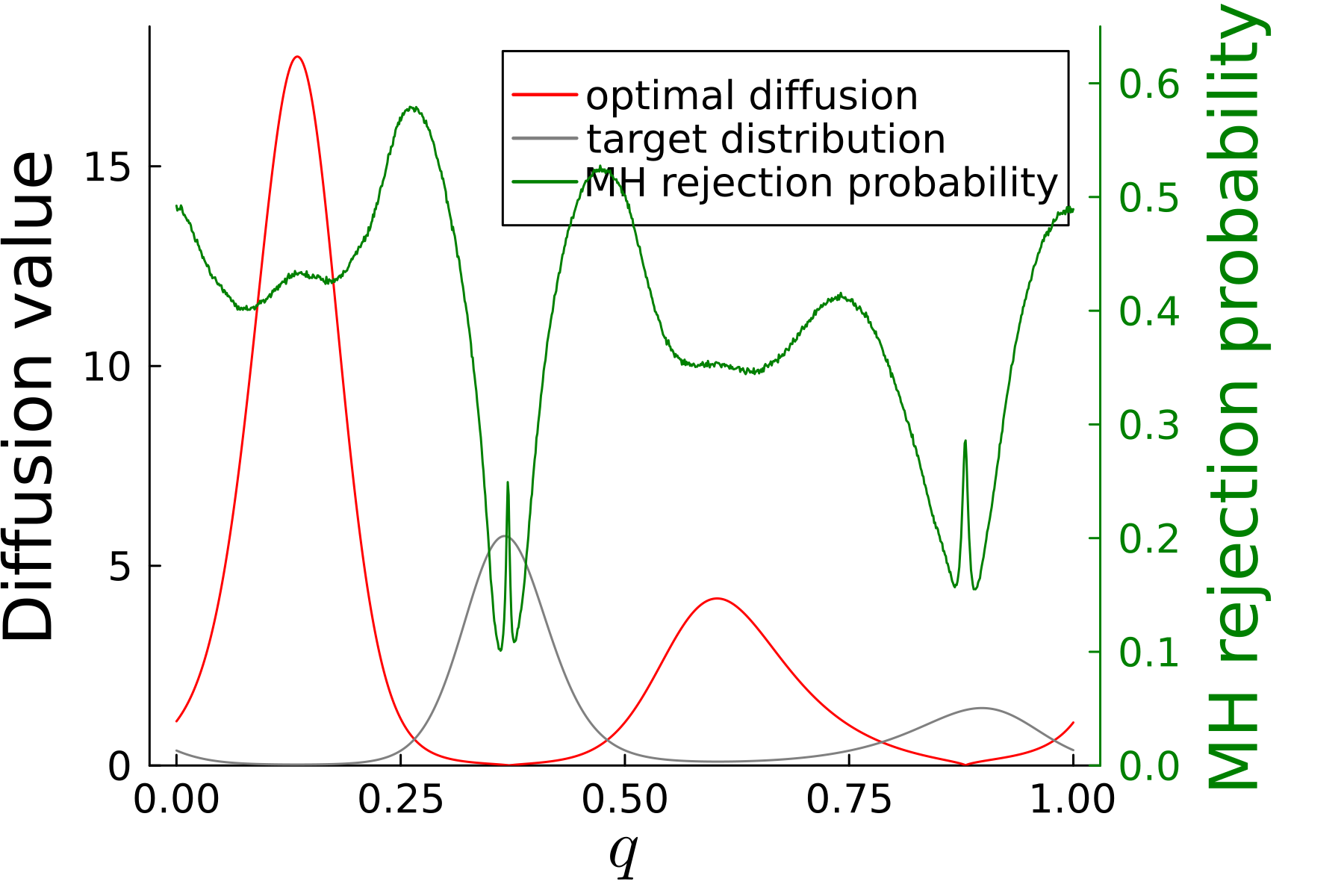}
    \caption{\label{fig:rejection_probabilities_optimal_diffusion}
    Optimal diffusion coefficient,~$a=0$.}
  \end{subfigure}
  \hfill
  \begin{subfigure}[t]{0.45\textwidth}
    \centering
    \includegraphics[width=\linewidth,draft=false]{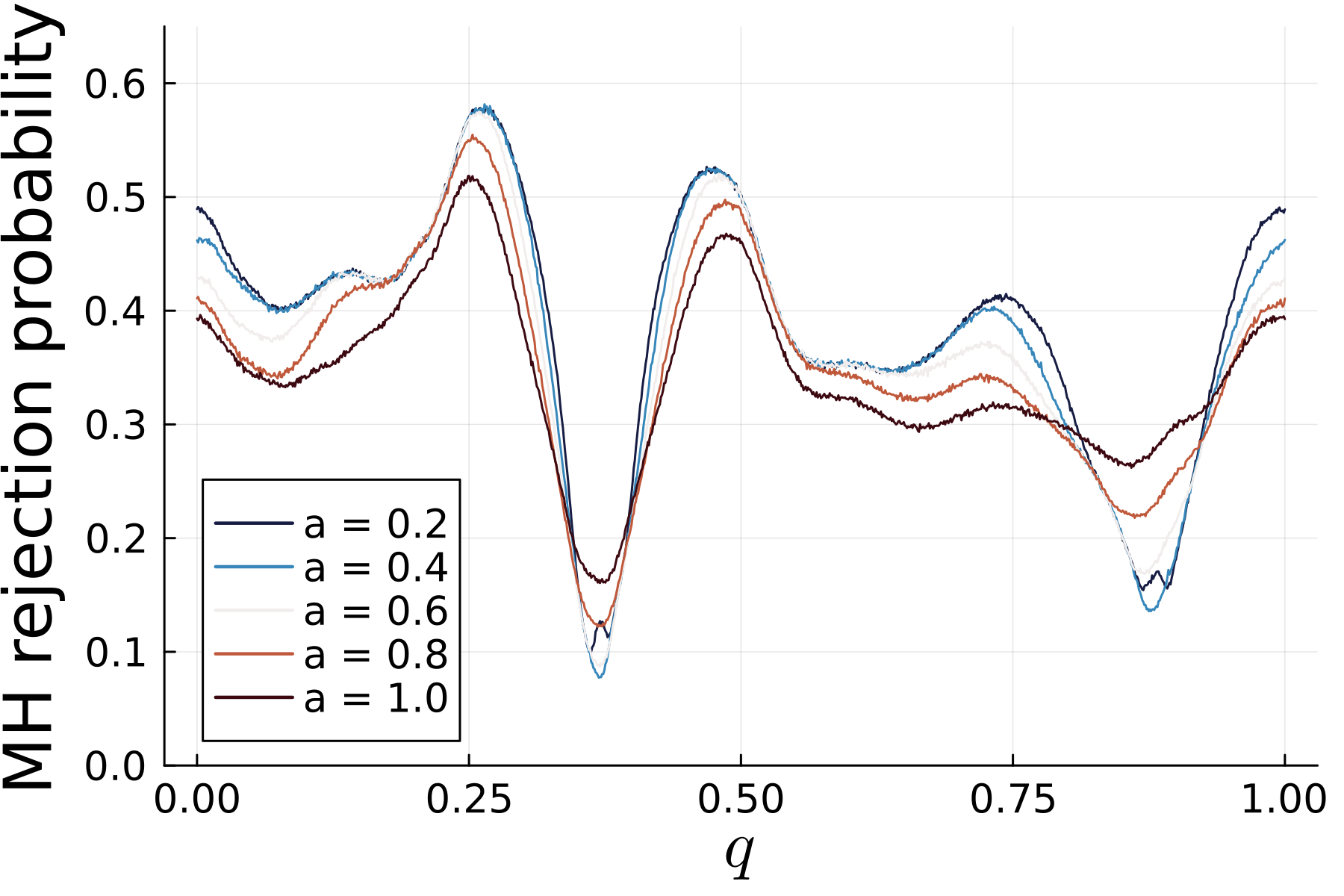}
    \caption{\label{fig:rejection_probabilities_various_lower_bounds}
    Optimal diffusion coefficient: influence of the lower bound~$a$ on the MH rejection probability.}
  \end{subfigure}
  \hfill
  \begin{subfigure}[t]{0.45\textwidth}
    \centering
    \includegraphics[width=\linewidth,draft=false]{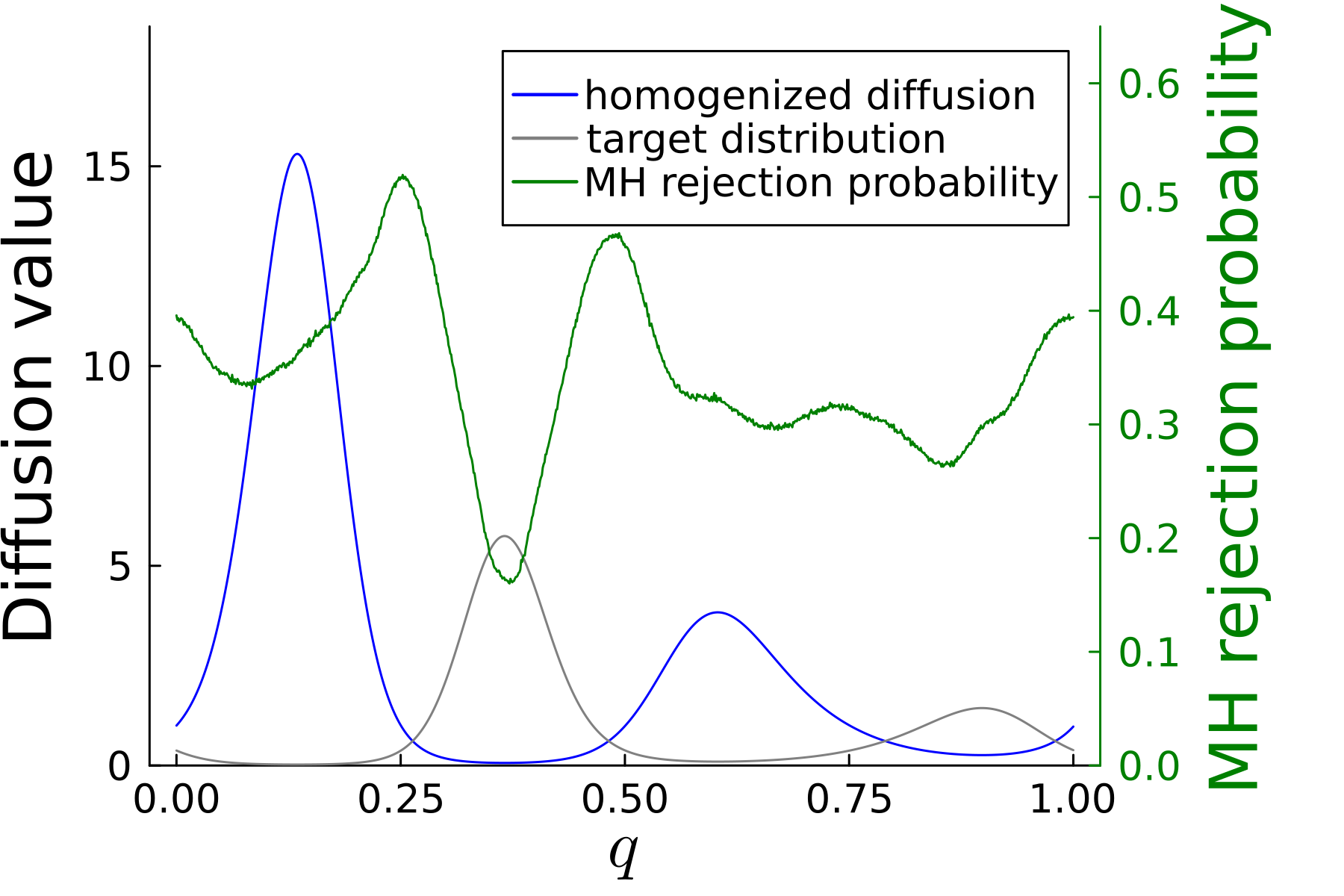}
    \caption{Homogenized diffusion coefficient.}
    \label{fig:rejection_probabilities_homogenized_diffusion}
  \end{subfigure}
  \caption{Spatial analysis of the Metropolis--Hastings rejection probability for the periodic, one-dimensional double well potential~$V(q)=\sin(4\pi q)(2 + \sin(2\pi q))$.}
  \label{fig:rejection_probabilities}
\end{figure}

\paragraph{Transition times.} We compute the mean transition time from the deepest well to any of its two nearby periodized copies. More precisely, we consider the initial condition~$q^0=x_0\approx 0.36544$ and compute the average physical time it takes for the system to cross~$x_0-1$ or~$x_0+1$. This average is computed over~$N_\mathrm{transitions}=10^5$ transitions. We show in~\Cref{fig:transition_times_pic} the mean transition time with respect to the time step~$\Delta t$, as well as the rejection probability of the Metropolis--Hastings procedure. The rejection probability due to the Metropolis--Hastings procedure scales as~$\mathrm{O}\left(\sqrt{\Delta t}\right)$; see~\cite[Section 4.7]{FS17} and~\Cref{lem:acceptance-rate-DL} in~\Cref{sec:technical_lemmas_pathwise_cv} below. We see that the mean transition time between the two metastable states is much smaller when using either the optimal diffusion coefficient or the homogenized diffusion coefficient than with the constant diffusion coefficient. For a time step~$\Delta t=5\times 10^{-5}$, the numerical values of the mean transition times when using~$\diff^{\star}$,~$\diff_{\mathrm{hom}}^{\star}$ and~$\diff_{\rm cst}$ are respectively~2.37, 1.77 and~17.78. This shows that the optimal diffusion coefficients obtained by our numerical procedure or the homogenized diffusion coefficient can be used to faster sample the configurational space, especially in the case of metastable dynamics.

\begin{figure}
  \centering
  \begin{subfigure}[t]{0.45\textwidth}
    \includegraphics[width=\linewidth,draft=false]{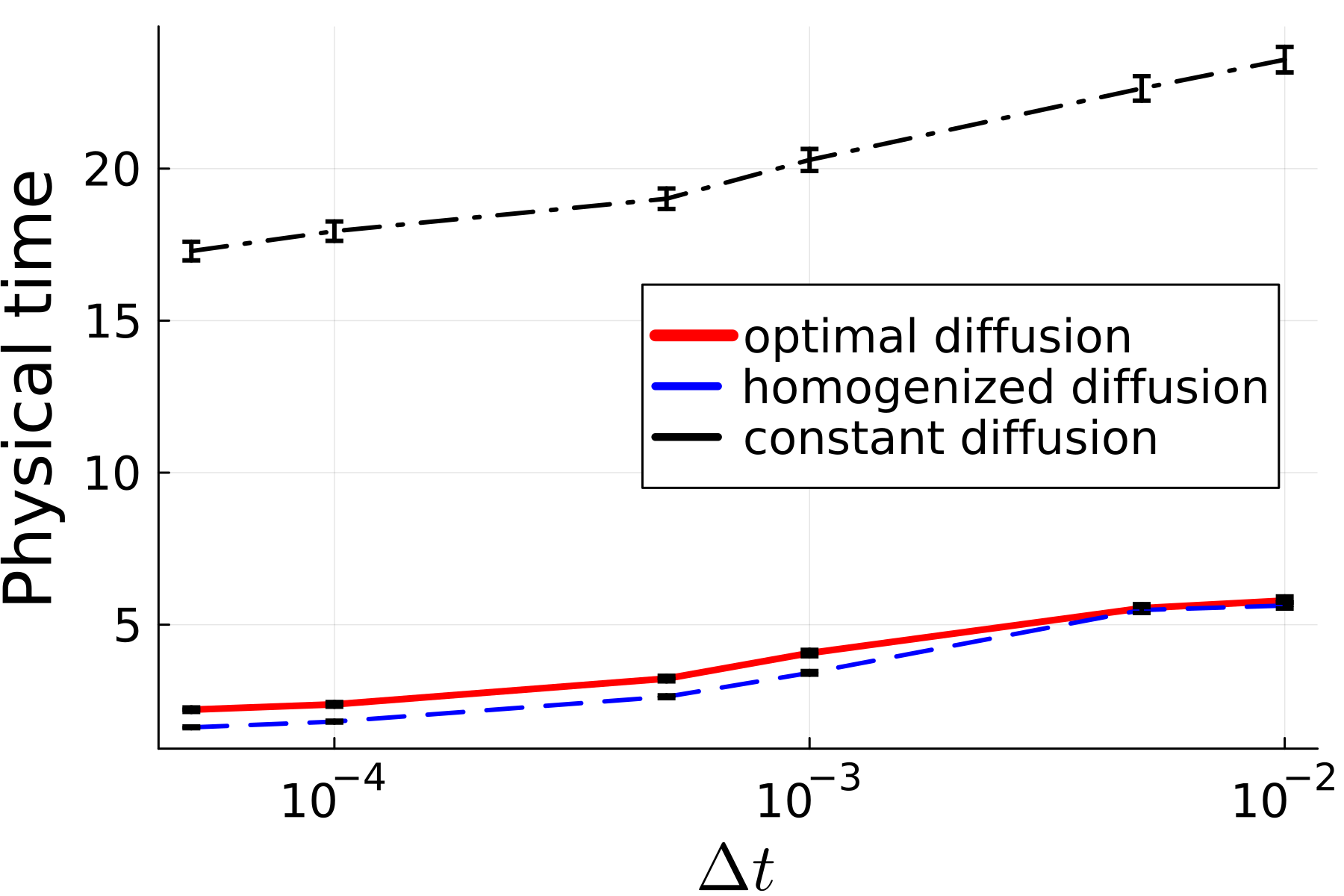}
    \caption{Mean transition time.}
  \end{subfigure}
  \hfill
  \begin{subfigure}[t]{0.45\textwidth}
    \includegraphics[width=\linewidth,draft=false]{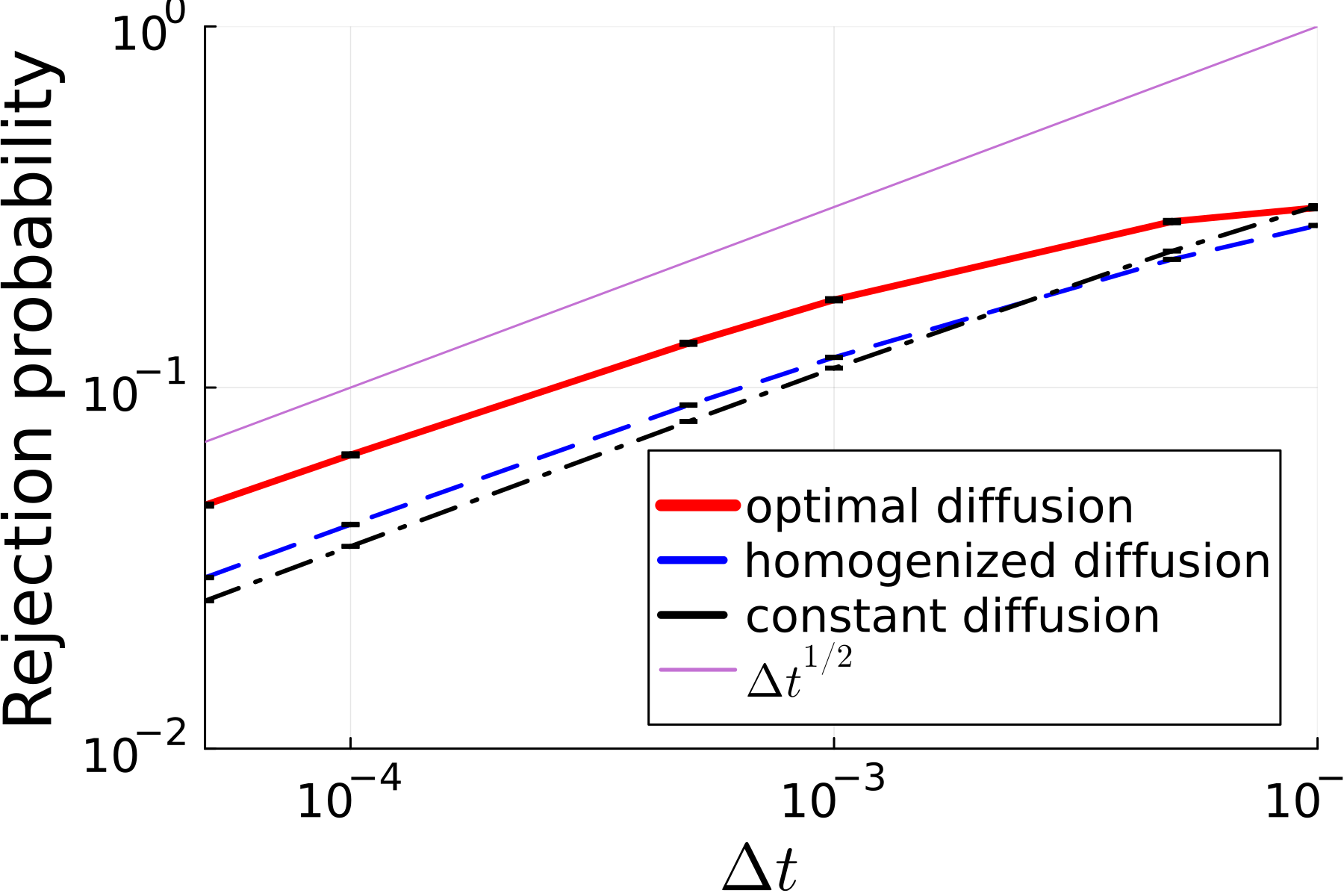}
    \caption{Metropolis--Hastings rejection probability.}
    \label{fig:transition_time_MH_ratio}
  \end{subfigure}
  \caption{Transition times with RWMH for the three optimal diffusion coefficients with respect to the time step~$\Delta t$.}
  \label{fig:transition_times_pic}
\end{figure}

\section{Perspectives}
\label{sec:perspectives}

This work calls for various perspectives and extensions, in particular:
\begin{itemize}
  \item A key issue remains to find a relevant normalization constraint \emph{e.g.}~based on a numerical criterion, as discussed in~\Cref{rmk:normalization_numerics}. One way to address this question would be to look at the fully time-discretized problem, \emph{e.g.}~optimizing the spectral gap of the transition kernel of the Metropolis Hastings algorithm introduced in Section~\ref{subsec:Donsker} with respect to $\Delta t \Diff$;
  \item In various contexts, we restrict ourselves to isotropic diffusion matrices or even to the one-dimensional setting (in particular for numerical simulations). Considering a larger class of positive definite symmetric matrices may be important to tackle anisotropic targets;
  \item In order to apply the methodology to actual systems of interest, one needs to adapt our numerical approach to higher dimensional scenarios, even in the case of a scalar diffusion coefficient, \emph{i.e.}~when the diffusion matrix is a multiple of the identity. It is indeed impossible to numerically optimize a diffusion matrix which genuinely depends on the full configuration~$q$ using discretization techniques such as finite elements, because of the curse of dimensionality of the corresponding approximation space. Apart from resorting to parsimonious functional representations of the diffusion matrix (which usually presuppose some form of smoothness), it seems appropriate to look for diffusion matrices parameterized by a low dimensional function~$\xi(q)$, a so-called reaction coordinate in molecular dynamics~\cite{LRS10}, which summarizes some important information on the system. In this context, the potential energy function is replaced by the free energy associated with~$\xi$, and the optimization problem can again be reformulated in a low dimensional setting where finite element methods can be used. Work in this direction is in progress~\cite{lelievre_2024};
  \item In this paper we consider the problem of maximizing the spectral gap. The self-adjointness of the generator implies that maximizing the spectral gap is typically equivalent to minimizing the asymptotic variance, uniformly over the space of $L^2(\mu)$ observables~\cite{AL21}. It would be interesting to explore the combination of spectral gap maximization and asymptotic variance minimization, as well as minimizing an appropriate cost functional that measures the true cost of the actual computation, by using the approach developed in~\cite{GlynnWhitt_1992};
  \item The objective in this paper is to optimize the convergence in $\chi^2$-divergence (see~\eqref{eq:cvg}). A natural question would be to consider convergence in other distances, such as relative entropy with respect to the target distribution (which amounts to optimizing Logarithmic Sobolev Inequality constants).
\end{itemize}


\appendix

\section{Proof of~\Cref{thm:well-posedness-1D}}
\label{app:thm:well-posedness-1D}

The fact that the optimization problem is well posed follows by a straightforward adaptation of~\cite[Theorem~10.3.1]{Henrot}. Recall that, by~\Cref{lem:concave-mat}, the application~$\Diff \mapsto\Lambda(\Diff)$ is concave. Moreover, the function~$\Lambda$ is upper-semicontinuous for the weak-*~$L^\infty_V$ topology since it is defined in~\eqref{eq:lambdaD-init} 
as the pointwise (in $\mathcal D$) infimum of the family of functions
\begin{equation*}
  \Diff \mapsto  \int_{\T^\dim} \nabla u (q)^{\top}\Diff(q) \nabla u (q) \, \mu(q)\,dq \Big / \int_{\T^\dim} u(q)^2 \mu(q)\,dq,
\end{equation*}
indexed by $u \in H^{1,0}(\mu) \setminus\{0\}$, which are continuous in $\mathcal D$ for the weak-*~$L^\infty_V$ topology.
In particular, $\Lambda$ is upper-semicontinuous for the weak-*~$L^\infty_V$ topology on~$\Diffset_p^{a,b}$ for any~$a \geq 0$,~$b\ge 0$ and~$p\geq 1$. Moreover, for any~$1\leq p \leq +\infty$, the subset~$\Diffset_{p}^{a,b}$ is weakly closed for the norm~$L^p_V$ defined in~\eqref{eq:norm_L^p_mu_matrices} (this follows directly from the weak closedness of standard Lebesgue spaces). We finally note that the set~$\Diffset_p^{a,b}$ is nonempty by~\Cref{lem:D_p^ab_nonempty} and compact for the weak-*~$L^\infty_V$ topology, for $a \ge 0$ and $b>0$. This already shows that there exists a solution to~\eqref{eq:maximizer_in_thm}.

To prove that $\Diff^\star$ necessarily takes a nonzero value on any open subset $\Omega$ of $\T^\dim$, we rely on the following lemma, which shows that the optimal diffusion matrix~$\Diff^{\star}(q)$ cannot be degenerate on an open set when~$a=0$ (for~$a>0$, the diffusion is of course nondegenerate).

\begin{lemma}
  \label{lem:D-positive-on-open-set}
  Fix~$a=0$. For any open set~$\Omega \subset \T^\dim$, there exists~$q\in\Omega$ such that~$\Diff^{\star}(q)\neq 0$.
\end{lemma}

\begin{proof}
  Let us first prove that $\Lambda(\Diff^\star)>0$.
  Since the potential function~$V\in\calC^{\infty}(\T^{\dim})$, it is bounded and the density~$\mu$ is bounded from below by a positive constant and bounded from above on~$\T^\dim$. The measure with density~\eqref{eq:mu} therefore satisfies a Poincaré inequality, \emph{i.e.}~there exists~$C_{\mu}>0$ such that
  \begin{equation}
    \label{eq:Poincare}
    \inf_{u\in H^{1,0}(\mu)} \left\{
    \int_{\T^\dim}|\nabla u(q)|^2 \, \mu(q)\,dq \Big/  \int_{\T^\dim} u(q)^2 \, \mu(q) \, dq \right\} = \frac{1}{C_{\mu}}.
  \end{equation}
  In addition, there exists~$u_{\mu}\in H^{1,0}(\mu)\setminus \{0\}$ such that the infimum~\eqref{eq:Poincare} is attained.
Consider
\begin{equation*}
    c = \min\left\lbrace  \frac{1}{\normF{\Id_\dim}} \left(\int_{\T^\dim} \rme^{- p V(q)} \, dq\right)^{-1/p}, \frac1b\right\rbrace.
  \end{equation*}
Notice that $\Diff(q) = c \Id_\dim$ is in the constrained set~$\Diffset_p^{a,b}$.
  Then,
  \begin{equation}
    \label{eq:positive-lambda}
    \Lambda(\Diff^\star) \geq \Lambda\left(c \Id_\dim\right) = c \inf\limits_{u\in H^{1,0}(\mu)\setminus\left\lbrace0\right\rbrace} \left\{ \int_{\T^\dim}|\nabla u(q)|^2 \mu(q) \, dq \Big / \dps \int_{\T^\dim} u(q)^2 \mu(q) \, dq\right\} = \frac{c}{C_{\mu}}>0.
  \end{equation}

To prove the result of Lemma~\ref{lem:D-positive-on-open-set}, we now proceed by contradiction. 
  Assume that there exists an open set~$\Omega\subset \T^\dim$ such that~$\Diff^{\star}(q) = 0$ for all~$q\in\Omega$. We can then consider a nonzero function~$u_{\Omega}\in \calC^{\infty}(\T^\dim)$ such that~$\supp(u_{\Omega})\subset \Omega$, which leads to
  \begin{equation*}
    0 \leq \Lambda(\Diff^{\star})\leq  \int_{\T^\dim} \nabla u_{\Omega}(q)^\top \Diff^{\star}(q)\nabla u_{\Omega}(q) \, \mu(q) \, dq \Big / \int_{\T^\dim} u_{\Omega}(q)^2 \mu(q) \, dq = 0.
  \end{equation*}
  The latter equality contradicts~\eqref{eq:positive-lambda} and thus provides the desired result.
\end{proof}

\section{Formal characterization of the maximizer using the smooth-min approach}
\label{app:smooth_max}

We first construct the counterpart of the map~$\mathsf{m}_\alpha$ in our framework. In order to do so, we fix~$\Diff\in L^{\infty}_{V}(\T^{\dim},\calM_{a,b})$ such that~$\Diff\geqslant c\Id_\dim$ for some~$c>0$ (see~\eqref{eq:unif_SDP}). We denote by~$g_{\alpha}\colon x\mapsto\rme^{-\alpha x}$ and~$h_{\alpha}\colon x\mapsto x\rme^{-\alpha x}$. By functional calculus, the maps~$g_\alpha$ and~$h_{\alpha}$ can be generalized to functions of the (positive) self-adjoint operator~$-\cLD$. With an abuse of notation, we write in this case~$g_\alpha(\Diff)$ and~$h_{\alpha}(\Diff)$ instead of~$g_\alpha(-\cLD)$ and~$h_{\alpha}(-\cLD)$. Since~$\Diff$ is bounded and uniformly bounded from below by a positive constant (in the sense of symmetric matrices), the positive operator~$-\cLD$ has discrete spectrum, its minimum eigenvalue being simple and equal to~0, see~\Cref{subsec:general-problem}. We therefore write
\begin{equation}
  \label{eq:decomposition_cld}
  -\cLD=\sum_{i\geqslant 1}\lambda_i P_i,
\end{equation}
where~$\lambda_1=0<\lambda_2<\ldots<\lambda_i<\ldots$ are the distinct eigenvalues of~$-\cLD$ and~$P_i$ is the orthogonal projection onto~$\Ker(-\cLD-\lambda_i\Id)$. The multiplicity of~$\lambda_i$ is denoted by~$N_i=\mathrm{dim}\,\Ran(P_i)\geqslant 1$. For~$\alpha>0$, the map~$f_{\alpha}$ defined in~\eqref{eq:f_alpha} thus rewrites
\begin{equation*}
  f_{\alpha}(\Diff)=\frac{\Tr(h_{\alpha}(\Diff))}{\Tr(g_{\alpha}(\Diff))-1},
\end{equation*}
where the operators~$g_{\alpha}(\Diff),h_{\alpha}(\Diff)$ are defined by functional calculus, \emph{i.e.}
\begin{equation*}
    g_{\alpha}(\Diff)  =\sum_{i\geqslant 1}g_{\alpha}(\lambda_i)P_i, \qquad
    h_{\alpha}(\Diff)  =\sum_{i\geqslant2}h_{\alpha}(\lambda_i)P_i,
\end{equation*}
where we used that~$h_{\alpha}(\lambda_1)=0$. The ratio~\eqref{eq:f_alpha} is well-defined in view of~\Cref{lem:trace_class_operators} and the fact that~$\Tr\left(g_\alpha(\Diff)\right)>1$. Formally, since~$\lambda_1=0$ and $N_1=1$,
\begin{equation*}
  f_{\alpha}(\Diff)=\sum_{i\geqslant2}N_i\lambda_i\rme^{-\alpha\lambda_i} \Big /{\sum_{i\geqslant2}}N_i\rme^{-\alpha\lambda_i},
\end{equation*}
so that, again formally (compare with~\eqref{eq:conv_m_alpha}),
\begin{equation}
  \label{eq:limit_f_alpha_to_infty}
  f_\alpha(\Diff)
  =\frac{
    \lambda_2\left(1+\dfrac{1}{N_2\lambda_2}\dps{\sum_{i\geqslant 3}}N_i\lambda_i\rme^{-\alpha(\lambda_i-\lambda_2)}\right)
  }{
    1+\dfrac{1}{N_2}\dps{\sum_{i\geqslant 3}}N_i\rme^{-\alpha(\lambda_i-\lambda_2)}
  }\xrightarrow[\alpha\to+\infty]{}\lambda_2=\min\limits_{i\geqslant2}\lambda_i.
\end{equation}
This limit is justified after the statement of~\Cref{lem:trace_class_operators} using the dominated convergence theorem. This shows that the quantity~$f_{\alpha}(\Diff)$ is an approximation of~$\Lambda(\Diff)$.

This motivates considering the approximation of the problem~\eqref{eq:maximizer_in_thm} considered in~\eqref{eq:maximizer_in_thm_alpha}. One can observe that, for a fixed~$\Diff$,~$f_\alpha(\Diff)$ is a convex combination of the eigenvalues~$\lambda_i$, which implies that~$f_\alpha(\Diff)\geqslant\Lambda(\Diff)$. Therefore, it holds~$f_\alpha(\Diff^{\star,\alpha})\geqslant\Lambda(\Diff^{\star})$, that we indeed observed numerically. 

As in~\Cref{subsec:euler_lagrange_nondegenerate}, we assume that there exists a solution to~\eqref{eq:maximizer_in_thm_alpha} for~$a=0$ and some~$b>0$, which satisfies (see~\eqref{eq:b_not_active}-\eqref{eq:unif_SDP})
\begin{equation}\label{eq:b_not_active_alpha}
  \exists b_+>b, \quad \Diff^{\star,\alpha}(q) \, \rme^{-V(q)} \le \frac{1}{b_+} \,   \Id_\dim \textrm{ for a.e. }q\in\T^\dim,
\end{equation}
and
\begin{equation}\label{eq:unif_SDP_alpha}
  \exists c>0,\quad \Diff^{\star,\alpha}(q)\geq c\, \Id_\dim \textrm{ for a.e. }q\in\T^\dim.
\end{equation}
In particular, it holds~$\Phi_p(\Diff^{\star,\alpha})=1$. In our numerical experiments, we observe that~\eqref{eq:b_not_active_alpha} is always satisfied for $b$ sufficiently small. Moreover, we use~\eqref{eq:unif_SDP_alpha} below for technical reasons (in order to rigorously define $f_\alpha$ and prove that it is differentiable). However, we will observe numerically that the characterization that we finally obtain on $\Diff^{\star,\infty}$ (by sending $\alpha \to +\infty$, see Equation~\eqref{eq:Diff_characterization_formal_alpha_to_infty_degenerate_case_1d_case}) seems to also hold even when~\eqref{eq:unif_SDP_alpha} is not satisfied.

\paragraph{Characterization of the optimal diffusion matrix for a finite value of~$\alpha$.}

The map~$f_{\alpha}$ defined by~\eqref{eq:f_alpha} is smooth over the set of functions~$\Diff\in L^{\infty}_{V}(\T^\dim,\calM_{a,b})$ such that~$\Diff\geqslant c\Id_\dim$ for some~$c>0$, and therefore its differential can be computed even when some eigenvalues are degenerate. The main result of this paragraph is the computation of the differentials of the map~$g_{\alpha}$ and~$h_{\alpha}$, whose proof is given in~\Cref{app:differential_g_h_alpha}. Note that the derivatives of the maps~$g_{\alpha}$ and~$h_{\alpha}$, namely the maps~$g'_{\alpha}$ and~$h'_{\alpha}$ defined by
\begin{equation*}
  \forall x\in\R, \quad \forall\alpha>0,\qquad
  g_{\alpha}'(x)=-\alpha\rme^{-\alpha x},\qquad 
  h_{\alpha}'(x)=(1-\alpha x)\rme^{-\alpha x},
\end{equation*}
can also be generalized to functions of the self-adjoint operator~$-\cLD$, and we then write~$g_{\alpha}'(\Diff)=g_{\alpha}'(-\cLD)$ and~$h_{\alpha}'(\Diff)=h_{\alpha}'(-\cLD)$.

\begin{lemma}
  \label{lem:differential_g_h_alpha}
  Fix~$\alpha>0$ and~$\Diff\in L^{\infty}_{V}(\T^\dim,\calM_{a,b})$ such that~$\Diff\geqslant c\Id_\dim$ for some~$c>0$. Then there exists~$K\in\R_{+}$ such that, for any~$\delta\Diff\in L^{\infty}(\T^{\dim},\calS_\dim)$ satisfying~$\left\lVert\delta\Diff\right\rVert_{L^{\infty}(\T^{\dim})}\leqslant c/2$, the operators~$g_{\alpha}(\Diff+\delta\Diff), g_{\alpha}(\Diff),g_{\alpha}'(\Diff)\cL_{\delta\Diff}$ and the operators~$h_{\alpha}(\Diff+\delta\Diff), h_{\alpha}(\Diff), h_{\alpha}'(\Diff)\cL_{\delta\Diff}$ are trace-class on~$L^{2}(\mu)$ and
  \begin{equation}
    \label{eq:differentials}
    \left\lbrace
    \begin{aligned}
      \left\lvert\Tr(g_{\alpha}(\Diff+\delta\Diff))-\Tr(g_{\alpha}(\Diff))-\Tr\left(g_{\alpha}'(\Diff)\left(-\cL_{\delta\Diff}\right)\right)\right\rvert
       & \leqslant K\left\lVert\delta\Diff\right\rVert^{2}_{L^{\infty}(\T^{\dim})}, \\
      \left\lvert\Tr(h_{\alpha}(\Diff+\delta\Diff))-\Tr(h_{\alpha}(\Diff))-\Tr\left(h_{\alpha}'(\Diff)\left(-\cL_{\delta\Diff}\right)\right)\right\rvert
       & \leqslant K\left\lVert\delta\Diff\right\rVert^{2}_{L^{\infty}(\T^{\dim})}.
    \end{aligned}
    \right.
  \end{equation}
\end{lemma}

This result immediately implies a differentiability result for the map~$f_{\alpha}$.
\begin{corollary}
  \label{cor:differential_f_alpha}
  Fix~$\alpha>0$ and~$\Diff\in L^{\infty}_{V}(\T^\dim,\calM_{a,b})$ such that~$\Diff\geqslant c\Id_\dim$ for some~$c>0$. Then there exists~$K\in\R_{+}$ such that, for any~$\delta\Diff\in L^{\infty}(\T^{\dim},\calS_\dim)$ satisfying~$\left\lVert\delta\Diff\right\rVert_{L^{\infty}(\T^{\dim})}\leqslant c/2$,
  \begin{equation*}
    \left\lvert f_{\alpha}(\Diff+\delta\Diff)-f_{\alpha}(\Diff)-\mathcal{F}_{\alpha}(\Diff)(\delta\Diff)\right\rvert \leqslant K\left\lVert\delta\Diff\right\rVert^{2}_{L^{\infty}(\T^{\dim})},
  \end{equation*}
  with
  \begin{equation*}
    \mathcal{F}_{\alpha}(\Diff)(\delta\Diff)=\frac{
    \Tr\left(h'_{\alpha}(\Diff)\left(-\cL_{\delta\Diff}\right)\right)\left[
      \Tr(g_{\alpha}(\Diff))-1
      \right]-\Tr(h_{\alpha}(\Diff))\Tr\left(g'_{\alpha}(\Diff)\left(-\cL_{\delta\Diff}\right)\right)
    }{
    \left(
    \Tr(g_{\alpha}(\Diff))-1
    \right)^{2}
    }.
  \end{equation*}
\end{corollary}

Straightforward computations lead to
\begin{equation*}
  \mathcal{F}_\alpha(\Diff)(\delta\Diff)
  =\sum\limits_{i\geqslant1}\left[\frac{\mathcal{G}_\alpha(1-\alpha\lambda_i)+\alpha \mathcal{H}_\alpha}{\mathcal{G}_\alpha^2}\rme^{-\alpha\lambda_i}\right]\Tr\left(
  P_i\left(-\cL_{\delta\Diff}\right)
  \right),
\end{equation*}
with
\begin{equation}
  \label{eq:G_alpha_H_alpha}
  \mathcal{G}_\alpha=\sum\limits_{j\geqslant2}N_j\rme^{-\alpha\lambda_j},\qquad
  \mathcal{H}_\alpha=\sum\limits_{j\geqslant2}N_j\lambda_j\rme^{-\alpha\lambda_j}.
\end{equation}

We now show how to obtain the characterization of~$\Diff^{\star,\alpha}$ solution to~\eqref{eq:maximizer_in_thm_alpha} given in~\eqref{eq:diff_alpha_F}, following the same computations as in the proof of~\Cref{prop:EL_degenerate}. This requires to introduce a Hilbert basis~$(e_k)_{k\geqslant1}$ composed of~$L^2(\mu)$ normalized eigenvectors of~$-\cLD$, namely
\begin{equation*}
  \forall k,\ell\geqslant 1,\qquad
  -\cLD e_{k}=\varlambda_k e_{k},\qquad \left\langle e_{k},e_{\ell}\right\rangle_{L^{2}(\mu)}=\delta_{k,\ell},
\end{equation*}
where~$\delta$ is the Kronecker symbol and the eigenvalues~$\varlambda_1=0\leqslant\varlambda_2\leqslant\dots\leqslant \varlambda_k\leqslant \dots$ are ordered and counted with their multiplicities:
\begin{equation*}
  \forall i\geqslant 1,\quad
  \forall k\in\left\lbrace 1+\sum_{j=1}^{i-1}N_j,\dots, \sum_{j=1}^{i}N_j\right\rbrace,\qquad
  \varlambda_k=\lambda_i.
\end{equation*}
We rewrite the spectral decomposition~\eqref{eq:decomposition_cld} of~$-\cLD$ in terms of this given Hilbert basis and with possibly degenerate eigenvalues: for a given smooth function~$\varphi$,
\begin{equation*}
  -\cLD\varphi=\sum_{k\geqslant1}\varlambda_k \left\langle \varphi,e_k\right\rangle_{L^{2}(\mu)} e_k.
\end{equation*}
We exploit the Euler--Lagrange equation for the objective function~$f_{\alpha}$, namely that~$\mathcal{F}_\alpha(\Diff^{\star,\alpha})$ is proportional to the differential of the  function $\Phi_p$~\eqref{eq:normalization_constraint}:
\begin{equation*}
  \forall \delta \Diff\in L^{\infty}(\T^{\dim},\calS_\dim),\qquad \mathcal{F}_{\alpha}(\Diff^{\star,\alpha})(\delta \Diff) \propto\int_{\T^{\dim}} \left\lvert \Diff^{\star,\alpha}(q)\right\rvert_{\mathrm{F}}^{p-2} \rme^{- p V(q)}  \Diff^{\star,\alpha}(q): \delta \Diff(q) \, dq.
\end{equation*}
Since
\begin{equation*}
  P_i=\sum_{k=1+\sum_{l=1}^{i-1}N_l}^{\sum_{l=1}^{i}N_l}\left\langle e_k,\cdot\right\rangle_{L^{2}(\mu)}e_k
\end{equation*}
and
\begin{align*}
  \Tr\left(\left\langle e_k,\cdot\right\rangle_{L^{2}(\mu)}e_k\nabla^*\delta\Diff\nabla\right)
   & =
  \int_{\T^{\dim}} e_k(q)\left[\nabla^*\delta\Diff\nabla e_k\right](q)\mu(q)\,dq \\
   & =
  \int_{\T^{\dim}} \nabla e_k(q)^{\top}\delta\Diff(q)\nabla e_k(q)\mu(q)\, dq          \\
   & =
  \int_{\T^{\dim}}\delta\Diff(q):\left(\nabla e_k(q)\otimes\nabla e_k(q)\right)\mu(q)\,dq,
\end{align*}
one readily obtains~\eqref{eq:diff_alpha_F}. Note that, since~$\Diff^{\star,\alpha}$ depends on~$\alpha$, its eigenvalues also depend on~$\alpha$ so that the quantities~$\mathcal{G}_{\alpha}$ and~$\mathcal{H}_{\alpha}$ first defined in~\eqref{eq:G_alpha_H_alpha} are changed accordingly, see~\eqref{eq:G_alpha_H_alpha_maximizer}. Note that we used that~$\nabla e_{1,\alpha}=0$ to start the sum at~$k=2$ in~\eqref{eq:diff_alpha_F} since the eigenspace associated to~$\varlambda_{1,\alpha}=0$, which is one-dimensional, is spanned by constant functions.

\paragraph{Formal characterization in the limit~$\alpha\to+\infty$ and in dimension 1.} We now show how to obtain the formal characterization~\eqref{eq:Diff_characterization_formal_alpha_to_infty_degenerate_case_1d_case}.

In the case~$\varlambda_{2,\infty}< \varlambda_{3,\infty}$, the limit when~$\alpha\to+\infty$ of the sum on the right-hand side of~\eqref{eq:diff_alpha_1d} can be computed as
\begin{align}
  \MoveEqLeft[2]
  \sum_{k\geqslant 2}\left[\frac{\mathcal{G}_\alpha(1-\alpha\varlambda_{k,\alpha})+\alpha \mathcal{H}_\alpha}{\mathcal{G}_\alpha^2}\rme^{-\alpha\varlambda_{k,\alpha}}\right]|e'_{k,\alpha}|^2\label{eq:expansion_nabla_e_i_alpha} \\
   & \underset{\alpha\to+\infty}{\sim}
  \frac{\rme^{-\alpha\varlambda_{2,\infty}}(1-\alpha \varlambda_{2,\infty})+\alpha\varlambda_{2,\infty}\rme^{-\alpha\varlambda_{2,\infty}}}{\rme^{-2\alpha\varlambda_{2,\infty}}}\rme^{-\alpha\varlambda_{2,\infty}}|e'_{2,\alpha}|^2=|e'_{2,\alpha}|^2.\nonumber
\end{align}
The characterization of the optimal diffusion matrix~\eqref{eq:diff_alpha_1d} becomes in the limit~$\alpha\to+\infty$:
\begin{equation*}
  \Diff^{\star,\infty}(q)=\widetilde{\gamma}_{\infty}\rme^{V(q)}
  \left\lvert e_{2,\infty}'(q)\right\rvert^{2/(p-1)},
\end{equation*}
with~$\widetilde{\gamma}_{\infty}>0$. We therefore recover Equation~\eqref{eq:EL-prop} of~\Cref{subsec:euler_lagrange_nondegenerate}.

When~$\varlambda_{2,\infty}=\varlambda_{3,\infty}$, the factor in front of~$|e'_{2,\alpha}|^2$ in~\eqref{eq:expansion_nabla_e_i_alpha} is asymptotically equivalent to
\begin{align*}
   & \frac{\rme^{-\alpha\varlambda_{2,\alpha}}}{\left(\rme^{-\alpha\varlambda_{2,\alpha}}+\rme^{-\alpha\varlambda_{3,\alpha}}\right)^{2}}\left[
    \left(
    \rme^{-\alpha\varlambda_{2,\alpha}}+\rme^{-\alpha\varlambda_{3,\alpha}}
    \right)\left(
    1-\alpha\varlambda_{2,\alpha}
    \right)
    +\alpha\left(
    \varlambda_{2,\alpha}\rme^{-\alpha\varlambda_{2,\alpha}}+
    \varlambda_{3,\alpha}\rme^{-\alpha\varlambda_{3,\alpha}}
    \right)
  \right]\\
   & \quad
  =
  \frac{1}{\left(1+\rme^{-\alpha(\varlambda_{3,\alpha}-\varlambda_{2,\alpha})}\right)^{2}}\left[
    \left(1+\rme^{-\alpha(\varlambda_{3,\alpha}-\varlambda_{2,\alpha})}\right)(1-\alpha\varlambda_{2,\alpha})+\alpha\varlambda_{2,\alpha}+\alpha\varlambda_{3,\alpha}\rme^{-\alpha(\varlambda_{3,\alpha}-\varlambda_{2,\alpha})}
  \right] \\
   & \quad
  =
  \frac{1}{\left(1+\rme^{-\alpha(\varlambda_{3,\alpha}-\varlambda_{2,\alpha})}\right)^{2}}\left[
    1+\rme^{-\alpha(\varlambda_{3,\alpha}-\varlambda_{2,\alpha})}+\alpha(\varlambda_{3,\alpha}-\varlambda_{2,\alpha})\rme^{-\alpha(\varlambda_{3,\alpha}-\varlambda_{2,\alpha})}
  \right]\\
   & \quad\xrightarrow[\alpha\to+\infty]{}\frac{1+\rme^{-\eta}+\eta\rme^{-\eta}}{(1+\rme^{-\eta})^{2}},
\end{align*}
while the factor in front of~$|e_{3,\alpha}'|^2$ is asymptotically equivalent to~$\frac{\rme^{-\eta}(1+\rme^{-\eta}-\eta)}{(1+\rme^{-\eta})^{2}}$. One therefore obtains the characterization given by~\eqref{eq:Diff_characterization_formal_alpha_to_infty_degenerate_case_1d_case}.

We check in~\Cref{fig:res_3_b} that the characterization of the optimal diffusion coefficient using a finite value of~$\alpha$, \emph{i.e.}~formula~\eqref{eq:diff_alpha_1d}, is valid, using similar approximations as for~\eqref{eq:approximation_d_star_infty}: more precisely, for~$n\in\left\lbrace1,\dots,N\right\rbrace$, the~$n$-th component of the discrete approximation of~\eqref{eq:diff_alpha_1d} is given by
\begin{equation*}
  \widetilde{\gamma}_{\alpha}\rme^{V((n-1)/N)}\left(
  \sum_{k=2}^{N}\left[
    \frac{G_{\alpha}(\diff^{\star,\alpha})(1-\alpha\sigma_k(\diff^{\star,\alpha}))+\alpha H_{\alpha}(\diff^{\star,\alpha})}{G_{\alpha}(\diff^{\star,\alpha})^{2}}\rme^{-\alpha\sigma_k(\diff^{\star,\alpha})}
    \right]\left\lvert U_{k,\alpha,n}-U_{k,\alpha,n-1}\right\rvert^2\right)^{1/(p-1)},
\end{equation*}
where~$G_\alpha,H_\alpha$ are defined in~\eqref{eq:g_h_alpha_discrete},~$(U_{k,\alpha})_{2\leqslant k\leqslant N}$ are the normalized eigenvectors associated with the eigenvalues~$(\sigma_k(\diff^{\star,\alpha}))_{2\leqslant k\leqslant N}$ and~$\widetilde{\gamma}_\alpha$ is a normalizing constant. Note that this formula requires all the eigenelements which can be computationally heavy.

\begin{figure}
    \centering
    \includegraphics[width=0.6\textwidth]{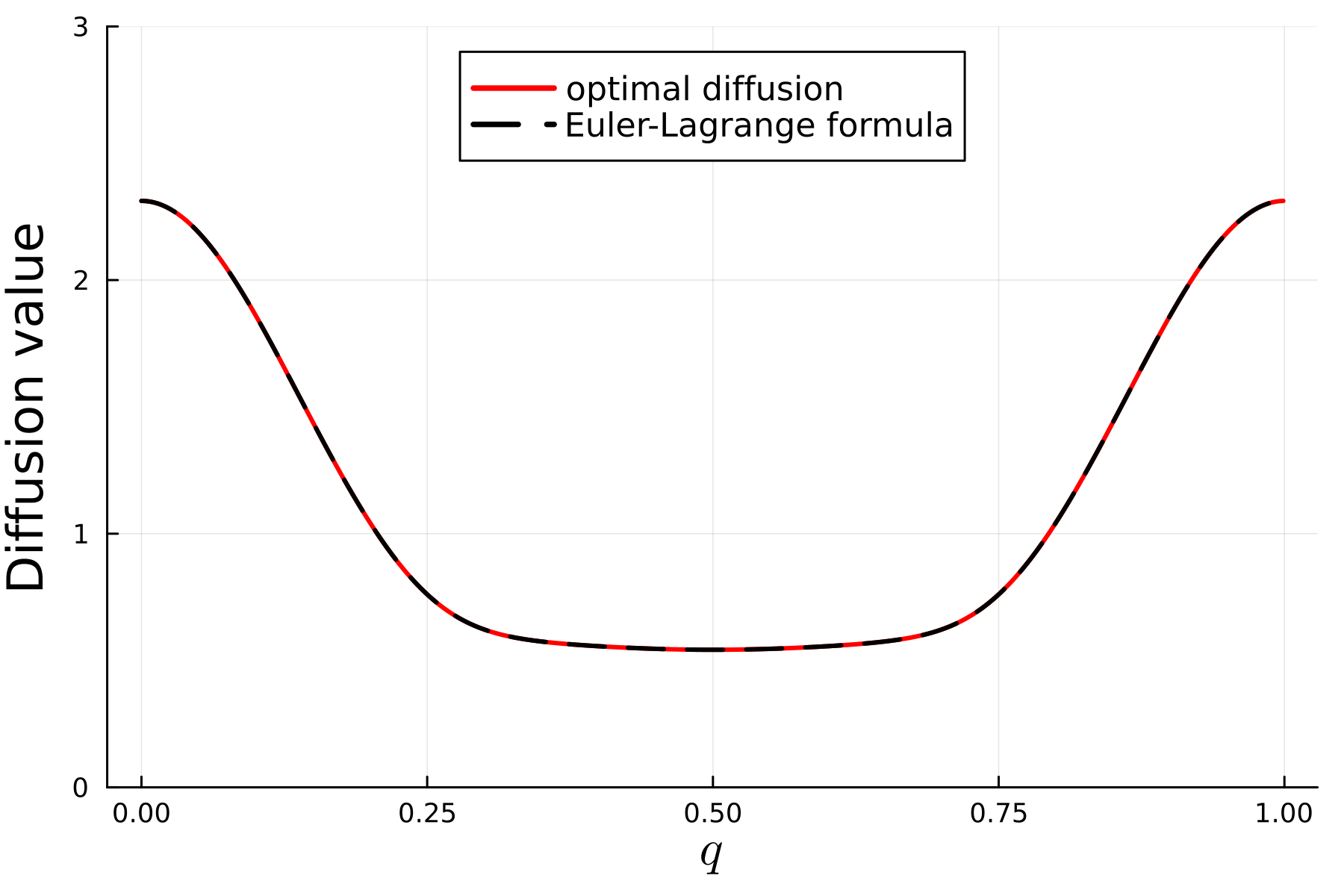}
    \caption{\label{fig:res_3_b}Comparison between~$D^{\star,\alpha}$ with~$\alpha=1.0$ and~$D^{\star}$.}
\end{figure}

\subsection{Proof of~\Cref{lem:differential_g_h_alpha}}
\label{app:differential_g_h_alpha}

\Cref{lem:differential_g_h_alpha} relies on technical results which are proved in~\Cref{app:differential_g_h_alpha_technical_results}. Let us introduce some notation and recall standard results for trace-class operators that will be useful below. For two Banach spaces~$E,F$, we denote by~$\calB(E,F)$ the Banach space of bounded linear operators from ~$E$ to~$F$. We denote by~$\left\lVert A\right\rVert_{1}=\Tr(A^{*}A)^{1/2}$ the trace norm of an operator~$A$ on~$L^{2}(\mu)$.  Note that if, for any~$s\in[0,1]$, the operator~$A(s)$ is a trace-class operator on~$L^{2}(\mu)$ and there exist~$K\in\R_{+}$ such that~$\left\lVert A(s)\right\rVert_{1}\leqslant K$ for any~$s\in[0,1]$, then
\begin{equation}
  \label{eq:bochner_integral_trace_norm}
  \left\lVert\int_{0}^{1} A(s)ds\right\rVert_{1}\leqslant\int_{0}^{1}\left\lVert A(s)\right\rVert_{1}ds\leqslant K.
\end{equation}
We also recall that if~$A$ is bounded on~$L^{2}(\mu)$ and~$B$ is trace-class on~$L^{2}(\mu)$, the operators~$AB$ and~$BA$ are trace-class on~$L^{2}(\mu)$ and
\begin{equation}
  \label{eq:trace_bounded_norm}
  \left\lVert AB\right\rVert_{1}\leqslant \left\lVert A\right\rVert_{\calB(L^{2}(\mu))}\left\lVert B\right\rVert_{1},\qquad \left\lVert BA\right\rVert_{1}\leqslant \left\lVert A\right\rVert_{\calB(L^{2}(\mu))}\left\lVert B\right\rVert_{1}
\end{equation}
We refer to~\cite{ReedSimon_Vol1,Simon_TraceIdeals} for more detailed results on trace-class operators.

To prove~\eqref{eq:differentials}, we first assume that~$\delta\Diff\in\calC^{3}(\T^{\dim},\calS_\dim)$ with~$\left\lVert\delta\Diff\right\rVert_{L^{\infty}(\T^{\dim})}\leqslant c/2$, and show the inequality~\eqref{eq:differentials} with a bound~$\left\lVert\delta\Diff\right\rVert_{L^{\infty}(\T^{\dim})}^{2}$ on the right-hand side. The desired result for any~$\delta\Diff\in L^{\infty}(\T^{\dim},\calS_\dim)$ such that~$\left\lVert\delta\Diff\right\rVert_{L^{\infty}(\T^{\dim})}\leqslant c/2$ is then obtained by a density argument.

We perform the computations for the map~$g_{\alpha}$, a similar reasoning leading to the result for the map~$h_{\alpha}$. In this proof, the letter~$K$ denotes a generic nonnegative finite constant which may change from line to line.

Consider~$\delta\Diff\in \calC^{3}(\T^\dim,\calS_\dim)$ such that~$\left\lVert\delta\Diff\right\rVert_{L^{\infty}(\T^{\dim})}\leqslant c/2$. Then~$\Diff+\delta\Diff\geqslant (c/2)\Id_\dim$ so that the operator~$\cL_{\Diff+\delta\Diff}$ is negative and self-adjoint on~$L^{2}(\mu)$. We can therefore use Duhamel's formula twice to obtain
\begin{align}
  \MoveEqLeft[2]
  g_{\alpha}(\Diff+\delta\Diff)-g_{\alpha}(\Diff)\nonumber                                                                                                                                                                                                \\
   & =
  \rme^{\alpha\cL_{\Diff+\delta\Diff}}-\rme^{\alpha\cLD}
  =\left[\rme^{s\alpha\cL_{\Diff+\delta\Diff}}\rme^{(1-s)\alpha\cLD}\right]_{s=0}^{s=1}
  \nonumber                                                                                                                                                                                                                                               \\
   & =
  \alpha\int_{0}^{1}\rme^{s\alpha\cL_{\Diff+\delta\Diff}}\cL_{\delta\Diff}\rme^{(1-s)\alpha\cLD}ds\nonumber                                                                                                                                               \\
   & =
  \alpha\int_{0}^{1}\left(\rme^{s\alpha\cLD}+\left[\rme^{s(1-u)\alpha\cLD}\rme^{su\alpha\cL_{\Diff+\delta\Diff}}\right]_{u=0}^{u=1}\right)\cL_{\delta\Diff}\rme^{(1-s)\alpha\cLD}ds\nonumber                                                              \\
   & =
  \alpha\int_{0}^{1}\rme^{s\alpha\cLD}\cL_{\delta\Diff}\rme^{(1-s)\alpha\cLD}ds+\alpha^{2}\int_{0}^{1}\int_{0}^{1}\rme^{s(1-u)\alpha\cLD}\cL_{\delta\Diff}\rme^{su\alpha\cL_{\Diff+\delta\Diff}}\cL_{\delta\Diff}\rme^{(1-s)\alpha\cLD}s\,du\,ds\nonumber \\
   & =
  \alpha\int_{0}^{1}\rme^{s\alpha\cLD}\cL_{\delta\Diff}\rme^{(1-s)\alpha\cLD}ds+\alpha^{2}\int_{0}^{1}\int_{0}^{s}\rme^{(s-u)\alpha\cLD}\cL_{\delta\Diff}\rme^{u\alpha\cL_{\Diff+\delta\Diff}}\cL_{\delta\Diff}\rme^{(1-s)\alpha\cLD}du\,ds,\label{eq:duhamel}
\end{align}
where we performed the change of variables~$u\leftarrow su$ for the last equality.

By~\Cref{lem:trace_class_operators} below, the operators~$g_{\alpha}(\Diff+\delta\Diff)$ and~$g_{\alpha}(\Diff)$ are trace-class on~$L^{2}(\mu)$. The operator defined by the first term on the right-hand side of~\eqref{eq:duhamel} is linear in~$\delta\Diff$. We now show that it is trace-class on~$L^{2}(\mu)$ and that there exists~$K\in\R_{+}$ such that
\begin{equation}
  \label{eq:differential_trace_class_operator}
  \left\lVert\alpha\int_{0}^{1}\rme^{s\alpha\cLD}\cL_{\delta\Diff}\rme^{(1-s)\alpha\cLD}ds\right\rVert_{1}
  \leqslant K\left\lVert\delta\Diff\right\rVert_{L^{\infty}(\T^\dim)}.
\end{equation}
For any~$s\in[0,1]$, let~$T_s=\rme^{s\alpha\cLD}\cL_{\delta\Diff}\rme^{(1-s)\alpha\cLD}$. We need to distinguish between the two cases~$0\leqslant s\leqslant 1/2$ and~$1/2\leqslant s\leqslant 1$: when~$0\leqslant s\leqslant 1/2$, the operator~$\rme^{(1-s)\alpha\cLD}$, possibly multiplied by any polynomial of~$\cLD$, is trace-class on~$L^{2}(\mu)$ uniformly in~$s$ thanks to~\Cref{lem:trace_class_operators}, while the operator~$\rme^{s\alpha\cLD}$ is bounded on~$L^{2}(\mu)$, so that~\eqref{eq:trace_bounded_norm} can be used. A similar result is obtained in the case~$1/2\leqslant s\leqslant 1$ by switching~$\rme^{(1-s)\alpha\cLD}$ and~$\rme^{s\alpha\cLD}$ in the reasoning.

More precisely, for~$s\in[0,1/2]$, we rewrite~$T_s$ as
\begin{equation*}
  T_s=
  \underbrace{\rme^{s\alpha\cLD}}_{=:A_1}
  \underbrace{\cL_{\delta\Diff}(\Id-\cLD)^{-1}}_{=:A_2}
  \underbrace{(\Id-\cLD)\rme^{(1-s)\alpha\cLD}}_{=:A_3}.
\end{equation*}
Note that~$\left\lVert A_1\right\rVert_{\calB(L^{2}(\mu))}\leqslant1$, and~$A_2$ is bounded on~$L^{2}(\mu)$ by~\Cref{lem:multiple_operators_bounded} (see~\eqref{eq:estimate_lemma_operators_bounded_1} below). The operator~$A_3$ is trace-class on~$L^{2}(\mu)$ by~\Cref{lem:trace_class_operators}, so that $T_s$ is trace-class on~$L^{2}(\mu)$ with trace norm uniformly bounded over~$s\in\left[0,1/2\right]$. In view of~\eqref{eq:trace_bounded_norm} and~\eqref{eq:estimate_lemma_operators_bounded_1}, there exists~$K\in\R_{+}$ such that
\begin{equation*}
  \forall s\in[0,1/2],\qquad
  \left\lVert T_s\right\rVert_{1}\leqslant K\left\lVert\delta\Diff\right\rVert_{L^{\infty}(\T^\dim)}.
\end{equation*}
In the case~$ 1/2\leqslant s\leqslant 1$, we rewrite $T_s$ as
\begin{equation*}
  T_s=\underbrace{\rme^{s\alpha\cLD}(\Id-\cLD)}_{:=A_3}\underbrace{(\Id-\cLD)^{-1}\cL_{\delta\Diff}}_{=:A_2}\underbrace{\rme^{(1-s)\alpha\cLD}}_{=:A_1},
\end{equation*}
and proceed with a similar argument as above. In view of~\eqref{eq:bochner_integral_trace_norm}, the operator defined by the first term on the right-hand side of~\eqref{eq:duhamel} is therefore trace-class on~$L^{2}(\mu)$ and~\eqref{eq:differential_trace_class_operator} holds.

We now show that the operator defined by the second term on the right-hand side of~\eqref{eq:duhamel} is trace-class on~$L^{2}(\mu)$, and that its trace norm is bounded from above by~$\left\lVert\delta\Diff\right\rVert_{\calC^{3}(\T^\dim)}^{2}$. First, note that the operator~$(\Id-\cL_{\Diff+\delta\Diff})^{2}\rme^{\theta\cLD}$ is trace-class on~$L^{2}(\mu)$ for any~$\theta>0$. Indeed, let us rewrite this operator as~$(\Id-\cL_{D+\delta\Diff})^{2}(\Id-\cLD)^{-2}(\Id-\cLD)^{2}\rme^{\theta\cLD}$, and let us set~$T=(\Id-\cL_{D+\delta\Diff})^{2}(\Id-\cLD)^{-2}$. \Cref{lem:bound_cl_d_delta_diff} then implies that
\begin{equation}
  \label{eq:to_ref_for_c_star_a_c}
  0\leqslant T^{*}T\leqslant(\Id-\cLD)^{-2}\left(\Id-\frac{\left\lVert\Diff\right\rVert_{L^{\infty}(\T^{\dim})}+\left\lVert\delta\Diff\right\rVert_{L^{\infty}(\T^{\dim})}}{c}\cLD\right)^{4}\left(\Id-\cLD\right)^{-2}.
\end{equation}
Recall that if~$A\leqslant B$ as symmetric operators on a Hilbert space, then~$C^{*}AC\leqslant C^{*}BC$ for any bounded operator~$C$ on the same Hilbert space. The operator on the right-hand side of~\eqref{eq:to_ref_for_c_star_a_c} is therefore bounded on~$L^{2}(\mu)$ by functional calculus, so that~$T$ is also bounded on~$L^{2}(\mu)$.~\Cref{lem:trace_class_operators} then implies that the operators~$(\Id-\cL_{\Diff+\delta\Diff})^{2}\rme^{\theta\cLD}$ are trace-class on~$L^{2}(\mu)$, and more precisely that for any~$\underline{\theta}>0$, there exists~$K\in\R_{+}$ (which depends on~$\underline{\theta}$) such that
\begin{equation}
  \label{eq:particular_operator_is_trace_class}
  \forall\theta>\underline{\theta},\qquad \left\lVert(\Id-\cL_{\Diff+\delta\Diff})^{2}\rme^{\theta\cLD}\right\rVert_{1}\leqslant K.
\end{equation}

For any~$s\in[0,1]$ and~$u\in[0,s]$, let~$T_{s,u}=\rme^{\alpha (s-u)\cLD}\cL_{\delta\Diff}\rme^{u\alpha\cL_{\Diff+\delta\Diff}}\cL_{\delta\Diff}\rme^{(1-s)\alpha\cLD}$. We need to distinguish between three cases: if~$0\leqslant s\leqslant 1/2$, the operator~$\rme^{(1-s)\alpha\cLD}$ is trace-class on~$L^{2}(\mu)$ uniformly in~$s$ and the other operators are bounded on~$L^{2}(\mu)$; if~$ 1/2\leqslant s\leqslant 1$, then the operator~$\rme^{\alpha(s-u)\cLD}$ is trace-class on~$L^{2}(\mu)$ uniformly in~$u,s$ for~$0\leqslant u\leqslant s/2$, while it is the operator~$\rme^{\alpha u\cL_{\Diff+\delta\Diff}}$ which is trace-class on~$L^{2}(\mu)$ uniformly in~$u,s$ for~$s/2\leqslant u\leqslant s$.

If $s\in[0,1/2]$, we rewrite~$T_{s,u}$ as
\begin{equation*}
  T_{s,u}=
  \underbrace{\rme^{(s-u)\alpha\cLD}}_{=:A_1}
  \underbrace{\cL_{\delta\Diff}(\Id-\cL_{\Diff+\delta\Diff})^{-1}}_{=:A_2}
  \underbrace{\rme^{u\alpha\cL_{\Diff+\delta\Diff}}}_{=:A_3}
  \underbrace{(\Id-\cL_{\Diff+\delta\Diff})\cL_{\delta\Diff}(\Id-\cL_{\Diff+\delta\Diff})^{-2}}_{=:A_4}
  \underbrace{(\Id-\cL_{\Diff+\delta\Diff})^{2}\rme^{(1-s)\alpha\cLD}}_{=:A_5}.
\end{equation*}
It holds~$\left\lVert A_{1}\right\rVert_{\calB(L^{2}(\mu))}\leqslant1$ and~$\left\lVert A_{3}\right\rVert_{\calB(L^{2}(\mu))}\leqslant1$. The operators~$A_2$ and~$A_4$ are bounded on~$L^{2}(\mu)$ by~\Cref{lem:multiple_operators_bounded}, while the operator~$A_5$ is trace-class on~$L^{2}(\mu)$ by~\eqref{eq:particular_operator_is_trace_class}. Moreover, in view of~\eqref{eq:trace_bounded_norm},~\eqref{eq:particular_operator_is_trace_class},~\eqref{eq:estimate_lemma_operators_bounded_1} and~\eqref{eq:estimate_lemma_operators_bounded_3}, there exists~$K\in\R_{+}$ such that, for any~$s\in[0,1/2]$ and any~$u\in[0,s]$,
\begin{equation}
  \label{eq:T_s_u_trace_norm}
  \left\lVert T_{s,u}\right\rVert_{1}\leqslant K\left\lVert\delta\Diff\right\rVert^{2}_{\calC^{3}(\T^\dim)}.
\end{equation}

If $s\in[1/2,1]$ and~$u\in[0,s/2]$, we rewrite~$T_{s,u}$ as
\begin{equation*}
  T_{s,u}=
  \underbrace{\rme^{(s-u)\alpha\cLD}(\Id-\cL_{\Diff+\delta\Diff})^{2}}_{=:A_5}
  \underbrace{(\Id-\cL_{\Diff+\delta\Diff})^{-2}\cL_{\delta\Diff}(\Id-\cL_{\Diff+\delta\Diff})}_{=:A_4}
  \underbrace{\rme^{u\alpha\cL_{\Diff+\delta\Diff}}}_{=:A_3}
  \underbrace{(\Id-\cL_{\Diff+\delta\Diff})^{-1}\cL_{\delta\Diff}}_{=:A_2}
  \underbrace{\rme^{(1-s)\alpha\cLD}}_{=:A_1},
\end{equation*}
and if~$s\in[1/2,1]$ and~$u\in[s/2,s]$, we rewrite~$T_{s,u}$ as
\begin{equation*}
  T_{s,u}=
  \underbrace{\rme^{(s-u)\alpha\cLD}}_{=:A_1}
  \underbrace{\cL_{\delta\Diff}(\Id-\cL_{\Diff+\delta\Diff})^{-1}}_{=:A_2}
  \underbrace{(\Id-\cL_{\Diff+\delta\Diff})\rme^{u\alpha\cL_{\Diff+\delta\Diff}}(\Id-\cL_{\Diff+\delta\Diff})}_{=:A_5}
  \underbrace{(\Id-\cL_{\Diff+\delta\Diff})^{-1}\cL_{\delta\Diff}}_{=:A_4}
  \underbrace{\rme^{(1-s)\alpha\cLD}}_{=:A_3}.
\end{equation*}
We then perform manipulations similar to the ones leading to~\eqref{eq:T_s_u_trace_norm} in the case~$0\leqslant s\leqslant 1/2$.

In view of~\eqref{eq:bochner_integral_trace_norm}, the operator in~\eqref{eq:duhamel} is therefore trace-class on~$L^{2}(\mu)$, and there exists~$K\in\R_{+}$ such that
\begin{equation*}
  \left\lVert\alpha^{2}\int_{0}^{1}\int_{0}^{s}\rme^{\alpha (s-u)\cLD}\cL_{\delta\Diff}\rme^{u\alpha\cL_{\Diff+\delta\Diff}}\cL_{\delta\Diff}\rme^{(1-s)\alpha\cLD}du\,ds\right\rVert_{1}\leqslant K\left\lVert\delta\Diff\right\rVert_{\calC^{3}(\T^\dim)}^{2}.
\end{equation*}

We can at this stage take the trace in the equality~\eqref{eq:duhamel}. In particular, using the linearity and the cyclicity of the trace, the trace of the operator defined by the first term on the right-hand side of~\eqref{eq:duhamel} is
\begin{align*}
  \Tr\left(
  \alpha\int_{0}^{1}\rme^{s\alpha\cLD}\cL_{\delta\Diff}\rme^{(1-s)\alpha\cLD}ds
  \right)
   & =
  \alpha\int_{0}^{1}\Tr(
  \rme^{s\alpha\cLD}\cL_{\delta\Diff}\rme^{(1-s)\alpha\cLD}
  )ds  \\
   & =
  \alpha\int_{0}^{1}\Tr(
  \rme^{\alpha\cLD}\cL_{\delta\Diff}
  )ds  \\
   & =
  \alpha\Tr(\rme^{\alpha\cLD}\cL_{\delta\Diff})=\Tr(g_{\alpha}'(\Diff)\left(-\cL_{\delta\Diff}\right)).
\end{align*}
This shows that~\eqref{eq:differentials} hold with the bound~$\left\lVert\delta\Diff\right\rVert_{\calC^{3}(\T^{\dim})}^{2}$ on the right-hand side.

Using the density of~$\calC^{3}(\T^{\dim},\calS_\dim)$ in~$L^{\infty}(\T^{\dim},\calS_\dim)$, we now prove that~\eqref{eq:differentials} holds. It suffices to show that the trace of the operator defined by the second term on the right-hand side of~\eqref{eq:duhamel} is bounded from above by~$\left\lVert\delta\Diff\right\rVert_{L^{\infty}(\T^{\dim})}^{2}$. We compute, using Fubini's theorem,
\begin{align*}
   & \Tr\left(\alpha^{2}
  \int_{0}^{1}\int_{0}^{s}
  \rme^{\alpha (s-u)\cLD}\cL_{\delta\Diff}\rme^{u\alpha\cL_{\Diff+\delta\Diff}}\cL_{\delta\Diff}\rme^{(1-s)\alpha\cLD}du\,ds
  \right)                \\
   & \qquad=
  \alpha^{2}
  \int_{0}^{1}\int_{0}^{s}
  \Tr\left(
  \rme^{(s-u)\alpha\cLD}\cL_{\delta\Diff}\rme^{u\alpha\cL_{\Diff+\delta\Diff}}\cL_{\delta\Diff}\rme^{(1-s)\alpha\cLD}
  \right)du\,ds          \\
   & \qquad=
  \alpha^{2}
  \int_{0}^{1}\int_{0}^{s}
  \Tr\left(
  \rme^{(1-u)\alpha\cLD}\cL_{\delta\Diff}\rme^{u\alpha\cL_{\Diff+\delta\Diff}}\cL_{\delta\Diff}
  \right)du\,ds          \\
   & \qquad=
  \alpha^{2}
  \int_{0}^{1}\int_{u}^{1}
  \Tr\left(
  \rme^{u\alpha\cL_{\Diff+\delta\Diff}}\cL_{\delta\Diff}\rme^{(1-u)\alpha\cLD}\cL_{\delta\Diff}
  \right)ds\,du          \\
   & \qquad=
  \alpha^{2}
  \int_{0}^{1}(1-u)
  \Tr\left(
  \rme^{u\alpha\cL_{\Diff+\delta\Diff}}\cL_{\delta\Diff}\rme^{(1-u)\alpha\cLD}\cL_{\delta\Diff}
  \right)du.
\end{align*}
Denote by~$V_{u}=\rme^{u\alpha\cL_{\Diff+\delta\Diff}}\cL_{\delta\Diff}\rme^{(1-u)\alpha\cLD}\cL_{\delta\Diff}$. If $u\in[0,1/2]$, we rewrite~$V_u$ as
\begin{equation*}
  V_u=
  \underbrace{\rme^{u\alpha\cL_{\Diff+\delta\Diff}}}_{=:A_1}
  \underbrace{\cL_{\delta\Diff}(\Id-\cLD)^{-1}}_{=:A_2}
  \underbrace{(\Id-\cLD)\rme^{(1-u)\alpha\cLD}(\Id-\cLD)}_{=:A_3}
  \underbrace{(\Id-\cLD)^{-1}\cL_{\delta\Diff}}_{=:A_4}.
\end{equation*}
It holds~$\left\lVert A_{1}\right\rVert_{\calB(L^{2}(\mu))}\leqslant1$, the operators~$A_2$ and~$A_4$ are bounded on~$L^{2}(\mu)$ by~\Cref{lem:multiple_operators_bounded}, while the operator~$A_3$ is trace-class on~$L^{2}(\mu)$ by~\Cref{lem:trace_class_operators}. In view of~\eqref{eq:trace_bounded_norm} and~\eqref{eq:estimate_lemma_operators_bounded_1}, there exists~$K\in\R_{+}$ such that, for all~$u\in[0,1/2]$,
\begin{equation*}
  \left\lVert V_{u}\right\rVert_{1}\leqslant K\left\lVert\delta\Diff\right\rVert^{2}_{L^{\infty}(\T^{\dim})}.
\end{equation*}
If $u\in[1/2,1]$, using the cyclicity of the trace, we write
\begin{equation*}
  \Tr(V_u)
  =
  \Tr(
  \underbrace{(\Id-\cLD)\rme^{u\alpha\cL_{\Diff+\delta\Diff}}(\Id-\cLD)}_{=:A_1}
  \underbrace{(\Id-\cLD)^{-1}\cL_{\delta\Diff}}_{=:A_2}
  \underbrace{\rme^{(1-u)\alpha\cLD}}_{=:A_3}
  \underbrace{\cL_{\delta\Diff}(\Id-\cLD)^{-1}}_{=:A_4}
  ),
\end{equation*}
and proceed as above.
The inequality
\begin{equation*}
  \left\lvert\int_{0}^{1}(1-u)\Tr\left(V_{u}\right)\,du\right\rvert\leqslant\int_{0}^{1}\left\lvert\Tr\left(V_{u}\right)\right\rvert\,du\leqslant\int_0^1\left\lVert V_u\right\rVert_1\,du,
\end{equation*}
then implies~\eqref{eq:differentials}, which concludes the proof.

\subsection{Technical results}
\label{app:differential_g_h_alpha_technical_results}

The following lemma (whose proof is given in~\Cref{app:trace_class_operators}) shows that~$f_{\alpha}$ is well-defined whenever~$\alpha>0$ and~$\Diff\in L^{\infty}_{V}(\T^{\dim},\calM_{a,b})$ satisfies~$\Diff\geqslant c\Id_\dim$ for some~$c>0$.

\begin{lemma}
  \label{lem:trace_class_operators}
  Fix~$\Diff\in L^{\infty}_{V}(\T^\dim,\calM_{a,b})$ such that~$\Diff\geqslant c\Id_\dim$ for some~$c>0$. Then, for any~$n\in\N$ and~$\theta>0$, the operators~$\left(\cLD\right)^{n}\rme^{\theta\cLD}$ and~$(\Id-\cLD)^{n}\rme^{\theta\cLD}$ are trace-class on~$L^{2}(\mu)$. Moreover, for any~$n\in\N$ and~$\underline{\theta}>0$, there exists~$\theta^{\star}>0$ and~$K\in\R_{+}$ (which both depend on~$n$ and~$\underline{\theta}$) such that
  \begin{equation}
    \label{eq:trace_bound}
    \forall\theta\geqslant\underline{\theta},\qquad
    \left\lVert \left(\Id-\cLD\right)^{n}\rme^{\theta\cLD}\right\rVert_{1}
    \leqslant
    K.
  \end{equation}
\end{lemma}

Using the notation of~\Cref{subsec:euler_lagrange_degenerate}, the limit in~\eqref{eq:limit_f_alpha_to_infty} is therefore obtained with the dominated convergence theorem, using that~$\rme^{-\alpha(\lambda_i-\lambda_2)}\leqslant\rme^{-(\lambda_i-\lambda_2)}$ for any~$\alpha\geqslant1$ and that the sums~$\sum_{i\geqslant3}N_i\lambda_i\rme^{-\lambda_i}$ and~$\sum_{i\geqslant3}N_i\rme^{-\lambda_i}$ are finite since the operators~$g_{1}(\Diff)$ and~$h_{1}(\Diff)$ are trace-class on~$L^{2}(\mu)$.

The three following lemmas are useful to bound various terms when identifying the differentials of~$\Diff\mapsto g_{\alpha}(\Diff)$ and~$\Diff\mapsto h_{\alpha}(\Diff)$. More precisely, the following lemma allows to compare the operators~$\cLD$ and~$\cL_{\delta\Diff}$ as a function of~$\left\lVert\delta\Diff\right\rVert_{L^{\infty}(\T^{\dim})}$. 
\begin{lemma}
  \label{lem:linear_bound}
  Fix~$\Diff\in L^{\infty}_{V}(\T^\dim,\calM_{a,b})$ such that~$\Diff\geqslant c\Id_\dim$ for some~$c>0$. Then for any~$\delta\Diff\in L^{\infty}(\T^\dim,\calS_\dim)$,
  \begin{equation}
    \label{eq:linear_bound}
    \frac{1}{c}\left\lVert\delta\Diff\right\rVert_{L^{\infty}(\T^\dim)}\cLD\leqslant\cL_{\delta\Diff}\leqslant -\frac{1}{c}\left\lVert\delta\Diff\right\rVert_{L^{\infty}(\T^\dim)}\cLD.
  \end{equation}
\end{lemma}
The result is an immediate consequence of the inequality
\begin{equation*}
    \delta\Diff\leqslant\left\lVert\delta\Diff\right\rVert_{L^\infty(\T^\dim)}\Id_\dim\leqslant\frac{1}{c}\left\lVert\delta\Diff\right\rVert_{L^\infty(\T^\dim)}\Diff.
\end{equation*}
Likewise, the following lemma (proved in~\Cref{app:bound_cl_d_delta_diff}) allows to compare the operators~$\cL_{\Diff}$ and~$\cL_{\Diff+\delta\Diff}$ for a small enough perturbation~$\delta\Diff$.
\begin{lemma}
  \label{lem:bound_cl_d_delta_diff}
  Fix~$\Diff\in L^{\infty}_{V}(\T^\dim,\calM_{a,b})$ such that~$\Diff\geqslant c\Id_\dim$ for some~$c>0$, and consider~$\delta\Diff\in L^{\infty}(\T^\dim,\calS_\dim)$ such that~$\left\lVert\delta\Diff\right\rVert_{L^{\infty}(\T^{\dim})}\leqslant c/2$. Then,
  \begin{equation}
    \label{eq:bound_cl_d_delta_diff}
    \frac{\left\lVert\Diff\right\rVert_{L^{\infty}(\T^\dim)}+\left\lVert\delta\Diff\right\rVert_{L^{\infty}(\T^\dim)}}{c}\cLD\leqslant\cL_{\Diff+\delta\Diff}\leqslant \frac{c}{2\left\lVert\Diff\right\rVert_{L^{\infty}(\T^\dim)}}\cLD\leqslant 0.
  \end{equation}
\end{lemma}

Lastly, the following result (proved in~\Cref{app:multiple_operators_bounded}) shows that some operators appearing in the proof of~\Cref{lem:differential_g_h_alpha} are bounded on~$L^{2}(\mu)$, with an explicit bound on~$\delta\Diff$ when needed. Note that the bound~\eqref{eq:estimate_lemma_operators_bounded_3} is not written in terms of the $L^{\infty}$ norm of~$\delta\Diff$, but rather as a function of the~$\calC^{3}$ norm of~$\delta\Diff$. This may be a technical limitation due to our technique of proof, which involves commutators, but the estimate is sufficient for our purposes.

\begin{lemma}
  \label{lem:multiple_operators_bounded}
  Fix~$\Diff\in L^{\infty}_{V}(\T^\dim,\calM_{a,b})$ such that~$\Diff\geqslant c\Id_\dim$ for some~$c>0$, and consider~$\delta\Diff\in \calC^{3}(\T^\dim,\calS_\dim)$ such that~$\left\lVert\delta\Diff\right\rVert_{L^{\infty}(\T^{\dim})}\leqslant c/2$. Then the following statements hold:
  \begin{itemize}
    \item The operators $\cL_{\delta\Diff}(\Id-\cL_{\Diff+\delta\Diff})^{-1}$ and~$\cL_{\delta\Diff}(\Id-\cLD)^{-1}$ are bounded on~$L^{2}(\mu)$. More precisely, there exists~$K\in\R_{+}$ such that
          \begin{equation}
            \label{eq:estimate_lemma_operators_bounded_1}
            \left\lVert\cL_{\delta\Diff}(\Id-\cL_{\Diff+\delta\Diff})^{-1}\right\rVert_{\calB(L^{2}(\mu))}+\left\lVert\cL_{\delta\Diff}(\Id-\cLD)^{-1}\right\rVert_{\calB(L^{2}(\mu))}\leqslant K\left\lVert\delta\Diff\right\rVert_{L^{\infty}(\T^\dim)};
          \end{equation}
    \item The operator~$(\Id-\cL_{\Diff+\delta\Diff})^{2}\rme^{\theta\cLD}$ is bounded on~$L^{2}(\mu)$ for any~$\theta>0$. More precisely, for any~$\underline{\theta}>0$, there exists~$K_{\underline{\theta}}\in\R_{+}$ such that, for all~$\theta>\underline{\theta}$,
          \begin{equation}
            \label{eq:estimate_lemma_operators_bounded_2}
            \left\lVert (\Id-\cL_{\Diff+\delta\Diff})^{2}\rme^{\theta \cLD}  \right\rVert_{\calB(L^{2}(\mu))}\leqslant K_{\underline{\theta}};
          \end{equation}
    \item The operator~$(\Id-\cL_{\Diff+\delta\Diff})\cL_{\delta\Diff}(\Id-\cL_{\Diff+\delta\Diff})^{-2}$ is bounded on~$L^{2}(\mu)$. More precisely, there exists~$K\in\R_{+}$ such that
          \begin{equation}
            \label{eq:estimate_lemma_operators_bounded_3}
            \left\lVert(\Id-\cL_{\Diff+\delta\Diff})\cL_{\delta\Diff}(\Id-\cL_{\Diff+\delta\Diff})^{-2}\right\rVert_{\calB(L^{2}(\mu))}
            \leqslant K\left\lVert\delta\Diff\right\rVert_{\calC^{3}(\T^\dim)},
          \end{equation}
          where~$\left\lVert\delta\Diff\right\rVert_{\calC^{3}(\T^\dim)}=\sup\limits_{1\leqslant i,j\leqslant\dim}\left\lVert\delta\Diff_{i,j}\right\rVert_{\calC^{3}(\T^\dim)}$.
  \end{itemize}
\end{lemma}

\subsubsection{Proof of~\Cref{lem:trace_class_operators}}
\label{app:trace_class_operators}

We first show that the operator~$\rme^{\theta\cLD}$ is trace-class on~$L^{2}(\mu)$ for~$\theta>0$. Using the lower bound on~$\Diff$, it holds~$\cLD\leqslant c\cL_{\Id_\dim}$ on~$L^{2}(\mu)$. 
It therefore suffices to show that the operator~$\rme^{\theta\cL_{\Id_\dim}}$ is trace-class on~$L^{2}(\mu)$. Define the unitary transformation
\begin{equation}
  \label{eq:unitary_transformation}
  U\colon
  \left\lbrace
  \begin{aligned}
    L^{2}(\mu) & \to L^{2}(\T^{\dim}), \\
    f          & \mapsto f\rme^{-V/2}.
  \end{aligned}
  \right.
\end{equation}
Straightforward computations show that
\begin{equation}
  \label{eq:unitary_transformation_laplacian}
  U\cL_{\Id_\dim}U^{\star}=\Delta +W\leqslant0
\end{equation}
on~$L^{2}(\T^{\dim})$, where~$\Delta$ is the Laplacian operator with periodic boundary conditions on~$L^{2}(\T^{\dim})$, and the smooth function
\begin{equation}
  \label{eq:W}
  W:=\frac{1}{2}\Delta V-\frac{1}{4}\left\lvert\nabla V\right\rvert^{2}
\end{equation}
is identified with a multiplication operator on~$L^{2}(\T^{\dim})$. This implies that
\begin{equation*}
  0\leqslant
  \Tr\left(
  \rme^{\theta\cLD}
  \right)\leqslant
  \Tr\left(
  \rme^{c\theta\cL_{\Id_\dim}}
  \right)=
  \Tr_{L^{2}(\T^{\dim})}\left(
  \rme^{c\theta(\Delta +W)}
  \right)\leqslant
  \rme^{c\theta\calW}\Tr_{L^{2}(\T^{\dim})}\left(
  \rme^{c\theta\Delta}
  \right),
\end{equation*}
where~$\calW=\max\limits_{q\in\T^{\dim}}\left\lvert W(q)\right\rvert$. Since the operator~$\rme^{c\theta \Delta}$ is trace-class on~$L^{2}(\T^\dim)$ for any~$\theta>0$ (see for instance~\cite[Theorem XIII.76]{ReedSimon_Vol4}), the operator~$\rme^{\theta\cLD}$ is trace-class on~$L^{2}(\mu)$ for any~$\theta>0$.

To conclude the proof, we rely on the fact that, for~$\underline{\theta}>0$ and~$n\in\N$ fixed, there exists~$\theta^{\star}>0$ and~$K>0$ such that~$(-x)^{n}\rme^{\theta x}\leqslant K\rme^{\theta^{\star}x}$ for any~$x\in\R_{-}$ and any~$\theta\geqslant\underline{\theta}$. This implies that the operator~$\left(\cLD\right)^{n}\rme^{\theta\cLD}$ is trace-class on~$L^{2}(\mu)$ for any~$n\in\N$. It follows that the operator~$(\Id-\cLD)^{n}\rme^{\theta\cLD}$ is also trace-class on~$L^{2}(\mu)$ as it is a linear combination of the operators~$\left(\cLD\right)^{k}\rme^{\theta\cLD}$ for~$0\leqslant k\leqslant n$. This also implies that the estimate~\eqref{eq:trace_bound} holds, and concludes the proof.



\subsubsection{Proof of~\Cref{lem:bound_cl_d_delta_diff}}
\label{app:bound_cl_d_delta_diff}

Since~$\frac{c}{2}\Id_\dim\leqslant\Diff+\delta\Diff\leqslant\left(\left\lVert\Diff\right\rVert_{L^{\infty}(\T^\dim)}+\left\lVert\delta\Diff\right\rVert_{L^{\infty}(\T^\dim)}\right)\Id_\dim$, it holds
\begin{equation}
  \label{eq:proof_bound_cl_d_delta_diff_1}
  \left(
  \left\lVert\Diff\right\rVert_{L^{\infty}(\T^\dim)}+\left\lVert\delta\Diff\right\rVert_{L^{\infty}(\T^\dim)}
  \right)\cL_{\Id_\dim}\leqslant\cL_{\Diff+\delta\Diff}\leqslant \frac{c}{2}\cL_{\Id_\dim}\leqslant 0.
\end{equation}
Using now that~$c\Id_\dim\leqslant\Diff\leqslant\left\lVert\Diff\right\rVert_{L^{\infty}(\T^\dim)}\Id_\dim$, it holds~$\left\lVert\Diff\right\rVert_{L^{\infty}(\T^\dim)}^{-1}\Diff\leqslant\Id_\dim\leqslant c^{-1}\Diff$ so that
\begin{equation}
  \label{eq:proof_bound_cl_d_delta_diff_2}
  c^{-1}\cLD\leqslant\cL_{\Id_\dim}\leqslant\left\lVert\Diff\right\rVert_{L^{\infty}(\T^\dim)}^{-1}\cLD.
\end{equation}
Combining~\eqref{eq:proof_bound_cl_d_delta_diff_1} with~\eqref{eq:proof_bound_cl_d_delta_diff_2} implies~\eqref{eq:bound_cl_d_delta_diff}, which concludes the proof.

\subsubsection{Proof of~\Cref{lem:multiple_operators_bounded}}
\label{app:multiple_operators_bounded}

We prove the first item. Denote by~$T=\cL_{\delta\Diff}(\Id-\cLD)^{-1}$. Using~\Cref{lem:linear_bound}, it holds
\begin{equation*}
  0\leqslant T^{*}T\leqslant \frac{1}{c^{2}}\left\lVert\delta\Diff\right\rVert_{L^{\infty}(\T^\dim)}^{2}(\Id-\cLD)^{-1}\cLD^{2}(\Id-\cLD)^{-1}.
\end{equation*}
The operator on the right-hand side of the inequality is bounded on~$L^{2}(\mu)$ since
\begin{equation*}
  \left\lVert
  (\Id-\cLD)^{-1}\cLD^{2}(\Id-\cLD)^{-1}
  \right\rVert_{\calB(L^{2}(\mu))}\leqslant\sup_{\lambda\geqslant0}\frac{\lambda^{2}}{(1+\lambda)^{2}}=1.
\end{equation*}
Therefore, the operator~$T$ is bounded on~$L^{2}(\mu)$. Now let~$S=\cL_{\delta\Diff}(\Id-\cL_{\Diff+\delta\Diff})^{-1}$. Using~\Cref{lem:bound_cl_d_delta_diff}, it holds
\begin{equation*}
  0\leqslant SS^{*}\leqslant \cL_{\delta\Diff}\left(\Id-\frac{c}{2\left\lVert\Diff\right\rVert_{L^{\infty}(\T^\dim)}}\cLD\right)^{-2}\cL_{\delta\Diff}.
\end{equation*}
The operator~$\cL_{\delta\Diff}\left(\Id-\frac{c}{2\left\lVert\Diff\right\rVert_{L^{\infty}(\T^\dim)}}\left\lVert\Diff\right\rVert_{L^{\infty}(\T^\dim)}^{-1}\cLD\right)^{-1}$ is bounded on~$L^{2}(\mu)$ using a similar argument as above, with an upper bound proportional to~$\left\lVert\delta\Diff\right\rVert_{L^{\infty}(\T^{\dim})}$. This implies that the operator~$S$ is bounded on~$L^{2}(\mu)$. The estimate~\eqref{eq:estimate_lemma_operators_bounded_1} easily follows.

We next prove the second item. Denote now by~$T_{\theta}=(\Id-\cL_{\Diff+\delta\Diff})^{2}\rme^{\theta\cLD}$. In view of~\Cref{lem:bound_cl_d_delta_diff},
\begin{equation*}
  0\leqslant T_{\theta}^{*}T_{\theta}\leqslant \rme^{\theta\cLD}\left(\Id-c^{-1}\left(\left\lVert\Diff\right\rVert_{L^{\infty}(\T^\dim)}+\left\lVert\delta\Diff\right\rVert_{L^{\infty}(\T^\dim)}\right)\cLD\right)^{4}\rme^{\theta\cLD}.
\end{equation*}
The operators~$\left(\cLD\right)^{i}\rme^{\theta\cLD}$ for~$0\leqslant i\leqslant 2$ are bounded on~$L^{2}(\mu)$ by functional calculus. This implies that the operator~$T_{\theta}$ is bounded on~$L^{2}(\mu)$, and the uniform estimate~\eqref{eq:estimate_lemma_operators_bounded_2} is obtained in the same manner as in the end of the proof of~\Cref{lem:trace_class_operators}.

We finally prove the third item. Let~$T=(\Id-\cL_{\Diff+\delta\Diff})\cL_{\delta\Diff}(\Id-\cL_{\Diff+\delta\Diff})^{-2}$. Consider~$\kappa>0$ (a parameter which will be made precise below) and rewrite the operator~$T$ as
\begin{equation}
  \label{eq:op_bounded_c3}
  T=\underbrace{(\Id-\cL_{\Diff+\delta\Diff})(\kappa\Id-\cL_{\Id_\dim})^{-1}}_{:=A_1}\underbrace{(\kappa\Id-\cL_{\Id_\dim})\cL_{\delta\Diff}(\kappa\Id-\cL_{\Id_\dim})^{-2}}_{:=A_2}\underbrace{(\kappa\Id-\cL_{\Id_\dim})^{2}(\Id-\cL_{\Diff+\delta\Diff})^{-2}}_{:=A_3}.
\end{equation}
The operators~$A_1$ and~$A_3$ are bounded on~$L^{2}(\mu)$. We prove this statement for the operator~$A_1$, the proof for~$A_3$ being similar. Using~\eqref{eq:proof_bound_cl_d_delta_diff_1}, it holds
\begin{equation*}
  0\leqslant A_1^{*}A_1\leqslant (\kappa\Id-\cL_{\Id_\dim})^{-1}\left(\Id-\left(\left\lVert\Diff\right\rVert_{L^{\infty}(\T^{\dim})}+\left\lVert\delta\Diff\right\rVert_{L^{\infty}(\T^{\dim})}\right)\cL_{\Id_\dim}\right)^{2}(\kappa\Id-\cL_{\Id_\dim})^{-1},
\end{equation*}
and the operator on the right-hand side of the previous inequality is bounded on~$L^{2}(\mu)$ since
\begin{equation*}
  \begin{aligned}
     & \left\lVert(\kappa\Id-\cL_{\Id_\dim})^{-2}\left(\Id-\left(\left\lVert\Diff\right\rVert_{L^{\infty}(\T^{\dim})}+\left\lVert\delta\Diff\right\rVert_{L^{\infty}(\T^{\dim})}\right)\cL_{\Id_\dim}\right)^{2}\right\rVert_{\calB(L^{2}(\mu))}              \\
     & \qquad\qquad\qquad\qquad\leqslant\left(\sup\limits_{\lambda\geqslant0}\frac{1+\left(\left\lVert\Diff\right\rVert_{L^{\infty}(\T^{\dim})}+\left\lVert\delta\Diff\right\rVert_{L^{\infty}(\T^{\dim})}\right)\lambda}{\kappa+\lambda}\right)^{2}<+\infty.
  \end{aligned}
\end{equation*}
In view of~\eqref{eq:op_bounded_c3}, it therefore suffices to show that the operator~$(\kappa\Id-\cL_{\Id_\dim})\nabla^{*}\delta\Diff\nabla(\kappa\Id-\cL_{\Id_\dim})^{-2}$ is bounded on~$L^{2}(\mu)$, or, equivalently, that the operator
\begin{equation}
  \label{eq:transformation_operator}
  U(\kappa\Id-\cL_{\Id_\dim})\nabla^{*}\delta\Diff\nabla(\kappa\Id-\cL_{\Id_\dim})^{-2}U^{*}=U(\kappa\Id-\cL_{\Id_\dim})U^{*}\left(U\nabla^{*}\delta\Diff\nabla U^{*}\right)U(\kappa\Id-\cL_{\Id_\dim})^{-2}U^{*},
\end{equation}
is bounded on~$L^{2}(\T^{\dim})$, where~$U$ denotes the unitary transformation defined by~\eqref{eq:unitary_transformation}. Using~\eqref{eq:unitary_transformation_laplacian}, it holds
\begin{equation*}
  U(\kappa\Id-\cL_{\Id_\dim})U^{*}=\kappa\Id-\Delta-W,\qquad
  U(\kappa\Id-\cL_{\Id_\dim})^{-2}U^{*}=(\kappa\Id -\Delta-W)^{-2}.
\end{equation*}
where~$W$ is defined in~\eqref{eq:W}. Straightforward computations lead to
\begin{equation*}
  U\nabla^{*}=\left(-\nabla+\frac{1}{2}\nabla V\right)U,\qquad
  \nabla U^{*}=U^{*}\left(\nabla+\frac{1}{2}\nabla V\right).
\end{equation*}
Therefore, the operator in~\eqref{eq:transformation_operator} can be rewritten as
\begin{equation*}
  \left(\kappa\Id-\Delta-W\right)
  \left[\sum_{1\leqslant i,j\leqslant\dim}\left(
    -\partial_{q_j}+\frac{1}{2}\partial_{q_j}V
    \right)\delta\Diff_{i,j}\left(
    \partial_{q_i}+\frac{1}{2}\partial_{q_i}V
    \right)\right]
  \left(\kappa\Id-\Delta-W\right)^{-2}.
\end{equation*}
We now choose~$\kappa$ such that~$\kappa-W\geqslant 1$ (which is possible since~$W$ is smooth on the compact space~$\T^{\dim}$), so that~$\kappa\Id-\Delta-W\geqslant\Id$ on~$L^{2}(\T^{\dim})$. The operators~$(\kappa\Id-\Delta-W)(\Id-\Delta)^{-1}$ and~$(\Id-\Delta)^{2}(\kappa\Id-\Delta-W)^{-2}$ are then bounded on~$L^{2}(\T^{\dim})$, so that it suffices to show that the operator
\begin{equation*}
  \left(\Id-\Delta\right)
  \left[\sum_{1\leqslant i,j\leqslant\dim}\left(
    -\partial_{q_j}+\frac{1}{2}\partial_{q_j}V
    \right)\delta\Diff_{i,j}\left(
    \partial_{q_i}+\frac{1}{2}\partial_{q_i}V
    \right)\right]
  \left(\Id-\Delta\right)^{-2}
\end{equation*}
is bounded on~$L^{2}(\T^{\dim})$ to conclude.

For~$k\geqslant1$, denote by~$H^{k}(\T^{\dim})$ the Sobolev space of square integrable functions on~$\T^{\dim}$ with all (distributional) derivatives up to order~$k$ belonging to~$L^{2}(\T^{\dim})$. Using the Fourier series characterization of Sobolev spaces and the fact that~$V$ is smooth, it holds
\begin{equation*}
  \left\lbrace
  \begin{aligned}
                                              & (\Id-\Delta)^{-2}\in\calB(L^{2}(\T^{\dim}),H^{4}(\T^{\dim})),                          \\
    \forall 1\leqslant i,j\leqslant\dim,\quad & \partial_{q_i}+\frac{1}{2}\partial_{q_i}V\in\calB(H^{4}(\T^{\dim}),H^{3}(\T^{\dim})),  \\
    \forall 1\leqslant i,j\leqslant\dim,\quad & -\partial_{q_j}+\frac{1}{2}\partial_{q_j}V\in\calB(H^{3}(\T^{\dim}),H^{2}(\T^{\dim})), \\
                                              & \Id-\Delta\in\calB(H^{2}(\T^{\dim}), L^{2}(\T^{\dim})).
  \end{aligned}
  \right.
\end{equation*}
To obtain~\eqref{eq:estimate_lemma_operators_bounded_3}, it therefore suffices that~$\delta\Diff_{i,j}$, seen as a multiplication operator on~$H^{3}(\T^{\dim})$, is bounded uniformly in~$i,j\in\left\lbrace1,\dots,\dim\right\rbrace$ with~$\left\lVert\delta\Diff_{i,j}\right\rVert_{\calB(H^{3}(\T^{\dim}))}\leqslant K_{i,j}\left\lVert\delta\Diff\right\rVert_{\calC^{3}(\T^{\dim})}$ for a constant~$K_{i,j}\geqslant 0$. This is easily seen to hold using the Leibniz rule since~$\delta\Diff\in\calC^{3}(\T^{\dim},\calS_\dim)$. Thus, the estimate~\eqref{eq:estimate_lemma_operators_bounded_3} holds, which concludes the proof of~\Cref{lem:multiple_operators_bounded}.

\section{Discretizations of the optimization problems~\eqref{eq:maximizer_in_thm} and~\eqref{eq:maximizer_in_thm_alpha}}
\label{app:numerical}

To solve the optimization problems~\eqref{eq:maximizer_in_thm} and~\eqref{eq:maximizer_in_thm_alpha} in practice, we introduce two finite-dimensional parametrizations: one for the diffusion matrix~$\Diff$, and one for the eigenfunction~$u$. We consider for simplicity the case of isotropic diffusion matrices
\begin{equation}
  \label{eq:scalar_diffusion_for_numerics}
  \Diff(q) = \Df(q) \Id_\dim.
\end{equation}
We introduce a piecewise constant parametrization for~$\Df$ in~\Cref{subsubsec:paramD}, and next a~$\mathbb{P}_1$ finite element parametrization (\emph{i.e.}~using a basis of continuous and piecewise affine functions) for the eigenfunctions~$u$ in~\Cref{subsec:FE}, in the one dimensional case~$\dim=1$ for simplicity of exposition. The optimization problem then boils down to a constrained generalized eigenvalue problem, see~\Cref{subsubsec:disc_optim}. We describe in~\Cref{subsubsec:optim-numeric} how to solve the latter problem.

\subsection{Parametrization of the diffusion matrix}
\label{subsubsec:paramD}

In practice, we parametrize the scalar valued function~$\Df$ in~\eqref{eq:scalar_diffusion_for_numerics} using a finite-dimensional vector subspace of~$L^\infty(\T^\dim,\R)$. More precisely, we introduce a set of~$N$ non-negative functions~$\{\psi_1,\ldots, \psi_N\} \subset L^\infty(\T^\dim,\R)$, and consider functions of the generic form
\begin{equation}
  \label{eq:Diff_parametrized}
  \Df_\diff(q) = \sum_{n=1}^N \diff_n \psi_n(q), \qquad \diff = (\diff_1,\ldots, \diff_N)\in\R^N.
\end{equation}
Nonnegativity conditions for~$\Df_\diff$ may be cumbersome to write when the functions~$\psi_n$ have supports which overlap in a nontrivial way, as is the case for instance for (tensor products of) trigonometric functions. This motivates us to choose
\begin{equation}
  \label{eq:choice_psi_i}
  \psi_n(q) = \mathbf{1}_{K_n}(q),
\end{equation}
the indicator function of a domain~$K_n \subset \T^\dim$, with the condition that~$(K_1,\ldots,K_N)$ forms a partition of~$\T^\dim$ (typically, the sets~$K_n$ are rectangles obtained from a Cartesian mesh). In order for the function~\eqref{eq:Diff_parametrized} to be in the set~\eqref{eq:constrained-set-matriciel} for the choice~\eqref{eq:choice_psi_i}, the components~$(D_n)_{1\leqslant n\leqslant N}$ should belong to the set
\begin{equation}
  \label{eq:def-Dhat}
  \diffset_p^{a,b} = \left\{\diff\in \left[a \rme^{ V_+(K_1)},\frac{1}{b}\rme^{ V_-(K_1)} \right] \times \dots \times \left[a \rme^{ V_+(K_N)},\frac{1}{b}\rme^{ V_-(K_N)} \right] \, \middle| \, \sum_{n=1}^N \omega_{p,n} D_n^p \leq 1 \right\},
\end{equation}
where
\begin{equation*}
  V_-(E) = \inf_E V, \qquad V_+(E) = \sup_E V, \qquad \omega_{p,n} = \int_{K_n} \rme^{- p V}\left\lvert\Id_\dim\right\rvert_{\rmF}^{p}>0.
\end{equation*}
We assume in the sequel that the conditions
\begin{equation}
  \label{eq:condition_ab_partition}
  ab \leq \min_{1 \leq n \leq N} \rme^{ [V_-(K_n)-V_+(K_n)]}, \qquad a  \leq \left( \sum_{n=1}^N \omega_{p,n} \rme^{ p V_+(K_n)} \right)^{-1/p},
\end{equation}
hold, so that the set~$\diffset_p^{a,b}$ is not empty. These conditions can be ensured by choosing~$a,b \geq 0$ sufficiently small. 

The maximization problem~\eqref{eq:maximizer_in_thm} can then be approximated by the finite dimensional maximization problem
\begin{equation}
  \label{eq:optim-PC}
  \textrm{Find }\diff^\star\in\diffset_p^{a,b}\textrm{ such that }\Lambda(\Df_{\diff^\star} \Id_\dim) = \sup_{\diff\in\diffset_p^{a,b}} \Lambda(\Df_\diff \Id_\dim).
\end{equation}
The well-posedness of this optimization problem is guaranteed by the following result.

\begin{proposition}
  \label{prop:well-posed-finite-dim}
  Fix~$N \geq 1$ and~$p\in [1,+\infty)$, and consider a partition~$(K_1,\dots,K_N)$ of~$\T^\dim$ and~$a,b \geq 0$ such that~\eqref{eq:condition_ab_partition} holds. There exists a solution~$\diff^{\star}$ to~\eqref{eq:optim-PC}.
\end{proposition}

Note that the value~$b=0$ is allowed for~$N$ finite, in contrast to~\Cref{thm:well-posedness-1D}, since the inequality in~\eqref{eq:def-Dhat} ensures that all the components~$(D_n)_{1 \leq n \leq N}$ are bounded from above by~$(\min_{1 \leq n \leq N} \omega_{p,n})^{-1/p}$.

\begin{proof}
  The function~$\diff \mapsto \Lambda(\Df_\diff \Id_\dim)$ is upper-semicontinuous on~$\diffset_p^{a,b}$ by arguments similar to the ones used in~\Cref{app:thm:well-posedness-1D}  (and whatever the choice of the norm on~$\R^N$). In addition,~$\diffset_p^{a,b}$ is closed and bounded, hence compact, which, combined with the upper-semicontinuity of the function to maximize, guarantees the existence of a maximum and thus proves the result.
\end{proof}

Likewise, the maximization problem~\eqref{eq:maximizer_in_thm_alpha} can be approximated by the finite dimensional maximization problem
\begin{equation*}
  \textrm{Find }\diff^{\star,\alpha}\in\diffset_p^{a,b}\textrm{ such that }f_\alpha(\Df_{\diff^{\star,\alpha}} \Id_\dim) = \sup_{\diff\in\diffset_p^{a,b}} f_\alpha(\Df_\diff \Id_\dim).
\end{equation*}
The well-posedness of this optimization problem can be obtained by similar arguments.

Notice that one can  prove that any solution~$\diff^{\star}$ of~\eqref{eq:optim-PC} is positive when the domains~$K_n$ have nonempty interiors.

\begin{proposition}
  \label{prop:well-posed-pwc}
  Fix~$N \geq 1$ and~$p\in [1,+\infty)$, and consider a partition~$(K_1,\dots,K_N)$ of~$\T^\dim$ for which the domains $(K_n)_{1 \le n \le N}$ have nonempty interiors. Assume that~$a,b \geq 0$ satisfy~\eqref{eq:condition_ab_partition}, and denote by~${\diff}^{\star} = (\diff^{\star}_{1},\ldots,{\diff}^{\star}_{N})\in\diffset_p^{a,b}$ a solution to the maximization problem~\eqref{eq:optim-PC}. Then,
  \begin{equation*}
    \forall n\in \{1,\ldots, N\}, \qquad {\diff}^{\star}_{n}>0.
  \end{equation*}
\end{proposition}

Note however that this result does not give any information on the behavior of the lower bound on the components of~$D^\star$ in the limit of a vanishing mesh size. The numerical results we present (see for instance~\Cref{fig:res_1}) indicate that the lower bound can indeed converge to~0 as~$N \to +\infty$.

\begin{proof}
There exists~$\gamma > 0$ such that~$\gamma \mathbbm{1}_N\in\diffset_p^{a,b}$ (where~$\mathbbm{1}_N$ is the~$N$-dimensional vector whose components are all equal to~1), and thus, since~$\Df_{\mathbbm{1}_N}(q) = 1$ for any~$q \in \T^\dim$, and in view of the Poincar\'e inequality~\eqref{eq:Poincare},
\begin{equation*}
  \Lambda(\Df_{\diff^{\star}} \Id_\dim) \geq \gamma\Lambda(\Df_{\mathbbm{1}_N} \Id_\dim) \geq \frac{\gamma}{C_{\mu}}.
\end{equation*}
Consider~$n \in \argmin_{m \in \{1,\ldots,N\}} \diff^{\star}_{p,m}$. Since~$K_n$ has a nonempty interior, there exists an open ball~$\mathscr{B}_n$ such that~$\overline{\mathscr{B}_n} \subset K_n$. We consider the function~$v_n$ defined on~$\mathscr{B}_n$ as a normalized eigenfunction associated with the first eigenvalue of the Laplace problem with Dirichlet boundary conditions:
\begin{equation}
  \label{eq:def_v_n_prop_well_posed_pwc}
  -\Delta v_n = \lambda_{1,n} v_n \ \mathrm{on} \ \mathscr{B}_n, \qquad v_n = 0 \ \mathrm{on} \ \partial \mathscr{B}_n, \qquad \int_{\mathscr{B}_n} v_n^2 = 1.
\end{equation}
Note that~$\lambda_{1,n} > 0$. The function~$v_n$ is then extended by~0 on~$\T^\dim \setminus \mathscr{B}_n$. It is easy to check that~$v_n \in H^1(\mu)$. We finally define~$u_n \in H^{1,0}(\mu)$ as
\begin{equation*}
  u_n = v_n - \int_{\mathscr{B}_n} v_n \, \mu.
\end{equation*}
In view of~\eqref{eq:def_v_n_prop_well_posed_pwc}, and denoting by~$\mu_+ = \max\limits_{\T^\dim} \mu < +\infty$,
\begin{equation*}
  \int_{\T^\dim} \Df_{\diff^{\star}} |\nabla u_n|^2 \mu = \diff^\star_{p,n} \int_{\mathscr{B}_n} |\nabla u_n|^2 \mu \leq \mu_+ \diff^\star_{p,n} \int_{\mathscr{B}_n} |\nabla v_n|^2 = \mu_+ \lambda_{1,n} \diff^\star_{p,n}.
\end{equation*}
Therefore, since~$\|u_n\|^2_{L^2(\mu)} > 0$ (otherwise~$v_n$ would be constant on~$\mathscr{B}_n$, which is impossible by~\eqref{eq:def_v_n_prop_well_posed_pwc}),
\begin{equation*}
  \begin{aligned}
    \Lambda(\Df_{\diff^{\star}}\Id_\dim) & = \inf_{u\in H^{1,0}(\T^\dim) \setminus\{0\}}\frac{\dps \int_{\T^\dim}\Df_{\diff^{\star}}(q)|\nabla u(q)|^2\mu(q)\,dq}{\dps\int_{\T^\dim}u(q)^2\mu(q)\,dq} \leq \frac{\dps \int_{\T^\dim}\Df_{\diff^{\star}}(q)|\nabla u_n(q)|^2\mu(q)\,dq}{\dps\int_{\T^\dim}u_n(q)^2\mu(q)\,dq} \leq \frac{\mu_+\lambda_{1,n} \diff^\star_{p,n} }{\|u_n\|^2_{L^2(\mu)}} .
  \end{aligned}
\end{equation*}
This allows to conclude that~$\diff^{\star}_{n}\geq \gamma\|u_n\|^2_{L^2(\mu)} / (\mu_+ \lambda_{1,n}C_{\mu}) > 0$, which leads to the claimed result.
\end{proof}

\subsection{Finite element approximation of the eigenvalue problem}
\label{subsec:FE}

We discuss in this section how to numerically approximate~$\Lambda(\Diff)$ for a given diffusion matrix~$\Diff$. For simplicity of exposition, we restrict the presentation to the one dimensional case~$\dim=1$, the extension to higher dimensional situations posing no difficulties in principle (but numerically viable only for~$\dim \leqslant 3$ in terms of computational costs).

The minimization in~\eqref{eq:lambdaD-init} is performed over a linear space of dimension $I$ obtained by a $\mathbb{P}_1$ finite element discretization over a regular mesh with nodes
$q_i = i/I$ for~$i\in\{0,\ldots,I-1\}$ with~$q_{-1} \equiv q_{I-1}$,~$q_0 \equiv q_{I}$ and~$q_1 \equiv q_{I+1}$ to comply with periodic boundary conditions. The continuous basis functions~$(\varphi_i)_{i=1,\dots,I}$ have support on~$[q_{i-1},q_{i+1}]$ and are piecewise affine:
\begin{equation*}
  \varphi_i(q) =
  \left\{ \begin{aligned}
    I(q-q_{i-1}) & \quad \text{if} \ q\in[q_{i-1},q_i],   \\
    I(q_{i+1}-q) & \quad \text{if} \ q\in[q_{i},q_{i+1}].
  \end{aligned} \right.
\end{equation*}
We denote by~$S_I={\rm Span} (\varphi_1, \ldots, \varphi_I)$, and by
\begin{equation*}
  S_{I,0} = \left\{u \in S_I \, \middle| \, \int_{\T} u(q) \mu(q)\,dq = 0 \right\}.
\end{equation*}
We can then introduce the following mapping, which corresponds to a finite element approximation of~$\Lambda(\mathscr{D}_\diff)$:
\begin{equation}
  \label{eq:lambda-FE}
  \lambda_{N,I}(\diff) = \inf_{u \in S_{I,0} \setminus\{0\}} \frac{\dps \int_{\T} \Df_\diff(q) u'(q)^2 \mu(q)\,dq}{\dps \int_{\T} u(q)^2 \mu(q)\,dq}.
\end{equation}

Let us recall the standard approach to solve~\eqref{eq:lambda-FE}. Upon writing~$u = \sum_{i=1}^I U_i \varphi_i$, the optimization problem~\eqref{eq:lambda-FE} can be rewritten as
\begin{equation*}
  \inf_{U \in \R^I \setminus \{0\}} \frac{U^\top A(\diff) U}{U^\top B U}, \qquad A(\diff) = \sum_{n=1}^N \diff_n A_n,
\end{equation*}
where the matrices~$A_1,\dots,A_N$ and~$B \in \R^{I \times I}$ have entries
\begin{equation}
  \label{eq:A-M-mat}
  [A_n]_{i,j} = \int_{K_n} \varphi_j'(q) \varphi_i'(q) \rme^{-V(q)}\,dq,
  \qquad
  B_{i,j} = \int_{\T} \varphi_j(q) \varphi_i(q) \rme^{-V(q)}\,dq.
\end{equation}
The Euler--Lagrange condition for the minimization problem~\eqref{eq:lambda-FE} is therefore the following finite dimensional eigenvalue problem:
\begin{equation}
  \label{eq:disc_ev}
  A(\diff)U = \sigma B U, \qquad U^\top B U = 1.
\end{equation}
Since~$A(\diff)$ and~$B$ are real symmetric matrices and~$B$ is positive definite, the eigenvalue problem~\eqref{eq:disc_ev} admits~$0 \leq \sigma_1(\diff)\leq \sigma_2(\diff)\leq\ldots\leq \sigma_{I}(\diff)$ real and nonnegative eigenvalues (counted without multiplicities), with associated~$B$-orthonormal eigenvectors~$\{U_1(\diff),U_2(\diff),\ldots, U_{I}(\diff)\}$. In addition,~$\sigma_1(\diff) = 0$ and~$U_1(\diff)$ is proportional to~$(1,\dots,1)$.

\subsection{The discrete optimization problems}\label{subsubsec:disc_optim}
The optimization problem~\eqref{eq:maximizer_in_thm} is finally approximated by replacing~$\Lambda(\mathscr{D}_\diff)$ in~\eqref{eq:optim-PC} by~$\lambda_{N,I}(\diff)$. More precisely, fix~$p \in [1,+\infty)$. For any~$a,b\geqslant 0$ such that~\eqref{eq:condition_ab_partition} holds,
\begin{equation}
  \label{eq:optim-discrete}
  \text{Find } \diff^\star \in \diffset^p_{a,b} \text{ such that }
  \sigma_2(\diff^\star) = \sup_{\diff \in \diffset_p^{a,b}} \sigma_2(\diff) = \lambda_{N,I}^{\star},
\end{equation}
where~$\sigma_2(\diff)$ is the first nonzero (second-smallest) eigenvalue of the problem~\eqref{eq:disc_ev}.

The following proposition guarantees the existence of a maximizer to~\eqref{eq:optim-discrete}, which is moreover positive under some compatibility  on the meshes used to discretize the diffusion coefficient and the eigenvalue problem.

\begin{proposition}
  \label{prop:well-posed-disc}
  Fix~$p \in [1,+\infty)$. Consider~$a,b \geq 0$ such that~\eqref{eq:condition_ab_partition} holds. Then, there exists~$\diff^{\star}\in \diffset_p^{a,b}$ such that
  \begin{equation*}
    \sigma_2(\diff^\star) = \sup_{\diff \in \diffset_p^{a,b}} \sigma_2(\diff)= \lambda_{N,I}^{\star} > 0.
  \end{equation*}
  Assume moreover that, for all $n \in \{1,\ldots,N\}$, there exists~$i \in \{0,\dots,I-1\}$ such that $[q_{i-1},q_{i+1}] \subset \overline{K_n}$. Then, any maximizer satisfies~$\dps \min_{1 \leq n \leq N} \diff^\star_{n}>0$.
\end{proposition}

\begin{proof}
  As already mentioned in the proof of~\Cref{prop:well-posed-finite-dim}, the set~$\diffset_p^{a,b}$ is compact. Thus, to prove the existence of a solution, it suffices to show that the application~$\diff\mapsto \sigma_2(\diff)$ is upper-semicontinuous; in fact, we prove that it is continuous. Indeed, the application~$M\in\mathcal{S}_{I}(\R)\mapsto \sigma_2(M)$ mapping any symmetric matrix to its second-smallest eigenvalue is continuous on~$\mathcal{S}_{I}(\R)$ (see {\em e.g.} \cite[Corollary~III.2.6]{Bhatia:2023544}). On the other hand,~$\diff \mapsto A(\diff)$ is also continuous. By composition,~$\diff \mapsto \sigma_2(\diff)$ is continuous on~$\delta_p^{a,b}$, which proves the existence of a maximizer.

  In addition, the Poincar\'e inequality~\eqref{eq:Poincare} guarantees with the choice~$\diff = \gamma \mathbbm{1}_N$ (where~$\mathbbm{1}_N$ is the~$N$-dimensional vector whose components are all equal to~1, and~$\gamma>0$ is a normalization constant chosen such that~$\diff\in \delta_{p}^{a,b}$) that
  \begin{equation}
    \label{eq:ev_disc-pos-lambda}
    \lambda_{N,I}^{\star} \geq \sigma_2(\gamma \mathbbm{1}_N) \geq \frac{\gamma}{C_\mu} > 0.
  \end{equation}
  To prove that~$\diff^{\star}_{n}>0$ for all~$n\in\{1,\ldots,N\}$, we proceed by contradiction. Assume that there exists~$n\in\{1,\ldots,N\}$ such that~$\diff^{\star}_{n} \leq 0$. Consider~$i\in\{0,\ldots, I-1\}$ such that~$[q_{i-1}, q_{i+1}]\subset \overline{K_n}$. Setting~$\phi_i = \varphi_i - \int_{\T}\varphi_i \, \mu$ (so that~$\phi_i\in S_{I,0}$), we obtain
  \begin{equation*}
    \sigma_2(\diff^{\star})\leq \frac{\dps \int_{\T} \diff^{\star}(q)|\nabla \phi_i(q)|^2\mu(q)\,dq}{\dps \int_{\T}\phi_i(q)^2\mu(q)\,dq} = \diff^{\star}_{n}\frac{\dps \int_{K_n} |\nabla \phi_i(q)|^2\mu(q)\,dq}{\dps \int_{\T}\phi_i(q)^2\mu(q)\,dq}\leq 0,
  \end{equation*}
  which contradicts \eqref{eq:ev_disc-pos-lambda} and thus proves the result.
\end{proof}

The optimization problem~\eqref{eq:maximizer_in_thm_alpha} is approximated as follows. Fix~$p \in [1,+\infty)$. For any~$a,b\geqslant0$ such that~\eqref{eq:condition_ab_partition} holds,
\begin{equation}
  \label{eq:optim-discrete-alpha}
  \text{Find } \diff^{\star,\alpha} \in \diffset^p_{a,b} \text{ such that }
  F_{\alpha}(\diff^{\star,\alpha})=  \sup_{\diff \in \diffset_p^{a,b}}  F_\alpha(\diff)=\lambda_{N,I}^{\star,\alpha},
\end{equation}
where
\begin{equation}
  \label{eq:f_alpha_discrete}
  F_\alpha(\diff)=\frac{H_\alpha(\diff)}{G_\alpha(\diff)},
\end{equation}
with
\begin{equation}
  \label{eq:g_h_alpha_discrete}
  G_\alpha(\diff)=\sum_{i=2}^{I}\rme^{-\alpha\sigma_i(\diff)},\qquad
  H_{\alpha}(\diff)=\sum_{i=2}^{I}\sigma_i(\diff)\rme^{-\alpha\sigma_i(\diff)}.
\end{equation}

The well-posedness of the optimization problem~\eqref{eq:optim-discrete-alpha} (without the positivity of the optimizer) can be obtained by similar arguments. Using similar (and even simpler) arguments as the ones presented after~\eqref{eq:limit_f_alpha_to_infty}, one can show that
\begin{equation*}
  \forall\alpha>0,\qquad
  \lambda_{N,I}^{\star,\alpha}\geqslant\lambda_{N,I}^{\star},
\end{equation*}
and the map~$\alpha\mapsto F_\alpha(\diff)$ is nonincreasing on~$\R_{+}^{*}$. However, the monotonicity of the map~$\alpha\mapsto\sigma_2(\diff^{\star,\alpha})$ is not clear: it is observed numerically that this map is nondecreasing, see~\Cref{fig:res_3_c,fig:res_cos4_b}.
Finally, let~$(\alpha_n)_{n\geqslant0}$ be a sequence of~$\R_{+}^{*}$ such that~$\alpha_n\to+\infty$ when~$n\to+\infty$. The sequence of maximizers~$(\diff^{\star,\alpha_n})_{n\geqslant0}$ is bounded on~$\delta_p^{a,b}$ so that, upon extraction of a subsequence (not indicated in the notation), there exists~$\diff^{\star,\infty}\in\delta_p^{a,b}$ such that
\begin{equation*}
  \diff^{\star,\alpha_n}\xrightarrow[n\to+\infty]{}\diff^{\star,\infty}.
\end{equation*}
By definition, it holds
\begin{equation*}
  F_{\alpha_n}(\diff^{\star,\alpha_n})=\lambda_{N,I}^{\star,\alpha_n}\geqslant\lambda_{N,I}^{\star}\geqslant\sigma_2(\diff^{\star,\infty}).
\end{equation*}
By continuity of the maps~$\diff\mapsto\sigma_i(\diff)$, it is then easy to show that~$F_{\alpha_n}(\diff^{\star,\alpha_n})\to\sigma_2(\diff^{\star,\infty})$, so that~$\lambda_{N,I}^{\star}=\sigma_2(D^{\star,\infty})$ and~$\diff^{\star,\infty}$ solves~\eqref{eq:optim-discrete}. Moreover,~$\lambda_{N,I}^{\star,\alpha}\to\lambda_{N,I}^{\star}$ as~$\alpha\to+\infty$. The latter convergence is illustrated in~\Cref{sec:numerical}.

In conclusion, solving~\eqref{eq:optim-discrete-alpha} for large values of~$\alpha$ should yield a tight upper bound on~$\lambda_{N,I}^{\star}$. In practice, solving~\eqref{eq:optim-discrete-alpha} for small values of~$\alpha$ already yields a good approximation in the situations we have considered.

\subsection{Practical aspects of the implementation}
\label{subsubsec:optim-numeric}
We discuss in this section how to numerically solve~\eqref{eq:maximizer_in_thm}, approximated as~\eqref{eq:optim-discrete}. The numerical solution of the smooth-min approximation~\eqref{eq:maximizer_in_thm_alpha}, approximated as~\eqref{eq:optim-discrete-alpha}, is made precise in~\Cref{sec:numerical_approx_smooth_max}.

To solve~\eqref{eq:optim-discrete} in practice, we used either (i) the Sequential Least Squares Quadratic Programming algorithm (SLSQP), which operates through linearization of the optimality conditions (see for instance~\cite[Section 15.1]{Bonnans_optim} for an introduction to such methods), relying on the implementation available in SciPy~\cite{2020SciPy-NMeth}; or (ii) the IPOPT solver~\cite{Wachter2006} within the Julia programming language, using the JuMP framework~\cite{DunningHuchetteLubin2017}. In order to use these methods, we need to compute the gradient of the objective function and the constraint, which we discuss next.

A first task is to approximate the integrals for the matrix elements in~\eqref{eq:A-M-mat} and the weight factors~$\omega_{p,n}$ in~\eqref{eq:def-Dhat} by some quadrature method. In our numerical results, we set~$N=I$ and~$K_n=[(n-1)/N,n/N)$ for any~$n\in\left\lbrace1,\dots,N\right\rbrace$, so that, for example,~$[A_n]_{i,j}$ can be approximated by (with periodic boundary conditions for the indices)
\begin{equation}
  \label{eq:An_approximation}
  N\rme^{-V\left(\frac{n-1}{N}\right)}(
  \delta_{i,n}\delta_{i,j}+\delta_{i,n-1}\delta_{i,j}-\delta_{i,n}\delta_{i,j+1}-\delta_{i,n-1}\delta_{i,j-1}
  ),
\end{equation}
where~$\delta$ is the Kronecker delta function, and~$B_{i,j}$ can be approximated by
\begin{equation}
  \label{eq:B_approximation}
  \rme^{-V\left(\frac{i-1}{N}\right)}\frac{4\delta_{i,j}+\delta_{i,j-1}+\delta_{i,j+1}}{6N}.
\end{equation}
We refer to~\Cref{sec:numerical_approx_matrix_elements} for more details in order to obtain~\eqref{eq:An_approximation}-\eqref{eq:B_approximation}. Likewise, the weights~$\omega_{p,n}$ can be approximated by
\begin{equation}
  \label{eq:weights_omega_p_n}
  \frac{1}{N}\rme^{-pV\left(\frac{n-1}{N}\right)}.
\end{equation}

To run the optimization algorithms, we next need the gradient of the constraint
\begin{equation*}
  \diff \mapsto \sum_{n=1}^N \omega_{p,n} \diff_n^p,
\end{equation*}
which is readily computed. A more subtle point is to compute the gradient of the target function~$\lambda_{N,I}(\diff) = \sigma_2(\diff)$ with the notation introduced in~\Cref{subsec:FE}. It is readily computed assuming that the second eigenvalue~$\sigma_2(\diff)$ is simple as the following classical result shows. Recall that the matrix~$B$ is symmetric positive definite, so that the matrix~$B^{-1/2}$ is well-defined.

\begin{proposition}
  \label{prop:gradient}
  Consider~$\diff^{0} \in \R^N$ such that the second-smallest eigenvalue~$\sigma_2(\diff^{0})$ of the symmetric positive matrix~$B^{-1/2}A(\diff^{0})B^{-1/2}$ is simple. Then,~$\diff \mapsto \sigma_2(\diff)$ is differentiable around~$\diff^{0}$, and the components of its gradient are
  \begin{equation}
    \label{eq:dlambda}
    \forall n \in \{1,\dots,N\}, \qquad \frac{\partial \sigma_2}{\partial \diff_n}(\diff^{0}) = U(\diff^{0})^\top A_n U(\diff^{0}),
  \end{equation}
  where~$U(\diff^{0})$ is an eigenvector normalized as in~\eqref{eq:disc_ev}.
\end{proposition}

\begin{proof}
  The proof involves manipulations similar to the ones used to prove~\Cref{prop:EL_degenerate}. When~$\sigma_2(\diff^{0})$ is simple,~$\sigma_2$ is smooth in an open neighborhood of~$\diff^{0}$, and~$D \mapsto U(D)$ can also be constructed to be smooth, so that, by taking the partial derivative with respect to~$\diff_n$ of the first condition in~\eqref{eq:disc_ev},
  \begin{equation*}
    \frac{\partial A}{\partial \diff_n}(\diff) U(\diff) + A(\diff) \frac{\partial U}{\partial \diff_n}(\diff) = \frac{\partial \sigma_2}{\partial \diff_n}(\diff) B U(\diff) + \sigma_2(\diff) B \frac{\partial U}{\partial \diff_n}(\diff).
  \end{equation*}
  The desired result follows by multiplying this equality by~$U(\diff)^\top$ on the left, using the normalization~$U(D)^\top B U(D)=1$ and that~$U(D)^{\top}A(D)=\sigma_2(D)U(D)^{\top}B$ by transposing the relation~\eqref{eq:disc_ev} as the matrices~$A(D)$ and~$B$ are real symmetric.
\end{proof}

Using the approximation considered in~\eqref{eq:An_approximation} (recalling that~$I=N$ in this case), the formula approximating the gradient~\eqref{eq:dlambda} reads: for any~$n\in\left\lbrace1,\dots,N\right\rbrace$,
\begin{align}
  \frac{\partial\sigma_2}{\partial D_n}(\diff) & \approx
  N\rme^{-V\left(\frac{n-1}{N}\right)}
  \begin{pmatrix}
    U(\diff)_{n-1},U(\diff)_{n}
  \end{pmatrix}^{\top}
  \begin{pmatrix}
    1 & -1 \\-1&1
  \end{pmatrix}
  \begin{pmatrix}
    U(\diff)_{n-1},U(\diff)_{n}
  \end{pmatrix}\nonumber \\
  & =\label{eq:gradient_An_approximation}
  N\rme^{-V\left(\frac{n-1}{N}\right)}\left[U(\diff)_{n}-U(\diff)_{n-1}\right]^{2}.
\end{align}

\begin{remark}
 \Cref{prop:gradient} provides a formula for the gradient of~$\sigma_2$ around points where this function is differentiable, namely for~$\diff \in\R^N$ such that~$\sigma_2(\diff)$ is simple. However, formula~\eqref{eq:dlambda} makes sense even if~$\sigma_2(\diff)$ is degenerate, although the resulting quantity is not necessarily the gradient of~$\sigma_2$ in this situation. In practice, we use this formula for the gradient for any~$\diff$, including if degeneracy happens. The downside is that the convergence may be slower around degeneracy points.
\end{remark}

\paragraph{Hyperparameters used to obtain the numerical results of~\Cref{sec:numerical}.}
We use a mesh characterized by~$N = I = 1000$ (recall that~$N$ is the dimension of the approximation space for the diffusion coefficient~$\Diff$, and that~$I$ is the number of basis functions for the $\mathbb{P}_1$ finite element method). In particular, we choose~$K_n=[(n-1)/N,n/N)$ for~$n\in\left\lbrace1,\dots,N\right\rbrace$. Note that we used~$N=I$ to simplify the implementation of~$A_n$ (see~\eqref{eq:An_approximation}) and the computation of the gradient of~$\sigma_2$ (see~\Cref{prop:gradient}), even though the assumption in~\Cref{prop:well-posed-disc} to ensure positivity does not hold in that case. Similar results were obtained using~$I=kN$ with~$k\geqslant 2$ which is a sufficient condition for the positivity result to hold. The integrals for the matrix elements in~\eqref{eq:A-M-mat} are approximated by~\eqref{eq:An_approximation}-\eqref{eq:B_approximation}, and the weights~$\omega_{p,n}$ are approximated by~\eqref{eq:weights_omega_p_n}. The initial condition for the diffusion coefficient in the optimization procedure solving either~\eqref{eq:optim-discrete} or~\eqref{eq:optim-discrete-alpha} is set to~$\diff_{\mathrm{hom}}^{\star}$.

The constant~$\gamma$ used to defined the constant diffusion coefficient in~\Cref{sec:numerical} is given by
\begin{equation*}
  \gamma=\left(\sum_{n=1}^{N}\omega_{p,n}\right)^{-1/p},
\end{equation*}
so that~$\sum_{n=1}^{N}\omega_{p,n} D_{\mathrm{cst},n}^{p}=1$. Similarly, the constant~$\widetilde{\gamma}_{\infty}$ in~\eqref{eq:approximation_d_star_infty} is determined by the equality~$\sum_{n=1}^{N}\omega_{p,n} (D^{\star,\infty}_n)^{p}=1$.

\begin{remark}
  \label{rem:numerical_issues}
  We have checked that all the numerical results conducted in~\Cref{sec:numerical,subsec:homog:numerical_results} are converged up to numerical precision (with a convergence threshold of $10^{-15}$). For some potentials, we were not able to get converged results, in the sense that the convergence threshold was not attained before the maximum number of iterations, set to 1000 in our examples. In some cases, the algorithm even stopped before the maximum number of iterations, because it was unable to find a descent direction compatible with the constraints.
  In practice, numerical convergence was harder to obtain when the optimal diffusion coefficient was almost zero at some points.
\end{remark}

\subsubsection{Numerical optimization of the smooth-min approximation}
\label{sec:numerical_approx_smooth_max}

In order to solve~\eqref{eq:optim-discrete-alpha}, we need to compute the gradient of~$F_\alpha$ defined in~\eqref{eq:f_alpha_discrete} with respect to~$\diff$. The expressions of the gradients of the eigenvalues with respect to~$\diff$ are similar to the expressions obtained in the proof of~\Cref{prop:gradient}: if the maps~$\diff\mapsto\sigma_i(\diff)$ with~$i\in\left\lbrace2,\dots,I\right\rbrace$ are differentiable, one finds that
\begin{equation*}
  \forall n\in\left\lbrace1,\dots,N\right\rbrace,\qquad\frac{\partial\sigma_i}{\partial D_n}(D)=U_i(D)^{\top}A_n U_i(\diff),
\end{equation*}
where~$U_i(\diff)$ is an eigenvector associated with~$\sigma_i(\diff)$ which satisfies~$U_i(D)^{\top}B U_i(\diff)=1$. Likewise, using the approximation~\eqref{eq:An_approximation}, and similarly to~\eqref{eq:gradient_An_approximation}, it holds
\begin{equation*}
  \forall n\left\lbrace1,\dots,N\right\rbrace,\qquad
  \frac{\partial \sigma_i}{\partial D_n}(\diff)\approx N\rme^{-V\left(\frac{n-1}{N}\right)}[U_i (\diff)_n-U_i (\diff)_{n-1}]^{2}.
\end{equation*}
Then, it holds
\begin{equation*}
  \frac{\partial G_{\alpha}}{\partial D_n}(\diff)=-\alpha\sum_{i=2}^{I}\rme^{-\alpha\sigma_i(\diff)}U_i(\diff)^{\top}A_n U_i(\diff),
\end{equation*}
and
\begin{equation*}
  \frac{\partial H_{\alpha}}{\partial D_n}(\diff)=\sum_{i=2}^{I}(1-\alpha\sigma_i(\diff))\rme^{-\alpha\sigma_i(\diff)}U_i(\diff)^{\top}A_n U_i(\diff),
\end{equation*}
from which the gradient of~$F_\alpha$ is easily computed.

\subsubsection{Approximation of the integral elements}
\label{sec:numerical_approx_matrix_elements}

\paragraph{Approximation of~$A_n$.} Fix~$n\in\left\lbrace1,\dots,N\right\rbrace$. On~$K_n$, the only basis functions which are nonzero are~$\varphi_{n}$ and~$\varphi_{n-1}$. Therefore, the only nonzero entries for~$A_n$ correspond to the (row, column) pairs~$(n-1,n-1),(n-1,n),(n,n-1)$ and~$(n,n)$. For example, the entry~$(n,n)$ is approximated as
\begin{equation*}
  [A_n]_{n,n}
  =\int\limits_{\frac{n-1}{N}}^{n/N}(\varphi_{n}'(q))^{2}\rme^{-V(q)}\,dq \approx
  \rme^{-V\left(\frac{n-1}{N}\right)}\int\limits_{\frac{n-1}{N}}^{\frac{n}{N}}(\varphi'(q))^{2}\,dq=N\rme^{-V\left(\frac{n-1}{N}\right)},
\end{equation*}
while the entry~$(n,n-1)$ is approximated as~$[A_n]_{n,n-1}\approx-N\rme^{-V\left(\frac{n-1}{N}\right)}$. The formula~\eqref{eq:An_approximation} then follows.

\paragraph{Approximation of~$B$.} Fix~$(i,j)\in\left\lbrace1,\dots,N\right\rbrace^{2}$. Similar approximations as above lead to
\begin{align*}
  B_{i,j}
   & =\int\limits_{0}^{1}\varphi_i(q)\varphi_j(q)\rme^{-V(q)}\,dq
  =\sum_{n=1}^{N}\int\limits_{\frac{n-1}{N}}^{\frac{n}{N}}\varphi_i(q)\varphi_j(q)\rme^{-V(q)}\,dq                                                                                            \\
   & \approx
  \sum_{n=1}^{N}\rme^{-V\left(\frac{n-1}{N}\right)}\int\limits_{\frac{n-1}{N}}^{\frac{n}{N}}\varphi_i(q)\varphi_j(q)\,dq                                                                      \\
   & =\frac{1}{6N}\sum_{n=1}^{N}\rme^{-V\left(\frac{n-1}{N}\right)}\left(2\delta_{j,n-1}\delta_{i,j}+\delta_{j,n-1}\delta_{i,j+1}+2\delta_{j,n}\delta_{i,j}+\delta_{j,n}\delta_{i,j-1}\right) \\
   & =
  \frac{1}{6N}\rme^{-V\left(\frac{i-1}{N}\right)}\left(2\delta_{i,j}+\delta_{i,j+1}\right)+\frac{1}{6N}\rme^{-V\left(\frac{i}{N}\right)}\left(2\delta_{i,j}+\delta_{i,j-1}\right)                   \\
   & \approx
  \frac{1}{6N}\rme^{-V\left(\frac{i-1}{N}\right)}\left(4\delta_{i,j}+\delta_{i,j+1}+\delta_{i,j-1}\right),
\end{align*}
which is exactly~\eqref{eq:B_approximation}.

\section{Proofs of the homogenization results}
\label{app:homogenization}

We fix~$a,b>0$ in all this section. The proofs of the results in~\Crefrange{app:thm:Hcvg-compact}{app:thm:Hcvg-spectral} are given for completeness, as they are obtained by straightforward modifications of standard results of the literature on homogenization problems with Dirichlet boundary conditions~\cite{allaire_homogeneisation,lebris_blanc_homog_book}.

We start by stating a result which will be repeatedly used in the following, and whose proof is based on the Lax--Milgram lemma (by an easy adaptation of~\cite[Lemma~1.3.21]{allaire_homogeneisation}).
\begin{lemma}
  \label{lem:LaxMilgram}
  Fix~$a,b>0$, and consider~$f\in H^{-1}(\T^\dim)$ such that~$\left\langle f,\mathbf{1}\right\rangle_{H^{-1}(\T^\dim),H^{1}_0(\T^\dim)}=0$ and~$\mathcal{A}\in L^{\infty}(\T^\dim,\mathcal{M}_{a,b})$. Then, there exists a unique solution~$u\in H^1(\T^\dim)$ to the problem
  \begin{equation*}
    \left\lbrace
    \begin{aligned}
      -\div{\mathcal{A}\nabla u} & =f  & \text{in }H^{-1}(\T^\dim), \\
      \int_{\T^\dim}u            & =0.
    \end{aligned}
    \right.
  \end{equation*}
  Moreover, there exists~$C>0$ (which does not depend on~$u$ nor~$f$) such that
  \begin{equation}
    \label{eq:bound_LaxMilgram}
    \left\lVert u\right\rVert_{H^1(\T^\dim)}\leqslant C\left\lVert f\right\rVert_{H^{-1}(\T^\dim)}.
  \end{equation}
\end{lemma}

\subsection{Proof of~\Cref{thm:Hcvg-compact}}
\label{app:thm:Hcvg-compact}

The proof boils down to adapting the proof of~\cite[Theorem~1.2.16]{allaire_homogeneisation}, which states a similar compactness result for equations with Dirichlet boundary conditions, to our case with periodic boundary conditions. To do so, we start by considering a modified problem on an extended domain with Dirichlet boundary conditions, and apply~\cite[Theorem 1.2.16]{allaire_homogeneisation} to obtain the~$H$-convergence for this modified problem. We then use a result on the independence of~$H$-convergence on boundary conditions, stated e.g. in~\cite[Proposition~1.2.19]{allaire_homogeneisation}, to obtain the~$H$-convergence in the case of periodic boundary conditions.

Fix~$f\in H^{-1}(\T^\dim)$ such that~$\langle f,\mathbf{1}\rangle_{H^{-1}(\T^\dim),H^1(\T^\dim)} = 0$. Consider the open set~$\Omega = (-1,2)^\dim$ and denote by~$\widetilde{\A}^k$ and~$\widetilde{f}$ the extensions to~$\Omega$ of~$\A^k$ and~$f$ obtained by periodicity. Consider~$\widetilde u_0^k\in H^1(\Omega)$ the solution to
\begin{equation*}
  \left\{\begin{aligned}
    -\div\left(\widetilde \A^k\nabla \widetilde u_0^k\right) = \widetilde  f & \quad \text{ on }\Omega,         \\
    \widetilde u_0^k = 0                                                     & \quad \text{ on }\partial\Omega.
  \end{aligned}\right.
\end{equation*}
In view of~\cite[Theorem 1.2.16]{allaire_homogeneisation}, and upon extraction of a subsequence (not explicitly indicated), there exists~$\widehat{\mathcal{A}}\in L^{\infty}(\Omega,\mathcal{M}_{a,b})$ such that
\begin{equation*}
  \left\{\begin{aligned}
    \widetilde u^k_0                       & \rightharpoonup \widetilde u_0                     & \text{   weakly in }H^1(\Omega),      \\
    \widetilde \A^k\nabla \widetilde u^k_0 & \rightharpoonup \widehat{\A} \nabla \widetilde u_0 & \text{   weakly in }L^2(\Omega)^\dim,
  \end{aligned}\right.
\end{equation*}
where~$\widetilde u_0$ is the solution of the homogenized problem
\begin{equation*}
  \left\{\begin{aligned}
    -\div\left(\widehat{\A} \nabla \widetilde u_0\right) = \widetilde f & \quad \text{ on }\Omega,         \\
    \widetilde u_0 = 0                                                  & \quad \text{ on }\partial\Omega.
  \end{aligned}\right.
\end{equation*}

We now turn back to our original problem with periodic boundary conditions. Denote by~$u^k\in H^1(\T^\dim)$ the unique solution (thanks to~\Cref{lem:LaxMilgram}) to
\begin{equation}
  \label{eq:periodic-bnd-pbm}
  \left\{\begin{aligned}
    -\div\left(\A^k\nabla u^k\right) = f & \quad \text{ on }\T^\dim, \\
    \int_{\T^\dim}u^k(q) \, dq = 0.      &
  \end{aligned}\right.
\end{equation}
Since~$\A^k \in L^\infty(\T^\dim,\mathcal{M}_{a,b})$ and~$u^k$ has average~0 with respect to the Lebesgue measure on~$\T^\dim$, standard estimates based on the Poincar\'e--Wirtinger inequality show that the sequence~$(u^k)_{k\geq 1}$ is uniformly bounded in~$H^1(\T^\dim)$, hence weakly converges in~$H^1(\T^\dim)$ to a function~$u$, up to the extraction of a subsequence. This extracted subsequence also strongly converges in~$L^2(\T^\dim)$. In fact, as made precise at the end of the proof, we will prove that the whole sequence converges. With some abuse of notation, we therefore still denote the extracted subsequence as~$(u^k)_{k\geq 1}$.
Note that, in particular,~$\int_{\T^\dim}u(q) \, dq = 0$.

Denoting by~$\widetilde{u}^k,\widetilde{u}$ the extensions to~$\Omega$ of~$u^k,u$ obtained by periodicity, we claim that
\begin{equation}
  \label{eq:ccl_Prop1.2.19_for_Thm2}
  \widetilde \A^k \nabla \widetilde{u}^k \rightharpoonup \widehat{\A} \nabla \widetilde u \quad \text{ weakly in }L^2_{\rm loc}(\Omega)^\dim.
\end{equation}
In view of~\cite[Proposition 1.2.19]{allaire_homogeneisation}, it suffices to this end to check that~$\widetilde u^k \rightharpoonup \widetilde u$ weakly in~$H^1_{\rm loc}(\Omega)$ and~$-\div\left(\widetilde \A^k \nabla \widetilde{u}^k\right) = \widetilde{f} \in H^{-1}_{\rm loc}(\Omega)$. To prove the first statement, we combine the weak convergence of~$(u^k)_{k \geq 1}$ in~$H^1(\T^\dim)$ with an argument based on a partition of unity. Fix~$\widetilde\phi\in H^{-1}(\Omega)$ with compact support~$K$ in~$\Omega$. Thanks to the compactness of the support and for~$\varepsilon \in (0,1/3)$ fixed, there exists a finite number~$L$ of points~$(x_\ell)_{1\leq \ell\leq L}$ such that~$K \subset \bigcup_{\ell=1}^L \mathcal{B}(x_{\ell},\varepsilon)$, where~$\mathcal{B}(x_{\ell},\varepsilon)$ denotes the open ball of radius~$\varepsilon$ centered at~$x_{\ell}$. Upon reducing~$\varepsilon$, it can be assumed that~$\overline{\mathcal{B}(x_{\ell},\varepsilon)} \subset \Omega$ for all~$\ell\in \{1,\ldots, L\}$. Note that, for all~$\ell\in \{1,\ldots, L\}$, the mapping~$i_\ell : \mathcal{B}(x_{\ell},\varepsilon)\to \T^\dim$ defined as~$i_\ell(x) = x \text{ mod. } \Z^\dim$ is injective. In addition, there exists an associated partition of unity, namely a family of nonnegative smooth functions~$(\chi_{\ell})_{1\leq \ell\leq L}$ such that~$\supp(\chi_{\ell})\subset \overline{\mathcal{B}(x_{\ell},\varepsilon)}$ for all~$\ell\in \{1,\ldots, L\}$, and~$\sum_{\ell=1}^L\chi_{\ell} = 1$ on~$K$. We use the decomposition~$\widetilde\phi = \sum_{\ell = 1}^L\widetilde\phi_{\ell}$, where~$\widetilde\phi_{\ell} = \widetilde\phi\chi_{\ell}$ has support on~$\overline{\mathcal{B}(x_{\ell}, \varepsilon)}$. In particular, for all~$k\geq 1$,
\begin{equation*}
  \left\langle \widetilde{\phi} , \widetilde{u}^k \right\rangle_{H^{-1}(\Omega),H^{1}(\Omega)} = \sum_{\ell=1}^L \left\langle \widetilde{\phi}_{\ell}, \widetilde{u}^k \right\rangle_{H^{-1}(\Omega),H^{1}(\Omega)}.
\end{equation*}
Define next~$\phi_{\ell} \in H^{-1}(\T^\dim)$ from~$\widetilde{\phi}_{\ell}$ as
\begin{equation*}
  \forall v \in H^1(\T^\dim), \qquad \left\langle \phi_{\ell}, v \right\rangle_{H^{-1}(\T^\dim),H^{1}(\T^\dim)} = \left\langle \widetilde{\phi}, \chi_\ell (v \circ i_\ell) \right\rangle_{H^{-1}(\Omega),H^{1}(\Omega)}.
\end{equation*}
In particular,~$\langle \widetilde{\phi}_{\ell}, \widetilde{u}^k\rangle_{H^{-1}(\Omega),H^{1}(\Omega)} = \langle \phi_{\ell}, u^k \rangle_{H^{-1}(\T^\dim),H^{1}(\T^\dim)}$ for all~$k\geq 1$ and~$\langle \widetilde{\phi}_{\ell}, \widetilde{u}\rangle_{H^{-1}(\Omega),H^{1}(\Omega)} = \langle \phi_{\ell}, u \rangle_{H^{-1}(\T^\dim),H^{1}(\T^\dim)}$, so that
\begin{equation*}
  \begin{aligned}
     & \left\langle \widetilde{\phi} , \widetilde{u}^k \right\rangle_{H^{-1}(\Omega),H^{1}(\Omega)} =  \sum_{\ell=1}^L \left\langle \phi_{\ell}, u^k \right\rangle_{H^{-1}(\T^\dim),H^{1}(\T^\dim)}                                                                                                                                                        \\
     & \qquad\qquad \xrightarrow[k\to+\infty]{} \sum_{\ell=1}^L \left\langle \phi_{\ell}, u\right\rangle_{H^{-1}(\T^\dim),H^{1}(\T^\dim)} = \sum_{\ell=1}^L \left\langle \widetilde{\phi}_{\ell}, \widetilde{u} \right\rangle_{H^{-1}(\Omega),H^{1}(\Omega)} = \left\langle \widetilde{\phi} , \widetilde{u} \right\rangle_{H^{-1}(\Omega),H^{1}(\Omega)},
  \end{aligned}
\end{equation*}
which allows to conclude that~$\widetilde u^k \rightharpoonup \widetilde u$ weakly in~$H^1_{\rm loc}(\Omega)$. A similar argument, with the same partition of unity, can be used to prove that~$-\div\left(\widetilde \A^k\nabla \widetilde u^k\right) = \widetilde f$ in~$H^{-1}_{\rm loc}(\Omega)$ for all~$k \ge 1$.

For any~$\Phi \in L^2(\T^\dim)^\dim$, define the function~$\Phi^\dagger:\Omega \to \R$ obtained by extending~$\Phi$ by~0 outside of the set~$(0,1)^\dim$ identified with~$\T^\dim$. More precisely,~$\Phi^\dagger(q) = \mathbf{1}_{(0,1)^\dim}(q) \Phi(i(q))$, where~$i:(0,1)^\dim \to \T^\dim$ is the canonical injection. Note that~$\Phi^\dagger\in L^2(\Omega)^\dim$ has compact support in~$\Omega$, and, using~\eqref{eq:ccl_Prop1.2.19_for_Thm2},
\begin{equation*}
  \begin{aligned}
    \int_{\T^\dim} \Phi(q)^\top \A^k(q)\nabla {u}^k(q) \, dq & = \int_{\Omega} \Phi^\dagger(q)^\top \widetilde \A^k(q)\nabla \widetilde{u}^k(q)\, dq \\ \xrightarrow[k\rightarrow +\infty]{} & \int_{\Omega} \Phi^\dagger(q)^\top \widehat{\A}(q) \nabla \widetilde{u}(q)\, dq = \int_{\T^\dim} \Phi(q)^\top \overline{\A}(q) \nabla {u}(q) \, dq,
  \end{aligned}
\end{equation*}
where~$\overline{\A} = \widehat{\A} \circ i^{-1}$. This shows that~$\A^k\nabla {u}^k$ weakly converges to~$\overline{\A} \nabla {u}$ in~$L^2(\T^\dim)^\dim$. In view of~\eqref{eq:periodic-bnd-pbm},
\begin{equation*}
  \left\{\begin{aligned}
    -\div\left(\overline \A\nabla u\right) = f & \quad \text{ in } H^{-1}(\T^\dim), \\
    \int_{\T^\dim}u(q) \, dq = 0.              &
  \end{aligned}\right.
\end{equation*}
Thus,~$u$ is uniquely defined, and the whole sequences~$(u^k)_{k \geq 1}$ and~$(\A^k \nabla u^k)_{k \geq 1}$ weakly converge respectively to~$u$ in~$H^1(\T^\dim)$ and~$\overline{\A} \nabla {u}$ in~$L^2(\T^\dim)^\dim$, which provides the desired convergence in the sense of~\Cref{def:Hcvg}.

%

\subsection{Proof of~\Cref{thm:Hcvg-spectral}}
\label{app:thm:Hcvg-spectral}

The proof of~\Cref{thm:Hcvg-spectral} is an adaptation of the proof of~\cite[Theorem~1.3.16]{allaire_homogeneisation}. The first step is to show that~$(u^k)_{k \geq 1}$ and~$(\lambda^k)_{k \geq 1}$ converge respectively (up to extraction) to an eigenvector and eigenvalue of the problem
\begin{equation}
  \label{eq:1219-2}
  \left\{\begin{aligned}
    -\div\left(\overline{\mathcal{A}}\nabla u(q)\right) & = \overline{\lambda}\rho(q)u(q) \quad \text{ on } \T^\dim, \\
    \int_{\T^\dim} u(q)^2 \, dq                         & = 1.
  \end{aligned}\right.
\end{equation}
The second step is to prove that~$\overline{\lambda}$ is indeed the smallest nonzero eigenvalue of~\eqref{eq:1219-2}.

\paragraph{Convergence to a solution of the problem~\eqref{eq:1219-2}.}
Let us first prove that, up to the extraction of a subsequence, the sequences~$(\lambda^k)_{k \geq 1}$ and~$(u^k)_{k \geq 1}$ respectively converge in~$\R_+$ and weakly in~$H^1(\T^\dim)$. By the min-max principle, using the fact that the eigenfunctions of the operator~$-\div(\A^k\nabla \cdot)$ associated with the eigenvalue~0 are constant functions,
the second-smallest eigenvalue~$\lambda^k$ satisfies
\begin{equation*}
  \lambda^k = \min_{v \in H^1(\T^\dim) \setminus \{0\}} \left\{ \frac{\dps \int_{\T^\dim} \nabla v(q)^\top \A^k(q)\nabla v(q) \, dq}{\dps \int_{\T^\dim}v^2(q) \rho^k(q) \, dq} \ \middle| \ \int_{\T^\dim} v \rho^k = 0 \right\}.
\end{equation*}
By ,~$a \Id_\dim \leq \A^k \leq b^{-1} \Id_\dim$ and~$0 < \rho_- \leq \rho^k \leq \rho_+$, so that, for all~$v\in H^1(\T^\dim)$,
\begin{equation*}
  \frac{a}{\rho_+} \frac{\dps \int_{\T^\dim} |\nabla v(q)|^2 \, dq}{\dps \int_{\T^\dim} v^2(q) \, dq} \leq \frac{\dps \int_{\T^\dim} \nabla v(q)^\top \A^k(q)\nabla v(q)\,dq}{\dps \int_{\T^\dim}\rho^k(q)v^2(q)\,dq} \leq \frac{1}{b \rho_-} \frac{\dps \int_{\T^\dim} |\nabla v(q)|^2 \, dq}{\dps \int_{\T^\dim} v^2(q) \, dq},
\end{equation*}
which implies
\begin{equation}
  \label{eq:bounds_lambda_k}
  0 < \frac{4\pi^2 a}{\rho_+} \leq \lambda^k \leq \frac{4\pi^2}{b \rho_-}
\end{equation}
since the first nonzero eigenvalue of the operator~$-\Delta$ on~$\T^\dim$ is~$4\pi^2$. This also implies
\begin{equation*}
  a \left\|\nabla u^k \right\|_{L^2(\T^\dim)}^2 \leq \int_{\T^\dim} \nabla u^k(q)^\top \A^k(q)\nabla u^k(q)\,dq = \lambda^k \int_{\T^\dim} \rho^k(q) u^k(q)^2 \, dq \leq \frac{4\pi^2 \rho_+}{b \rho_-},
\end{equation*}
since~$u^k$ is normalized in~$L^2(\T^\dim)$. The sequences~$(\lambda^k)_{k \geq 1}$ and~$(u^k)_{k\geq 1}$ are therefore bounded respectively in~$\R_+$ and~$H^1(\T^\dim)$.

Upon extraction (not indicated explicitly in the notation), there exists a subsequence such that~$\lambda^k \to \overline{\lambda}$, and~$u^k \rightharpoonup u$ weakly in~$H^1(\T^\dim)$ and strongly in~$L^2(\T^\dim)$. In particular,~$\left\| u \right\|_{L^2(\T^\dim)} = 1$. Since~$\lambda^k \rho^k u^k$ is the product of two converging subsequences, respectively for the weak-$*$~$L^{\infty}$ topology and for the strong~$L^2$ topology, it converges to~$\overline{\lambda}\rho u$ weakly in~$L^2(\T^\dim)$, and thus strongly in~$H^{-1}(\T^\dim)$. Note also that
\begin{equation*}
  \lambda^k \int_{\T^\dim}\rho^k u^k = \left\langle -\div\left(\A^k\nabla u^k\right), \mathbf{1} \right\rangle_{H^{-1}(\T^\dim),H^1(\T^\dim)} = 0,
\end{equation*}
so that, since~$\lambda^k \neq 0$ by~\eqref{eq:bounds_lambda_k} and~$\rho^k u^k$ weakly converges to~$\rho u$ in~$L^2(\T^\dim)$,
\begin{equation*}
  0= \int_{\T^\dim}\rho^k u^k=\left\langle \rho^{k}u^{k},\mathbf{1}\right\rangle_{L^{2}(\T^{\dim})}\xrightarrow[k\to+\infty]{}\int_{\T^\dim}\rho u.
\end{equation*}
Since~$\rho u$ has integral~0, the solution~$v^k$ to the following problem is uniquely defined by~\Cref{lem:LaxMilgram}:
\begin{equation*}
  -\div\left(\A^k\nabla v^k\right)=\overline{\lambda}\rho u,\qquad \int_{\T^\dim}v^k=0.
\end{equation*}
By definition of the~$H$-convergence~$\A^k\xrightarrow[k\to+\infty]{H}\overline{\A}$, the sequence~$(v^k)_{k \geq 1}$ converges weakly in~$H^1(\T^\dim)$ to~$v \in H^1(\T^\dim)$, and~$\A^k\nabla v^k\rightharpoonup\overline{\A}\nabla v$ weakly in~$L^2(\T^\dim)^\dim$, with~$v$ the unique solution in~$H^1(\T^\dim)$ to
\begin{equation}
  \label{eq:eq_on_v_for_limit_on_u}
  -\div\left(\overline{\A}\nabla v\right)=\overline{\lambda}\rho u,\qquad\int_{\T^\dim}v=0.
\end{equation}
As
\begin{equation*}
  -\div\left(\A^k\nabla\left(u^k-v^k - \int_{\T^\dim}u^k\right)\right)=\lambda^k\rho^k u^k-\overline{\lambda}\rho u\xrightarrow[k\to+\infty]{}0
\end{equation*}
strongly in~$H^{-1}(\T^\dim)$, it follows by~\eqref{eq:bound_LaxMilgram} that~$u^k-v^k-\int_{\T^\dim}u^k \to 0$ strongly in~$H^1(\T^\dim)$. Therefore, since~$(u^k)_{k \geq 1}$ weakly converges to~$u$ in~$H^1(\T^d)$, we obtain that~$v^k\to u-\int_{\T^\dim}u$ weakly in~$H^1(\T^\dim)$. By uniqueness of the weak limit in~$H^1(\T^\dim)$, it holds~$v=u-\int_{\T^\dim}u$. By plugging this equality in~\eqref{eq:eq_on_v_for_limit_on_u}, we can finally conclude that~$(u^k)_{k\geqslant1}$ converges weakly in~$H^1(\T^\dim)$ to  a solution~$u$ of~\eqref{eq:1219-2}.

\paragraph{Proving that~$\overline{\lambda}$ is the smallest nonzero eigenvalue.}
To conclude the proof, we show that~$\overline{\lambda}$ is the smallest nonzero eigenvalue of~\eqref{eq:1219-2}, proceeding by contradiction. Assume that there exist~$0<\widehat{\lambda}<\overline{\lambda}$ and~$\widehat{u}\in H^1(\T^\dim)$ such that
\begin{equation*}
  \left\{\begin{aligned}
    -\div\left(\overline{\mathcal{A}}(q)\nabla \widehat{u}(q)\right) & = \widehat{\lambda}\rho(q)\widehat{u}(q)  \quad \text{ on } \T^\dim, \\
    \int_{\T^\dim} \widehat{u}(q)^2 \, dq                            & = 1.
  \end{aligned}\right.
\end{equation*}
In particular,
\begin{equation}
  \label{eq:int_rho_uchapeau_egal_0}
  \int_{\T^\dim} \rho \widehat{u} = 0.
\end{equation}
For~$k \geq 1$, define~$w^k \in H^1(\T^\dim)$ as the unique solution to (well-defined by~\Cref{lem:LaxMilgram})
\begin{equation}
  \label{eq:wk-sequence-spectral}
  -\div(\mathcal{A}^k\nabla w^k) = \widehat{\lambda} \rho \widehat{u}, \qquad \int_{\T^\dim} w^k = 0.
\end{equation}
The sequence~$(w^k)_{k \geq 1}$ is bounded in~$H^1(\T^\dim)$ by~\eqref{eq:bound_LaxMilgram}, and hence weakly converges in~$H^1(\T^\dim)$ and strongly converges in~$L^2(\T^\dim)$ up to extraction of a subsequence (not explicitly indicated). The limit~$\widetilde{u}$ satisfies, by definition of~$H$-convergence,
\begin{equation}
  \label{eq:homog_vp_widetilde_u}
  \left\{\begin{aligned}
    -\div\left(\overline{\mathcal{A}}(q)\nabla \widetilde{u}(q)\right) & = \widehat{\lambda}\rho(q)\widehat{u}(q)  \quad \text{ on } \T^\dim, \\
    \int_{\T^\dim} \widetilde{u}(q) \, dq                              & = 0.
  \end{aligned}\right.
\end{equation}
The function~$\widehat{u} - \int_{\T^\dim} \widehat{u}$ also satisfies~\eqref{eq:homog_vp_widetilde_u}, and is therefore equal to~$\widetilde{u}$ by uniqueness of the solution to~\eqref{eq:homog_vp_widetilde_u}. The second condition in~\eqref{eq:homog_vp_widetilde_u} then implies that
\begin{equation}
  \label{eq:egalite_integrales_widehat_widetilde}
  \int_{\T^\dim} \widehat{u} = -\frac{\dps \int_{\T^\dim} \rho \widetilde{u}}{\dps \int_{\T^\dim} \rho}.
\end{equation}

Let us next introduce~$\widetilde{w}^k = w^k - \int_{\T^\dim} \rho^k w^k/\int_{\T^\dim} \rho^k$. Note that~$\int_{\T^\dim} \rho^k \widetilde{w}^k = 0$ by construction and~$\widetilde{w}^k \to \widetilde{u} + \int_{\T^\dim} \widehat{u} = \widehat{u}$ strongly in~$L^2(\T^\dim)$ as~$k \to +\infty$ in view of~\eqref{eq:egalite_integrales_widehat_widetilde}. Then, using first the definition of~$\lambda^k$, and next~\eqref{eq:wk-sequence-spectral},
\begin{equation*}
  \begin{aligned}
    \lambda^k \leq \frac{\dps \int_{\T^\dim}\nabla \widetilde{w}^k(q)^\top \A^k(q)\nabla \widetilde{w}^k(q)\,dq}{\dps \int_{\T^\dim} \widetilde{w}^k(q)^2 \rho^k(q) \, dq} = \frac{\dps \int_{\T^\dim}\nabla w^k(q)^\top \A^k(q)\nabla w^k(q)\,dq}{\dps \int_{\T^\dim} \widetilde{w}^k(q)^2 \rho^k(q) \, dq} & = \widehat{\lambda} \frac{\dps \int_{\T^\dim} w^k(q)\widehat{u}(q) \rho(q)\, dq}{\dps \int_{\T^\dim} \widetilde{w}^k(q)^2 \rho^k(q) \, dq}            \\
    \xrightarrow[k \to +\infty]{} \widehat{\lambda} \frac{\dps \int_{\T^\dim} \widetilde{u}(q) \widehat{u}(q) \rho(q) \, dq}{\dps \int_{\T^\dim} \widehat{u}(q)^2 \rho(q) \, dq}                                                                                                                             & = \widehat{\lambda} \frac{\dps \int_{\T^\dim} \widehat{u}(q)^2 \rho(q) \, dq}{\dps \int_{\T^\dim} \widehat{u}(q)^2\rho(q) \, dq} = \widehat{\lambda},
  \end{aligned}
\end{equation*}
where we used~\eqref{eq:int_rho_uchapeau_egal_0} in the last but one equality. We therefore obtain that~$\overline{\lambda} \leq \widehat{\lambda}$, which contradicts the .

The argument above can in fact be applied to any converging subsequence of the original sequence~$(\lambda^k)_{k \geq 1}$, henceforth leading to the uniqueness of the limit and the convergence of the whole sequence~$(\lambda^k)_{k \geq 1}$, which completes the proof.

\subsection{Proof of~\Cref{prop:periodicity}}
\label{app:lem:periodicity}

Consider a maximizer~$\Diff_0 \in \Diffset_{\#,k,p}^{a,b}$ of~\eqref{eq:lambda-periodic-optim} (which indeed exists in view of~\Cref{thm:well-posedness-1D}), and~$u_0 \in H^1(\T^\dim) \setminus \{0\}$ a minimizer associated with~$\Lambda^k(\Diff_0)$ in~\eqref{eq:lambda-periodic-def}. For~$\ell=(\ell_1,\dots,\ell_\dim) \in \{0,\dots,k-1\}^\dim$, introduce
\begin{equation*}
  \Diff_\ell(q) = \Diff_0\left(q+\frac{\ell}{k}\right), \qquad u_\ell(q) = u_0\left(q+\frac{\ell}{k}\right).
\end{equation*}
For these two functions we have, using the change of variable~$Q = q + \ell/k$ and the~$\mathbb{Z}^\dim$-periodicity of~$V$,
\begin{align*}
  \frac{\dps \int_{\T^\dim}\nabla u_\ell(q)^\top\Diff_\ell(q)\nabla u_\ell(q)\, \rme^{-V_{\#,k}(q)} \, dq}{\dps \int_{\T^\dim} u_\ell(q)^2 \, \rme^{-V_{\#,k}(q)} \, dq} & = \frac{\dps \int_{\T^\dim}\nabla u_0\left(q+\frac{\ell}{k}\right)^\top \Diff_0\left(q+\frac{\ell}{k}\right)\nabla  u_0\left(q+\frac{\ell}{k}\right)\,\rme^{-V(kq)}\,dq}{\dps \int_{\T^\dim}  u_0\left(q+\frac{\ell}{k}\right)^2 \, \rme^{-V(kq)}\,dq} \\
                                                                                                                                                                         & = \frac{\dps \int_{\T^\dim}\nabla u_0(Q)^\top\Diff_0(Q)\nabla  u_0(Q)\,\rme^{-V(kQ-\ell)} \, dQ}{\dps \int_{\T^\dim}  u_0(Q)^2 \rme^{-V(kQ-\ell)} \, dQ} = \Lambda^k(\Diff_0).
\end{align*}
Thus, for all~$\ell \in \{0,\dots,k-1\}^\dim$, the function~$\mathcal{D}_\ell$ also satisfies~$\Lambda^k(\Diff_\ell) = \Lambda^{k,\star}$. Define now the diffusion matrix
\begin{equation*}
  \Diff_0^k(q) = \frac{1}{k^\dim}\sum_{\ell_1=0}^{k-1}\dots \sum_{\ell_\dim=0}^{k-1}\Diff_\ell(q).
\end{equation*}
By concavity of the application~$\Diff\mapsto\Lambda^k(\Diff)$ (see~\Cref{lem:concave-mat}),
\begin{equation*}
  \Lambda^k(\Diff_0^k)\geq \frac{1}{k^\dim}\sum_{\ell_1=0}^{k-1}\dots \sum_{\ell_\dim=0}^{k-1}\Lambda^k(\Diff_\ell) = \Lambda^{k,\star}.
\end{equation*}
Since~$\Diff_0^k \in\Diffset_{\#,k,p}^{a,b}$ by the triangle inequality,~$\Diff_0^k$ also satisfies~$\Lambda^k(\Diff_0^k) = \Lambda^{k,\star}$. In addition,~$\Diff_0^k$ is~$(\Z/k)^\dim$-periodic by construction. There exists therefore~$\Diff^{k,\star} \in \Diffset_{p}^{a,b}$ such that~$\Diff_0^k(q) = \Diff^{k,\star}(kq)$, which concludes the proof.

\subsection{Proof of~\Cref{thm:commutation}}
\label{app:thm:commutation}

The following lemma, proved in~\Cref{subsec:proof-Hcvg-constant}, shows that in the particular case where one considers a sequence of diffusion matrices~$(\A^k)_{k \geq 1}$ such that~$\A^k$ is~$(\Z/k)^\dim$-periodic for any~$k \geq 1$, the~$H$-limit~$\overline{\A}$ is in fact a constant matrix. It also includes comparison results between weak-* and homogenized limits, which are needed in the proof of~\Cref{thm:commutation}. This result is an adaptation to the case of periodic boundary conditions of standard results for divergence operators with Dirichlet boundary conditions (see for instance~\cite[Lemmas~1.3.13 and~1.3.14]{allaire_homogeneisation}).

\begin{lemma}
  \label{lem:Hcvg-constant}
  Let~$a,b>0$ and~$(\A^k)_{k\geq 1} \subset L^\infty(\mathbb{T}^\dim,\mathcal{M}_{a,b})$ be a sequence of diffusion matrices such that~$\A^k$ is~$(\Z/k)^\dim$-periodic for any~$k \geq 1$. Then,~$(\A^k)_{k\geq 1}$~$H$-converges (up to the extraction of a subsequence) to a constant matrix~$\overline{\A}\in \mathcal{M}_{a,b}$.

  Moreover, if~$(\calB^{k})_{k\geqslant1}\subset L^{\infty}(\T^{\dim},\calM_{a,b})$ is another sequence of diffusion matrices such that~$\calB^{k}$ is~$(\Z/k)^{\dim}$-periodic for any~$k\geqslant 1$, and
  \begin{equation*}
    \forall k\geqslant 1,\quad
    \forall\xi\in\T^{\dim},\qquad
    \xi^{\top}\calA^{k}\xi\leqslant\xi^{\top}\calB^{k}\xi,
  \end{equation*}
  then the homogenized limits are also ordered, that is, denoting by~$\overline{\calB}\in\calM_{a,b}$ the~$H$-limit of~$(\calB^{k})_{k\geqslant1}$,
  \begin{equation*}
    \forall\xi\in\T^{\dim},\qquad
    \xi^{\top}\overline{\calA}\xi\leqslant\xi^{\top}\overline{\calB}\xi.
  \end{equation*}

  Furthermore, if
  \begin{equation*}
    \begin{aligned}
      \calA^{k}                   & \rightharpoonup \calA_{+} \text{ weakly-* in }L^{\infty}(\T^{\dim},\calM_{a,b}),      \\
      \left(\calA^{k}\right)^{-1} & \rightharpoonup \calA_{-}^{-1} \text{ weakly-* in }L^{\infty}(\T^{\dim},\calM_{a,b}),
    \end{aligned}
  \end{equation*}
  then it holds
  \begin{equation*}
    \forall\xi\in\T^{\dim},\qquad
    \xi^{\top}\calA_{-}\xi\leqslant\xi^{\top}\overline{\calA}\xi\leqslant\xi^{\top}\calA_{+}\xi.
  \end{equation*}
\end{lemma}

For fixed~$k\geq 1$, recalling the notation~$\Diff_{\#,k}(q) = \Diff(kq)$ for any~$\Diff \in \Diffset_p^{a,b}$, we obtain by definition of the diffusion matrix~$\Diff^{k,\star}$ in~\Cref{prop:periodicity} that
\begin{equation}
  \label{eq:commut-1}
  \forall\Diff\in \Diffset_p^{a,b},\qquad
  \Lambda^k(\Diff_{\#,k}^{k,\star})\geq \Lambda^k(\Diff_{\#,k}).
\end{equation}
Fix now~$\Diff\in\Diffset_p^{a,b}$, and introduce the sequences~$(\A^{k,\star})_{k\geq 1}$ and~$(\A^k)_{k\geq 1}$ defined as
\begin{equation*}
  \A^{k,\star} = \Diff_{\#,k}^{k,\star} \rme^{-V_{\#,k}},
  \qquad
  \A^k = \Diff_{\#,k} \rme^{-V_{\#,k}}.
\end{equation*}
By~\Cref{lem:Hcvg-constant}, upon extraction a subsequence (still indexed by~$k$ with some abuse of notation), the sequence~$(\A^{k,\star})_{k\geq 1}$~$H$-converges to a constant homogenized matrix~$\overline\A^\star\in \mathcal{M}_{a,b}$. By~\Cref{thm:Hcvg}, the sequence~$(\A^k)_{k\geq 1}$ also~$H$-converges to a constant matrix~$\overline{\A}\in\mathcal{M}_{a,b}$. In addition, in view of the definitions~\eqref{eq:lambda_per_k} and~\eqref{eq:lambda-periodic-def},
\begin{equation*}
  \Lambda^k(\Diff_{\#,k}^{k,\star}) = \Lambda_{\#,k}(\Diff^{k,\star}),
  \qquad
  \Lambda^k(\Diff_{\#,k})  = \Lambda_{\#,k}(\Diff).
\end{equation*}
Now, by~\Cref{thm:Hcvg-spectral} and~\Cref{cor:Hcvg-spectral},~$H$-convergence implies the convergence of the associated spectral gap, and thus
\begin{equation*}
  \Lambda_{\#,k}(\Diff^{k,\star})\xrightarrow[k\to+\infty]{} \overline \Lambda^\star = \min_{u \in H^{1,0}(\T^\dim) \setminus\{0\}} \frac{\dps \int_{\T^\dim} \nabla u^\top \overline{\A}^\star \nabla u}{\dps \int_{\T^\dim} u^2 \int_{\T^\dim} \rme^{-V}}, \qquad\qquad \Lambda_{\#,k}(\Diff) \xrightarrow[k\to+\infty]{} \Lambda_{\mathrm{hom}}(\Diff).
\end{equation*}
Thus, passing to the limit~$k\to +\infty$ in~\eqref{eq:commut-1}, we obtain that~$\overline \Lambda^\star \geq \Lambda_{\mathrm{hom}}(\Diff)$. Since this result holds for any~$\Diff\in\Diffset_p^{a,b}$, we finally obtain that
\begin{equation*}
  \overline \Lambda^\star \geq \max_{\Diff \in \Diffset_p^{a,b}}\Lambda_{\mathrm{hom}}(\Diff) = \Lambda_{\mathrm{hom}}^\star,
\end{equation*}
where we recall that~$\Lambda_{\mathrm{hom}}^\star$ is defined in~\eqref{eq:optim-homog-def}.

To obtain the equality~$\overline \Lambda^\star = \max_{\Diff \in \Diffset_p^{a,b}}\Lambda_{\mathrm{hom}}(\Diff)$, it therefore suffices to find a diffusion matrix~$\widehat{\Diff}\in \Diffset_p^{a,b}$ such that~$\Lambda_{\mathrm{hom}}(\widehat{\Diff}) = \overline \Lambda^\star$. The natural candidate is~$\widehat{\Diff}(q) = \mathrm{e}^{V(q)}\overline{\A}^\star$. Note first that~$\widehat{\A}_{\#,k} = \widehat{\Diff}_{\#,k} \rme^{-V_{\#,k}} = \overline{\A}^\star$ for any~$k \geq 1$, so that the sequence~$(\widehat{\A}_{\#,k})_{k \geq 1}$~$H$-converges to~$\overline{\A}^\star$, and therefore~$\Lambda_{\mathrm{hom}}(\widehat{\Diff}) = \overline \Lambda^\star$ by~\Cref{cor:Hcvg-spectral}. To conclude the proof, it suffices to check that~$\widehat{\Diff}\in\Diffset_p^{a,b}$. Note first that, since~$\Diff^{k,\star} \in \Diffset_p^{a,b}$, it holds for any~$\xi \in \R^\dim$ and almost all~$q \in \T^\dim$,
\begin{equation*}
  a |\xi|^2 \leq \mathrm{e}^{- V_{\#,k}(q)} \xi^{\top}\Diff_{\#,k}^{k,\star}(q)\xi \leq \frac1b |\xi|^2.
\end{equation*}
By~\Cref{lem:Hcvg-constant}, the above inequality passes to the~$H$-limit, so that,
\begin{equation*}
  \forall\xi\in\R^{\dim},\qquad
  a |\xi|^2 \leq \xi^{\top}\overline{\A}^{\star}\xi \leq \frac1b |\xi|^2.
\end{equation*}
Plugging~$\widehat{\Diff}(q) =\mathrm{e}^{ V(q)}\overline{\mathcal{A}}^\star$ in the last inequality allows to conclude that~$\widehat{\Diff}\in L^{\infty}_{V}(\T^\dim,\mathcal{M}_{a,b})$. Let us next check that
\begin{equation*}
  \left\|\widehat{\Diff}\right\|_{L^p_V}^p = \int_{\T^\dim} \normF{\widehat\Diff(q)}^p \rme^{- p V(q)} \, dq \leq 1.
\end{equation*}
In view of~\Cref{lem:Hcvg-constant}, it holds~$0 \leq \overline{\A}^\star \leq \A_{+}$ in the sense of symmetric matrices for any weak-* limit~$\A_{+}$ of (a subsequence of)~$(\A^{k,\star})_{k \geq 1}$. By lower semicontinuity of the norm, and since the norm~$|\cdot|_{\rmF}$ is compatible with the order on symmetric positive matrices and weak-*~$L^\infty(\T^\dim)$ convergence implies weak~$L^p(\T^\dim)$ convergence, 
\begin{equation}\label{eq:CV_norm_homog_fin}
  \left\|\widehat{\Diff}\right\|_{L^p_V}^p = \normF{\overline{\A}^\star}^p \leq \int_{\T^\dim} \normF{\A_{+}(q)}^p \, dq \leq \liminf_{k \to +\infty} \int_{\T^\dim} \normF{\mathcal{A}^{k,\star}(q)}^p \, dq = \liminf_{k \to +\infty} \left\|\Diff^{k,\star}\right\|_{L^p_V}^p \leq 1,
\end{equation}
which gives the claimed bound and completes the proof of~\Cref{thm:commutation}.

\begin{remark}
\label{rmk:choice_normalization_motivated}
  The convergence result~\eqref{eq:CV_norm_homog_fin} in fact holds for other choices of normalizations, in particular normalizations for the diffusion matrix of the form
  \[
  \int_{\T^d} \left|\Diff(q)\right|_\F^p \, \rme^{-p V(q)} \, \omega(q) \, dq
  \]
  for~$\omega \in L^\infty(\T^d)$ with~$\omega \geq 0$. Therefore, \Cref{thm:commutation} can be adapted for other normalizations than the one introduced in~\eqref{eq:constrained-set-matriciel}.
\end{remark}

\subsubsection{Proof of~\Cref{lem:Hcvg-constant}}
\label{subsec:proof-Hcvg-constant}

\Cref{thm:Hcvg-compact} guarantees the~$H$-convergence of~$(\A^k)_{k\geq 1}$, up to the extraction of a subsequence, towards a general diffusion matrix~$\overline{\A}\in L^\infty(\mathbb{T}^p,\mathcal{M}_{a,b})$. Thus, we only need to show that~$\overline{\A}$ is constant. To do so, we follow the proof of~\cite[Theorem~1.3.18]{allaire_homogeneisation}. Since~$\A^k$ is~$(\mathbb{Z}/k)^\dim$-periodic, there exists a~$\mathbb{Z}^\dim$-periodic matrix-valued function~$\mathfrak{A}^k \in L^\infty(\mathbb{T}^\dim,\mathcal{M}_{a,b})$ such that~$\A^k(q) = \mathfrak{A}^k(kq)$. Fix a smooth function~$\varphi$ defined on~$\T^{\dim}$, and consider the sequence of oscillating test functions:
\begin{equation}
  \label{eq:corr_periodique}
  \widetilde{w}^k(q) = \varphi(q) + \frac1k \sum_{i=1}^{\dim}w_i^k(kq)\partial_{q_i}\varphi(q),
\end{equation}
where~$w_i^k \in H^1(\T^\dim)$ is the periodic function uniquely defined (in view of~\Cref{lem:LaxMilgram}) by the following equation on~$\mathbb{T}^\dim$:
\begin{equation}
  \label{eq:def_wtilde_i_k}
  -\div\left(\mathfrak{A}^k ( e_i + \nabla w_i^k)\right) = 0, \qquad \int_{\T^\dim} w_i^k = 0.
\end{equation}
The choice~\eqref{eq:corr_periodique} is suggested in the proof of~\cite[Theorem~1.3.1]{AllaireLectNotes} 
and was used for instance in the proof of~\cite[Proposition~7]{LLT12}. By construction, it holds
\begin{equation}
  \label{eq:gradient_w_k}
  \nabla \widetilde{w}^k(q) = \sum_{i=1}^{\dim}\partial_{q_i}\varphi(q)\left(e_i+\nabla w_i^k(kq)\right)+\frac1k\sum_{i=1}^{\dim}w_i^k(kq)\nabla\partial_{q_i}\varphi(q).
\end{equation}

\paragraph{The sequence~$(w_i^k)_{k\geqslant 1}$ is uniformly bounded in~$H^{1}(\T^\dim)$.}
By~\eqref{eq:def_wtilde_i_k}, $w_i^k$ solves 
\begin{equation*}
  -\div\left(\mathfrak{A}^k\nabla w_i^k\right)=\div\left(\mathfrak{A}^k e_i\right),\qquad \int_{\T^\dim}w_i^k=0.
\end{equation*}
Since~$\left\langle\div\left(\mathfrak{A}^k e_i\right),\mathbf{1}\right\rangle_{H^{-1}(\T^\dim),H^1(\T^\dim)}=0$,~\eqref{eq:bound_LaxMilgram} implies that there exists a constant~$C>0$ (which does not depend on~$k$) such that
\begin{equation}
  \label{eq:w_ik_LM}
  \left\lVert w_i^k\right\rVert_{H^{1}(\T^\dim)}\leqslant C\left\lVert\div\left(\mathfrak{A}^{k}e_i\right)\right\rVert_{H^{-1}(\T^\dim)}.
\end{equation}
Since~$\mathfrak{A}^k\in L^{\infty}(\T^\dim,\calM_{a,b})$ with~$b>0$, one easily bounds the right-hand side of~\eqref{eq:w_ik_LM} by a constant which does not depend on~$k$.

\paragraph{The sequence~$(\widetilde{w}^k)_{k\geqslant 1}$ weakly converges in~$H^{1}(\T^\dim)$ to~$\varphi$.}
Since~$\left\lVert w_i^k(k\cdot)\right\rVert_{L^{2}(\T^\dim)}=\left\lVert w_i^k\right\rVert_{L^{2}(\T^\dim)}$ and knowing that~$(w_i^k)_{k\geqslant 1}$ is uniformly bounded in~$L^{2}(\T^\dim)$, it is clear in view of~\eqref{eq:corr_periodique} that~$(\widetilde{w}^k)_{k\geqslant1}$ strongly converges to~$\varphi$ in~$L^{2}(\T^\dim)$. As for the weak~$L^{2}(\T^\dim)$ convergence of the gradient, we use~\Cref{lem:homog_lim} in~\Cref{sec:lemma_avg}: the sequence~$(e_j\cdot\nabla w_i^k)_{k\geqslant 1}$ is uniformly bounded in~$L^{2}(\T^\dim)$ for any~$1\leqslant j\leqslant\dim$, and it holds
\begin{equation*}
  \int_{\T^\dim}e_j\cdot\nabla w_i^k(q)\,dq =0.
\end{equation*}
Therefore, the sequence~$(e_j\cdot\nabla w_i^k(k\cdot))_{k\geqslant 1}$ weakly converges in~$L^{2}(\T^\dim)$ to 0, so that~$(\nabla\widetilde{w}^k)_{k\geqslant 1}$ weakly converges in~$L^{2}(\T^\dim)^{\dim}$ to~$\sum_{i=1}^{d}\partial_{q_i}\varphi\,e_i=\nabla\varphi$.

\paragraph{The sequence~$\left(\A^k(e_i+\nabla w_i^k(k\cdot))\right)_{k\geqslant 1}$ weakly converges in~$L^{2}(\T^\dim)^\dim$ to a constant vector~$\mathfrak{F}_i$.}

Since~$\left\lVert \nabla w_i^k(k\cdot)\right\rVert_{L^{2}(\T^\dim)}=\left\lVert\nabla w_i^k\right\rVert_{L^{2}(\T^\dim)}$, and knowing that~$(\nabla w_i^k)_{k\geqslant1}$ is uniformly bounded in~$L^{2}(\T^\dim)$, the sequence~$\left(\A^k(e_i+\nabla w_i^k(k\cdot))\right)_{k\geqslant 1}$ is uniformly bounded in~$L^{2}(\T^\dim)$. Therefore, we can introduce, up to extraction, the weak limit~$\mathfrak{F}_i$ of this sequence in~$L^2(\T^\dim)^\dim$. Let us show that this weak limit is constant using again~\Cref{lem:homog_lim}. Indeed, for~$1\leqslant j\leqslant \dim$, let~$a_k^j=e_j\cdot \mathfrak{A}^k(e_i+\nabla w_i^k)\in L^2(\T^\dim)$. The family~$(a_k^j)_{k\geqslant0}$ is uniformly bounded in~$L^2(\T^\dim)$ so that, using the Cauchy--Schwarz inequality,
\begin{equation*}
  \sup\limits_{k\geqslant 1}\int_{\T^\dim}\left\lvert a_k^j\right\rvert \leq \sup\limits_{k\geqslant 1}\| a_k^j \|_{L^2(\T^\dim)} < +\infty.
\end{equation*}
Moreover, up to extraction, there exists~$\rho_j\in\R$ such that~$\int_{\T^\dim}a_k^j$ converges to~$\rho_j$. This allows us to conclude that~$\mathfrak{F}_{i,j}$, the~$j$-th component of~$\mathfrak{F}_{i}$, is constant. By multiplying~\eqref{eq:gradient_w_k} by~$\A^k$, we therefore obtain that the sequence~$\left(\A^k\nabla\widetilde{w}^k\right)_{k\geqslant1}$ weakly converges to~$\mathfrak{F}\nabla\varphi$ in~$L^{2}(\T^\dim)^\dim$.

\paragraph{A preliminary computation.} We show a convergence result that will be used multiple times in the proof. This can be seen as the counterpart of the celebrated div-curl lemma in our context. Let~$(u^{k})_{k\geqslant1}$ be a sequence converging weakly in~$H^{1}(\T^\dim)$ and strongly in~$L^{2}(\T^\dim)$ to~$u\in H^{1}(\T^\dim)$. It then holds
\begin{equation}
  \label{eq:homog_preliminary_comp}
  \left\langle-\div\left(\A^k\nabla \widetilde{w}^k\right),u^k\right\rangle_{H^{-1}(\T^\dim),H^{1}(\T^\dim)}\xrightarrow[k\to+\infty]{}\left\langle-\div\left(\mathfrak{F}\nabla\varphi\right), u\right\rangle_{H^{-1}(\T^\dim),H^{1}(\T^\dim)},
\end{equation}
where the constant matrix~$\mathfrak{F}$ has entries~$\mathfrak{F}_{i,j}$.
Indeed, using~\eqref{eq:gradient_w_k}, it holds
\begin{equation}
  \label{eq:ak_nablawk}
  \int_{\T^\dim}\A^k\nabla \widetilde{w}^k\cdot\nabla u^k = \sum_{i=1}^\dim\int_{\T^\dim}\left(e_i+\nabla w^k_i(k\cdot)\right)\partial_{q_i}\varphi \cdot\A^k\nabla u^k + \frac1k \sum_{i=1}^\dim\int_{\T^\dim} w_i^k(k\cdot)\nabla\partial_{q_i}\varphi \cdot\A^k\nabla u^k.
\end{equation}
The second term on the right-hand side of the previous equality converges to~0 as~$k \to +\infty$ since~$\left(\A^k\right)_{k\geqslant 1}$ is uniformly bounded in~$L^{\infty}(\T^\dim,\calM_{a,b})$ with~$b>0$, $\left(\nabla u^k\right)_{k\geqslant1}$ weakly converges in~$L^{2}(\T^\dim)^\dim$ to~$\nabla u$ and the sequence~$(w_i^k(k\cdot))_{k\geqslant1}$ is uniformly bounded in~$L^{2}(\T^\dim)$. For the first term, we  use an integration by parts and~\eqref{eq:def_wtilde_i_k} to write
\begin{equation}
  \label{eq:identify_weak_limit}
  \begin{aligned}
     & \sum_{i=1}^\dim\int_{\T^\dim}\left(e_i+\nabla w^k_i(k\cdot)\right)\partial_{q_i}\varphi \cdot\A^k\nabla u^k = -\sum_{i=1}^\dim\int_{\T^\dim}\A^k\left(e_i+\nabla w_i^k(k\cdot)\right)\cdot\left(\nabla\partial_{q_i}\varphi\right) u^k \\
     & \qquad \xrightarrow[k\to+\infty]{}
    -\sum_{i=1}^\dim\int_{\T^\dim} \mathfrak{F}_i\cdot\left(\nabla\partial_{q_i}\varphi\right) u
    =
    -\sum_{i,j=1}^\dim\int_{\T^\dim}\mathfrak{F}_{i,j}\left(\partial_{q_i,q_j}^2\varphi\right) u=\int_{\T^\dim} \mathfrak{F}\nabla\varphi\cdot\nabla u.
  \end{aligned}
\end{equation}
This shows that~\eqref{eq:homog_preliminary_comp} holds.

\paragraph{The matrix~$\mathfrak{F}$ is equal to~$\overline{\A}$.}

Let us next identify the weak limit~$\mathfrak{F}$. Fix~$u\in H^1(\T^\dim)$, and introduce~$f = -\div(\overline{\A}\nabla u) \in H^{-1}(\T^\dim)$. Note that~$\left\langle f,\mathbf{1}\right\rangle_{H^{-1}(\T^\dim),H^1(\T^\dim)}=0$. Introduce next the solution~$u^k$ (again well-defined by~\Cref{lem:LaxMilgram}) of
\begin{equation*}
  -\div\left(\A^k\nabla u^k\right)=f,\qquad \int_{\T^\dim}u^k=0.
\end{equation*}
By definition of~$H$-convergence,~$u^k\rightharpoonup u-\int_{\T^\dim}u$ weakly in~$H^1(\T^\dim)$, and strongly in~$L^{2}(\T^\dim)$ upon extraction. Testing against~$\widetilde{w}^k$, it holds
\begin{equation}
  \label{eq:eq_to_pass_to_limit_oscillating_test_fct}
  \left\langle f, \widetilde{w}^k \right\rangle_{H^{-1}(\T^\dim),H^1(\T^\dim)} = \int_{\T^\dim}\nabla \widetilde{w}^k \cdot \A^k\nabla u^k = \left\langle-\div\left(\A^k\nabla\widetilde{w}^k\right),u^k\right\rangle_{H^{-1}(\T^\dim),H^{1}(\T^\dim)}.
\end{equation}
By~\eqref{eq:homog_preliminary_comp}, it holds
\begin{equation*}
  \int_{\T^\dim}\nabla\widetilde{w}^k\cdot\A^k\nabla u^k\xrightarrow[k\to+\infty]{}\int_{\T^\dim}\mathfrak{F}\nabla\varphi\cdot\nabla \left(u-\int_{\T^\dim}u\right)=\int_{\T^\dim}\mathfrak{F}\nabla\varphi\cdot\nabla u.
\end{equation*}
As~$\widetilde{w}^k$ weakly converges in~$H^{1}(\T^\dim)$ to~$\varphi$, it follows, by passing to the limit~$k \to +\infty$ in~\eqref{eq:eq_to_pass_to_limit_oscillating_test_fct}, that
\begin{equation*}
   \left\langle f, \varphi \right\rangle_{H^{-1}(\T^\dim),H^1(\T^\dim)} = \int_{\T^\dim}\mathfrak{F}\nabla u\cdot\nabla \varphi.
\end{equation*}
By definition of~$f$,
\begin{equation*}
   \left\langle f, \varphi \right\rangle_{H^{-1}(\T^\dim),H^1(\T^\dim)} = \int_{\T^\dim}\overline{\mathcal{A}}\nabla u\cdot\nabla \varphi.
\end{equation*}
As the previous equalities hold true for any smooth function~$\varphi$, we obtain that~$\overline{\A} \nabla u=\mathfrak{F} \nabla u$, which indeed allows us to conclude that~$\overline{\A} = \mathfrak{F}$ since~$u \in H^1(\T^\dim)$ was arbitrary. This shows finally that the homogenized matrix is constant.

\begin{remark}[The periodic case]
  \label{rmk:cv_H_periodic}
  In the case when~$\A^k(q) = \mathfrak{A}(kq)$ (which corresponds to the setting of~\Cref{thm:Hcvg}), no extraction is needed for the sequence of functions~$(\A^k(e_i+\nabla w_i^k(k\cdot)))_{k \geq 1}$ to weakly converge in~$L^2(\T^\dim)$. This is due to~\Cref{lem:homog_lim}, and the fact that the functions~$w_i^k$ obtained from~\eqref{eq:def_wtilde_i_k} are independent of~$k$, and actually correspond to the functions~$w_i$ defined in~\eqref{eq:Abar2}. Therefore, there is a unique possible limiting constant matrix~$\mathfrak{F}$, with
  \begin{equation*}
    \mathfrak{F}_i = \int_{\T^\dim} \A (e_i + \nabla w_i) = \sum_{j=1}^d e_j \int_{\T^\dim} e_j \cdot \A (e_i + \nabla w_i) = \sum_{j=1}^d e_j \int_{\T^\dim} (e_j + \nabla w_j)\cdot \A (e_i + \nabla w_i),
  \end{equation*}
  where we used~\eqref{eq:def_wtilde_i_k} for the last equality. This allows us to recognize the matrix in~\eqref{eq:Abar1} applied to~$e_i$.
  Since all converging subsequences have the same limit, this shows in fact that the full sequence~$(\A^k)_{k \geq 1}$ converges, as stated in~\Cref{thm:Hcvg}.
\end{remark}
We now show the second part of~\Cref{lem:Hcvg-constant}, whose proof is an adaptation of~\cite[Lemma~1.3.13]{allaire_homogeneisation}. Define the sequence of oscillating test functions corresponding to the sequence~$(\calB^{k})_{k\geqslant0}$, namely
\begin{equation*}
  \widetilde{v}^{k}(q)=\varphi(q)+\frac{1}{k}\sum_{i=1}^{\dim}v_i^{k}(kq)\partial_{q_{i}}\varphi(q),
\end{equation*}
where~$v_{i}^{k}$ is such that
\begin{equation*}
  -\div\left(\mathfrak{B}^{k}(e_i+\nabla v_{i}^{k})\right)=0,\quad
  \int_{\T^{\dim}}v_i^k=0,
\end{equation*}
with~$\mathfrak{B}^{k}\in L^{\infty}(\T^{\dim},\calM_{a,b})$ such that~$\calB^{k}(q)=\mathfrak{B}^{k}(kq)$ for almost all~$q\in\T^{\dim}$  (this is the counterpart of~\eqref{eq:corr_periodique}-\eqref{eq:def_wtilde_i_k} with~$\A^k$ replaced by~$\calB^k$). The sequence~$\left(\widetilde{v}^{k}\right)_{k\geqslant0}$ then weakly converges in~$H^{1}(\T^{\dim})$ to~$\varphi$ in the same way as the sequence~$(\widetilde{w}^k)_{k\geqslant0}$ weakly converges in~$H^{1}(\T^{\dim})$ to~$\varphi$. Since for almost all~$q\in\T^{\dim}$, the matrix~$\calA^{k}$ is symmetric and positive, it holds
\begin{equation*}
  \calA^{k}\nabla \widetilde{w}^{k}\cdot\nabla \widetilde{w}^{k}-2\calA^{k}\nabla \widetilde{w}^{k}\cdot\nabla \widetilde{v}^{k}+\calA^{k}\nabla \widetilde{v}^{k}\cdot\nabla \widetilde{v}^{k}=\calA^{k}\left(\nabla \widetilde{w}^{k}-\nabla \widetilde{v}^{k}\right)\cdot (\nabla \widetilde{w}^{k}-\nabla \widetilde{v}^{k})\geqslant0.
\end{equation*}
Using that~$\calA^{k}\leqslant \calB^{k}$, this implies that
\begin{equation}
  \label{eq:ak_bk_bound}
  \calA^{k}\nabla \widetilde{w}^{k}\cdot\nabla \widetilde{w}^{k}-2\calA^{k}\nabla \widetilde{w}^{k}\cdot\nabla \widetilde{v}^{k}+\calB^{k}\nabla \widetilde{v}^{k} \cdot\nabla \widetilde{v}^{k}\geqslant0.
\end{equation}
We now pass to the limit~$k\to+\infty$. To do so, consider~$\psi\in\calC^{\infty}(\T^{\dim})$ with~$\psi\geqslant0$. Multiplying~\eqref{eq:ak_bk_bound} by~$\psi$ and integrating over~$\T^{\dim}$, it holds
\begin{equation}
  \label{eq:int_ipp_div_curl_lemma}
  \int_{\T^{\dim}}\calA^{k}\nabla \widetilde{w}^{k}\cdot\nabla \widetilde{w}^{k}\psi-2\int_{\T^{\dim}}\calA^{k}\nabla \widetilde{w}^{k}\cdot\nabla \widetilde{v}^{k}\psi+\int_{\T^{\dim}}\calB^{k}\nabla \widetilde{v}^{k}\cdot\nabla \widetilde{v}^{k}\psi\geqslant0.
\end{equation}
For any~$k\geqslant0$, it holds~$\widetilde{w}^{k}\psi\in H^{1}(\T^{\dim})$. Therefore, one obtains
\begin{equation}
  \label{eq:ipp_div_curl_lemma}
  \int_{\T^{\dim}}\calA^{k}\nabla \widetilde{w}^{k}\cdot\nabla \widetilde{w}^{k} \, \psi
  =\left\langle-\div(\calA^{k}\nabla \widetilde{w}^{k}),\widetilde{w}^{k}\psi\right\rangle_{H^{-1}(\T^{\dim}),H^{1}(\T^{\dim})}-\int_{\T^{\dim}}\calA^{k}\nabla \widetilde{w}^{k}\cdot\nabla\psi \, \widetilde{w}^{k}.
\end{equation}
The sequence~$(\widetilde{w}^{k}\psi)_{k\geqslant0}$ converges to~$\varphi\psi$, weakly in~$H^{1}(\T^{\dim})$ and strongly in~$L^{2}(\T^\dim)$, so that, using~\eqref{eq:homog_preliminary_comp}, it holds
\begin{equation}
  \label{eq:first_term_ipp_div_curl_lemma}
  \left\langle-\div(\calA^{k}\nabla \widetilde{w}^{k}),\widetilde{w}^{k}\psi\right\rangle_{H^{-1}(\T^{\dim}),H^{1}(\T^{\dim})}\xrightarrow[k\to+\infty]{}
  \left\langle-\div(\overline{\calA}\nabla \varphi),\varphi\psi\right\rangle_{H^{-1}(\T^{\dim}),H^{1}(\T^{\dim})}.
\end{equation}
As~$(\calA^{k}\nabla \widetilde{w}^{k})_{k\geqslant0}$ weakly converges in~$L^{2}(\T^\dim)^\dim$ to~$\overline{\calA}\nabla\varphi$, and the sequence~$\left(\widetilde{w}^{k}\nabla\psi\right)_{k\geqslant0}$ strongly converges to~$\varphi\nabla\psi$ in~$L^{2}(\T^{\dim})^{\dim}$, one obtains
\begin{equation}
\label{eq:second_term_ipp_div_curl_lemma}
  \int_{\T^{\dim}}\calA^{k}\nabla \widetilde{w}^{k}\cdot\nabla \psi \, \widetilde{w}^{k}
  \xrightarrow[k\to+\infty]{}
  \int_{\T^{\dim}}\overline{\calA}\nabla \varphi\cdot\nabla \psi \, \varphi.
\end{equation}
Combining~\eqref{eq:first_term_ipp_div_curl_lemma} and~\eqref{eq:second_term_ipp_div_curl_lemma} in~\eqref{eq:ipp_div_curl_lemma}, it holds
\begin{equation*}
  \int_{\T^{\dim}}\calA^{k}\nabla \widetilde{w}^{k}\cdot\nabla \widetilde{w}^{k}\psi\xrightarrow[k\to+\infty]{}
  \left\langle-\div(\overline{\calA}\nabla \varphi),\varphi\psi\right\rangle_{H^{-1}(\T^{\dim}),H^{1}(\T^{\dim})}-\int_{\T^{\dim}}\overline{\calA}\nabla \varphi\cdot\nabla \psi \, \varphi=\int_{\T^{\dim}}\overline{\calA}\nabla\varphi\cdot\nabla\varphi\, \psi.
\end{equation*}
Using a similar reasoning to identify  the limits of the two other terms in~\eqref{eq:int_ipp_div_curl_lemma} (using the convergence result~\eqref{eq:homog_preliminary_comp} for the sequence~$\left(\calB^k\right)_{k\geqslant1}$), it holds
\begin{equation*}
  \int_{\T^{\dim}}\left(\overline{\calA}\nabla\varphi\cdot\nabla\varphi-2\overline{\calA}\nabla\varphi\cdot\nabla\varphi+\overline{\calB}\nabla\varphi\cdot\nabla\varphi\right)\psi\geqslant0.
\end{equation*}
Since~$\psi\geqslant0$ is arbitrary, this implies the following pointwise inequality:
\begin{equation*}
  0\leqslant\overline{\calA}\nabla\varphi\cdot\nabla\varphi-2\overline{\calA}\nabla\varphi\cdot\nabla\varphi+\overline{\calB}\nabla\varphi\cdot\nabla\varphi=\left(\overline{\calB}-\overline{\calA}\right)\nabla\varphi\cdot\nabla\varphi.
\end{equation*}
Since~$\varphi$ is an arbitrary smooth function, this shows that~$\overline{\calA}\leqslant\overline{\calB}$.

We finally show the third part of~\Cref{lem:Hcvg-constant}, which is an adaptation of~\cite[Lemma~1.3.14]{allaire_homogeneisation}. Using the coercivity and symmetry of~$\calA^{k}$, it holds
\begin{equation*}
\calA^{k}\nabla \widetilde{w}^{k}\cdot\nabla \widetilde{w}^{k}-2\calA^{k}\nabla \widetilde{w}^{k}\cdot\nabla\varphi+\calA^{k}\nabla\varphi\cdot\nabla\varphi = \calA^{k}\left(\nabla \widetilde{w}^{k}-\nabla\varphi\right)\cdot\left(\nabla \widetilde{w}^{k}-\nabla\varphi\right)\geqslant0.
\end{equation*}
Passing to the limit in the previous inequality using similar arguments as above, one obtains
\begin{equation*}
  \overline{\calA}\nabla\varphi\cdot\nabla\varphi-2\overline{\calA}\nabla\varphi\cdot\nabla\varphi+\calA_{+}\nabla\varphi\cdot\nabla\varphi\geqslant0,
\end{equation*}
which implies that~$\calA_{+}\geqslant\overline{\calA}$. Likewise, denoting by~$\sigma=\calA_{-}\nabla\varphi$,
\begin{equation*}
\calA^{k}\nabla \widetilde{w}^{k}\cdot\nabla \widetilde{w}^{k}-2\nabla \widetilde{w}^{k}\cdot\sigma+\left(\calA^{k}\right)^{-1}\sigma\cdot\sigma=   \left(\calA^{k}\right)^{-1}\left(
  \calA^{k}\nabla \widetilde{w}^{k}-\sigma
  \right)\cdot\left(
  \calA^{k}\nabla \widetilde{w}^{k}-\sigma
  \right)
  \geqslant0.
\end{equation*}
Passing to the limit~$k\to+\infty$ yields
\begin{equation*}
  \overline{\calA}\nabla\varphi\cdot\nabla\varphi-2\nabla\varphi\cdot\calA_{-}\nabla\varphi+\nabla\varphi\cdot\calA_{-}\nabla\varphi\geqslant0.
\end{equation*}
This shows that~$\overline{\calA}\geqslant\calA_{-}$, which concludes the proof.


\subsubsection{A technical result of periodic averaging}
\label{sec:lemma_avg}

We conclude this section by stating and proving the following technical result, obtained by a simple adaptation of~\cite[Lemma~1.3.19]{allaire_homogeneisation}.

\begin{lemma}
  \label{lem:homog_lim}
  Consider a sequence of functions~$(a_k)_{k \geq 1} \subset L^2(\T^\dim)$ such that
  \begin{equation}
    \label{eq:averaging_s}
    \int_{\T^\dim} a_k \xrightarrow[k \to +\infty]{} \varrho \in \mathbb{R},
    \qquad
    \sup_{k \geq 1} \int_{\T^\dim} |a_k| < +\infty.
  \end{equation}
  Then the sequence~$(a_k(k\cdot))_{k \geq 1}$ weakly converges in~$L^2(\T^\dim)$ to the constant function~$\varrho$.
\end{lemma}

\begin{proof}
  The proof is a simple extension of~\cite[Lemma~1.3.19]{allaire_homogeneisation}. Consider~$\phi\in \calC^0(\T^\dim)$. We decompose~$\T^\dim$ as the union of~$k^\dim$ cubes~$q_i^k + (\T/k)^\dim$ for~$1 \leq i \leq k^\dim$, so that
  \begin{equation*}
    \int_{\T^\dim} a_k(kq) \phi(q) \, dq = \sum_{i=1}^{k^\dim} \int_{(\T/k)^\dim} a_k\left(k(q_i^k+Q)\right) \phi\left(q_i^k+Q\right) dQ.
  \end{equation*}
  For a given~$1 \leq i \leq k^\dim$, by the periodicity of~$a_k$,
  \begin{equation*}
    \begin{aligned}
       & \left|\int_{(\T/k)^\dim} a_k(k(q_i^k+Q)) \phi(q_i^k+Q) \, dQ - \frac{1}{k^\dim} \phi(q_i^k) \int_{\T^\dim} a_k \right|     \\
       & \qquad\qquad\qquad\leq \int_{(\T/k)^\dim} \left|a_k(k(q_i^k+Q))\right| \, \left|\phi(q_i^k+Q)-\phi(q_i^k)\right| dQ        \\
       & \qquad\qquad\qquad \leq \frac{1}{k^\dim} \left(\int_{\T^\dim} |a_k|\right)\max_{|q-q'|_\infty \leq 1/k}|\phi(q')-\phi(q)|.
    \end{aligned}
  \end{equation*}
  Therefore,
  \begin{equation*}
    \left| \int_{\T^\dim} a_k(kq) \phi(q) \, dq - \sum_{i=1}^{k^\dim} \frac{1}{k^\dim} \phi(q_i^k) \int_{\T^\dim} a_k \right| \leq \left(\int_{\T^\dim} |a_k|\right)\max_{|q-q'|_\infty \leq 1/k}|\phi(q')-\phi(q)|,
  \end{equation*}
  the right-hand side of the above inequality going to~0 as~$k \to +\infty$ since~$\phi$ is continuous on the compact set~$\T^\dim$, hence uniformly continuous. Moreover,
  \begin{equation*}
    \sum_{i=1}^{k^\dim} \frac{1}{k^\dim} \phi(q_i^k) \int_{\T^\dim} a_k \xrightarrow[k \to +\infty]{} \varrho \int_{\T^\dim} \phi
  \end{equation*}
  by results on the convergence of Riemann sums. This shows finally that
  \begin{equation*}
    \int_{T^\dim} a_k(kq) \phi(q) \, dq \xrightarrow[k \to +\infty]{} \varrho \int_{\T^\dim} \phi.
  \end{equation*}
  The claimed result follows by the density of~$\calC^0(\T^\dim)$ in~$L^2(\T^\dim)$.
\end{proof}

\section{Proof of~\Cref{thm:Donsker}}
\label{app:thm:Donsker}

The proof consists of three main steps. For the ease of the presentation, we consider the specific case of processes associated with a sequence of time steps~$(\Delta t_n)_{n \geq 1}$, which we further assume to be~$\Delta t_n = 1/n$ in all this section, the proof in the general case following by straightforward modifications.

We first show in~\Cref{prop:tightness} that the sequence~$\{P_{1/n}\}_{n\in\mathbb{N}^{*}}$ is tight in the family of probability measures on $\calC^0([0,+\infty), \T^\dim)$. By Prokhorov's theorem,~$\{P_{1/n}\}_{n\in\mathbb{N}^{*}}$ then admits a weakly convergent subsequence. We next show that the weak limit solves the martingale problem for the limiting diffusion process. Finally, the uniqueness of the martingale problem implies that the weak limit is the law of the solution of the SDE defined in~\eqref{eq:RWMH-variance}.

Let us now state more precisely the above results, starting with the tightness of the sequence~$\{P_{1/n}\}_{n\in\N^{*}}$ (proved in~\Cref{sec:proof_prop:tightness}, with some technical results postponed to~\Cref{sec:technical_lemmas_pathwise_cv}).

\begin{proposition}[Tightness]
  \label{prop:tightness}
  The sequence~$\{P_{1/n}\}_{n\in\N^{*}}$ is tight in the family of probability measures on $\calC^0([0,+\infty), \T^\dim)$.
\end{proposition}
~\Cref{prop:tightness} and Prokhorov's theorem guarantee that~$\{P_{1/n}\}_{n\in\N^{*}}$ admits a weakly converging subsequence~$P_\star$.  To conclude the proof, we identify~$P_\star$ as the law of the solution to the SDE~\eqref{eq:dynamics_donsker}, with generator~\eqref{eq:generator_cLD}, by considering the martingale problem.

\begin{proposition}[Limit martingale problem]
  \label{prop:limit-martingale}
  Any weak limit $P_\star$ of the sequence $\{P_{1/n}\}_{n\in\N^{*}}$ is a solution to the martingale formulation of the stochastic differential equation~\eqref{eq:dynamics_donsker}. More precisely,
for any $\mathcal C^\infty$ function $\varphi: \T^\dim \to \R$, let us define the process~$(M_t)_{t \ge 0}$  by:
\begin{equation*}
\forall t \ge 0, \qquad  M_t=  \varphi(Q^\star_t) - \varphi(q_0) - \int_{0}^t (\cLD\varphi)(Q^\star_s) \, ds
  \end{equation*}
  where~$(Q^\star_t)_{t\geq 0}$ is the canonical process on $\calC^0([0,+\infty), \T^\dim)$.
  Then, $(M_t)_{t \ge 0}$ is a $P_\star$-martingale.
\end{proposition}

The proof of this result can be read in~\Cref{sec:proof_prop:limit-martingale} (with, again, some technical results postponed to~\Cref{sec:technical_lemmas_pathwise_cv}). The uniqueness of the solution to the martingale problem then implies that~$P_\star$ is the law of the process~\eqref{eq:dynamics_donsker}.

To prove the results and for the sake of conciseness, we use the following notation, specific to our context. Let $f:\T^d \times \R^d \times [0,1] \to \R$ be a measurable function. For any couple of independent random variables $(G,U)$ where~$G$ is a centered Gaussian vector with identity covariance and $U$ is uniformly distributed on~$[0,1]$, we say that 
$$f(q,G,U) = \mathrm{O}(n^{-a})$$
for some~$a \geq 0$ if there exists an integer~$k_f \in\N$ such that, for any~$b \in \mathbb{N}$, there is~$K_{b,f} \in \mathbb{R}_+$ for which, almost surely,
\begin{equation}
  \label{eq:bound_remainder_f_b}
  \forall q \in \T^\dim, \qquad \E\left( |f(q,G,U)|^b \, \middle| \, G \right) \leq \frac{K_{b,f}}{n^{ab}} \left(1+|G|^{k_f}\right)^b.
\end{equation}
In particular,~$f(q,G,U) = \mathrm{O}(n^{-a})$ implies that, for all~$b \in \mathbb{N}$, there exists a constant~$C_{f,b} > 0$ such that
\begin{equation*}
  \E\left(|f(q,G,U)|^b\right) \leq \frac{C_{f,b}}{n^{ab}}.
\end{equation*}

\subsection{Proof of~\Cref{prop:tightness}}
\label{sec:proof_prop:tightness}

Recall that, by definition, the sequence $\{P_{1/n}\}_{n\in\N^{*}}$ of probability measures on $\calC^0([0,+\infty), \T^\dim)$ is tight if, for any~$\varepsilon>0$, there exists a relatively compact set~$K\subseteq \calC^0([0,+\infty), \T^\dim)$ such that~$\sup_{n\in\N^{*}} P_{1/n}(K^c)\leq \varepsilon$ (see, e.g.,~\cite[Section~5]{billing}). The characterization of relatively compact subsets of~$\calC^0([0,+\infty), \T^\dim)$ by the Arzelà--Ascoli theorem (see, e.g. \cite[Theorem~7.2]{billing}) allows to derive conditions for the tightness of~$\{P_{1/n}\}_{n\in\N^{*}}$  (see \cite[Theorem~7.3]{billing}). We consider the following pair of conditions to get tightness:
\begin{enumerate}[(i)]
  \item For any~$\varepsilon>0$, there exists a compact set~$\mathscr{K}\subseteq \T^\dim$ such that
        \begin{equation*}
          \limsup_{n\in\N^{*}} P_{1/n}\left(Q_0^{1/n}\notin \mathscr{K}\right)\leq \varepsilon.
        \end{equation*}
  \item For any~$T\in\R_+$ and~$\varepsilon>0$,
        \begin{equation*}
          \lim_{\delta\to 0}\limsup_{n\in\N^{*}}P_{1/n}\left(\sup_{\substack{s,t\in[0,T]\\|s-t|\leq \delta}}\left|Q^{1/n}_s-Q^{1/n}_t\right|\geq \varepsilon\right)=0.
        \end{equation*}
\end{enumerate}
The first condition holds trivially since the initial condition~$Q_0^{1/n} = q_0$ is fixed. To prove the second condition, we start by observing that, for all~$n\in\N^{*}$,
\begin{equation*}
  \sup_{\substack{s,t\in[0,T]\\|s-t|\leq \delta}}\left|Q^{1/n}_s-Q^{1/n}_t\right| \leq 3\max_{0\leq k \leq \lfloor T/\delta \rfloor}\sup_{t\in[0,\delta]}\left|Q^{1/n}_{k\delta +t}-Q^{1/n}_{k\delta}\right|.
\end{equation*}
By~\Cref{lem:tightness-subres-1} in~\Cref{sec:technical_lemmas_pathwise_cv}, for any~$k\leq \lfloor T/\delta \rfloor$,
\begin{equation*}
  \begin{aligned}
     & \sup_{\substack{s,t\in[0,T] \\|s-t|\leq \delta}}\left|Q^{1/n}_s-Q^{1/n}_t\right| \\
     & \quad \leq
    3\max\left(\max_{0\leq k \leq \lfloor T/\delta \rfloor}\max_{\lfloor nk\delta \rfloor \leq \ell\leq \lceil n(k+1)\delta \rceil} \left|q_{1/n}^\ell - q_{1/n}^{\lfloor nk\delta \rfloor}\right|,\max_{0\leq k \leq \lfloor T/\delta \rfloor}\max_{\lceil nk\delta \rceil \leq \ell\leq \lceil n(k+1)\delta \rceil} \left|q_{1/n}^\ell - q_{1/n}^{\lceil nk\delta \rceil}\right|\right).
  \end{aligned}
\end{equation*}
Thus,
\begin{equation}
  \label{eq:decomposition_a_b_tightness}
  \begin{aligned}
    \P \left( \sup_{\substack{s,t\in[0,T]                                                                                                                                                                                                                    \\|s-t|\leq \delta}}\left|Q^{1/n}_s-Q^{1/n}_t\right|\geq \varepsilon\right)
     & \leq \underbrace{\P \left(\max_{0\leq k \leq \lfloor T/\delta \rfloor}\max_{\lfloor nk\delta \rfloor \leq \ell\leq \lceil n(k+1)\delta \rceil} \left|q_{1/n}^\ell - q_{1/n}^{\lfloor nk\delta \rfloor}\right| \geq\frac{\varepsilon}{3}\right)}_{(a)} \\
     & \ \ +  \underbrace{\P \left(\max_{0\leq k \leq \lfloor T/\delta \rfloor}\max_{\lceil nk\delta \rceil \leq \ell\leq \lceil n(k+1)\delta \rceil} \left|q_{1/n}^\ell - q_{1/n}^{\lceil nk\delta \rceil}\right| \geq\frac{\varepsilon}{3}\right)}_{(b)}.
  \end{aligned}
\end{equation}
The two terms~$(a)$ and~$(b)$ are controlled using similar arguments; for simplicity we make precise only the control of~$(a)$.  We start by writing, for~$\ell \geq \lfloor nk\delta \rfloor$,
\begin{equation*}
  q_{1/n}^\ell - q_{1/n}^{\lfloor nk\delta \rfloor} =\sum_{i = \lfloor nk\delta \rfloor}^{\ell - 1}\Delta_n(q^i_{1/n}, G^{i+1}, U^{i+1}) ,
\end{equation*}
where (in view of~\eqref{eq:RWMH-variance} and recalling the notation~$R_{\Delta t}$ from~\eqref{eq:def_R_dt})
\begin{equation}
  \label{eq:increment}
  \Delta_n(q, G, U) = \sqrt{\frac{2}{n}}\Diff(q)^{1/2}G\mathbbm{1}_{\{U\leq R_{1/n}(q,G)\}}.
\end{equation}
By~\Cref{lem:martingale-part} in~\Cref{sec:technical_lemmas_pathwise_cv}, the following decomposition holds:
\begin{equation}
  \label{eq:decomposition_q}
  q_{1/n}^\ell - q_{1/n}^{\lfloor nk\delta \rfloor} =\sum_{i = \lfloor nk\delta \rfloor}^{\ell - 1}M_n(q^i_{1/n}, G^{i+1}, U^{i+1})+\sum_{i = \lfloor nk\delta \rfloor}^{\ell - 1}T_n(q^i_{1/n}),
\end{equation}
with~$(M_n(q^i_{1/n}, G^{i+1}, U^{i+1}))_{i\geq 0}$ a sequence of martingale increments satisfying~$M_n(q, G, U) = \mathrm{O}(n^{-1/2})$ and~$T_n(q) = \mathrm{O}(n^{-1})$. Explicit expressions for~$M_n(q^i_{1/n}, G^{i+1}, U^{i+1})$ and~$T_n(q^i_{1/n})$ are given in~\Cref{lem:martingale-part}. From~\eqref{eq:decomposition_q}, we obtain
\begin{equation*}
  \begin{aligned}
     & \max_{0\leq k \leq \lfloor T/\delta \rfloor}\max_{\lfloor nk\delta \rfloor \leq \ell\leq \lceil n(k+1)\delta \rceil} \left|q_{1/n}^\ell - q_{1/n}^{\lfloor nk\delta \rfloor} \right|                                               \\
     & \qquad\qquad \leq \max_{0\leq k \leq \lfloor T/\delta \rfloor}\max_{\lfloor nk\delta \rfloor \leq \ell\leq \lceil n(k+1)\delta \rceil}\left|\sum_{i = \lfloor nk\delta \rfloor}^{\ell - 1}M_n(q^i_{1/n}, G^{i+1}, U^{i+1}) \right| \\
     & \qquad\qquad \ \ + \max_{0\leq k \leq \lfloor T/\delta \rfloor}\max_{\lfloor nk\delta \rfloor \leq \ell\leq \lceil n(k+1)\delta \rceil}\left |\sum_{i = \lfloor nk\delta \rfloor}^{\ell - 1}T_n(q^i_{1/n}) \right|,
  \end{aligned}
\end{equation*}
so that 
\begin{equation*}
  \begin{aligned}
     & \P \left( \max_{0\leq k \leq \lfloor T/\delta \rfloor}\max_{\lfloor nk\delta \rfloor \leq \ell\leq \lceil n(k+1)\delta \rceil} \left|q_{1/n}^\ell - q_{1/n}^{\lfloor nk\delta \rfloor}\right|\geq \frac{\varepsilon}{3} \right)                                                                \\
     & \qquad\qquad \leq  \underbrace{\P \left(\max_{0\leq k \leq \lfloor T/\delta \rfloor}\max_{\lfloor nk\delta \rfloor \leq \ell\leq \lceil n(k+1)\delta \rceil}\left|\sum_{i = \lfloor nk\delta \rfloor}^{\ell - 1}M_n(q^i_{1/n}, G^{i+1}, U^{i+1}) \right| \geq\frac{\varepsilon}{6}\right)}_{I} \\
     & \qquad\qquad \ \ + \underbrace{\P \left( \max_{0\leq k \leq \lfloor T/\delta \rfloor}\max_{\lfloor nk\delta \rfloor \leq \ell\leq \lceil n(k+1)\delta \rceil}\left |\sum_{i = \lfloor nk\delta \rfloor}^{\ell - 1}T_n(q^i_{1/n}) \right| \geq\frac{\varepsilon}{6}\right)}_{II}.
  \end{aligned}
\end{equation*}
We next bound each of the terms on the right-hand side of the above inequality. We use~$C$ to denote a constant which can change from line to line.

\paragraph{Control of~$I$.} The term~$I$ is a sum of martingale increments, and can therefore be bounded using martingale inequalities. Consider to this end the sequence~$(S_k^{j})_{\ell\geq 0}$ for~$0\leq k \leq \left\lfloor T/\delta\right\rfloor$, defined as~$S_k^{0} = 0$ and
\begin{equation*}
  \forall 1\leq j\leq \lceil n(k+1)\delta \rceil-\lfloor nk\delta \rfloor, \qquad S_k^{j}=\sum_{i = \lfloor nk\delta \rfloor}^{\lfloor nk\delta \rfloor+j-1 }M_n(q^i_{1/n}, G^{i+1}, U^{i+1}).
\end{equation*}
For~$j>\lceil n(k+1)\delta\rceil-\lfloor nk\delta \rfloor$, we set the increments to~$0$, which yields~$S_k^{j}=S_k^{\lceil n(k+1)\delta\rceil-\lfloor nk\delta \rfloor}$. The sequence~$(S_k^{j})_{j\geq 1}$ is a martingale. We start by applying the union bound and Markov's inequality, which yields
\begin{equation}
  \label{eq:proba-martingale}
  \begin{aligned}
     & \P\left(\max_{0\leq k \leq \lfloor T/\delta \rfloor}\max_{j\geq 1}\left |S_k^{j}\right|  \geq \frac{\varepsilon}{6}\right) = \P\left(\max_{0\leq k \leq \lfloor T/\delta \rfloor}\max_{j \geq 1}\left |S_k^{j}\right|^4  \geq \frac{\varepsilon^4}{6^4}\right) \\
     & \qquad\qquad\qquad\qquad \leq \sum_{k=0}^{\lfloor T/\delta\rfloor} \P\left(\max_{j \geq 1}|S_k^{j}|^4  \geq \frac{\varepsilon^4}{6^4}\right)\leq \frac{6^4}{\varepsilon^4}\sum_{k=0}^{\lfloor T/\delta\rfloor} \E\left(\max_{j \geq 1}|S_k^{j}|^4 \right).
  \end{aligned}
\end{equation}
We are then in position to resort to the martingale inequality from~\cite[Theorem 1.1]{burkholder1972integral}, recalled here for convenience.

\begin{theorem}[Theorem~1.1 in~\cite{burkholder1972integral}]
  \label{thm:bounds-moments-martingale}
  Suppose that~$X=(X_1,X_2,\ldots)$ is a martingale, and introduce (with the convention $X_0=0$)
  \begin{equation*}
    X^* = \sup_{n \geq 1} |X_n|,
    \qquad
    S(X) = \left(\sum_{k=1}^{+\infty} (X_k-X_{k-1})^2\right)^{1/2}.
  \end{equation*}
  Consider a convex function~$\Phi : [0,+\infty) \to [0,+\infty)$ such that~$\Phi(0)=0$, satisfying the following growth condition: There exists~$c \in \mathbb{R}_+$ such that
  \begin{equation*}
    \forall u > 0, \qquad \Phi(2u)\leq c\Phi(u).
  \end{equation*}
  Set~$\Phi(\infty) = \lim_{u\to +\infty} \Phi(u)$. Then there exists~$C_-,C_+ \in \mathbb{R}_+$ such that
  \begin{equation*}
    C_- \E\left[ \Phi(S(X))\right]\leq \E\left[\Phi(X^*)\right] \leq C_+ \E\left[\Phi(S(X))\right].
  \end{equation*}
\end{theorem}

Applying~\Cref{thm:bounds-moments-martingale} to the sequence~$(S_k^j)_{j \geq 0}$ with~$\Phi(u) = u^4$, and using a discrete Cauchy--Schwarz inequality, it follows that $  \forall 0\leq k \leq \lfloor T/\delta \rfloor$,
\begin{equation*}
  \begin{aligned}
    \E\left(\max_{j\geq 1}|S_k^j|^4 \right) & \leq C_+ \E\left[ \left(\sum_{i=\lfloor nk\delta \rfloor}^{\lceil n(k+1)\delta \rceil-1}\left|M_n\left(q_{1/n}^{i}, G^{i+1},U^{i+1}\right)\right|^2\right)^2\right] \\
    & \leq C_+(n\delta+1) \sum_{i=\lfloor nk\delta \rfloor}^{\lceil n(k+1)\delta \rceil-1} \E\left[\left|M_n\left(q_{1/n}^i, G^{i+1},U^{i+1}\right)\right|^4\right].
  \end{aligned}
\end{equation*}
Since~$M_n\left(q_{1/n}^{i}, G^{i+1},U^{i+1}\right) = \mathrm{O}(n^{-1/2})$ with bounds uniform in~$i$, it follows that
\begin{equation*}
  \forall 0\leq k \leq \lfloor T/\delta \rfloor,
  \qquad
  \E\left(\max_{1 \leq j\leq \lceil n\delta \rceil}|S_k^j|^4 \right)\leq C \frac{(n\delta + 1)^2}{n^2}.
\end{equation*}
Combining the latter inequality with~\eqref{eq:proba-martingale}, we finally obtain that
\begin{equation*}
  \P\left(\max_{0\leq k \leq \lfloor T/\delta \rfloor}\max_{1 \leq j\leq \lceil n\delta \rceil}\left |S_\ell\right|  \geq \frac{\varepsilon}{6}\right) \leq C \left(\frac{T}{\delta}+1\right)\frac{(n\delta + 1)^2}{n^2 \varepsilon^4}.
\end{equation*}
The control of~$I$ is concluded by noticing that~$\dps \lim_{\delta\to 0}\limsup_{n\to+\infty} \left(T/\delta+1\right)(n\delta + 1)^2 n^{-2} = 0$.

\paragraph{Control of~$II$.} Using Markov's inequality, then a Cauchy--Schwarz inequality,
\begin{align}
   & \P \left( \max_{0\leq k \leq \lfloor T/\delta \rfloor}\max_{\lfloor nk\delta \rfloor \leq \ell\leq \lceil n(k+1)\delta \rceil}\left |\sum_{i = \lfloor nk\delta \rfloor}^{\ell - 1}T_n(q^i_{1/n}) \right| \geq\frac{\varepsilon}{6}\right)\nonumber                                           \\
   & \qquad\qquad\qquad\qquad \leq \frac{36}{\varepsilon^2}\E\left[\max_{0\leq k \leq \lfloor T/\delta \rfloor}\max_{\lfloor nk\delta \rfloor \leq \ell\leq \lceil n(k+1)\delta \rceil}\left |\sum_{i = \lfloor nk\delta \rfloor}^{\ell - 1}T_n(q^i_{1/n}) \right|^2 \right]\nonumber              \\
   & \qquad\qquad\qquad\qquad \leq \frac{36 (n \delta+1)}{\varepsilon^2}\E\left[\max_{0\leq k \leq \lfloor T/\delta \rfloor}\max_{\lfloor nk\delta \rfloor \leq \ell\leq \lceil n(k+1)\delta \rceil} \sum_{i = \lfloor nk\delta \rfloor}^{\ell - 1}\left|T_n(q^i_{1/n}) \right|^2 \right]\nonumber \\
   & \qquad\qquad\qquad\qquad \leq\frac{36 (n \delta+1)}{\varepsilon^2}\E\left[\max_{0\leq k \leq \lfloor T/\delta \rfloor}\sum_{i = \lfloor nk\delta \rfloor}^{\lceil n(k+1)\delta\rceil- 1}\left|T_n(q^i_{1/n}) \right|^2 \right]\nonumber                                                       \\
   & \qquad\qquad\qquad\qquad \leq \frac{36 (n \delta+1)}{\varepsilon^2}\sum_{k=0}^{\lfloor T/\delta \rfloor}\sum_{i = \lfloor nk\delta \rfloor}^{\lceil n(k+1)\delta\rceil- 1}\E\left[\left|T_n(q^i_{1/n}) \right|^2 \right]. \label{eq:tight-F}
\end{align}
Yet, by~\Cref{lem:martingale-part},~$T_n(q^i_{1/n}) = \mathrm{O}(n^{-1})$ uniformly in~$i$, so that
\begin{equation*}
  \sum_{i = \lfloor nk\delta \rfloor}^{\lceil n(k+1)\delta\rceil- 1}\E\left[\left|T_n(q^i_{1/n}) \right|^2 \right] \leq \frac{C(n\delta +1)}{n^2} .
\end{equation*}
Plugging the last inequality in~\eqref{eq:tight-F}, we obtain that
\begin{equation*}
  \P \left( \max_{0\leq k \leq \lfloor T/\delta \rfloor}\max_{\lfloor nk\delta \rfloor \leq \ell\leq \lceil n(k+1)\delta \rceil}\left |\sum_{i = \lfloor nk\delta \rfloor}^{\ell - 1}T_n(q^i_{1/n}) \right| \geq\frac{\varepsilon}{6}\right) \leq C\left(\frac{T}{\delta}+1\right) \frac{(n\delta+1)^2}{n^{2} \varepsilon^{2}}.
\end{equation*}
The control of~$II$ is concluded by noticing that~$\dps \lim_{\delta\to 0}\limsup_{n\to\infty} C(T/\delta+1) (n\delta+1)^2 n^{-2} =0$.

\paragraph{Conclusion of the proof.} Combining the control of the two terms~$I$ and~$II$, we conclude that
\begin{equation*}
  \lim_{\delta\to 0}\limsup_{n\to +\infty}\P \left( \max_{0\leq k \leq \lfloor T/\delta \rfloor}\max_{\lfloor nk\delta \rfloor \leq \ell\leq \lceil n(k+1)\delta \rceil}\left|q_{1/n}^\ell - q_{1/n}^{\lfloor nk\delta \rfloor}\right|\geq \frac{\varepsilon}{3} \right) = 0.
\end{equation*}
It can similarly be proved that
\begin{equation*}
  \lim_{\delta\to 0}\limsup_{n\to +\infty} \P \left( \max_{0\leq k \leq \lfloor T/\delta \rfloor}\max_{\lceil nk\delta \rceil \leq \ell\leq \lceil n(k+1)\delta \rceil}\left|q_{1/n}^\ell - q_{1/n}^{\lceil nk\delta \rceil}\right|\geq \frac{\varepsilon}{3} \right) = 0.
\end{equation*}
In view of~\eqref{eq:decomposition_a_b_tightness}, we finally obtain that
\begin{equation*}
  \lim_{\delta\to 0}\limsup_{n\to +\infty} \P\left(\sup_{\substack{s,t\in[0,T]\\|s-t|\leq \delta}}  \left|Q^{1/n}_s-Q^{1/n}_t\right|\right) = 0,
\end{equation*}
which concludes the proof of~\Cref{prop:tightness}.

\subsection{Proof of~\Cref{prop:limit-martingale}}
\label{sec:proof_prop:limit-martingale}

For~$n\in\N^{*}$, let us denote by~$\mathcal{L}^{n}$ the operator defined by:
 for any bounded measurable function~$\phi : \T^\dim \to \mathbb{R}$,
\begin{equation}
  \label{eq:def_Ln_varphi}
  (\mathcal{L}^n \phi)(q) = \mathbb{E}\left[ \phi(q_{1/n}^{1}) \, \middle| \, q_{1/n}^0 = q\right] - \phi(q).
\end{equation}
For a given function~$\varphi \in \calC^3(\T^d)$, we start by showing in~\Cref{sec:convergence_generator_Ln} that the sequence of functions~$(n\mathcal{L}^{n}\varphi)_{n\geq 1}$ converges uniformly towards~$\cLD\varphi$ as~$n\to\infty$, in the sense of uniform convergence of continuous functions over the compact set~$\T^d$. We then write the martingale problem associated to~$\mathcal{L}^n$ in~\Cref{sec:limit_martingale_pbm}, and prove that the limit~$P_\infty$ solves the martingale problem for the limit generator~$\cLD$.

\subsubsection{Convergence of the sequence of functions~$(n\mathcal{L}^{n}\varphi)_{n\geq 1}$}
\label{sec:convergence_generator_Ln}

In order to determine the limit of the sequence of functions~$(n\mathcal{L}^{n}\varphi)_{n\geq 1}$, we work out a Taylor expansion of the function~\eqref{eq:def_Ln_varphi}. We expand to this end~$\varphi(q_{1/n}^{1})$ at second order around~$q$, seeing~$q^1_{1/n}-q$ as a perturbation along the direction~$\Diff(q)^{1/2}G^1$ with step size~$\sqrt{2/n}\mathbbm{1}_{\{U^1\leq R_{1/n}(q,G^1)\}} = \mathrm{O}(n^{-1/2})$. Writing~$G,U$ instead of~$G^1,U^1$ to alleviate the notation,
\begin{equation*}
  \begin{aligned}
    \varphi\left(q+\sqrt{\frac{2}{n}}\mathbbm{1}_{\{U\leq R_{1/n}(q,G)\}}\Diff(q)^{1/2}G\right) = \varphi(q)  + \sqrt{\frac{2}{n}}\mathbbm{1}_{\{U\leq R_{1/n}(q,G)\}}\nabla\varphi(q)^\top\Diff(q)^{1/2}G & \\
    + \frac{1}{n}\mathbbm{1}_{\{U\leq R_{1/n}(q,G)\}} G^\top \Diff(q)^{1/2} \nabla^2\varphi(q)\Diff(q)^{1/2}G + \mathrm{O}(n^{-3/2}),                                                                      &
  \end{aligned}
\end{equation*}
where the remainder (obtained for instance with a Taylor expansion with exact remainder) is uniformly controlled using derivatives of~$\varphi$ of order at most~3, by which we mean that the remainder satisfies an inequality such as~\eqref{eq:bound_remainder_f_b}, with a bound~$K_{b,f} \leq c_{b,f} \|\varphi\|_{\calC^3(\T^\dim)}$. Subtracting~$\varphi(q)$ and taking the expectation with respect to~$U$,
\begin{equation*}
  \begin{aligned}
    \E\left[\varphi\left(q+\sqrt{\frac{2}{n}}\mathbbm{1}_{\{U\leq R_{1/n}(q,G)\}}\Diff(q)^{1/2}G\right) \middle| G \right]-\varphi(q) = \sqrt{\frac{2}{n}} R_{1/n}(q,G)\nabla\varphi(q)^\top\Diff(q)^{1/2}G & \\
    + \frac{1}{n}R_{1/n}(q,G)G^\top \Diff(q)^{1/2} \nabla^2\varphi(q)\Diff(q)^{1/2}G + \mathrm{O}(n^{-3/2}).                                                                                      &
  \end{aligned}
\end{equation*}
In view of the expansion of the acceptance rate provided by~\Cref{lem:acceptance-rate-DL} (with the analytical expression~\eqref{eq:zeta-def} for~$\zeta(q,G)$), and further taking the expectation with respect to~$G$, we obtain
\begin{equation*}
  (\mathcal{L}^n \varphi)(q) =
  -\frac{2}{n} \nabla\varphi(q)^\top\Diff(q)^{1/2}\E[G\max\left(0,\zeta(q,G)\right)] + \frac{1}{n} \Diff(q) : \nabla^2\varphi(q) + \mathrm{O}(n^{-3/2}),
\end{equation*}
with a remainder term which is uniformly bounded in~$q \in \T^\dim$.
A simple expression for the expectation
\begin{equation*}
  \E[G\max\left(0,\zeta(q,G)\right)] = \int_{\R^\dim}g\max\left(0,\zeta(q,g)\right)\frac{\mathrm{e}^{-g^2/2}}{(2\pi)^{\dim/2}} \, dg
\end{equation*}
can be derived as in~\cite{FS17} using the anti-symmetry of the function~$\zeta$. Indeed, for all~$q\in\T^\dim$ and for all~$g\in\R^\dim$, a direct inspection of~\eqref{eq:zeta-def} reveals that~$\zeta(q,g) = -\zeta(q,-g)$, so that~$g\zeta(q,g) = -g\zeta(q,-g)$. Therefore, using~\cite[Lemma~10]{FS17} to obtain the second line,
\begin{equation*}
  \begin{aligned}
    \E[G\max\left(0,\zeta(q,G)\right)] & = \int_{\{g\in\R^d \, | \, \zeta(q,g)\geq 0\}} g\zeta(q,g)\frac{\mathrm{e}^{-|g|^2/2}}{(2\pi)^{\dim/2}} \, dg =  \frac{1}{2}\int_{\R^\dim}g\zeta(q,g)\frac{\mathrm{e}^{-|g|^2/2}}{(2\pi)^{\dim/2}}\, dg \\
    & = \frac12 \left[ \Diff(q)^{1/2}\nabla V(q) -\Diff(q)^{-1/2}\div(\Diff)(q) \right].
  \end{aligned}
\end{equation*}
We conclude that
\begin{equation*}
  2 \nabla\varphi(q)^\top\Diff(q)^{1/2}\E[G\min\left(0,\zeta(q,G)\right)] =  \left(\Diff(q)\nabla V(q) - \div(\Diff)(q)\right)^\top \nabla\varphi(q).
\end{equation*}
This ensures finally that~$n(\mathcal{L}^n\varphi)(q) \to (\cLD\varphi)(q)$ uniformly in~$q \in \T^\dim$ as~$n\to +\infty$.

\subsubsection{Limit martingale problem}
\label{sec:limit_martingale_pbm}

The argument is standard, and we refer for example to~\cite[Section 11.2]{StroockVaradhan} and~\cite[Theorem 5.8]{CometsMeyre} for details. Let us however outline the reasoning for the sake of completeness.

For any~$n\geq 1$ and any test function~$\varphi\in \calC^3(\T^\dim)$, the definition of~$\mathcal{L}^n$ ensures that
\begin{equation*}
  \varphi(q^{k}_{1/n}) - \varphi(q^0_{1/n}) - \frac1n \sum_{i=1}^{k-1} (n\mathcal{L}^n\varphi)(q^{i}_{1/n})
\end{equation*}
is a martingale with respect to~$P_{1/n}$. Now, using the definition~\eqref{eq:process},
\begin{equation*}
  \begin{aligned}
    \varphi(q^{k}_{1/n}) & - \varphi(q^0_{1/n}) - \frac1n \sum_{i=1}^{k-1} (n\mathcal{L}^n\varphi)(q^{i}_{1/n})\\
    & = \varphi\left(Q^{1/n}_{k/n}\right)-\varphi\left(Q^{1/n}_{0}\right) - \frac1n \sum_{i=1}^{k-1} (n\mathcal{L}^n\varphi)\left(Q^{1/n}_{i/n}\right)\\
    & = \varphi\left(Q^{1/n}_{k/n}\right)-\varphi\left(Q^{1/n}_{0}\right) - \frac1n \sum_{i=1}^{k-1} (\mathcal{L}_\Diff \varphi)\left(Q^{1/n}_{i/n}\right) + \frac1n \sum_{i=1}^{k-1} (n\mathcal{L}^n\varphi - \mathcal{L}_\Diff \varphi)\left(Q^{1/n}_{i/n}\right).
  \end{aligned}
\end{equation*}
The last term converges to~0 as~$n \to +\infty$ in view of the uniform convergence~$n(\mathcal{L}^n\varphi)(q) \to (\cLD\varphi)(q)$. The first sum can be rewritten as
\begin{equation*}
  \begin{aligned}
    \frac1n \sum_{i=1}^{k-1} (\mathcal{L}_\Diff \varphi)\left(Q^{1/n}_{i/n}\right) & = \sum_{i=1}^{k-1} \int_{i/n}^{(i+1)/n} (\mathcal{L}_\Diff \varphi)\left(Q^{1/n}_{i/n}\right) dt \\
    & = \int_{1/n}^{k/n} (\mathcal{L}_\Diff \varphi)\left(Q^{1/n}_t\right) dt + \mathrm{O}\left( \frac 1 n\right).
  \end{aligned}
\end{equation*}
For~$t > 0$ fixed, we conclude by letting~$n\to +\infty$ with~$k = \lfloor nt \rfloor$ that
\begin{equation*}
  \varphi(Q_t^\star) - \varphi(Q_0^\star) - \int_{0}^t(\mathcal{L}_\mathcal{D}\varphi)(Q_s^\star)
\end{equation*}
is a martingale process for any weak limit~$P_\ast$ of~$\{P_{1/n}\}_{n \geq 1}$, which concludes the proof.

\subsection{Technical results}
\label{sec:technical_lemmas_pathwise_cv}

We gather in this section some technical results used in~\Cref{sec:proof_prop:tightness,sec:proof_prop:limit-martingale}.

\begin{lemma}
  \label{lem:tightness-subres-1}
  For any~$\delta>0$, $n \in \N$ and~$k\in\N$, it holds
  \begin{equation*}
    \begin{aligned}
       & \sup_{t\in[0,\delta]} \left|Q^{1/n}_{k\delta +t}-Q^{1/n}_{k\delta}\right|\\
       & \qquad \leq \max\left\{\max_{\lfloor nk\delta \rfloor \leq \ell\leq \lceil n(k+1)\delta \rceil} \left|q_{1/n}^\ell - q_{1/n}^{\lfloor nk\delta \rfloor}\right|,\max_{\lceil nk\delta \rceil \leq \ell\leq \lceil n(k+1)\delta \rceil} \left|q_{1/n}^\ell - q_{1/n}^{\lceil nk\delta \rceil}\right|\right\}.
    \end{aligned}
  \end{equation*}
\end{lemma}

\begin{proof}
  Note first that, since~$Q^{1/n}_{k\delta}$ linearly interpolates between~$q_{1/n}^{\lfloor nk\delta \rfloor}$ and~$q_{1/n}^{\lceil nk\delta \rceil}$,
  \begin{equation}
    \label{eq:sup_Q_to_be_bounded}
    \sup_{t\in[0,\delta]} \left|Q^{1/n}_{k\delta +t}-Q^{1/n}_{k\delta}\right| \leq \max \left\{ \sup_{s\in[k\delta,(k+1)\delta]} \left|Q^{1/n}_{s}-q_{1/n}^{\lfloor nk\delta \rfloor}\right|, \sup_{s\in[k\delta,(k+1)\delta]} \left|Q^{1/n}_{s}-q_{1/n}^{\lceil nk\delta \rceil}\right|\right\}.
  \end{equation}
  We next decompose the supremums over~$s \in [k\delta,(k+1)\delta]$ into supremums over intervals of length~$1/n$, by noting that
  \begin{equation*}
    [k\delta,(k+1)\delta] \subset \bigcup_{\ell = \lfloor nk\delta \rfloor}^{\lceil n(k+1)\delta \rceil-1} \left[ \frac{\ell}{n},\frac{\ell+1}{n}\right].
  \end{equation*}
  For instance,
  \begin{equation*}
    \sup_{s\in[k\delta,(k+1)\delta]} \left|Q^{1/n}_{s}-q_{1/n}^{\lfloor nk\delta \rfloor}\right| \leq \max_{\lfloor nk\delta \rfloor \leq \ell\leq \lceil n(k+1)\delta \rceil-1} \sup_{\tau\in[0,1/n]} \left|Q^{1/n}_{\ell/n+\tau}-q_{1/n}^{\lfloor nk\delta \rfloor}\right|.
  \end{equation*}
  Since~$Q^{1/n}_{\ell/n+\tau}$ linearly interpolates between~$q_{1/n}^\ell$ and~$q_{1/n}^{\ell+1}$, it holds
  \begin{equation*}
    \sup_{\tau\in[0,1/n]} \left|Q^{1/n}_{\ell/n+\tau}-q_{1/n}^{\lfloor nk\delta \rfloor}\right| \leq \max \left\{ \left|q_{1/n}^\ell-q_{1/n}^{\lfloor nk\delta \rfloor}\right|, \left|q_{1/n}^{\ell+1}-q_{1/n}^{\lfloor nk\delta \rfloor}\right|\right\}.
  \end{equation*}
  The claimed conclusion then follows from a similar bound for the second maximum over~$s$ on the right-hand side of~\eqref{eq:sup_Q_to_be_bounded}.
\end{proof}

The following lemma provides an expansion of the acceptance rate in inverse powers of~$n$.

\begin{lemma}[Expansion of~$R_{1/n}(q, G)$]
  \label{lem:acceptance-rate-DL}
  For any~$q \in \mathbb{T}^\dim$ and~$G \in \mathbb{R}^\dim$, it holds
  \begin{equation*}
    R_{1/n}(q, G) = 1-\sqrt{\frac{2}{n}}\max(0,\zeta(q, G) ) + \mathrm{O}(n^{-1}),
  \end{equation*}
  with
  \begin{equation}
    \label{eq:zeta-def}
    \begin{aligned}
      \zeta(q, G) = &\nabla V(q)^{\top}\Diff(q)^{1/2}G  + \frac{1}{2}\Tr\left(\Diff(q)^{-1}\left[\textsf{D}\Diff(q)\cdot\Diff(q)^{1/2}G\right]\right)           \\
    & -\frac{1}{2}G^{\top}\Diff(q)^{-1/2}\left[\textsf{D}\Diff(q)\cdot\Diff(q)^{1/2}G\right]\Diff(q)^{-1/2}G,
    \end{aligned}
  \end{equation}
  where~$\textsf{D}\Diff(q)$ is the differential matrix of~$q\mapsto\Diff(q)$, \emph{i.e.}, for any~$h \in \R^\dim$, $\textsf{D}\Diff(q) \cdot h$  is the matrix with entries~$\nabla \Diff_{ij}(q)^\top h$.
\end{lemma}

\begin{proof}
  The computations are similar to the ones in~\cite[Section~4.7]{FS17}. 
  To control remainder terms, we use the fact that the functions at hand are sufficiently smooth, the diffusion matrix~$\Diff$ is positive and the domain~$\T^\dim$ is compact. The magnitude of the acceptance probability in the Metropolis algorithm reads
  \begin{equation*}
    R_{1/n}(q, G) = \min \left\{ 1, \exp\left[-\alpha\left(q, \Phi_{1/n}(q,G)\right)\right] \right\}, \qquad \Phi_{1/n}(q,G) = q + \sqrt{\frac{2}{n}}\Diff( q)^{1/2}G,
  \end{equation*}
  with
  \begin{equation*}
    \begin{aligned}
      \alpha(q,\widetilde q) & = \frac{1}{2}\left( \ln \det[\Diff(\widetilde q)] - \ln \det[\Diff(q)] \right) + \left[V(\widetilde q)-V(q)\right]                                    \\
                             & \qquad \qquad + \frac{ n}{4}\left[ \left|\Diff(\widetilde q)^{-1/2}(q-\widetilde q)\right|^2 - \left|\Diff(q)^{-1/2}(\widetilde q-q)\right|^2\right].
    \end{aligned}
  \end{equation*}
  First,
  \begin{equation}
    \label{eq:grad-term}
    V(\Phi_{1/n}(q,G)) - V(q) = \sqrt{\frac{2}{n}}\nabla V(q)^{\top}\Diff( q)^{1/2}G + \mathrm{O}(n^{-1}).
  \end{equation}
  We next consider the terms involving~$\det(\Diff)$. Since~$\Diff$ is $\calC^2$ on $\T^\dim$, 
  \begin{equation*}
    \Diff(\Phi_{1/n}(q,G)) =  \Diff(q) + \sqrt{\frac{2}{n}}\textsf{D}\Diff(q)\cdot \Diff( q)^{1/2}G + \mathrm{O}(n^{-1}).
  \end{equation*}
It then holds, using the fact that $\Diff$ is uniformly bounded from below by a positive constant,
  \begin{equation*}
    \begin{aligned}
      \det[\Diff( \Phi_{1/n}(q,G))] & =  \det[\Diff( q)]\det( \Id +  \sqrt{\frac{2}{n}}\Diff( q)^{-1}\textsf{D}\Diff(q)\cdot \Diff(q)^{1/2}G + \mathrm{O}(n^{-1}))                        \\
    & =  \det[\Diff( q)]\left(1+  \sqrt{\frac{2}{n}}\Tr\left(\Diff( q)^{-1}[\textsf{D}\Diff(q)\cdot \Diff( q)^{1/2}G]\right) + \mathrm{O}(n^{-1})\right),
    \end{aligned}
  \end{equation*}
  so that
  \begin{equation}
    \label{eq:det-term}
    \ln \det[\Diff(\Phi_{1/n}(q,G))] - \ln \det[\Diff( q)] = \sqrt{\frac{2}{n}}\Tr\left(\Diff( q)^{-1} \left[\textsf{D}\Diff(q)\cdot \Diff( q)^{1/2}G\right]\right) + \mathrm{O}(n^{-1}).
  \end{equation}

 We finally turn to the term~$|\Diff(\Phi_{1/n}(q,G))^{-1/2}(q-\Phi_{1/n}(q,G))|^2 -|\Diff(q)^{-1/2}(\Phi_{1/n}(q,G)-q)|^2$. We use the expansion~$\Diff(q + x)^{-1} = \Diff(q)^{-1} - \Diff(q)^{-1}\left[\textsf{D}\Diff(q)\cdot x\right] \Diff(q)^{-1} + \mathrm{O}(|x|^2)$ (using again the uniform positive lower bound on $\Diff$) to write
  \begin{equation*}
    \begin{aligned}
       & \left|\Diff(\Phi_{1/n}(q,G))^{-1/2}(q-\Phi_{1/n}(q,G))\right|^2 - \left|\Diff(q)^{-1/2}(\Phi_{1/n}(q,G)-q)\right|^2                                                            \\
       & \qquad \qquad \qquad \qquad = (\Phi_{1/n}(q,G)-q)^{\top} \left[\Diff(\Phi_{1/n}(q,G))^{-1} -  \Diff(q)^{-1}\right](\Phi_{1/n}(q,G)-q)                                          \\
       & \qquad \qquad \qquad \qquad = -\left(\frac{2}{n}\right)^{3/2}G^{\top}\Diff(q)^{-1/2}\left[\textsf{D}\Diff(q)\cdot \Diff(q)^{1/2}G\right]\Diff(q)^{-1/2}G + \mathrm{O}(n^{-2}).
    \end{aligned}
  \end{equation*}
  Combining~\eqref{eq:grad-term} and~\eqref{eq:det-term} with the latter equality leads to~$\alpha(q,\Phi_{1/n}(q,G)) = (2/n)^{1/2} \zeta(q, G) + \mathrm{O}(n^{-1})$, with~$\zeta(q,G)$ given by~\eqref{eq:zeta-def}. The desired conclusion then follows from the inequality~$\max(x,0) - \max(x,0)^2/2 \leq 1-\min(1,\rme^{-x}) \leq \max(x,0)$, obtained by distinguishing the cases~$x \leq 0$ and~$x \geq 0$.
\end{proof}

The last result makes precise the martingale decomposition of the increments~$\Delta_n(q, G, U)$ defined in~\eqref{eq:increment}.

\begin{lemma}
  \label{lem:martingale-part}
  Fix~$q \in \T^\dim$. The increment~$\Delta_n(q, G, U)$ writes
  \begin{equation*}
    \Delta_n(q, G, U) =  M_n(q, G, U) + T_n(q),
  \end{equation*}
  where~$M_n(q, G, U) = \Delta_n(q, G, U)- \E\left[ \Delta_n(q, G, U) \right]$ and~$T_n(q)  =\E\left[ \Delta_n(q, G, U) \right]$, the latter expectations being with respect to realizations of the independent random variables~$G$ and~$U$, with~$G$ a~$\dim$-dimensional standard Gaussian random variable and~$U$ a uniform random variable on~$[0,1]$. Moreover,
  \begin{equation}
\label{eq:properties_martingale_decomposition}
    \E[M_n(q, G, U)] = 0,
    \qquad
    M_n(q, G, U) = \mathrm{O}(n^{-1/2}),
    \qquad
    T_n(q) = \mathrm{O}(n^{-1}),
  \end{equation}
  the last two estimates holding uniformly in~$q \in \T^\dim$.
\end{lemma}

\begin{proof}
  The first property of~\eqref{eq:properties_martingale_decomposition} holds by construction, so it suffices to prove the two other estimates. For the last one, we note that, by definition,
  \begin{equation*}
    T_n(q) = \sqrt{\frac{2}{n}}\Diff(q)^{1/2} \E\left[ G R_{1/n}(q,G) \right] = \sqrt{\frac{2}{n}}\Diff(q)^{1/2} \E\left[ G \left(R_{1/n}(q,G)-1\right) \right],
  \end{equation*}
  so that the conclusion follows from~\Cref{lem:acceptance-rate-DL} and the fact that~$\Diff$, its derivatives and its inverse, are uniformly bounded in~$q$. The other estimate is obtained by noting that
  \begin{equation*}
    M_n(q, G, U) = \sqrt{\frac{2}{n}} \Diff(q)^{1/2} \left( G\mathbbm{1}_{\{U\leq R_{1/n}(q,G)\}} - \E\left[ G R_{1/n}(q,G) \right]\right),
  \end{equation*}
  so that, in view of~\Cref{lem:acceptance-rate-DL},
  \begin{equation*}
    \forall u \in [0,1], \qquad |M_n(q, G, u)| \leq \frac{C}{\sqrt{n}} \normF{\Diff(q)}^{1/2} |G| + \mathrm{O}(n^{-1}).
  \end{equation*}
  This concludes the proof.
\end{proof}

\section{Derivation of the differential of~$\Lambda$}
\label{app:differential_lambda}

This derivation follows the proof of the Hellmann--Feynman theorem. It can be seen as differentiating the Rayleigh--Ritz quotient that defines~$\Lambda(\Diff^{\star}+t\delta\Diff)$, using the chain rule. We detail here the computations:
\begin{align*}
    \frac{d}{dt}\Lambda(\Diff^{\star}+t\delta\Diff)
    &=\frac{d}{dt}\left(
        \frac{\left\langle-\cL_{\Diff^{\star}+t\delta\Diff}u_{\Diff^{\star}+t\delta\Diff},u_{\Diff^{\star}+t\delta\Diff}\right\rangle_{L^{2}(\mu)}}{\left\lVert u_{\Diff^{\star}+t\delta\Diff}\right\rVert^{2}_{L^{2}(\mu)}}
    \right)\\
    &=\frac{d}{dt}
        \left\langle
            -\cL_{\Diff^{\star}+t\delta\Diff}u_{\Diff^{\star}+t\delta\Diff},u_{\Diff^{\star}+t\delta\Diff}        \right\rangle_{L^{2}(\mu)}\\
    &=\left\langle -\cL_{\delta\Diff} u_{\Diff^{\star}+t\delta\Diff}, u_{\Diff^{\star}+t\delta\Diff}\right\rangle_{L^{2}(\mu)}+2\left\langle -\cL_{\Diff^{\star}+t\delta\Diff}u_{\Diff^{\star}+t\delta\Diff},\frac{d}{dt} u_{\Diff^{\star}+t\delta\Diff}\right\rangle_{L^{2}(\mu)}\\
    &=\left\langle -\cL_{\delta\Diff} u_{\Diff^{\star}+t\delta\Diff}, u_{\Diff^{\star}+t\delta\Diff}\right\rangle_{L^{2}(\mu)}+2\Lambda(\Diff^{\star}+t\Diff)\left\langle u_{\Diff^{\star}+t\delta\Diff},\frac{d}{dt}u_{\Diff^{\star}+t\delta\Diff}\right\rangle_{L^{2}(\mu)}\\
    &=\left\langle -\cL_{\delta\Diff} u_{\Diff^{\star}+t\delta\Diff}, u_{\Diff^{\star}+t\delta\Diff}\right\rangle_{L^{2}(\mu)},
\end{align*}
where we used that~$u_{\Diff^{\star}+t\delta\Diff}$ is normalized in~$L^{2}(\mu)$ for the second equality, that~$\cL_\Diff$ is linear in~$\Diff$ and that it is a symmetric operator on~$L^{2}(\mu)$ for the third equality, that~$u_{\Diff^{\star}+t\delta\Diff}$ is an eigenvector of~$-\cL_{\Diff^{\star}+t\delta\Diff}$ associated with the eigenvalue~$\Lambda(\Diff^{\star}+t\delta\Diff)$ for the fourth equality and that~$\frac{d}{dt}\left\lVert u_{\Diff^{\star}+t\delta\Diff}\right\rVert^{2}_{L^{2}(\mu)}=0=2\left\langle u_{\Diff^{\star}+t\delta\Diff},\frac{d}{dt}u_{\Diff^{\star}+t\delta\Diff}\right\rangle_{L^{2}(\mu)}$ for the last equality. Evaluating this equality at~$t=0$ leads to the expression of the differential.

\paragraph{Acknowledgements.}
We thank Grégoire Allaire, Benjamin Bogosel, Andrew Duncan, Federico Ghimenti, Frédéric Legoll, Mathieu Lewin, Danny Perez, Samuel Power, Frédéric van Wijland Jonathan Weare and Olivier Zahm for stimulating discussions. The works of T.L., R.S. and G.S. benefit from fundings from the European Research Council (ERC) under the European Union's Horizon 2020 research and innovation programme (project EMC2, grant agreement No 810367), and from the Agence Nationale de la Recherche through the grants ANR-19-CE40-0010-01 (QuAMProcs) and ANR-21-CE40-0006 (SINEQ). G.P. is partially supported by an ERC-EPSRC Frontier Research Guarantee through Grant No. EP/X038645, ERC through Advanced Grant No. 247031. This project was initiated as T.L. was a visiting professor at Imperial College of London (ICL), with a visiting professorship grant from the Leverhulme Trust. The Department of Mathematics at ICL and the Leverhulme Trust are warmly thanked for their support.

\bibliographystyle{abbrv}
\bibliography{biblio.bib}

\end{document}